\documentclass{memo-l}

\usepackage{mathrsfs}
\usepackage{amsmath}
\usepackage{amssymb}
\usepackage{amsfonts}
\usepackage{graphicx}

\usepackage{color}

\newtheorem{theorem}{Theorem}[chapter]
\newtheorem{lemma}[theorem]{Lemma}
\newtheorem{proposition}[theorem]{Proposition}
\newtheorem{corollary}[theorem]{Corollary}

\theoremstyle{definition}
\newtheorem{definition}[theorem]{Definition}
\newtheorem{example}[theorem]{Example}

\theoremstyle{remark}
\newtheorem{remark}[theorem]{Remark}
\newtheorem*{assumption}{Assumption}

\numberwithin{section}{chapter}
\numberwithin{equation}{chapter}
\numberwithin{figure}{chapter}
\def\RR{\mathbb{R}}       % one dimension real space
\def\Rd{\mathbb{R}^d}     % d dimension real space
\def\NN{\mathbb{N}}       % positive integer
\def\CC{\mathbb{C}}       % one dimension complex space
\def\PP{\mathbb{P}}
\def\TT{\mathbb{T}}
\def\ZZ{\mathbb{Z}}

\def\ud{\mathrm{d}}       % integral symbol d

\def\vx{\boldsymbol{x}}   %
\def\vy{\boldsymbol{y}}   %
\def\vz{\boldsymbol{z}}
\def\vc{\boldsymbol{c}}

\def\va{\boldsymbol{a}}
\def\valpha{\boldsymbol{\alpha}}
\def\vtheta{\boldsymbol{\theta}}
\def\vvarphi{\boldsymbol{\varphi}}
\def\vbeta{\boldsymbol{\beta}}
\def\vxi{\boldsymbol{\xi}}

\def\ve{\boldsymbol{e}}

\def\vv{\boldsymbol{v}}
\def\vn{\boldsymbol{n}}

\def\vl{\boldsymbol{l}}
\def\v0{\boldsymbol{0}}
\def\vb{\boldsymbol{b}}

\def\vl{\boldsymbol{l}}
\def\veta{\boldsymbol{\eta}}
\def\v0{\boldsymbol{0}}

\def\vvartheta{\boldsymbol{\vartheta}}

\def\vw{\boldsymbol{w}}
\def\vTheta{\boldsymbol{\Theta}}

\def\vA{\mathsf{A}}
\def\vI{\mathsf{I}}

\def\vD{\mathsf{D}}
\def\vU{\mathsf{U}}
\def\vE{\mathsf{E}}

\newcommand{\norm}[1]{\left\lVert#1\right\rVert}  % norm of R^d
\newcommand{\abs}[1]{\left\lvert#1\right\rvert}   % absolute value of real or complex number

\def\Banach{\mathcal{B}}
\def\Hilbert{\mathcal{H}}
\def\Fun{\mathcal{F}}

\def\Cont{\mathrm{C}}
\def\Leb{\mathrm{L}}
\def\Linfty{\mathbb{L}_{\infty}}

\def\Span{\mathrm{span}}
\def\Sign{\mathrm{sign}}
\def\Space{\mathcal{N}}

\def\Kset{\mathcal{K}}
\def\Aset{\mathcal{A}}
\def\Eset{\mathcal{E}}

\def\Sset{\mathcal{S}}
\def\Yset{\mathcal{Y}}
\def\VecSpace{\mathcal{V}}

\def\Domain{\Omega}

\def\risk{\mathcal{L}}
\def\svm{\mathcal{T}}
\def\measure{\upsilon}
\def\Borel{\mathcal{B}}

\def\Gateaux{d_G}
\def\GateauxNorm{\iota}
\def\Frechet{d_F}
\def\adjK{K'} %adjoint kernel
\def\normopt{\mathcal{P}}
\def\order{o}

\def\adjphi{\phi'}

\def\adjvarphi{\varphi'}

\def\ii{\mathrm{i}}

\def\lp{\mathrm{l}_p}
\def\lq{\mathrm{l}_q}
\def\lr{\mathrm{l}_r}
\def\lone{\mathrm{l}_1}
\def\linfty{\mathrm{l}_{\infty}}
\def\czero{\mathrm{c}_0}
\def\ltwo{\mathrm{l}_2}
\def\lspace{\mathrm{l}}

\def\DualMeasure{\mathcal{M}}
\def\supp{\mathrm{supp}}

\def\diag{\mathrm{diag}}

\makeindex

\begin{document}

\frontmatter

\title[Generalized Mercer Kernels and RKBSs]{Generalized Mercer Kernels and Reproducing Kernel Banach Spaces}

%    Remove any unused author tags.

%    author one information
\author{Yuesheng Xu\footnote{The first author is a professor emeritus of Syracuse University, Syracuse, NY 13244, USA.}}
\address{School of Data and Computer Science, Guangdong Province Key Laboratory of Computational Science, Sun Yat-Sen University, Guangzhou, Guangdong 510275 P. R. China}
%\curraddr{}
\email{xuyuesh@mail.sysu.edu.cn, yxu06@syr.edu}
%\thanks{}

%    author two information
\author{Qi Ye\footnote{The second author is a corresponding author.}}
\address{School of Mathematical Sciences, South China Normal University, Guangzhou, Guangdong 510631 P. R. China}
%\curraddr{}
\email{yeqi@m.scnu.edu.cn, qiye@syr.edu}
%\thanks{}

%    \date is required; it is the date received by the editor.
\date{}

%    Recognition of the 2010 edition of the Mathematics Subject
%    Classification requires a version of amsbook.cls from July 2009
%    or later.  If "2010" is not recognized, please upgrade.


%\dedicatory{Dedication text (use \\[2pt] for line break if necessary)}

\begin{abstract}

This article studies constructions of reproducing kernel Banach spaces (RKBSs) which may be viewed as a generalization of reproducing kernel Hilbert spaces (RKHSs).
A key point is to endow Banach spaces with reproducing kernels such that machine learning in RKBSs can be well-posed and of easy implementation.
First we verify many advanced properties of the general RKBSs such as density, continuity, separability, implicit representation, imbedding, compactness, representer theorem for learning methods, oracle inequality, and universal approximation.
Then, we develop a new concept of generalized Mercer kernels to construct $p$-norm RKBSs for $1\leq p\leq\infty$.
The $p$-norm RKBSs preserve the same simple format as the Mercer representation of RKHSs.
Moreover, the $p$-norm RKBSs are isometrically equivalent to the standard $p$-norm spaces of countable sequences. Hence, the $p$-norm RKBSs possess more geometrical structures than RKHSs including sparsity.
To be more precise, the suitable countable expansion terms of the generalized Mercer kernels can be used to represent the pairs of Schauder bases and biorthogonal systems of the $p$-norm RKBSs such that the generalized Mercer kernels become the reproducing kernels of the $p$-norm RKBSs.
The generalized Mercer kernels also cover many well-known kernels, for example, min kernels, Gaussian kernels, and power series kernels.
Finally, we propose to solve the support vector machines in the $p$-norm RKBSs, which are to minimize the regularized empirical risks over the $p$-norm RKBSs.
We show that the infinite dimensional support vector machines in the $p$-norm RKBSs can be equivalently transferred to finite dimensional convex optimization problems such that we obtain the finite dimensional representations of the support vector machine solutions for practical applications.
In particular,
we verify that some special support vector machines in the $1$-norm RKBSs are equivalent to the classical $1$-norm sparse regressions. This gives fundamental supports of a novel learning tool called sparse learning methods to be investigated in our next research project.

\end{abstract}

\maketitle

\tableofcontents

%    Include unnumbered chapters (preface, acknowledgments, etc.) here.
%\include{Acknowledgments}

\mainmatter
%    Include main chapters here.

%------------------------------------------------------------------------------------------------------------------------
%------------------------------------------------------------------------------------------------------------------------
\chapter{Introduction}\label{char-Intr}
%------------------------------------------------------------------------------------------------------------------------
%------------------------------------------------------------------------------------------------------------------------

Machine learning in Hilbert spaces has become a useful modeling and prediction tool in many areas of science and engineering. Recently, there was an emerging interest in developing learning algorithms in Banach spaces \cite{Zhou2003, MicchelliPontil2004, ZhangXuZhang2009}. Machine learning is usually well-posed in reproducing kernel Hilbert spaces (RKHSs). It is desirable to solve learning problems in Banach spaces endowed with certain reproducing kernels. Paper \cite{ZhangXuZhang2009} introduced a concept of reproducing kernel Banach spaces (RKBSs) in the context of machine learning by employing the notion of semi-inner products. The use of semi-inner products in construction of  RKBSs has its limitation. The main purpose of this paper is to systematically study the construction of RKBSs without using the semi-inner products.

%------------------------------------------------------------------------------------------------------------------------
\section{Machine Learning in Banach Spaces}\label{s:Motivation}
%------------------------------------------------------------------------------------------------------------------------
\sectionmark{Motivations of Machine Learning in Banach Spaces}

After the success of machine learning in Hilbert spaces, people naturally would like to know whether machine learning can be achieved in Banach spaces~\cite{Zhou2003, MicchelliPontil2004} because a Hilbert space is a special case of Banach spaces. There are many advantages of machine learning in Banach spaces. Banach spaces may have richer geometrical structures than Hilbert spaces. Hence, it is feasible to develop more efficient learning methods in Banach spaces than in Hilbert spaces. In particular, a variety of norms of Banach spaces can be employed to construct the function spaces and to measure the margins of the objects even if they do not satisfy the parallelogram law of a Hilbert space. For example,
we may develop the sparse learning methods in the $1$-norm space, making use of the geometrical property of a ball of the $1$-norm. This may not be accomplished in the $2$-norm space due its lack of the geometrical property. A special example of such learning methods is the sparse regression
\[
\min_{\vxi\in\lone}\left\{\frac{1}{2N}\norm{\vy-\vA\vxi}_2^2+\sigma\norm{\vxi}_1\right\}.
\]

It is well-known that the learning methods in Hilbert spaces have vigorous abilities. The main reason is that the certain Hilbert spaces can be endowed with reproducing kernels such that the solution of a learning problem can be written as the finite dimensional kernel-based representations for practical implementations and computations.
Informally, the reproducing properties can reconstruct the functions values using the reproducing kernels.
In a Hilbert space, the reproducing properties are defined by its inner product. Clearly, a Banach space may not always possess an inner product.
Nevertheless, paper \cite{ZhangXuZhang2009} shows that the reproducing properties can still be achieved with the semi-inner products. This indicates that the machine learning can be well-posed in the Banach spaces having the semi-inner-products.
However, the original definitions of RKBSs require the reflexivity condition. We know that the infinite dimensional $\lone$ space is a Banach space but it is not reflexive. There is a question whether the reproducing properties can be defined in the general Banach spaces. Here, we redefine the reproducing properties in the dual bilinear products and investigate the learning solutions by the associated kernels.
Moreover, we complete and improve the theoretical analysis of RKBSs.
This can give the support of theoretical analysis to develop another vigorous tools for machine learning such as classifications and data mining.
Then the reproducing properties given in Banach spaces could be a feasible way to obtain the kernel-based solutions of learning problems.

%------------------------------------------------------------------------------------------------------------------------
\section{Overview of Kernel-based Function Spaces}\label{s:Overview-Ker-FunSpace}
%------------------------------------------------------------------------------------------------------------------------
\sectionmark{Overview of Kernel-based Function Spaces}

Kernel-based methods have been a focus of attention in high dimensional approximation and machine learning.
The first deep investigation of reproducing kernels went back to paper~\cite{Aronszajn1950}.
There is another good source~\cite{Meschkowski1962} for RKHSs.

More recent, books~\cite{Buhmann2003,Wendland2005,Fasshauer2007} show how to do scattered data approximation and meshfree approximation
in native spaces induced by radial basis functions and (conditionally) positive definite kernels,
and books~\cite{Wahba1990,ScholkopfSmola2002,BerlinetThomas-Agnan2004,SteinwartChristmann2008,HastieTibshiraniFriedman2009} give another theory for statistical learning in RKHSs. Actually,
the native spaces and the RKHSs are the equivalent concepts.
Paper~\cite{SchabackWendland2006} connects machine learning and meshfree approximation in RKHSs and native spaces.
Papers~\cite{FasshauerYe2011Dist,FasshauerYe2011DiffBound} further establish
a connection of reproducing kernels and RKHSs with Green functions and generalized Sobolev spaces driven by differential operators and possible
boundary operators.

Currently people exert a great interest in generalization of kernel-based methods in Hilbert spaces to Banach spaces,
because Banach spaces have more geometrical structures produced by their norms of various kinds.
For the meshfree approximation methods, paper~\cite{EricksonFasshauer2008} extends native spaces to Banach spaces called
generalized native spaces to do the high dimensional interpolations.
In another way, paper~\cite{Zhou2003} shows the capacity of the learning theory in Banach spaces by using the imbedding properties of
RKHSs, and papers~\cite{MicchelliPontil2004,ArgyriouMicchelliPontil2009,ArgyriouMicchelliPontil2010} start with the representer theorem for learning methods in Banach spaces. Moreover, paper~\cite{ZhangXuZhang2009} first proposes a new concept of RKBSs for machine learning.
Combining the ideas of generalized native spaces and RKBSs, recent paper~\cite{FasshauerHickernellYe2013} ensures that we obtain the finite dimensional representations of support vector machine solutions in Banach spaces in terms of the positive definite functions.
Paper~\cite{Ye2014RKBS} shows that the machine learning can be well-posed and of easy implementation in the Banach spaces endowed with the reproducing kernels.
Base on theory of RKBSs given in~\cite{ZhangXuZhang2009}, papers~\cite{SriperumbudurFukumizuLanckriet2011,SongZhang2011,GarciaPortal2012,ZhangZhang2012,VillmannHaaseKastner2013,SongZhangHickernell2013}
give other applications of RKBSs such as embedding probability measures, least square regressions, sampling, regularized learning methods, and gradients based learning methods.
However, there are still couple important issues that need discussions:
\begin{itemize}
\item how to find the reproducing kernel of a given Banach space,
\item how to obtain representations of RKBSs by a given kernel.
\end{itemize}
In this article, we answer these two questions in a new setting of the reproducing properties of Banach spaces.

%------------------------------------------------------------------------------------------------------------------------
\section{Main Results}\label{s:Results}
%------------------------------------------------------------------------------------------------------------------------
\sectionmark{Main Results}

This article redefines the RKBSs given in Definition~\ref{d:RKBS} -- differently from \cite[Definition~1]{ZhangXuZhang2009} -- such that the reflexivity condition of RKBSs becomes unnecessary.
Roughly speaking, the RKBSs given here can be seen as the weaker format of the original reflexive RKBSs and the semi-inner-product RKBSs. The redefined RKBSs further cover the $1$-norm geometrical structures. This ensures that the sparse learning methods can be well-posed in RKBSs.
This development of RKBSs is to generalize the inner products of Hilbert spaces to the dual bilinear products of Banach spaces.
To be more precise, the reproducing properties of the RKHS $\Hilbert$ given by the inner products such as
\[
\left(f,K(\vx,\cdot)\right)_{\Hilbert}=f(\vx),\quad \left(K(\cdot,\vy),g\right)_{\Hilbert}=g(\vy),
\]
are generalized to the dual-bilinear-product format of the RKBS $\Banach$, that is,
\[
\langle f,K(\vx,\cdot)\rangle_{\Banach}=f(\vx),\quad \langle K(\cdot,\vy),g\rangle_{\Banach}=g(\vy),
\]
where $K$ is the reproducing kernel.

In Chapter~\ref{char:RKBS}, we verify the advanced properties of RKBSs such as density, continuity, separability, implicit representation, imbedding, and compactness.
Moreover, we generalize the representer theorem for learning methods over RKHSs to RKBSs including the regularized empirical risks
\[
\min_{f\in\Banach}
\left\{\frac{1}{N}\sum_{k\in\NN_N}L\left(\vx_k,y_k,f(\vx_k)\right)+R\left(\norm{f}_{\Banach}\right)\right\},
\]
where $\NN_N:=\left\{1,\ldots,N\right\}$
and the regularized infinite-sample risks
\[
\min_{f\in\Banach}
\left\{\int_{\Domain\times\RR}L\left(\vx,y,f(\vx)\right)\PP\left(\ud\vx,\ud y\right)+R\left(\norm{f}_{\Banach}\right)\right\},
\]
where $\Domain$ is a locally compact Hausdorff space (see Theorems~\ref{t:RKBS-opt-rep} and~\ref{t:RKBS-opt-rep-gen}).
In Proposition~\ref{p:RKBS-oracle}, we give the oracle inequality
\[
\PP^N\left(\widetilde{\risk}(s)\geq\tilde{r}^{\ast}+\kappa\right)\leq e^{-\tau},
\]
to measure the errors of the minimizer $s$ of the empirical risks over the closed ball $B_M$ of the RKBS $\Banach$
such that we obtain the approximation of the minimization $\tilde{r}^{\ast}$ of the infinite-sample risks over the closed set $\Eset_M$ of $B_M$ with respect to the uniform norm.
At the end of the chapter, we show that the dual space $\Banach'$ of the RKBS $\Banach$
has the universal approximation in Proposition~\ref{p:RKBS-universial-approx}, that is,
$\Banach'$ is dense in the space of continuous functions with the uniform norm. This property also indicates that the collection of all $\Eset_M'$ for $M>0$ covers the whole space of continuous functions.

In practical implementation of machine learning, the explicit representations of RKBSs and reproducing kernels need to be known specifically.
Now we look at the initial ideas of the generalized Mercer kernels and their RKBSs.
In tutorial lectures of RKHSs, we usually start with a simple example of the Hilbert space $\Hilbert$ spanned by the finite orthonormal basis
$\left\{\phi_k(x):=\sin(k\pi x):k\in\NN_n\right\}$ of $\Leb_2(-1,1)$, that is,
\[
\Hilbert:=\left\{f:=\sum_{k\in\NN_n}a_k\phi_k: a_1,\ldots,a_n\in\RR\right\},
\]
equipped with the norm
\[
\norm{f}_{\Hilbert}:=\left(\sum_{k\in\NN_n}\abs{a_k}^2\right)^{1/2}.
\]
Obviously, the normed space $\Hilbert$ is isometrically equivalent to the standard $2$-norm space $\ltwo^n$ composed of $n$-dimensional vectors; hence we check the reproducing properties of
the Hilbert space $\Hilbert$ by the well-defined reproducing kernel
\[
K(x,y):=\sum_{k\in\NN_n}\phi_k(x)\phi_k(y),\quad
\text{for all }x,y\in(-1,1),
\]
that is,
\[
\left(f,K(x,\cdot)\right)_{\Hilbert}=\sum_{k\in\NN_n}a_k\phi_k(x)=f(x),\quad
\left(K(\cdot,y),g\right)_{\Hilbert}=\sum_{k\in\NN_n}b_k\phi_k(y)=g(y),
\]
for all $f:=\sum_{k\in\NN_n}a_k\phi_k,g:=\sum_{k\in\NN_n}b_k\phi_k\in\Hilbert$.
For this special case, we find that the inner product of $\Hilbert$ is consistent with the integral product, that is,
\[
\int_{-1}^{1}f(x)g(x)\ud x=\sum_{j,k\in\NN_n}a_jb_k\int_{-1}^{1}\phi_j(x)\phi_k(x)\ud x=\sum_{k\in\NN_n}a_kb_k=\left(f,g\right)_{\Hilbert}.
\]

Based on the constructions of $\Hilbert$ we introduce another normed spaces with the same basis $\phi_1,\ldots,\phi_n$ and the same reproducing kernel $K$.
A natural idea is to extend the $2$-norm space to a $p$-norm space for $1\leq p\leq\infty$, that is,
\[
\Banach^p:=\left\{f:=\sum_{k\in\NN_n}a_k\phi_k: a_1,\ldots,a_n\in\RR\right\},
\]
equipped with a norm
\[
\norm{f}_{\Banach^p}:=
\left(\sum_{k\in\NN_n}\abs{a_k}^p\right)^{1/p},\text{ when }1\leq p<\infty,
\quad
\norm{f}_{\Banach^p}:=\sup_{k\in\NN_n}\abs{a_k},\text{ when }p=\infty.
\]
This construction ensures that the normed space $\Banach^p$ is isometrically equivalent to the standard $p$-norm space $\lp^n$ composed of $n$-dimensional real vectors. Moreover, the dual space of $\Banach^p$ is isometrically equivalent to $\Banach^q$ when $q$ is conjugate to $p$ because the dual space of $\lp^n$ and the normed space $\lq^n$ are isometrically isomorphic.
This guarantees that the dual bilinear product of $\Banach^p$ is consistent with the integral product, that is,
\[
\int_{-1}^{1}f(x)g(x)\ud x=\sum_{j,k\in\NN_n}a_jb_k\int_{-1}^{1}\phi_j(x)\phi_k(x)\ud x=\sum_{k\in\NN_n}a_kb_k=\langle f,g\rangle_{\Banach^p},
\]
for all $f:=\sum_{k\in\NN_n}a_k\phi_k\in\Banach^p$ and all $g:=\sum_{k\in\NN_n}b_k\phi_k\in\Banach^q$.
Therefore, we obtain the reproducing properties of the Banach space $\Banach^p$, that is,
\[
\langle f,K(x,\cdot)\rangle_{\Banach^p}=\sum_{k\in\NN_n}a_k\phi_k(x)=f(x),\quad
\langle K(\cdot,y),g\rangle_{\Banach^p}=\sum_{k\in\NN_n}b_k\phi_k(y)=g(y).
\]
In particular, the RKHS $\Hilbert$ is consistent with $\Banach^2$ and the existence of the $1$-norm RKBS $\Banach^1$ is verified.
However, these RKBSs $\Banach^p$ are just finite dimensional Banach spaces.

In the above example, for any finite number $n$, the reproducing kernel $K$ is always well-defined; but the kernel $K$ does not converge pointwisely when $n\to\infty$; hence the above countable basis $\phi_1,\ldots,\phi_n,\ldots$ can not be used to set up the infinite dimensional RKBSs.
Currently, people are most interested in the infinite dimensional Banach spaces for applications of machine learning. In this article, we mainly consider the infinite dimensional RKBSs such that the learning algorithms can be chosen from the enough large amounts of suitable solutions.
In the following chapters, we shall extend the initial ideas of the RKBSs $\Banach^p$ to construct the infinite dimensional RKBSs and we shall discuss what kinds of kernels can become reproducing kernels.
We know that any positive definite kernel defined on a compact Hausdorff space can always be written as the classical Mercer kernel that is a sum of its countable eigenvalues multiplying eigenfunctions in the form of equation~\eqref{eq:pdk-Mercer}, that is,
\[
K(\vx,\vy)=\sum_{n\in\NN}\lambda_ne_n(\vx)e_n(\vy).
\]
It is also well-known that the classical Mercer kernel is always a reproducing kernel of some separable RKHS and
the Mercer representation of this RKHS
guarantees the isometrical isomorphism onto the standard $2$-norm space of countable sequences (see the Mercer representation theorem of RKHSs in \cite[Theorem~10.29]{Wendland2005} and \cite[Theorem~4.51]{SteinwartChristmann2008}).
This leads us to generalize the Mercer kernel to develop the reproducing kernel of RKBSs.
To be more precise, the generalized Mercer kernels are the sums of countable symmetric or nonsymmetric expansion terms such as
\[
K(\vx,\vy)=\sum_{n\in\NN}\phi_n(\vx)\adjphi_n(\vy),
\]
given in Definition~\ref{d:GMK}.
Hence, we extend
the Mercer representations of the separable RKHSs to the infinite dimensional RKBSs in Sections~\ref{s:p-RKBS} and~\ref{s:1-RKBS}.

Chapter~\ref{char-GMK} primarily focuses on the construction of the $p$-norm RKBSs induced by the generalized Mercer kernels.
These $p$-norm RKBSs are infinite dimensional.
With additional conditions, the generalized Mercer kernels can become the reproducing kernels of the $p$-norm RKBSs. The $p$-norm RKBSs for $1<p<\infty$ are uniformly convex and smooth; but the $1$-norm RKBSs and the $\infty$-norm RKBSs are not reflexive.
The main idea of the $p$-norm RKBSs is that the expansion terms of the generalized Mercer kernels are viewed as the Schauder bases and the biorthogonal systems of the $p$-norm RKBSs (see Theorems~\ref{t:RKBS-MercerKer-p}, \ref{t:RKBS-MercerKer-1}, and~\ref{t:RKBS-MercerKer-infty}).
The techniques of their proofs are to verify that the $p$-norm RKBSs have the same geometrical structures of the standard $p$-norm spaces of countable sequences by the associated isometrical isomorphisms. This means that we transfer the kernel bases into the Schauder bases of the $p$-norm RKBSs. The advantage of these constructions is that we deal with the reproducing properties in a fundamental way.
Moreover, the imbedding, compactness, and universal approximation of the $p$-norm RKBSs can be checked by the expansion terms of the generalized Mercer kernels.

The positive definite kernels defined on the compact Hausdorff spaces are the special cases of the generalized Mercer kernels because these positive definite kernels can be expanded by their eigenvalues and eigenfunctions. Chapter~\ref{char-PDK} shows that the $p$-norm RKBSs of many well-known positive definite kernels, such as min kernels, Gaussian kernels, and power series kernels, are also well-defined.

At last, we discuss the solutions of the support vector machines in the $p$-norm RKBSs in Chapter~\ref{char-SVM}.
The theoretical results cover not only the classical support vector machines driven by the hinge loss but also general support vector machines driven by many other loss functions, for example, the least square loss, the logistic loss, and the mean loss.
Theorem~\ref{t:RKBS-MercerKer-svm-rep-pq} provides that the support vector machine solutions in the $p$-norm RKBSs for $1<p<\infty$ have the finite dimensional representations.
To be more precise, the infinite dimensional support vector machine solutions in these $p$-norm RKBSs can be equivalently transferred into the finite dimensional convex optimization problems such that the learning solutions can be set up by the suitable finite parameters.
This guarantees that the support vector machines in the $p$-norm RKBSs for $1<p<\infty$ can be well-computable and easy in implementation.
According to Theorem~\ref{t:opt-svm-s1}, the support vector machines in the $1$-norm RKBSs can be approximated by the support vector machines in the $p$-norm RKBSs when $p\to1$.
Moreover, we show that the support vector machine solutions in the special $p_m$-norm RKBSs can be written as a linear combination of the kernel bases even when these $p_m$-norm RKBSs are just Banach spaces without inner products (see Theorem~\ref{t:SVM-opt-rep-typ}).
It is well-known that the norm of the support vector machine solutions in RKHSs depends on the positive definite matrices. We also find that the norm of the support vector machine solutions in the $p_m$-norm RKBSs is related to the positive definite tensors (high-order matrices).
Hence, a tensor decomposition will be a good numerical tool to speed up the computations of the support vector machines in the $p_m$-norm RKBSs.
In particular, some special support vector machines in the $1$-norm RKBSs and the classical $\lone$-sparse regressions are equivalent. This offers a fresh direction to design a vigorous algorithm for sparse sampling. We shall develop fast algorithms and consider practical applications for the sparse learning methods in our future work for data analysis.

%------------------------------------------------------------------------------------------------------------------------
%------------------------------------------------------------------------------------------------------------------------
\chapter{Reproducing Kernel Banach Spaces}\label{char:RKBS}
%------------------------------------------------------------------------------------------------------------------------
%------------------------------------------------------------------------------------------------------------------------

In this chapter we provide an alternative definition of RKBSs originally introduced in~\cite{ZhangXuZhang2009}.
The definition of RKBSs given here is a natural generalization of RKHSs by viewing the inner products as the dual bilinear products to introduce the reproducing properties.
Moreover, we verify many other properties of RKBSs including density, continuity, separability, implicit representation, imbedding, compactness, representer theorem for learning methods, oracle inequality, and universal approximation.

%------------------------------------------------------------------------------------------------------------------------
\section{Reproducing Kernels and Reproducing Kernel Banach Spaces}\label{s:RKBS-RK}
%------------------------------------------------------------------------------------------------------------------------
\sectionmark{Reproducing Kernels and RKBS}

In this section we show a way to construct the reproducing properties in Banach spaces with reproducing kernels.

We begin with a review of the classical RKHSs.
Let $\Domain$ be a locally compact Hausdorff space and $\Hilbert$ be a Hilbert space composed of functions $f\in\Leb_0(\Domain)$. Here $\Leb_0(\Domain)$\footnote{\color{black}{In this article, all measurable functions are real-valued functions if it is not specified, for example, $\textrm{range}(f)\subseteq\RR$ if $f\in\Leb_0(\Domain)$.}} is the collection of all measurable functions defined on $\Domain$. The Hilbert space $\Hilbert$ is called a RKHS with the reproducing kernel $K\in\Leb_0(\Domain\times\Domain)$ if it satisfies the following two conditions
\begin{align*}
&\text{(i)}~K(\vx,\cdot)\in\Hilbert,\quad\text{for each }\vx\in\Domain,\\
&\text{(ii)}~( f,K(\vx,\cdot) )_{\Hilbert}=f(\vx),\quad\text{for
all }f\in\Hilbert\text{ and all }\vx\in\Domain,
\end{align*}
(see, \cite[Definition~10.1]{Wendland2005}).
We find that the dual space $\Hilbert'$ of $\Hilbert$ is isometrically equivalent to itself and the inner product $\left(\cdot,\cdot\right)_{\Hilbert}$ defined on $\Hilbert$ and $\Hilbert$ can be seen as an equivalent format of the dual bilinear product $\langle \cdot,\cdot \rangle_{\Hilbert}$ defined on $\Hilbert$ and $\Hilbert'$. This means that the classical reproducing properties can be represented by the dual space and the dual bilinear product. This gives an idea to extend the reproducing properties of Hilbert spaces to Banach spaces.

We need the notation and concepts of the Banach space.
A normed space $\Banach$ is called a Banach space if its norm induces a complete metric, or more precisely, every Cauchy sequence of $\Banach$ is convergent.
The dual space $\Banach'$ of the Banach space $\Banach$ is the collection of all continuous (bounded) linear functionals defined on $\Banach$ with the norm
$$
\norm{G}_{\Banach'}:=\sup_{f\in\Banach,f\neq0}\frac{\abs{G(f)}}{\norm{f}_{\Banach}}
$$
whenever $G\in\Banach'$.
The dual bilinear product $\langle\cdot,\cdot\rangle_{\Banach}$ is defined on the Banach space $\Banach$ and its dual space $\Banach'$ as
\[
\langle f,G \rangle_{\Banach}:=G(f),\quad\text{for all }f\in\Banach\text{ and all }G\in\Banach'.
\]
Since the natural map is the isometrically imbedding map from $\Banach$ into $\Banach''$, we have that
$\langle f, G \rangle_{\Banach}=\langle G,f \rangle_{\Banach'}$ for all $f\in\Banach$ and all $G\in\Banach'$.
Normed spaces $\Banach_1$ and $\Banach_2$ are said to be isometrically isomorphic, if there is an isometric isomorphism $T:\Banach_1\to\Banach_2$ from $\Banach_1$ onto $\Banach_2$, or more precisely, the linear operator $T$ is bijective and continuous such that $\norm{T(f)}_{{\Banach}_2}=\norm{f}_{{\Banach}_1}$ whenever $f\in\Banach_1$. By $\Banach_1\cong\Banach_2$ we mean that $\Banach_1$ and $\Banach_2$ are isometrically isomorphic.
This shows that the isometric isomorphism $T$ provides a way of identifying both the vector space structure and the topology of $\Banach_1$ and $\Banach_2$.
In particular, a Banach space $\Banach$ is said to be reflexive if $\Banach\cong\Banach''$.

Next, we give the definition of RKBSs and reproducing kernels to be based on the idea of \cite[Definition~3.1]{FasshauerHickernellYe2013}.

%/////////////////////////////////////////////////////////////////////////////////////////////////////////////////////
\begin{definition}\label{d:RKBS}
Let $\Domain$ and $\Domain'$ be locally compact Hausdorff spaces equipped with regular Borel measures $\mu$ and $\mu'$, respectively,
let a kernel $K\in\Leb_0(\Domain\times\Domain')$, and let
a normed space $\Banach$ be a Banach space composed of functions
$f\in\Leb_0(\Domain)$ such that the dual space $\Banach'$ of $\Banach$ is isometrically equivalent to a normed space $\Fun$ composed of functions $g\in\Leb_0(\Domain')$.
We call $\Banach$ a \emph{right-sided reproducing kernel Banach space} and $K$ its \emph{right-sided reproducing kernel} if
\begin{align*}
&\text{(i)}~K(\vx,\cdot)\in\Fun\cong\Banach',\quad\text{for each }\vx\in\Domain,\\
&\text{(ii)}~\langle f,K(\vx,\cdot) \rangle_{\Banach}=f(\vx),\quad\text{for
all }f\in\Banach\text{ and all }\vx\in\Domain.
\end{align*}
If the Banach space $\Banach$ reproduces from the other side, that is,
\begin{align*}
&\text{(iii)}~K(\cdot,\vy)\in\Banach,\quad\text{for each }\vy\in\Domain',\\
&\text{(iv)}~\langle K(\cdot,\vy),g \rangle_{\Banach}=g(\vy),\quad\text{for all }g\in\Fun\cong\Banach'\text{ and all }\vy\in\Domain',
\end{align*}
then $\Banach$ is called a \emph{left-sided reproducing kernel Banach space} and $K$ its \emph{left-sided reproducing kernel}.
If two-sided reproducing properties (i)-(iv) as above are satisfied, then we say that $\Banach$ is a \emph{two-sided reproducing kernel Banach space} with the \emph{two-sided reproducing kernel} $K$.
\end{definition}
%/////////////////////////////////////////////////////////////////////////////////////////////////////////////////////

Observing the conjugated structures, the \emph{adjoint kernel} $\adjK$ of the reproducing kernel $K$ is defined by
\[
\adjK(\vy,\vx):=K(\vx,\vy),\quad\text{for all }\vx\in\Domain\text{ and all }\vy\in\Domain'.
\]
The reproducing properties can also be described by the adjoint kernel $\adjK$, that is,
\begin{align*}
&\text{(i)}~\adjK(\cdot,\vx)\in\Banach',\quad\text{for each }\vx\in\Domain,\\
&\text{(ii)}~\langle f,\adjK(\cdot,\vx) \rangle_{\Banach}=f(\vx),\quad
\text{for all }f\in\Banach\text{ and all }\vx\in\Domain,\\
&\text{(iii)}~\adjK(\vy,\cdot)\in\Banach,\quad \text{for each }\vy\in\Domain',\\
&\text{(iv)}~\langle \adjK(\vy,\cdot),g \rangle_{\Banach}=g(\vy),
\quad
\text{for all }g\in\Banach'\text{ and all }\vy\in\Domain'.
\end{align*}
The features of the two-sided reproducing kernel $K$ and its adjoint kernel $\adjK$ can be represented by the dual bilinear products in the following manner
\[
K(\vx,\vy)
=
\langle \adjK(\vy,\cdot),K(\vx,\cdot) \rangle_{\Banach}
=
\langle K(\cdot,\vy),\adjK(\cdot,\vx) \rangle_{\Banach}
=\adjK(\vy,\vx),
\]
for all $\vx\in\Domain$ and all $\vy\in\Domain'$.
This means that both $\vx\mapsto K(\vx,\cdot)$, $\vy\mapsto K(\cdot,\vy)$ and $\vx\mapsto\adjK(\cdot,\vx)$, $\vy\mapsto\adjK(\vy,\cdot)$ can be viewed as the feature maps of RKBSs.

Clearly, the classical reproducing kernels of RKHSs can be seen as the two-sided reproducing kernels of the two-sided RKBS. It is well-known that the reproducing kernels of RKHSs are always positive definite while the reproducing kernels of RKBSs may be neither symmetric nor positive definite.

Our definition of RKBSs is different from the original definition of RKBSs originally introduced in \cite[Definition~1]{ZhangXuZhang2009} defined by the point evaluation functions.
The two-sided reproducing properties ensure that all point evaluation functionals $\delta_{\vx}$ and $\delta_{\vy}$ defined on $\Banach$ and $\Banach'$ are continuous linear functionals, because
$$
\delta_{\vx}(f)=\langle f,K(\vx,\cdot)\rangle_{\Banach}=f(\vx)
$$
and
$$
\delta_{\vy}(g)=\langle g,K(\cdot,\vy)\rangle_{\Banach'}=\langle K(\cdot,\vy),g \rangle_{\Banach}=g(\vy).
$$
This means that $K(\vx,\cdot)$ and $K(\cdot,\vy)$ can be seen as equivalent elements of $\delta_{\vx}$ and $\delta_{\vy}$, respectively.
When our two-sided RKBSs are reflexive and $\Domain=\Domain'$,
these two-sided RKBSs also fit \cite[Definition~1]{ZhangXuZhang2009}.
This shows that our definition of RKBSs contains that of \cite[Definition~1]{ZhangXuZhang2009} as a special example.

There are two reasons for us to extend the definition of RKBSs. The new definition allows the reproducing kernels to be defined on the nonsymmetric domains. Moreover, the reflexivity condition necessary for the original definition is not required for our RKBSs in order to introduce the $1$-norm RKBSs for sparse learning methods.
Even though we have a non-reflexive Banach space $\Banach$ such that all point evaluation functionals defined on $\Banach$ and $\Banach'$ are continuous linear functionals, this Banach space $\Banach$ still may not own a two-sided reproducing kernel. The reason is that $\Banach$ is merely isometrically imbedded into $\Banach''$, and then $\delta_{\vx}$ may not have the equivalent element in $\Banach$ for the reproducing property (i).

\subsection*{Uniqueness}

Now, we study the uniqueness of RKBSs and reproducing kernels.
It was shown in \cite[Theorem 10.3]{Wendland2005} that the reproducing kernel of a RKHS is unique.
But, the situation for the reproducing kernels of RKBSs is different. There may be many choices of the normed space $\Fun$ to which $\Banach'$ is isometrically equivalent. Moreover, the left-sided domain $\Domain$ of the kernel $K$ is corrective to the space $\Banach$ while the right-sided domain $\Domain'$ of the kernel $K$ is corrective to the space $\Fun$. This means that the change of $\Fun$ affects the selection of the kernel $K$ which implies that the reproducing kernel $K$ of the RKBS $\Banach$ may not be unique (see Section~\ref{s:p-RKBS}).
However, we have the following results.

%////////////////////////////////////////////////////////////////////////////////////////////////////////////////////////
\begin{proposition}\label{p:RKBS-uniquess-1}
If $\Banach$ is the left-sided reproducing kernel Banach space with
the fixed left-sided reproducing kernel $K$, then the isometrically isomorphic function space $\Fun$ of the dual space $\Banach'$ of $\Banach$ is unique.
\end{proposition}
%////////////////////////////////////////////////////////////////////////////////////////////////////////////////////////
\begin{proof}
Suppose that the dual space $\Banach'$ has two isometrically isomorphic function spaces $\Fun_1$ and $\Fun_2$. Take any element $G\in\Banach'$. Let $g_1\in\Fun_1$ and $g_2\in\Fun_2$ be the isometrically equivalent elements of $G$. Because of the left-sided reproducing properties, we have for all $\vy\in\Domain'$ that
\[
g_1(\vy)=\langle K(\cdot,\vy),g_1 \rangle_{\Banach}=\langle K(\cdot,\vy),G \rangle_{\Banach}
=\langle K(\cdot,\vy),g_2 \rangle_{\Banach}=g_2(\vy).
\]
This ensures that $g_1=g_2$ and thus, $\Fun_1=\Fun_2$.
\end{proof}

%////////////////////////////////////////////////////////////////////////////////////////////////////////////////////////
\begin{proposition}\label{p:RKBS-uniquess-2}
If the dual space $\Banach'$ of the two-sided reproducing kernel Banach space $\Banach$
is isometrically equivalent to the fixed function space $\Fun$, then the two-sided reproducing kernel $K$ of $\Banach$ is unique.
\end{proposition}
%////////////////////////////////////////////////////////////////////////////////////////////////////////////////////////
\begin{proof}
Assume that the two-sided RKBS $\Banach$ has two two-sided reproducing kernels $K$ and $W$ such that the corresponding isometrically isomorphic function space $\Fun$ of the dual space $\Banach'$ of $\Banach$ are the same, that is, $K(\vx,\cdot),W(\vx,\cdot)\in\Fun\cong\Banach'$ for all $\vx\in\Domain$. By the right-sided reproducing properties of $K$, we have that
\begin{equation}\label{eq:RK-unique-1}
\langle W(\cdot,\vy),K(\vx,\cdot) \rangle_{\Banach}=W(\vx,\vy),
\quad\text{for all }\vx\in\Domain\text{ and all }\vy\in\Domain'.
\end{equation}
Using the left-sided reproducing properties of $W$, we obtain
\begin{equation}\label{eq:RK-unique-2}
\langle W(\cdot,\vy),K(\vx,\cdot) \rangle_{\Banach}=K(\vx,\vy),
\quad\text{for all }\vx\in\Domain\text{ and all }\vy\in\Domain'.
\end{equation}
Combining equations~\eqref{eq:RK-unique-1} and~\eqref{eq:RK-unique-2}, we conclude that
\[
K(\vx,\vy)=W(\vx,\vy),\quad
\text{for all }\vx\in\Domain\text{ and all }\vy\in\Domain'.
\]
\end{proof}

%////////////////////////////////////////////////////////////////////////////////////////////////////////////////////////
\begin{remark}\label{r:RKBS-Fun}
According to the uniqueness results shown above, we do NOT consider different choices of $\Fun$. Therefore, the isometrically equivalent function space $\Fun$ is FIXED such that $\Banach'$ and $\Fun$ can be thought as the SAME in the following sections.
\end{remark}
%////////////////////////////////////////////////////////////////////////////////////////////////////////////////////////

%------------------------------------------------------------------------------------------------------------------------
\section{Density, Continuity, and Separability}\label{s:DensityContinuity}
%------------------------------------------------------------------------------------------------------------------------
\sectionmark{Density, Continuity, and Separability}

In this section, we show that RKBSs have similar density and continuity properties as RKHSs.

\subsection*{Density}

First we prove that a RKBS or its dual space can be seen as a completion of the linear vector space spanned linearly by its reproducing kernel (see Propositions~\ref{p:RKBS-dense} and~\ref{p:RKBS-dense-pointevalue}).
Let the right-sided and left-sided kernel sets be
\[
\Kset_{K}':=\left\{K(\vx,\cdot): \vx\in\Domain\right\}\subseteq\Banach',\quad
\Kset_K:=\left\{K(\cdot,\vy): \vy\in\Domain'\right\}\subseteq\Banach,
\]
respectively.
We now verify the density of $\Span\left\{\Kset_{K}'\right\}$ and $\Span\left\{\Kset_{K}\right\}$ in the RKBS $\Banach$ and its dual space $\Banach'$, respectively, using similar techniques of \cite[Theorem~2]{ZhangXuZhang2009} with which the density of point evaluation functionals defined on RKBSs was verified.
In other words, we look at whether the collection of all finite linear combinations of $\Kset_{K}'$ or $\Kset_{K}$ is dense in $\Banach'$ or $\Banach$, that is,
$$
\overline{\Span\left\{\Kset_{K}'\right\}}=\Banach'\ \ \mbox{or}\ \ \overline{\Span\left\{\Kset_{K}\right\}}=\Banach.
$$

To this end, we apply the Hahn-Banach extension theorem (\cite[Theorem~1.9.1 and Corollary~1.9.7]{Megginson1998}) that ensures that if $\Space$ is a closed subspace of a Banach space $\Banach$ such that $\Space\subsetneqq\Banach$, then there is a continuous linear functional $G$ defined on $\Banach$ such that $\norm{G}_{\Banach'}=1$ and
$$
\Space\subseteq\text{ker}(G)=\left\{f\in\Banach:\langle f,G \rangle_{\Banach}=0\right\}.
$$

%////////////////////////////////////////////////////////////////////////////////////////////////////////////////////////
\begin{proposition}\label{p:RKBS-dense}
If $\Banach$ is the left-sided reproducing kernel Banach space with the left-sided reproducing kernel $K\in\Leb_0(\Domain\times\Domain')$,
then $\Span\left\{\Kset_{K}\right\}$ is dense in $\Banach$.
\end{proposition}
%////////////////////////////////////////////////////////////////////////////////////////////////////////////////////////
\begin{proof}
Let $\Space$ be the completion (closure) of $\Span\left\{\Kset_{K}\right\}$ in the $\Banach$-norm.
If we verify that $\Space=\Banach$, then the proof is complete.

Since $\Banach$ is a Banach space, we have that $\Space\subseteq\Banach$. Assume to the contrary that $\Space\subsetneqq\Banach$.
By the Hahn-Banach extension theorem, there is a $g\in\Banach'$ such that $\norm{g}_{\Banach'}=1$ and
\begin{equation}\label{eq:RKBS-dense}
\Space\subseteq\text{ker}(g):=\left\{f\in\Banach: \langle f,g \rangle_{\Banach}=0\right\}.
\end{equation}
Combining equation~\eqref{eq:RKBS-dense} and the reproducing properties (iii)-(iv), we observe that
\[
g(\vy)=\langle K(\cdot,\vy),g \rangle_{\Banach}=0,\quad\text{for all }\vy\in\Domain',
\]
because $\Span\left\{\Kset_{K}\right\}\subseteq\Space$.
This contradicts the fact that $\norm{g}_{\Banach'}=1$ and $g=0$. Therefore, we must reject the assumption that $\Space\subsetneqq\Banach$. Consequently, $\Space=\Banach$.
\end{proof}
%////////////////////////////////////////////////////////////////////////////////////////////////////////////////////////

Proposition~\ref{p:RKBS-dense} shows that the left-sided RKBS $\Banach$ is separable if the left-sided kernel sets
$\Kset_{K}$ has the countable dense subset.

\begin{remark}\label{r:RKBS-dense}
Clearly, the closure of $\Span\left\{\Kset_{K}'\right\}$ is a closed subspace of the dual space $\Banach'$ of the right-sided RKBS $\Banach$.
However, $\Span\left\{\Kset_{K}'\right\}$ may not be dense in $\Banach'$.
Let us look at a counter example of the $1$-norm RKBS $\Banach$ with the right-sided reproducing kernel $K$ which is the integral min kernel given in Section~\ref{s:min}.
In Section~\ref{s:1-RKBS}, we show that the dual space $\Banach'$ of the $1$-norm RKBS $\Banach$ is isometrically isomorphic onto the space $\linfty$;
hence $\Span\left\{\Kset_{K}'\right\}$ is isometrically imbedding into $\linfty$.
It is obvious that the left-sided domain $\Domain$ of the integral min kernel $K$ is a separable space. This shows that $\Domain$ has a countable dense subset $X$. Moreover, we verify that $\Kset_{K|_{X\times\Domain'}}':=\left\{K(\vx,\cdot):\vx\in X\right\}$ is dense in the right-sided kernel set $\Kset_{K}'$.
Thus, the closure of all finite linear combinations of $\Kset_{K|_{X\times\Domain'}}'$ is equal to the closure of $\Span\left\{\Kset_{K}'\right\}$. This show that the closure of $\Span\left\{\Kset_{K}'\right\}$ is separable. Since $\linfty$ is non-separable, the closure of $\Span\left\{\Kset_{K}'\right\}$ is a proper closed subspace of $\Banach'$. Therefore, $\Span\left\{\Kset_{K}'\right\}$ is not dense in $\Banach'$.
\end{remark}

For the right-sided RKBSs, we need an additional condition on the reflexivity of the space to ensure the density of $\Span\left\{\Kset_{K}'\right\}$ in $\Banach'$.

%////////////////////////////////////////////////////////////////////////////////////////////////////////////////////////
\begin{proposition}\label{p:RKBS-dense-pointevalue}
If $\Banach$ is the reflexive right-sided reproducing kernel Banach space,
then $\Span\left\{\Kset_{K}'\right\}$ is dense in the dual space $\Banach'$ of $\Banach$.
\end{proposition}
%////////////////////////////////////////////////////////////////////////////////////////////////////////////////////////
\begin{proof}
The main technique used in this proof is a combination of the reflexivity of $\Banach$ and Proposition~\ref{p:RKBS-dense}.
According to the construction of the adjoint kernel $\adjK$ of the right-sided reproducing kernel $K$, we find that
\[
\Span\left\{\Kset_{K}'\right\}=\Span\left\{K(\vx,\cdot):\vx\in\Domain\right\}=\Span\left\{\adjK(\cdot,\vx): \vx\in\Domain\right\}=\Span\big\{\Kset_{\adjK}\big\}.
\]
Since $\Banach$ is reflexive, we have that
$
\Banach''\cong\Banach.
$
Combining the right-sided reproducing properties of $\Banach$,
we observe that
$$
\adjK(\cdot,\vx)=K(\vx,\cdot)\in\Banach', \ \ \mbox{for all}\ \ \vx\in\Domain
$$
and
$$
\langle \adjK(\cdot,\vx),f \rangle_{\Banach'}=\langle f,K(\vx,\cdot) \rangle_{\Banach}=f(\vx), \ \
\mbox{for all}\ \ f\in\Banach\cong\Banach'' \
\mbox{and for all}\ \ \vx\in\Domain.
$$
Hence, $\adjK$ is the left-sided reproducing kernel of the left-sided RKBS $\Banach'$.
Proposition~\ref{p:RKBS-dense} ensures that $\Span\big\{\Kset_{\adjK}\big\}$ is dense in $\Banach'$.
\end{proof}
%////////////////////////////////////////////////////////////////////////////////////////////////////////////////////////

By the preliminaries in \cite[Section~1.1]{Megginson1998}, a set $\Eset$ of a normed space $\Banach_1$ is called linearly independent if, for any $N\in\NN$ and any finite pairwise distinct elements $\phi_1,\ldots,\phi_N\in\Eset$, their linear combination $\sum_{k\in\NN_N}c_k\phi_k=0$ implies $c_1=\ldots=c_N=0$.
Moreover, a linearly independent set $\Eset$ is said to be a basis of a normed space $\Banach_1$ if the collocation of all finite linear combinations of $\Eset$ equals to  $\Banach_1$, that is, $\Span\left\{\Eset\right\}=\Banach_1$.

Moreover,
if $\Kset_{K}$ (resp. $\Kset_{K}'$) is linearly independent, then
$\Kset_{K}$ (resp. $\Kset_{K}'$) is a basis of
$\Span\big\{\Kset_{K}\big\}$ (resp. $\Span\big\{\Kset_{K}'\big\}$).
Propositions~\ref{p:RKBS-dense} and~\ref{p:RKBS-dense-pointevalue} also show that
the RKBS $\Banach$ and its dual space $\Banach'$
can be seen as the completion of
the linear vector spaces $\Span\big\{\Kset_{K}\big\}$ and $\Span\big\{\Kset_{K}'\big\}$ respectively.
Since no Banach space has a countably infinite basis \cite[Theorem~1.5.8]{Megginson1998},
the Banach spaces $\Banach$ (resp. $\Banach'$) can not be equal to $\Span\big\{\Kset_{K}\big\}$ (resp. $\Span\big\{\Kset_{K}'\big\}$) when the domains $\Domain'$ (resp. $\Domain$) is not a finite set.

\subsection*{Continuity}

In Propositions~\ref{p:weakly-conv-RKBS} and~\ref{p:weaklystar-conv-RKBS}., we discuss the relationships of the weak and weak* convergence and the pointwise convergence of RKBSs, respectively.
According to
\cite[Propositions~2.4.4 and 2.4.13]{Megginson1998}, we obtain the equivalent definitions of the \emph{weak and weak* convergence} of sequences of the Banach space $\Banach$ and its dual space $\Banach'$ as follows:
$$
f_n\in\Banach\overset{\text{weak}-\Banach}{\longrightarrow}f\in\Banach\ \ \mbox{if}\ \langle f_n,G \rangle_{\Banach}\to \langle f,G \rangle_{\Banach}\  \mbox{for all} \ G\in\Banach',
$$
and
$$
G_n\in\Banach'\overset{\text{weak*}-\Banach}{\longrightarrow}G\in\Banach'\ \ \mbox{if}\ \langle f,G_n \rangle_{\Banach}\to \langle f,G \rangle_{\Banach}\ \mbox{for all}\ f\in\Banach.
$$

%////////////////////////////////////////////////////////////////////////////////////////////////////////////////////////
\begin{proposition}\label{p:weakly-conv-RKBS}
The weak convergence in the right-sided reproducing kernel Banach space implies the pointwise convergence in this Banach space.
\end{proposition}
%////////////////////////////////////////////////////////////////////////////////////////////////////////////////////////
\begin{proof}
Let $\Banach$ be a right-sided RKBS with the right-sided reproducing kernel $K$.
We shall prove that
$$
\lim_{n\to\infty}f_n(\vx)=f(\vx)\ \mbox{for all}\ \vx\in\Domain\ \ \mbox{if} \ f_n\in\Banach\overset{\text{weak}-\Banach}{\longrightarrow}f\in\Banach, n\to\infty.
$$

Take any $\vx\in\Domain$ and any sequence $f_n\in\Banach$ for $n\in\NN$ which is weakly convergent to $f\in\Banach$ when $n\to\infty$.
By the reproducing property (i), we find that $K(\vx,\cdot)\in\Banach'$. This ensures that
\[
\lim_{n\to\infty}\langle f_n,K(\vx,\cdot) \rangle_{\Banach}
=\langle f,K(\vx,\cdot) \rangle_{\Banach}.
\]
Moreover, the reproducing property (ii) shows that
\[
\langle f_n,K(\vx,\cdot) \rangle_{\Banach}=f_n(\vx),\quad
\langle f,K(\vx,\cdot) \rangle_{\Banach}=f(\vx).
\]
Therefore, we have that
\[
\lim_{n\to\infty}f_n(\vx)=f(\vx).
\]
\end{proof}
%////////////////////////////////////////////////////////////////////////////////////////////////////////////////////////

The pointwise convergence may not imply the weak convergence in the right-sided RKBS $\Banach$ because $\Banach$ may not be reflexive so that we only determine that $\overline{\Span\left\{\Kset_{K}'\right\}}\subseteq\Banach'$.
If the right-sided RKBS $\Banach$ is reflexive, then Proposition~\ref{p:RKBS-dense-pointevalue} ensures the density of $\Span\left\{\Kset_{K}'\right\}$ in $\Banach'$. Hence, the weak convergence is equivalent to the pointwise convergence by the continuous extension.
In other words, the weak* convergence and the pointwise convergence are equivalent in left-sided RKBSs.

%////////////////////////////////////////////////////////////////////////////////////////////////////////////////////////
\begin{proposition}\label{p:weaklystar-conv-RKBS}
The weak* convergence in the dual space of the left-sided reproducing kernel Banach space is equivalent to the pointwise convergence in this dual space.
\end{proposition}
%////////////////////////////////////////////////////////////////////////////////////////////////////////////////////////
\begin{proof}
Let $\Banach'$ be the dual space of a left-sided RKBS $\Banach$ with the right-sided reproducing kernel $K$.
We shall prove that $g_n\in\Banach'\overset{\text{weak*}-\Banach}{\longrightarrow}g\in\Banach'$ as $n\to\infty$ if and only if $\lim_{n\to\infty}g_n(\vy)=g(\vy)$ for all $\vy\in\Domain_2$.

Take any $\vy\in\Domain'$ and any sequence $g_n\in\Banach'$ for $n\in\NN$ which is weakly* convergent to $g\in\Banach'$ when $n\to\infty$.
The reproducing property (iii) ensures that $K(\cdot,\vy)\in\Banach$; hence
\[
\lim_{n\to\infty}\langle K(\cdot,\vy),g_n \rangle_{\Banach}
=\langle K(\cdot,\vy),g \rangle_{\Banach}.
\]
Moreover, the reproducing property (iv) shows that
\[
\langle K(\cdot,\vy),g_n \rangle_{\Banach}=g_n(\vy),\quad
\langle K(\cdot,\vy),g \rangle_{\Banach}=g(\vy).
\]
Thus, we have that
\[
\lim_{n\to\infty}g_n(\vy)=g(\vy).
\]

Conversely, suppose that $g_n,g\in\Banach'$ such that $\lim_{n\to\infty}g_n(\vy)=g(\vy)$ for all $\vy\in\Domain'$ when $n\to\infty$.
We take any $f\in\Span\left\{\Kset_K\right\}$. Then $f$ can be written as a linear combination of some finite terms $K(\cdot,\vy_1),\ldots,K(\cdot,\vy_N)\in\Kset_K$, that is,
$$
f=\sum_{k\in\NN_N}c_kK(\cdot,\vy_k).
$$
According to the reproducing properties (iii)-(iv), we have that
\[
\langle f,g_n \rangle_{\Banach}=\sum_{k\in\NN_N}c_k\langle K(\cdot,\vy_k),g_n \rangle_{\Banach}
=\sum_{k\in\NN_N}c_kg_n(\vy_k),
\]
and
\[
\langle f,g \rangle_{\Banach}=\sum_{k\in\NN_N}c_k\langle K(\cdot,\vy_k),g \rangle_{\Banach}=\sum_{k\in\NN_N}c_kg(\vy_k);
\]
hence
\[
\lim_{n\to\infty}\langle f,g_n \rangle_{\Banach}=\langle f,g \rangle_{\Banach}.
\]
In addition, Proposition~\ref{p:RKBS-dense} ensures that $\overline{\Span\left\{\Kset_K\right\}}=\Banach$. Therefore, we verify the general case of $f\in\Banach$ by the continuous extension, that is,
\[
\lim_{n\to\infty}\langle f,g_n \rangle_{\Banach}=\langle f,g \rangle_{\Banach},\quad \text{for all }f\in\Banach.
\]
This ensures that $g_n\overset{\text{weak*}-\Banach}{\longrightarrow}g$ as $n\to\infty$.
\end{proof}
%////////////////////////////////////////////////////////////////////////////////////////////////////////////////////////

It is well-known that the continuity of reproducing kernels ensures that all functions in their RKHSs are continuous.
The RKBSs also have a similar property of continuity depending on their reproducing kernels. We present it below.

%////////////////////////////////////////////////////////////////////////////////////////////////////////////////////////
\begin{proposition}\label{p:RKBS-continue}
If $\Banach$ is the right-sided reproducing kernel Banach space with the right-sided reproducing kernel $K\in\Leb_0(\Domain\times\Domain')$ such that the map $\vx\mapsto K(\vx,\cdot)$ is continuous on $\Domain$, then $\Banach$ composes of continuous functions.
\end{proposition}
%////////////////////////////////////////////////////////////////////////////////////////////////////////////////////////
\begin{proof}
Take any $f\in\Banach$. We shall prove that $f\in\Cont(\Domain)$.
For any $\vx,\vz\in\Domain$, the reproducing properties (i)-(ii) imply that
\[
f(\vz)-f(\vx)=\langle f,K(\vz,\cdot)-K(\vx,\cdot)\rangle_{\Banach}.
\]
Hence, we have that
\begin{equation}\label{eq:RKBS-continuous}
\abs{f(\vz)-f(\vx)}
\leq\norm{f}_{\Banach}\norm{K(\vz,\cdot)-K(\vx,\cdot)}_{\Banach'}, \ \ \vx,\vz\in\Domain.
\end{equation}
Combining inequality~\eqref{eq:RKBS-continuous} and the continuity of the map $\vx\mapsto K(\cdot,\vx)$, we conclude that $f$ is a continuous function.
\end{proof}
%////////////////////////////////////////////////////////////////////////////////////////////////////////////////////////

\subsection*{Separability}
We already know that a RKHS can be non-separable as well as the example of the Hilbert space $\lspace_2(\Omega)$ with the uncountable domain $\Omega$.
Thus, a RKBS is not necessarily separable.
Now we show the sufficient conditions of separable RKBSs the same as separable RKHSs in \cite[Theorem~15]{BerlinetThomas-Agnan2004}.
The following criterion of the separability will be proved by the result that the closure of the linear combination of a countable subset of a normed space is separable (see
\cite[Proposition~1.12.1]{Megginson1998}).
We review the definition of the annihilator in \cite[Definition~1.10.14]{Megginson1998}.
Let $\Eset$ and $\Eset'$ be subsets of a Banach space $\Banach$ and its dual space $\Banach'$, respectively.
The annihilator of $\Eset$ in $\Banach'$ and the annihilator of $\Eset'$ in $\Banach$ are defined by
\[
\Eset^{\perp}:=\left\{g\in\Banach':\langle f,g\rangle_{\Banach}=0\text{ for all }f\in\Eset\right\},
\]
and
\[
{}^{\perp}\Eset':=\left\{f\in\Banach:\langle f,g\rangle_{\Banach}=0\text{ for all }g\in\Eset'\right\},
\]
respectively.

%////////////////////////////////////////////////////////////////////////////////////////////////////////////////////////
\begin{proposition}\label{p:RKBS-separable}
If $\Banach$ is the left-sided reproducing kernel Banach space and
the right-sided domain $\Domain'$ of the dual space $\Banach'$ of $\Banach$ contains a countable subset $X'\subseteq\Domain'$ such that for any $g\in\Banach'$, $g|_{X'}=0$ if and only if $g=0$, then $\Banach$ is separable.
\end{proposition}
%////////////////////////////////////////////////////////////////////////////////////////////////////////////////////////
\begin{proof}
Take any $g\in\Banach'$ such that $\langle K(\cdot,\vy),g \rangle_{\Banach}=0$ for all $\vy\in X'$.
So $g=0$. This ensures that the annihilator of $\Span\left\{K(\cdot,\vy):\vy\in X'\right\}$ is equal to $\{0\}$, that is, $\Span\left\{K(\cdot,\vy):\vy\in X'\right\}^{\perp}=\{0\}$.
Therefore, by \cite[Proposition~1.10.15]{Megginson1998}, the closure of the linear combination of $\left\{K(\cdot,\vy):\vy\in X'\right\}$ is equal to the annihilator of $\{0\}$ which is the whole space $\Banach$, that is,
$$
\overline{\Span\left\{K(\cdot,\vy):\vy\in X'\right\}}={}^{\perp}\{0\}=\Banach.
$$
\end{proof}
%////////////////////////////////////////////////////////////////////////////////////////////////////////////////////////

The following corollary exhibits a class of Banach spaces composed of continuous functions which have no reproducing kernels.

%////////////////////////////////////////////////////////////////////////////////////////////////////////////////////////
\begin{corollary}\label{c:RKBS-separable}
If $\Banach$ is a non-separable Banach space such that the dual space $\Banach'$ of $\Banach$ is composed of continuous functions defined on a separable domain $\Domain'$, then $\Banach$ is not a left-sided reproducing kernel Banach space.
\end{corollary}
%////////////////////////////////////////////////////////////////////////////////////////////////////////////////////////
\begin{proof}
Let $X'$ be a countable dense subset of $\Domain'$. As any element of $\Banach$ is continuous the conditions of Proposition~\ref{p:RKBS-separable} are satisfied. Therefore, if we assume that $\Banach$ had a left-sided reproducing kernel, then it would be separable. This would contradict the hypothesis.
\end{proof}
%////////////////////////////////////////////////////////////////////////////////////////////////////////////////////////

%------------------------------------------------------------------------------------------------------------------------
\section{Implicit Representation}\label{s:ImplicitRepresentation}
%------------------------------------------------------------------------------------------------------------------------
\sectionmark{Implicit Representation}

In this section, we find the implicit representation of RKBSs using the techniques of \cite[Theorem~10.22]{Wendland2005} for the implicit representation of RKHSs.
Specifically,
we employ a given kernel to set up a Banach space such that this Banach space becomes a two-sided RKBS with this kernel.

Suppose that
the right-sided kernel set $\Kset_{K}'$ of a given kernel $K\in\Leb_0(\Domain\times\Domain')$ is
linearly independent and the linear span of $\Kset_{K}'$ is endowed with some norm $\norm{\cdot}_{K}$.
Further suppose that
the completion $\Fun$ of $\Span\left\{\Kset_{K}'\right\}$ is reflexive, that is,
\[
\Fun:=\overline{\Span\left\{\Kset_{K}'\right\}}, \ \Fun\cong\Fun'',
\]
and the point evaluation functionals $\delta_{\vy}$ are continuous on $\Fun$, that is, $\delta_{\vy}\in\Fun'$ for all $\vy\in\Domain'$.

We denote by $\Delta_K$ the linear vector space spanned by the point evaluation functionals $\delta_{\vx}$ defined on $\Leb_0(\Domain)$, namely,
\[
\Delta_K:=\Span\left\{\delta_{\vx}: \vx\in\Domain\right\}.
\]
Clearly, $\left\{\delta_{\vx}:\vx\in\Domain\right\}$ is a basis of the linear vector space $\Delta_K$.
Moreover, $\Delta_K$ can be endowed with a norm by the equivalent norm $\norm{\cdot}_{K}$, that is, taking each $\lambda\in\Delta_K$, we have that $\lambda=\sum_{k\in\NN_N}c_k\delta_{\vx_k}$ for some finite pairwise distinct points $\vx_1,\ldots,\vx_N\in\Domain$. So, the linear independence of $\Kset_{K}'$ ensures that the norm of $\lambda$ is well-defined by
\[
\norm{\lambda}_{K}:=\norm{\sum_{k\in\NN_N}c_kK(\vx_k,\cdot)}_{K}.
\]
Comparing the norms of $\Delta_K$ and $\Span\left\{\Kset_{K}'\right\}$, we find that $\Delta_K$ and $\Span\left\{\Kset_{K}'\right\}$ are isometrically isomorphic by the linear operator
\[
T(\lambda):=\sum_{k\in\NN_N}c_kK(\vx_k,\cdot).
\]
This ensures that $\Delta_K$ is isometrically imbedded into the function space $\Fun$.

Next, we show how a two-sided RKBS $\Banach$ is constructed such that its two-sided reproducing kernel is equal to the given kernel $K$
and $\Banach'\cong\Fun$.

%////////////////////////////////////////////////////////////////////////////////////////////////////////////////////////
\begin{proposition}
If the kernel $K$, the reflexive Banach space $\Fun$, and the normed space $\Delta_K$ are defined as above,
then the space
\[
\Banach:=\left\{f\in\Leb_0(\Domain):\exists~C_f>0\text{ s.t. }\abs{\lambda(f)}\leq C_f\norm{\lambda}_{K}\text{ for all }\lambda\in\Delta_K\right\},
\]
equipped with the norm
\[
\norm{f}_{\Banach}:=\sup_{\lambda\in\Delta_K,\lambda\neq0}\frac{\abs{\lambda(f)}}{\norm{\lambda}_K},
\]
is a two-sided reproducing kernel Banach space with the two-sided reproducing kernel $K$, and the dual space $\Banach'$ of $\Banach$ is isometrically equivalent to $\Fun$.
\end{proposition}
%////////////////////////////////////////////////////////////////////////////////////////////////////////////////////////
\begin{proof}
The primary idea of implicit representation is to use the point evaluation functional $\delta_{\vx}$ to set up the normed space $\Banach$ composed of function $f\in\Leb_0(\Domain)$ such that
$\delta_{\vx}$ is continuous on $\Banach$.

By the standard definition of the dual space of $\Delta_K$,
the normed space $\Banach$ is isometrically equivalent to the dual space of $\Delta_K$, that is, $\Banach\cong\left(\Delta_K\right)'$.
Since the dual space of $\Delta_K$ is a Banach space, the space $\Banach$ is also a Banach space.
Next, we verify the right-sided reproducing properties of $\Banach$. Since
$$
\Delta_K\cong\Span\left\{\Kset_{K}'\right\},
$$
we have that
$$
\Banach\cong\left(\Span\left\{\Kset_{K}'\right\}\right)'.
$$
According to the density of $\Span\left\{\Kset_{K}'\right\}$ in $\Fun$,
we determine that $\Banach\cong\Fun'$ by the Hahn-Banach extension theorem.
The reflexivity of $\Fun$ ensures that
$$
\Banach'\cong\Fun''\cong\Fun.
$$
This implies that $\Delta_K$ is isometrically imbedded into $\Banach'$.
Since $K(\vx,\cdot)\in\Kset_{K}'$ is the equivalent element of $\delta_{\vx}\in\Delta_K$ for all $\vx\in\Domain$, the right-sided reproducing properties of $\Banach$ are well-defined, that is, $K(\vx,\cdot)\in\Banach'$ and
$$
\langle f,K(\vx,\cdot)\rangle_{\Banach}=\langle f,\delta_{\vx} \rangle_{\Banach}=f(\vx), \ \  \mbox{for all}\ \ f\in\Banach.
$$

Finally, we verify the left-sided reproducing properties of $\Banach$. Let $\vy\in\Domain'$ and $\lambda\in\Delta_K$. Hence, $\lambda$ can be represented as a linear combination of some $\delta_{\vx_1},\ldots,\delta_{\vx_N}\in\Delta_K$, that is,
$$
\lambda=\sum_{k\in\NN_N}c_k\delta_{\vx_k}.
$$
Moreover, we have that
\begin{equation}\label{eq:RKBS-implict-1}
\abs{\lambda\left(K(\cdot,\vy)\right)}\leq
\norm{\delta_{\vy}}_{\Fun'}\norm{\sum_{k\in\NN_N}c_kK(\vx_k,\cdot)}_{K}
=\norm{\delta_{\vy}}_{\Fun'}\norm{\lambda}_{K},
\end{equation}
because $\delta_{\vy}\in\Fun'$ and
\begin{equation}\label{eq:RKBS-implict-2}
\lambda\left(K(\cdot,\vy)\right)=\sum_{k\in\NN_N}c_kK(\vx_k,\vy)
=\delta_{\vy}\left(\sum_{k\in\NN_N}c_kK(\vx_k,\cdot)\right).
\end{equation}
Inequality~\eqref{eq:RKBS-implict-1} guarantees that $K(\cdot,\vy)\in\Banach$. Let $g$ be the equivalent element of $\lambda$ in $\Span\left\{\Kset_K'\right\}$.
Because of $\lambda\in\Banach'$, we have that
$$
\langle K(\cdot,\vy),g \rangle_{\Banach}=\lambda\left(K(\cdot,\vy)\right)=g(\vy)
$$
by equation~\eqref{eq:RKBS-implict-2}. Using the density of $\Span\left\{\Kset_K'\right\}$ in $\Fun$, we verify the general case of $g\in\Fun\cong\Banach'$ by the continuous extension, that is, $\langle K(\cdot,\vy),g \rangle_{\Banach}=g(\vy)$.
\end{proof}
%////////////////////////////////////////////////////////////////////////////////////////////////////////////////////////

%------------------------------------------------------------------------------------------------------------------------
\section{Imbedding}\label{s:Imbedding}
%------------------------------------------------------------------------------------------------------------------------
\sectionmark{Imbedding}

In this section, we establish the imbedding of RKBSs, an important property of RKBSs. In particular, it is useful in machine learning.
Note that the imbedding of RKHSs was given in \cite[Lemma~10.27 and Proposition~10.28]{Wendland2005}.

We say that a normed space $\Banach_1$ is imbedded into another normed space $\Banach_2$ if there exists an injective and continuous linear operator $T$ from $\Banach_1$ into $\Banach_2$.

Let $\Leb_p(\Domain)$ and $\Leb_q(\Domain')$ be the standard $p$-norm and $q$-norm Lebesgue spaces defined on $\Domain$ and $\Domain'$, respectively, where $1\leq p,q\leq \infty$. We already know that any $f\in\Leb_p(\Domain)$ is equal to $0$ almost everywhere if and only if
$f$ satisfies the condition
\[
\tag{C-$\mu$}
\mu\left(\left\{\vx\in\Domain:f(\vx)\neq0\right\}\right)=0.
\]
By the reproducing properties, the functions in the RKBSs are distinguished point wisely.
Hence, the RKBSs need another condition such that the zero element of RKBSs is equivalent to the zero element defined by the measure $\mu$ almost everywhere.
Then we say that the RKBS $\Banach$ satisfies the \emph{$\mu$-measure zero condition} if for any $f\in\Banach$, we have that $f=0$ if and only if $f$ satisfies the condition (C-$\mu$).
The $\mu$-measure zero condition of the RKBS $\Banach$ guarantees that for any $f,g\in\Banach$, we have that $f=g$ if and only if $f$ is equal to $g$ almost everywhere.
For example, when $\Banach\subseteq\Cont(\Domain)$ and $\supp(\mu)=\Domain$, then $\Banach$ satisfies the $\mu$-measure zero condition (see the Lusin Theorem).
Here, the support $\supp(\mu)$ of the Borel measure $\mu$ is defined as the set of all points $\vx$ in $\Domain$ for which every open neighbourhood $A$ of $\vx$ has positive measure. (More details of Borel measures and Lebesgue integrals can be found in \cite[Chapters 2 and 3]{Rudin1987}.)

We first study the imbedding of right-sided RKBSs.

%////////////////////////////////////////////////////////////////////////////////////////////////////////////////////////
\begin{proposition}\label{p:RKBS-imbedding}
Let $1\leq q\leq\infty$.
If $\Banach$ is the right-sided reproducing kernel Banach space with the right-sided reproducing kernel $K\in\Leb_0(\Domain\times\Domain')$ such that $\vx\mapsto\norm{K(\vx,\cdot)}_{\Banach'}\in\Leb_q(\Domain)$,
then the identity map from $\Banach$ into $\Leb_q(\Domain)$ is continuous.
\end{proposition}
%////////////////////////////////////////////////////////////////////////////////////////////////////////////////////////
\begin{proof}
We shall prove the imbedding by verifying that the identify map from $\Banach$ into $\Leb_q(\Domain)$ is continuous.
Let
\[
\Phi(\vx):=\norm{K(\vx,\cdot)}_{\Banach'},\quad \text{for }\vx\in\Domain.
\]
Since $\Phi\in\Leb_q(\Domain)$, we obtain a positive constant $\norm{\Phi}_{\Leb_q(\Domain)}$.
We take any $f\in\Banach$.
By the reproducing properties (i)-(ii), we have that
\[
f(\vx)=\langle f,K(\cdot,\vx) \rangle_{\Banach},\quad \text{for all }\vx\in\Domain.
\]
Hence, it follows that
\begin{equation}\label{eq:RKBS-imbedding}
\abs{f(\vx)}=\abs{\langle f,K(\cdot,\vx) \rangle_{\Banach}}\leq\norm{f}_{\Banach}\norm{K(\vx,\cdot)}_{\Banach'}=\norm{f}_{\Banach}\Phi(\vx),
\quad
\text{for }\vx\in\Domain.
\end{equation}
Integrating both sides of inequality~\eqref{eq:RKBS-imbedding} yields the continuity of the identify map
\[
\norm{f}_{\Leb_q(\Domain)}=\left(\int_{\Domain}\abs{f(\vx)}^q\mu(\ud\vx)\right)^{1/q}\leq
\norm{f}_{\Banach}\left(\int_{\Domain}\abs{\Phi(\vx)}^q\mu(\ud\vx)\right)^{1/q}
=\norm{\Phi}_{\Leb_q(\Domain)}\norm{f}_{\Banach},
\]
when $1\leq q<\infty$, or
\[
\norm{f}_{\Leb_{\infty}(\Domain)}=\underset{\vx\in\Domain}{\text{ess sup}}\abs{f(\vx)}\leq
\norm{f}_{\Banach}\underset{\vx\in\Domain}{\text{ess sup}}\abs{\Phi(\vx)}
=\norm{\Phi}_{\Leb_{\infty}(\Domain)}\norm{f}_{\Banach},
\]
when $q=\infty$.
\end{proof}
%////////////////////////////////////////////////////////////////////////////////////////////////////////////////////////

%////////////////////////////////////////////////////////////////////////////////////////////////////////////////////////
\begin{corollary}\label{c:RKBS-imbedding}
Let $1\leq q\leq\infty$ and let $\Banach$ be the right-sided reproducing kernel Banach space with the right-sided reproducing kernel $K\in\Leb_0(\Domain\times\Domain')$ such that $\vx\mapsto\norm{K(\vx,\cdot)}_{\Banach'}\in\Leb_q(\Domain)$.
If $\Banach$ satisfies the $\mu$-measure zero condition, then $\Banach$ is imbedded into $\Leb_q(\Domain)$.
\end{corollary}
%////////////////////////////////////////////////////////////////////////////////////////////////////////////////////////
\begin{proof}
According to Proposition~\ref{p:RKBS-imbedding}, the identity map $I:\Banach\to\Leb_q(\Domain)$ is continuous. Since $\Banach$ satisfies the $\mu$-measure zero condition, the identity map $I$ is injective.
Therefore, the RKBS $\Banach$ is imbedded into $\Leb_q(\Domain)$ by the identity map.
\end{proof}
%////////////////////////////////////////////////////////////////////////////////////////////////////////////////////////

%////////////////////////////////////////////////////////////////////////////////////////////////////////////////////////
\begin{remark}
If $\Banach$ does not satisfy the $\mu$-measure zero condition, then the above identity map is not injective. In this case, we can not say that $\Banach$ is imbedded into $\Leb_q(\Domain)$. But, to avoid the duplicated notations, the RKBS $\Banach$ can still be seen as a subspace of $\Leb_q(\Domain)$ in this article.
This means that functions of $\Banach$, which are equal almost everywhere, are seen as the same element in $\Leb_q(\Domain)$.
\end{remark}
%////////////////////////////////////////////////////////////////////////////////////////////////////////////////////////

For the compact domain $\Domain$ and the two-sided reproducing kernel $K\in\Leb_0(\Domain\times\Domain')$,  we define the left-sided integral operator $I_K:\Leb_p(\Domain)\to\Leb_0(\Domain')$ by
\[
I_K(\zeta)(\vy):=\int_{\Domain}K(\vx,\vy)\zeta(\vx)\mu(\ud\vx),\quad \text{for all }\zeta\in\Leb_p(\Domain)\text{ and all } \vy\in\Domain'.
\]
When $K(\cdot,\vy)\in\Leb_q(\Domain)$ for all $\vy\in\Domain'$, the linear operator $I_K$ is well-defined.
Here $p^{-1}+q^{-1}=1$.
We verify that the integral operator $I_K$ is also a continuous operator from $\Leb_p(\Domain)$ into the dual space $\Banach'$ of the two-sided RKBS $\Banach$ with the two-sided reproducing kernel $K$.

Same as \cite[Definition~3.1.3]{Megginson1998}, we call the linear operator $T^{\ast}:\Banach_2'\to\Banach_1'$ the adjoint operator of the continuous linear operator $T:\Banach_1\to\Banach_2$ if
\[
\langle Tf,g\rangle_{\Banach_1}=\langle f,T^{\ast}g \rangle_{\Banach_2}.
\]
\cite[Theorem~3.1.17]{Megginson1998} ensures that $T$ is injective if and only if the range of $T^{\ast}$ is weakly* dense in $\Banach_1'$, and $T^{\ast}$ is injective if and only if the range of $T$ is dense in $\Banach_2$. Here, a subset $\Eset$ is weakly* dense in a Banach space $\Banach$ if for any $f\in\Banach$, there exists a sequence $\left\{f_n:n\in\NN\right\}\subseteq\Eset$ such that $f_n$ is weakly* convergent to $f$ when $n\to\infty$.
Moreover, we shall show that $I_K$ is the adjoint operator of the identify map $I$ mentioned in Proposition~\ref{p:RKBS-imbedding}.
For convenience, we denote that the general division $1/\infty$ is equal to $0$.

%////////////////////////////////////////////////////////////////////////////////////////////////////////////////////////
\begin{proposition}\label{p:RKBS-imbedding-dual}
Let $1\leq p,q\leq\infty$ such that $p^{-1}+q^{-1}=1$.
If the kernel $K\in\Leb_0(\Domain\times\Domain')$
is the two-sided reproducing kernel of the two-sided reproducing kernel Banach space $\Banach$ such that $K(\cdot,\vy)\in\Leb_q(\Domain)$ for all $\vy\in\Domain'$ and the map $\vx\mapsto\norm{K(\vx,\cdot)}_{\Banach'}\in\Leb_q(\Domain)$,
then the left-sided integral operator $I_K$ maps $\Leb_p(\Domain)$ into the dual space $\Banach'$ of $\Banach$ continuously and
\begin{equation}\label{eq:imbeding-RKBS-dual}
\int_{\Domain}f(\vx)\zeta(\vx)\mu(\ud\vx)=\langle f,I_K\zeta \rangle_{\Banach},\quad \text{for all }\zeta\in\Leb_p(\Domain)\text{ and all }f\in\Banach.
\end{equation}
\end{proposition}
%////////////////////////////////////////////////////////////////////////////////////////////////////////////////////////
\begin{proof}
Since $K(\cdot,\vy)\in\Leb_q(\Domain)$ for all $\vy\in\Domain'$, the integral operator $I_K$ is well-defined on $\Leb_p(\Domain)$.
First we prove that the integral operator $I_K$ is a continuous linear operator from $\Leb_p(\Domain)$ into $\Banach'$.
Then we take a $\zeta\in\Leb_p(\Domain)$ and show that $I_K\zeta\in\Fun\cong\Banach'$.
To simplify the notation in the proof, we let the subspace $\VecSpace:=\Span\left\{\Kset_K\right\}$.
Define a linear functional $G_{\zeta}$ on $\VecSpace$ of $\Banach$ by $\zeta$, that is,
\[
G_{\zeta}(f):=
\int_{\Domain}f(\vx)\zeta(\vx)\mu(\ud\vx),
\quad
\text{for }f\in\VecSpace.
\]
Since $\Banach$ is the two-sided RKBS with the two-sided reproducing kernel $K$ and $\vx\mapsto\norm{K(\vx,\cdot)}_{\Banach'}\in\Leb_q(\Domain)$, Proposition~\ref{p:RKBS-imbedding} ensures that the identity map from $\Banach$ into $\Leb_q(\Domain)$ is continuous; hence there exists a positive constant $\norm{\Phi}_{\Leb_q(\Domain)}$ such that
\begin{equation}\label{eq:RKBS-imbedding-dual-e1}
\norm{f}_{\Leb_q(\Domain)}\leq\norm{\Phi}_{\Leb_q(\Domain)}\norm{f}_{\Banach},\quad
\text{for }f\in\Banach.
\end{equation}
By the H\"{o}lder inequality, we have that
\begin{equation}\label{eq:RKBS-imbedding-dual-e2}
\abs{G_{\zeta}(f)}=\abs{\int_{\Domain}f(\vx)\zeta(\vx)\mu(\ud\vx)}
\leq\norm{f}_{\Leb_q(\Domain)}\norm{\zeta}_{\Leb_p(\Domain)},
\quad
\text{for }f\in\VecSpace.
\end{equation}
Combining inequalities~\eqref{eq:RKBS-imbedding-dual-e1} and~\eqref{eq:RKBS-imbedding-dual-e2}, we find that
\begin{equation}\label{eq:RKBS-imbedding-dual-e3}
\abs{G_{\zeta}(f)}\leq\norm{\Phi}_{\Leb_q(\Domain)}\norm{f}_{\Banach}\norm{\zeta}_{\Leb_p(\Domain_1)},
\quad
\text{for }f\in\VecSpace,
\end{equation}
which implies that $G_{\zeta}$ is a continuous linear functional on $\VecSpace$.
On the other hand, Proposition~\ref{p:RKBS-dense} ensures that $\VecSpace$ is dense in $\Banach$. As a result, $G_{\zeta}$ can be uniquely extended to a continuous linear functional on $\Banach$ by the Hahn-Banach extension theorem, that is, $G_{\zeta}\in\Banach'$.
Since $\Fun$ and $\Banach'$ are isometrically isomorphic, there exists a function $g_{\zeta}\in\Fun$ which is an equivalent element of $G_{\zeta}$, such that
$$
G_{\zeta}(f)=\langle f,g_{\zeta} \rangle_{\Banach}\ \  \mbox{for all}\ \ f\in\Banach.
$$
Using the reproducing properties (iii)-(iv), we observe for all $\vy\in\Domain'$ that
\[
G_{\zeta}\left(K(\cdot,\vy)\right)=\langle K(\cdot,\vy),g_{\zeta} \rangle_{\Banach}
=g_{\zeta}(\vy);
\]
hence,
\[
(I_K\zeta)(\vy)=\int_{\Domain}K(\vx,\vy)\zeta(\vx)\mu(\ud\vx)=G_{\zeta}\left(K(\cdot,\vy)\right)
=g_{\zeta}(\vy).
\]
This ensures that
$$
I_K\zeta=g_{\zeta}\in\Fun\cong\Banach'.
$$
Moreover, inequality~\eqref{eq:RKBS-imbedding-dual-e3} gives
\[
\norm{g_{\zeta}}_{\Banach'}=\norm{G_{\zeta}}_{\Banach'}\leq \norm{\Phi}_{\Leb_q(\Domain)}\norm{\zeta}_{\Leb_p(\Domain')}.
\]
Therefore, the integral operator $I_K$ is also a continuous linear operator.

Next we verify equation~\eqref{eq:imbeding-RKBS-dual}. Let $f\in\VecSpace$ be arbitrary. Then $f$ can be represented as a linear combination of finite elements $K(\cdot,\vy_1),\ldots,K(\cdot,\vy_N)$ in $\Kset_K$ in the form that
$$
f=\sum_{k\in\NN_N}c_kK(\cdot,\vy_k).
$$
By the reproducing properties~(iii)-(iv), we have for all $\zeta\in\Leb_p(\Domain)$ that
\begin{align*}
\int_{\Domain}f(\vx)\zeta(\vx)\mu(\ud\vx)&=\sum_{k\in\NN_N}c_k\int_{\Domain}K(\vx,\vy_k)\zeta(\vx)\mu(\ud\vx)\\
&=\sum_{k\in\NN_N}c_k\left(I_K\zeta\right)(\vy_k)\\
&=\sum_{k\in\NN_N}c_k\langle K(\cdot,\vy_k),I_K(\zeta) \rangle_{\Banach}\\
&=\langle f,I_K(\zeta) \rangle_{\Banach}.
\end{align*}
Using the density of $\VecSpace$ in $\Banach$ and the continuity of the identity map from $\Banach$ into $\Leb_q(\Domain)$, we obtain the general case of $f\in\Banach$ by the continuous extension. Therefore, equation \eqref{eq:imbeding-RKBS-dual} holds.
\end{proof}
%////////////////////////////////////////////////////////////////////////////////////////////////////////////////////////

%////////////////////////////////////////////////////////////////////////////////////////////////////////////////////////
\begin{corollary}\label{c:RKBS-imbedding-dual}
Let $1\leq p,q\leq\infty$ such that $p^{-1}+q^{-1}=1$ and
let the kernel $K\in\Leb_0(\Domain\times\Domain')$
be the two-sided reproducing kernel of the two-sided reproducing kernel Banach space $\Banach$ such that $K(\cdot,\vy)\in\Leb_q(\Domain)$ for all $\vy\in\Domain'$ and the map $\vx\mapsto\norm{K(\vx,\cdot)}_{\Banach'}\in\Leb_q(\Domain)$.
If $\Banach$ satisfies the $\mu$-measure zero condition,
then the range $I_K\left(\Leb_p(\Domain)\right)$ is weakly* dense in the dual space $\Banach'$
and
the range $I_K\left(\Leb_p(\Domain)\right)$ is dense in the dual space $\Banach'$ when $\Banach$ is reflexive and $1<p,q<\infty$.
\end{corollary}
%////////////////////////////////////////////////////////////////////////////////////////////////////////////////////////
\begin{proof}
The proof follows from a consequence of the general properties of adjoint mappings in \cite[Theorem~3.1.17]{Megginson1998}.

According to Proposition~\ref{p:RKBS-imbedding-dual}, the integral operator $I_K:\Leb_p(\Domain)\to\Banach'$ is continuous.
Equation~\eqref{eq:imbeding-RKBS-dual} shows that the integral operator $I_K:\Leb_p(\Domain)\to\Banach'$ is the adjoint operator of the identity map $I:\Banach\to\Leb_q(\Domain)$.
Moreover, since $\Banach$ satisfies the $\mu$-measure zero condition, the identity map $I:\Banach\to\Leb_q(\Domain)$ is injective. This ensures that the weakly* closure of the range of $I_K$ is equal to $\Banach'$.

The last statement is true, because the reflexivity of $\Banach$ and $\Leb_q(\Domain)$ for $1<q<\infty$ ensures that
the closure of the range of $I_K$ is equal to $\Banach'$.
\end{proof}
%////////////////////////////////////////////////////////////////////////////////////////////////////////////////////////

To close this section, we discuss two applications of the imbedding theorems of RKBSs.

\subsection*{Imbedding Probability Measures}

Since the injectivity of the integral operator $I_K$ may not hold, the operator $I_K$ may not be the imbedding operator from $\Leb_p(\Domain)$ into $\Banach$.
Nevertheless, Proposition~\ref{p:RKBS-imbedding-dual} provides a useful tool to compute the imbedding probability measure in Banach spaces.
We consider computing the integral
\[
I_f:=\int_{\Domain}f(\vx)\mu_{\zeta}(\ud\vx),
\]
where the probability measure $\mu_{\zeta}$ is written as $\mu_{\zeta}(\ud\vx)=\zeta(\vx)\mu(\ud\vx)$ for a function $\zeta\in\Leb_p(\Domain)$ and
$f$ belongs to the two-sided RKBS $\Banach$ given in Proposition~\ref{p:RKBS-imbedding-dual}.
Usually it is impossible to calculate the exact value $I_f$.
In practice, we approximate $I_f$ by a countable sequence of easy-implementation integrals
\[
I_n:=\int_{\Domain}f_n(\vx)\mu_{\zeta}(\ud\vx),
\]
hoping that $\lim_{n\to\infty}I_n=I_f$.
If we find a countable sequence $f_n\in\Banach$ such that $\norm{f_n-f}_{\Banach}\to0$ when $n\to\infty$, then we have that
\[
\lim_{n\to\infty}\abs{I_n-I_f}=0.
\]
The reason is that Proposition~\ref{p:RKBS-imbedding-dual} ensures that
\begin{align*}
\abs{I_n-I_f}&=\abs{\int_{\Domain}\left(f_n(\vx)-f(\vx)\right)\mu_{\zeta}(\ud\vx)}\\
&=\abs{\int_{\Domain}\left(f_n-f\right)(\vx)\zeta(\vx)\mu(\ud\vx)}\\
&=\abs{\langle f_n-f, I_K(\zeta) \rangle_{\Banach}}\\
&\leq\norm{f_n-f}_{\Banach}\norm{I_K(\zeta)}_{\Banach'}.
\end{align*}
The estimator $f_n$ may be constructed by the reproducing kernel $K$ for the discrete data information $\left\{\left(\vx_k,y_k\right):k\in\NN_N\right\}\subseteq\Domain\times\RR$ induced by $f$.

\subsection*{Fr\'{e}chet Derivatives of Loss Risks}

In the learning theory, one often considers the expected loss of $f\in\Leb_{\infty}(\Domain)$ given by
\[
\int_{\Domain\times\RR}L(\vx,y,f(\vx))\PP\left(\ud\vx,\ud y\right)
=\frac{1}{C_{\mu}}\int_{\Domain}\int_{\RR}L(\vx,y,f(\vx))\PP(\ud y|\vx)\mu(\ud\vx),
\]
where $L:\Domain\times\RR\times\RR\to[0,\infty)$ is the given loss function, $C_{\mu}:=\mu(\Domain)$ and $\PP(y|\vx)$ is the regular conditional probability of the probability $\PP(\vx,y)$.
Here, the measure $\mu(\Domain)$ is usually assumed to be finite for practical machine problems.
In this case, we define a function
\[
H(\vx,t):=\frac{1}{C_{\mu}}\int_{\RR}L(\vx,y,t)\PP(\ud y|\vx),\quad\text{for }\vx\in\Domain\text{ and }t\in\RR.
\]

Next, we suppose that the function $t\mapsto H(\vx,t)$ is differentiable for any fixed $\vx\in\Domain$. We then write
\[
H_t(\vx,t):=\frac{\ud}{\ud t}H(\vx,t),\quad\text{for }\vx\in\Domain\text{ and }t\in\RR.
\]
Furthermore, we suppose that $\vx\mapsto H(\vx,f(\vx))\in\Leb_1(\Domain)$ for all $f\in\Leb_{\infty}(\Domain)$. Then the operator
\[
\risk_{\infty}(f):=\int_{\Domain}H(\vx,f(\vx))\mu(\ud\vx),\quad\text{for }f\in\Leb_{\infty}(\Domain)
\]
is clearly well-defined.
Finally, we suppose that $\vx\mapsto H_t(\vx,f(\vx))\in\Leb_1(\Domain)$ whenever $f\in\Leb_{\infty}(\Domain)$. Then the Fr\'{e}chet derivative of $\risk_{\infty}$ can be written as
\begin{equation}\label{eq:Frechet-risk}
\langle h,\Frechet\risk_{\infty}(f)\rangle_{\Leb_{\infty}(\Domain)}
=\int_{\Domain}h(\vx)H_t(\vx,f(\vx))\mu(\ud\vx),\quad
\text{for all }h\in\Leb_{\infty}(\Domain).
\end{equation}
Here, the operator $T$ from a normed space $\Banach_1$ into another normed space $\Banach_2$ is said to be Fr\'{e}chet differentiable at $f\in\Banach_1$ if there is a continuous linear operator
$\Frechet T(f):\Banach_1\to\Banach_2$ such that
\[
\lim_{\norm{h}_{\Banach_1}\to0}\frac{\norm{T(f+h)-T(f)-\Frechet T(f)(h)}_{\Banach_2}}{\norm{h}_{\Banach_1}}=0,
\]
and the continuous linear operator $\Frechet T(f)$ is called the Fr\'{e}chet derivative of $T$ at $f$ (see~\cite[Definition~A.5.14]{SteinwartChristmann2008}).
We next discuss the Fr\'{e}chet derivatives of loss risks in two-sided RKBSs based on the above conditions.

%////////////////////////////////////////////////////////////////////////////////////////////////////////////////////////
\begin{corollary}\label{c:RKBS-FrechetDerivative}
Let $K\in\Leb_0(\Domain\times\Domain')$ be the two-sided reproducing kernel of the two-sided reproducing kernel Banach space $\Banach$ such that $\vx\mapsto K(\vx,\vy)\in\Leb_{\infty}(\Domain)$ for all $\vy\in\Domain'$ and the map $\vx\mapsto\norm{K(\vx,\cdot)}_{\Banach'}\in\Leb_{\infty}(\Domain)$.
If the function $H:\Domain\times\RR\to[0,\infty)$ satisfies that the map $t\mapsto H(\vx,t)$ is differentiable for any fixed $\vx\in\Domain$ and $\vx\mapsto H(\vx,f(\vx)),\vx\mapsto H_t(\vx,f(\vx))\in\Leb_1(\Domain)$ whenever $f\in\Leb_{\infty}(\Domain)$, then
the operator
\[
\risk(f):=\int_{\Domain}H(\vx,f(\vx))\mu(\ud\vx),\quad\text{for }f\in\Banach,
\]
is well-defined and
the Fr\'{e}chet derivative of $\risk$ at $f\in\Banach$
can be represented as
\[
\Frechet\risk(f)
=\int_{\Domain}H_t(\vx,f(\vx))K(\vx,\cdot)\mu(\ud\vx).
\]
\end{corollary}
%////////////////////////////////////////////////////////////////////////////////////////////////////////////////////////
\begin{proof}
According to the imbedding of RKBSs (Propositions~\ref{p:RKBS-imbedding} and~\ref{p:RKBS-imbedding-dual}), the identity map $I:\Banach\to\Leb_{\infty}(\Domain)$ and the integral operator $I_K:\Leb_1(\Domain)\to\Banach'$ are continuous.
This ensures that the operator $\risk$ is well-defined for all $f\in\Banach$ and
\[
I_K(\zeta_f)\in\Banach',
\]
for all
\[
\zeta_f(\vx):=H_t(\vx,f(\vx)),\quad\text{for }\vx\in\Domain,
\]
driven by any $f\in\Banach$.

Clearly, $\risk_{\infty}=\risk\circ I$.
Using the chain rule of Fr\'{e}chet derivatives, we have that
\[
\Frechet\risk_{\infty}(f)(h)=
\Frechet(\risk\circ I)(f)(h)
=\Frechet\risk(If)\circ\Frechet I(f)(h)
=\Frechet\risk(f)(h),
\]
whenever $f,h\in\Banach$.
Therefore, we combine equations~\eqref{eq:imbeding-RKBS-dual} and \eqref{eq:Frechet-risk} to conclude that
\begin{align*}
\langle h,\Frechet\risk(f)\rangle_{\Banach}
=\langle h,\Frechet\risk_q(f)\rangle_{\Leb_{\infty}(\Domain)}
=\int_{\Domain}h(\vx)\zeta_f(\vx)\mu(\ud\vx)
=\langle h,I_K(\zeta_f) \rangle_{\Banach},
\end{align*}
for all $h\in\Banach$, whenever $f\in\Banach$. This means that $\Frechet\risk(f)=I_K(\zeta_f)$.
\end{proof}
%////////////////////////////////////////////////////////////////////////////////////////////////////////////////////////

%------------------------------------------------------------------------------------------------------------------------
\section{Compactness}\label{s:Compactness}
%------------------------------------------------------------------------------------------------------------------------
\sectionmark{Compactness}

In this section, we investigate the compactness of RKBSs. The compactness of RKBSs will play an important role in applications such as support vector machines.

Let the function space
\[
\Linfty(\Domain):=\left\{f\in\Leb_0(\Domain): \sup_{\vx\in\Domain}\abs{f(\vx)}<\infty\right\},
\]
be equipped with the uniform norm
\[
\norm{f}_{\infty}:=\sup_{\vx\in\Domain}\abs{f(\vx)}.
\]
Since $\Linfty(\Domain)$ is the collection of all bounded functions, the space $\Linfty(\Domain)$ is a subspace of the infinity-norm Lebesgue space $\Leb_{\infty}(\Domain)$.
We next verify that the identity map $I:\Banach\to \Linfty(\Domain)$ is a compact operator, where $\Banach$ is a RKBS.

We first review some classical results of the
compactness of normed spaces for the purpose of establishing the compactness of RKBSs.
A set $\Eset$ of a normed space $\Banach_1$ is called relatively compact if the closure of $\Eset$ is compact in the completion of $\Banach_1$.
A linear operator $T$ from a normed space $\Banach_1$ into another normed space $\Banach_2$ is compact if $T(\Eset)$ is a relatively compact set of $\Banach_2$ whenever $\Eset$ is a bounded set of $\Banach_1$ (see \cite[Definition~3.4.1]{Megginson1998}).

Let $\Cont(\Domain)$ be the collection of all continuous functions defined on a compact Hausdorff space $\Domain$ equipped with the uniform norm. We say that a set $\Sset$ of $\Cont(\Domain)$ is equicontinuous if, for any $\vx\in\Domain$ and any $\epsilon>0$, there is a neighborhood $U_{\vx,\epsilon}$ of $\vx$ such that $\abs{f(\vx)-f(\vy)}<\epsilon$ whenever $f\in\Sset$ and $\vy\in U_{\vx,\epsilon}$ (see \cite[Definition~3.4.13]{Megginson1998}).
The Arzel\`{a}-Ascoli theorem (\cite[Theorem~3.4.14]{Megginson1998}) will be applied in the following development because
the theorem guarantees that a set $\Sset$ of $\Cont(\Domain)$ is relatively compact if and only if the set $\Sset$ is bounded and equicontinuous.

%////////////////////////////////////////////////////////////////////////////////////////////////////////////////////////
\begin{proposition}\label{p:RKBS-compact}
Let $\Banach$ be the right-sided RKBS with the right-sided reproducing kernel $K\in\Leb_0(\Domain\times\Domain')$.
If the right-sided kernel set $\Kset_{K}'$ is compact in the dual space $\Banach'$ of $\Banach$, then the identity map from $\Banach$ into $\Linfty(\Domain)$ is compact.
\end{proposition}
%////////////////////////////////////////////////////////////////////////////////////////////////////////////////////////
\begin{proof}
It suffices to show that the unit ball
$$
B_{\Banach}:=\left\{f\in\Banach:\norm{f}_{\Banach}\leq1\right\}
$$
of $\Banach$ is relatively compact in $\Linfty(\Domain)$.

First, we construct a new semi-metric $d_K$ on the domain $\Domain$. Let
\[
d_K(\vx,\vy):=\norm{K(\vx,\cdot)-K(\vy,\cdot)}_{\Banach'},\quad\text{for all }\vx,\vy\in\Domain.
\]
Since $\Kset_{K}'$ is compact in $\Banach'$, the semi-metric space $\left(\Domain,d_K\right)$ is compact.
Let $\Cont(\Domain,d_K)$ be the collection of all continuous functions defined on $\Domain$ with respect to $d_K$ such that its norm is endowed with the uniform norm. Thus, $\Cont(\Domain,d_K)$ is a subspace of $\Linfty(\Domain)$.

Let $f\in\Banach$. According to the reproducing properties (i)-(ii), we have that
\begin{equation}\label{eq:RKBS-compact-1}
\abs{f(\vx)-f(\vy)}=\abs{\langle f,K(\vx,\cdot)-K(\vy,\cdot)\rangle_{\Banach}},\quad\text{for }\vx,\vy\in\Domain.
\end{equation}
By the Cauchy-Schwarz inequality, we observe that
\begin{equation}\label{eq:RKBS-compact-2}
\abs{\langle f,K(\vx,\cdot)-K(\vy,\cdot)\rangle_{\Banach}}
\leq\norm{f}_{\Banach}\norm{K(\vx,\cdot)-K(\vy,\cdot)}_{\Banach'},\quad\text{for }\vx,\vy\in\Domain.
\end{equation}
Combining equation~\eqref{eq:RKBS-compact-1} and inequality~\eqref{eq:RKBS-compact-2}, we obtain that
\[
\abs{f(\vx)-f(\vy)}\leq\norm{f}_{\Banach}d_K(\vx,\vy),\quad\text{for }\vx,\vy\in\Domain,
\]
which ensures that $f$ is Lipschitz continuous on $\Cont(\Domain,d_K)$ with a Lipschitz constant not larger than $\norm{f}_{\Banach}$.
This ensures that the unit ball $B_{\Banach}$ is equicontinuous on $\left(\Domain,d_K\right)$.

The compactness of $\Kset_{K}'$ of $\Banach'$ guarantees that $\Kset_{K}'$ is bounded in $\Banach'$. Hence, $\Phi\in \Linfty(\Domain)$.
By the reproducing properties (i)-(ii) and the Cauchy-Schwarz inequality, we confirm the uniform-norm boundedness of the unit ball $B_{\Banach}$, namely,
\[
\norm{f}_{\infty}
=\sup_{\vx\in\Domain}\abs{f(\vx)}
=\sup_{\vx\in\Domain}\abs{\langle f,K(\vx,\cdot) \rangle_{\Banach}}
\leq\norm{f}_{\Banach}\sup_{\vx\in\Domain}\norm{K(\vx,\cdot)}_{\Banach'}
\leq \norm{\Phi}_{\infty}<\infty,
\]
for all $f\in B_{\Banach}$.
Therefore, the Arzel\`{a}-Ascoli theorem guarantees that $B_{\Banach}$ is relatively compact in $\Cont(\Domain,d_K)$ and thus in $\Linfty(\Domain)$.
\end{proof}
%////////////////////////////////////////////////////////////////////////////////////////////////////////////////////////

The compactness of RKHSs plays a crucial role in estimating the error bounds and the convergence rates of the learning solutions in RKHSs.
In Section~\ref{s:OracleInequality}, the fact that the bounded sets of RKBSs are equivalent to the relatively compact sets of $\Linfty(\Domain)$ will be used to prove the oracle inequality for the RKBSs as shown in book~\cite{SteinwartChristmann2008} and papers~\cite{CuckerSmale2002,SmaleZhou2004} for the RKHSs.

%------------------------------------------------------------------------------------------------------------------------
\section{Representer Theorem}\label{s:RepThm}
%------------------------------------------------------------------------------------------------------------------------
\sectionmark{Representer Theorem}

In this section, we focus on finding the finite dimensional optimal solutions to minimize the regularized risks over RKBSs. Specifically,
we solve the following learning problems in the RKBS $\Banach$
\[
\min_{f\in\Banach}
\left\{\frac{1}{N}\sum_{k\in\NN_N}L\left(\vx_k,y_k,f(\vx_k)\right)+R\left(\norm{f}_{\Banach}\right)\right\},
\]
or
\[
\min_{f\in\Banach}
\left\{\int_{\Domain\times\RR}L\left(\vx,y,f(\vx)\right)\PP(\ud\vx,\ud y)+R\left(\norm{f}_{\Banach}\right)\right\}.
\]
These learning problems connect to the support vector machines to be discussed in Chapter~\ref{char-SVM}.

In this article, the representer theorems in RKBSs are expressed by the combinations of G\^{a}teaux/Fr\'echet derivatives
and reproducing kernels. The representer theorems presented next are developed from a point of view different from those in papers~\cite{ZhangXuZhang2009,FasshauerHickernellYe2013}
where the representer theorems were derived from the dual elements and the semi-inner products.
The classical representer theorems in RKHSs~\cite[Theorem 5.5 and 5.8]{SteinwartChristmann2008} can be viewed as a special case of the representer theorems in RKBSs.

\subsection*{Optimal Recovery}

We begin by considering the optimal recovery in RKBSs for given pairwise distinct data points $X:=\left\{\vx_k:k\in\NN_N\right\}\subseteq\Domain$ and associated data values $Y:=\left\{y_k:k\in\NN_N\right\}\subseteq\RR$, that is,
\[
\min_{f\in\Banach}\norm{f}_{\Banach}\text{ s.t. }f(\vx_k)=y_k,\text{ for all }k\in\NN_N,
\]
where $\Banach$ is the right-sided RKBS with the right-sided reproducing kernel $K\in\Leb_0(\Domain\times\Domain')$.

Let
\[
\Space_{X,Y}:=\left\{f\in\Banach: f(\vx_k)=y_k,\text{ for all }k\in\NN_N\right\}.
\]
If $\Space_{X,Y}$ is a null set, then the optimal recovery may have no meaning for the given data $X$ and $Y$. Hence, we assume that $\Space_{X,Y}$ is always non-null for any choice of finite data.

\emph{Linear Independence of Kernel Sets:}
Now we show that the linear independence of $K(\vx_1,\cdot),\ldots,K(\vx_N,\cdot)$ implies that $\Space_{X,Y}$ is always non-null.
The reproducing properties (i)-(ii) show that
\[
\left\langle f,\sum_{k\in\NN_N}c_kK(\vx_k,\cdot) \right\rangle_{\Banach}
=\sum_{k\in\NN_N}c_k\langle f,K(\vx_k,\cdot) \rangle_{\Banach}
=\sum_{k\in\NN_N}c_kf(\vx_k),\quad
\text{for all }f\in\Banach,
\]
which ensures that
$$
\sum_{k\in\NN_N}c_kK(\vx_k,\cdot)=0\ \ \mbox{if and only if}\ \ \sum_{k\in\NN_N}c_kf(\vx_k)=0, \ \ \mbox{for all}\ \ f\in\Banach.
$$
In other words, $\vc:=\left(c_k:k\in\NN_N\right)=\v0$ if and only if
\[
\sum_{k\in\NN_N}c_kz_k=0,\quad \text{for all }\vz:=\left(z_k:k\in\NN_N\right)\in\RR^N.
\]
If we check that $\vc=\v0$ if and only if
\[
\sum_{k\in\NN_N}c_kK(\vx_k,\cdot)=0,
\]
then
there exists at least one $f\in\Banach$ such that $f(\vx_k)=y_k$ for all $k\in\NN_N$.
Therefore, when $K(\vx_1,\cdot),\ldots,K(\vx_N,\cdot)$ are linearly independent, then $\Space_{X,Y}$ is non-null.
Since the data points $X$ are chosen arbitrarily, the linear independence of the right-sided kernel set $\Kset_{K}'$ is needed.

To further discuss optimal recovery, we require the RKBS $\Banach$ to satisfy additional conditions such as reflexivity, strict convexity, and smoothness (see \cite{Megginson1998}).

\emph{Reflexivity:} The RKBS $\Banach$ is reflexive if $\Banach''\cong\Banach$.

\emph{Strict Convexity:} The RKBS $\Banach$ is strictly convex (rotund) if
$$
\norm{tf+(1-t)g}_{\Banach}<1\ \ \mbox{whenever}\ \  \norm{f}_{\Banach}=\norm{g}_{\Banach}=1, 0<t<1.
$$
By \cite[Corollary~5.1.12]{Megginson1998}, the strict convexity of $\Banach$ implies that
$$
\norm{f+g}_{\Banach}<\norm{f}_{\Banach}+\norm{g}_{\Banach}\ \ \mbox{when}\ \ f\neq g.
$$

Moreover, \cite[Corollary~5.1.19]{Megginson1998} (M. M. Day, 1947) provides that any nonempty closed convex subset of a Banach space is a Chebyshev set if the Banach space is reflexive and strictly convex. This guarantees that, for each $f\in\Banach$ and any nonempty closed convex subset $\Eset\subseteq\Banach$,
there exists an exactly one $g\in\Eset$ such that $\norm{f-g}_{\Banach}=\text{dist}(f,\Eset)$ if the RKBS $\Banach$ is reflexive and strictly convex.

\emph{Smoothness:} The RKBS $\Banach$ is smooth (G\^{a}teaux differentiable) if
the normed operator $\normopt$ is G\^{a}teaux differentiable for all $f\in\Banach\backslash\{0\}$, that is,
\[
\lim_{\tau\to0}\frac{\norm{f+\tau h}_{\Banach}-\norm{f}_{\Banach}}{\tau}\text{ exists},\quad\text{for all }h\in\Banach,
\]
where
the normed operator $\normopt$ is given by
\[
\normopt(f):=\norm{f}_{\Banach}.
\]
The smoothness of $\Banach$ implies that
the G\^{a}teaux derivative $\Gateaux\normopt$ of $\normopt$ is well-defined on $\Banach$. In other words,
for each $f\in\Banach\backslash\{0\}$, there exists a continuous linear functional $\Gateaux\normopt(f)$ defined on $\Banach$ such that
\[
\langle h,\Gateaux\normopt(f) \rangle_{\Banach}
=\lim_{\tau\to0}\frac{\norm{f+\tau h}_{\Banach}-\norm{f}_{\Banach}}{\tau},\quad\text{for all }h\in\Banach.
\]
This means that $\Gateaux\normopt$ is an operator from $\Banach\backslash\{0\}$ into $\Banach'$.
According to \cite[Proposition~5.4.2 and Theorem~5.4.17]{Megginson1998} (S. Banach, 1932), the G\^{a}teaux derivative $\Gateaux\normopt$
restricted to the unit sphere
$$
S_{\Banach}:=\left\{f\in\Banach:\norm{f}_{\Banach}=1\right\}
$$
is equal to the spherical image from $S_{\Banach}$ into $S_{\Banach'}$, that is,
\[
\norm{\Gateaux\normopt(f)}_{\Banach'}=1\text{ and }\langle f,\Gateaux\normopt(f) \rangle_{\Banach}=1,\quad\text{for }f\in S_{\Banach}.
\]
The equation
\[
\frac{\norm{\alpha f+\tau h}_{\Banach}-\norm{\alpha f}_{\Banach}}{\tau}
=\frac{\norm{f+\alpha^{-1}\tau h}_{\Banach}-\norm{f}_{\Banach}}{\alpha^{-1}\tau},
\quad\text{for }f\in\Banach\backslash\{0\}\text{ and }\alpha>0,
\]
implies that
$$
\Gateaux\normopt(\alpha f)=\Gateaux\normopt(f).
$$
Hence, by taking $\alpha=\norm{f}_{\Banach}^{-1}$, we also check that
\begin{equation}\label{eq:Gateaux-Der-norm}
\norm{\Gateaux\normopt(f)}_{\Banach'}=1\text{ and }
\norm{f}_{\Banach}=\langle f,\Gateaux\normopt(f) \rangle_{\Banach}.
\end{equation}
Furthermore, the map $\Gateaux\normopt:S_{\Banach}\to S_{\Banach'}$ is bijective, because
the spherical image is a one-to-one map from $S_{\Banach}$ onto $S_{\Banach'}$ when $\Banach$ is reflexive, strictly convex, and smooth (see \cite[Corollary~5.4.26]{Megginson1998}).
In the classical sense, the G\^{a}teaux derivative $\Gateaux\normopt$ does not exist at $0$. However, to simplify the notation and presentation, we
redefine the G\^{a}teaux derivative of the map $f\mapsto\norm{f}_{\Banach}$ by
\begin{equation}\label{eq:Gateaux-norm}
\GateauxNorm:=
\begin{cases}
\Gateaux\normopt,&\text{ when }f\neq0,\\
0,&\text{ when }f=0.
\end{cases}
\end{equation}
According to the above discussion,
we determine that $\norm{\GateauxNorm(f)}_{\Banach'}=1$ when $f\neq0$ and the restricted map $\GateauxNorm|_{S_{\Banach}}$ is bijective.
On the other hand, the isometrical isomorphism between the dual space $\Banach'$ and the function space $\Fun$ implies that
the G\^{a}teaux derivative $\GateauxNorm(f)$ can be viewed as an equivalent element in $\Fun$ so that it can be represented as the reproducing kernel $K$.

To complete the proof of optimal recovery of RKBS, we need to find the relationship of G\^{a}teaux derivatives and orthogonality.
According to \cite{James1947},
a function $f\in\Banach\backslash\{0\}$ is said to be orthogonal to another function $g\in\Banach$ if
\[
\norm{f+\tau g}_{\Banach}\geq\norm{f}_{\Banach},\quad \text{for all }\tau\in\RR.
\]
This indicates that a function $f\in\Banach\backslash\{0\}$ is orthogonal to a subspace $\Space$ of $\Banach$ if $f$ is orthogonal to each element $h\in\Space$, that is,
\[
\norm{f+h}_{\Banach}\geq\norm{f}_{\Banach},\quad \text{for all }h\in\Space.
\]

%/////////////////////////////////////////////////////////////////////////////////////////////////////////////////////
\begin{lemma}\label{l:Gateaux-Der-Orth}
If the Banach space $\Banach$ is smooth, then $f\in\Banach\backslash\{0\}$ is orthogonal to $g\in\Banach$ if and only if $\langle g,\Gateaux\normopt(f) \rangle_{\Banach}=0$ where $\Gateaux\normopt(f)$ is the G\^{a}teaux derivative of the norm of $\Banach$ at $f$, and a nonzero $f$ is orthogonal to a subspace $\Space$ of $\Banach$ if and only if
\[
\langle h,\Gateaux\normopt(f) \rangle_{\Banach}=0,\quad \text{for all }h\in\Space.
\]
\end{lemma}
%/////////////////////////////////////////////////////////////////////////////////////////////////////////////////////
\begin{proof}
First we
suppose that $\langle g,\Gateaux\normopt(f) \rangle_{\Banach}=0$.
For any $h\in\Banach$ and any $0<t<1$, we have that
\[
\norm{f+th}_{\Banach}=\norm{(1-t)f+t(f+h)}_{\Banach}
\leq(1-t)\norm{f}_{\Banach}+t\norm{f+h}_{\Banach},
\]
which ensures that
\[
\norm{f+h}_{\Banach}\geq \norm{f}_{\Banach}+\frac{\norm{f+th}_{\Banach}-\norm{f}_{\Banach}}{t}.
\]
Thus,
\[
\norm{f+h}_{\Banach}\geq \norm{f}_{\Banach}+\lim_{t\downarrow0}\frac{\norm{f+th}_{\Banach}-\norm{f}_{\Banach}}{t}
=\norm{f}_{\Banach}+\langle h,\Gateaux\normopt(f) \rangle_{\Banach}.
\]
Replacing $h=\tau g$ for any $\tau\in\RR$, we obtain that
\[
\norm{f+\tau g}_{\Banach}\geq\norm{f}_{\Banach}+\langle \tau g,\Gateaux(f) \rangle_{\Banach}
=\norm{f}_{\Banach}+\tau\langle g,\Gateaux\normopt(f) \rangle_{\Banach}=\norm{f}_{\Banach}.
\]
That is, $f$ is orthogonal to $g$.

Next, we
suppose that $f$ is orthogonal to $g$. Then $\norm{f+\tau g}_{\Banach}\geq\norm{f}_{\Banach}$ for all $\tau\in\RR$; hence
\[
\frac{\norm{f+\tau g}_{\Banach}-\norm{f}_{\Banach}}{\tau}\geq0\text{ when }\tau>0,\quad
\frac{\norm{f+\tau g}_{\Banach}-\norm{f}_{\Banach}}{\tau}\leq0\text{ when }\tau<0,
\]
which ensures that
\[
\lim_{\tau\downarrow0}\frac{\norm{f+\tau g}_{\Banach}-\norm{f}_{\Banach}}{\tau}\geq0,\quad
\lim_{\tau\uparrow0}\frac{\norm{f+\tau g}_{\Banach}-\norm{f}_{\Banach}}{\tau}\leq0.
\]
Therefore
\[
\langle g,\Gateaux\normopt(f) \rangle_{\Banach}=
\lim_{\tau\to0}\frac{\norm{f+\tau g}_{\Banach}-\norm{f}_{\Banach}}{\tau}=0.
\]

This guarantees that $f$ is orthogonal to $\Space$ if and only if
$\langle h,\Gateaux\normopt(f) \rangle_{\Banach}=0$ for all $h\in\Space$.
\end{proof}
%/////////////////////////////////////////////////////////////////////////////////////////////////////////////////////

Using the G\^{a}teaux derivative we reprove the optimal recovery of RKBSs in a way similar to that used in~\cite[Lemma~3.1]{FasshauerHickernellYe2013} for proving the optimal recovery of semi-inner-product RKBSs.

%/////////////////////////////////////////////////////////////////////////////////////////////////////////////////////
\begin{lemma}\label{l:RKBS-opt-rep}
Let $\Banach$ be a right-sided reproducing kernel Banach space with the right-sided reproducing kernel $K\in\Leb_0(\Domain\times\Domain')$.
If the right-sided kernel set $\Kset_{K}'$ is linearly independent and $\Banach$ is reflexive, strictly convex, and smooth,
then for any pairwise distinct data points $X\subseteq\Domain$ and any associated data values $Y\subseteq\RR$,  the minimization problem
\begin{equation}\label{e:opt-simple}
\min_{f\in\Space_{X,Y}}\norm{f}_{\Banach},
\end{equation}
has a unique optimal solution $s$ with the G\^{a}teaux derivative $\GateauxNorm(s)$, defined by equation \eqref{eq:Gateaux-norm},
being a linear combination of $K(\vx_1,\cdot),\ldots,K(\vx_N,\cdot)$.
\end{lemma}
%/////////////////////////////////////////////////////////////////////////////////////////////////////////////////////
\begin{proof}
We shall prove this lemma by combining the G\^{a}teaux derivatives and reproducing properties to describe the orthogonality.

We first verify the uniqueness and existence of the optimal solution $s$.

The linear independence of $\Kset_{K}'$ ensures that the set $\Space_{X,Y}$ is not empty.
Now we show that $\Space_{X,Y}$ is convex. Take any $f_1,f_2\in\Space_{X,Y}$ and any $t\in(0,1)$. Since
$$
t f_1(\vx_k)+(1-t)f_2(\vx_k)=t y_k+(1-t)y_k=y_k, \ \ \mbox{for all}\ \ k\in\NN_N,
$$
we have that $t f_1+(1-t)f_2\in\Space_{X,Y}$.
Next, we verify that $\Space_{X,Y}$ is closed.
We see that any convergent sequence $\left\{f_n:n\in\NN\right\}\subseteq\Space_{X,Y}$ and $f\in\Banach$ satisfy $\norm{f_n-f}_{\Banach}\to0$ as $n\to\infty$. Then $f_n\overset{\text{weak}-\Banach}{\longrightarrow}f$ as $n\to\infty$. According to Proposition~\ref{p:weakly-conv-RKBS}, we obtain that
$$
f(\vx_k)=\lim_{n\to\infty}f_n(\vx_k)=y_k, \ \ \mbox{for all}\ \ k\in\NN_N,
$$
which ensures that $f\in\Space_{X,Y}$.
Thus $\Space_{X,Y}$ is a nonempty closed convex subset of the RKBS $\Banach$.

Moreover, the reflexivity and strict convexity of $\Banach$ ensure that the closed convex set $\Space_{X,Y}$ is a Chebyshev set.
This ensures that there exists a unique $s$ in $\Space_{X,Y}$ such that
$$
\norm{s}_{\Banach}=\text{dist}\left(0,\Space_{X,Y}\right)=\inf_{f\in\Space_{X,Y}}\norm{f}_{\Banach}.
$$

When $s=0$, $\GateauxNorm(s)=\Gateaux(s)=0$. Hence, in this case the conclusion holds true. Next, we assume that $s\neq0$.
Recall that
\[
\Space_{X,0}=\left\{f\in\Banach: f(\vx_k)=0,\text{ for all }k\in\NN_N\right\}.
\]
For any $f_1,f_2\in\Space_{X,0}$ and any $c_1,c_2\in\RR$,
since $c_1f_1(\vx_k)+c_2f_2(\vx_k)=0$ for all $k\in\NN_N$, we have that $c_1f_1+c_2f_2\in\Space_{X,0}$,
which guarantees that $\Space_{X,0}$ is a subspace of $\Banach$.
Because $s$ is the optimal solution of the minimization problem~\eqref{e:opt-simple} and $s+\Space_{X,0}=\Space_{X,Y}$,
we determine that
$$
\norm{s+h}_{\Banach}\geq\norm{s}_{\Banach}, \ \ \mbox{for all}\ \  h\in\Space_{X,0}.
$$
This means that $s$ is orthogonal to the subspace $\Space_{X,0}$.
Since the RKBS $\Banach$ is smooth, the orthogonality of the G\^{a}teaux derivatives of the normed operators given in Lemma~\ref{l:Gateaux-Der-Orth} ensures that
\[
\langle h,\GateauxNorm(s) \rangle_{\Banach}=0,\text{ for all }h\in\Space_{X,0},
\]
which shows that
\begin{equation}\label{eq:RKBS-opt-rep-l1}
\GateauxNorm(s)\in\Space_{X,0}^{\perp}.
\end{equation}
Using the reproducing properties (i)-(ii), we check that
\begin{equation}\label{eq:RKBS-opt-rep-l2}
\begin{split}
\Space_{X,0}&=\left\{f\in\Banach: \langle f,K(\vx_k,\cdot)\rangle_{\Banach}=f(\vx_k)=0,\text{ for all }k\in\NN_N\right\}\\
&=\left\{f\in\Banach: \langle f,h\rangle_{\Banach}=0,\text{ for all }h\in\Span\left\{K(\vx_k,\cdot):k\in\NN_N\right\}\right\}\\
&={}^{\perp}\Span\left\{K(\vx_k,\cdot):k\in\NN_N\right\}.
\end{split}
\end{equation}
Combining equations~\eqref{eq:RKBS-opt-rep-l1} and~\eqref{eq:RKBS-opt-rep-l2}, we obtain that
\[
\GateauxNorm(s)\in\left({}^{\perp}\Span\left\{K(\vx_k,\cdot):k\in\NN_N\right\}\right)^{\perp}
=\Span\left\{K(\vx_k,\cdot):k\in\NN_N\right\},
\]
by the characterizations of annihilators given in \cite[Proposition~2.6.6]{Megginson1998}.
There exist suitable parameters $\beta_1,\ldots,\beta_N\in\RR$ such that
\[
\GateauxNorm(s)=\sum_{k\in\NN_N}\beta_kK(\vx_k,\cdot).
\]
\end{proof}
%/////////////////////////////////////////////////////////////////////////////////////////////////////////////////////

Generally speaking, the G\^{a}teaux derivative $\GateauxNorm(s)$ is seen as a
continuous linear functional. More precisely, $\GateauxNorm(s)$ can also be rewritten as a linear combination of $\delta_{\vx_1},\ldots,\delta_{\vx_N}$, that is,
\begin{equation}\label{eq:RKBS-opt-rep-delta}
\GateauxNorm(s)=\sum_{k\in\NN_N}\beta_k\delta_{\vx_k},
\end{equation}
because the function $K(\vx_k,\cdot)$ is identical to the point evaluation functional $\delta_{\vx_k}$ for all $k\in\NN_N$.

\emph{Comparisons:}
Now we comment the difference of Lemma~\ref{l:RKBS-opt-rep} from the classical optimal recovery in Hilbert and Banach spaces.
\begin{itemize}
\item
When the RKBS $\Banach$ reduces to a Hilbert space framework, $\Banach$ becomes a RKHS and $s=\norm{s}_{\Banach}\GateauxNorm(s)$; hence Lemma~\ref{l:RKBS-opt-rep} is the generalization of the optimal recovery in RKHSs \cite[Theorem~13.2]{Wendland2005}, that is, the minimizer of the reproducing norms under the interpolatory functions in RKHSs is a linear combination of the kernel basis $K(\cdot,\vx_1),\ldots,K(\cdot,\vx_N)$.
The equation~\eqref{eq:RKBS-opt-rep-delta} implies that $\GateauxNorm(s)$ is a linear combination of the continuous linear functionals $\delta_{\vx_1},\ldots,\delta_{\vx_N}$.

\item
If the Banach space $\Banach$ is reflexive, strictly convex, and smooth, then $\Banach$ has a semi-inner product (see \cite{James1947,Lumer1961}).
In original paper \cite{ZhangXuZhang2009}, the RKBSs are mainly investigated by the semi-inner products.
In paper \cite{FasshauerHickernellYe2013}, it points out that the RKBSs can be reconstructed without the semi-inner products.
But, the optimal recovery in the renew RKBSs is still shown by the orthogonality semi-inner products in \cite[Lemma~3.1]{FasshauerHickernellYe2013} the same as in \cite[Theorem~19]{ZhangXuZhang2009}.
In this article, we avoid the concept of semi-inner products to investigate all properties of the new RKBSs.
Since the semi-inner product is set up by the G\^{a}teaux derivative of the norm, the key technique of the proof of Lemma~\ref{l:RKBS-opt-rep}, which is to solve the minimizer by the G\^{a}teaux derivative directly, is similar to \cite[Theorem~19]{ZhangXuZhang2009} and \cite[Lemma~3.1]{FasshauerHickernellYe2013}.
Hence, we discuss the optimal recovery in RKBSs in a totally different way here.
This prepares a new investigation of another weak optimal recovery in the new RKBSs in our current research work of sparse learning.

\item
Since $\langle s,\GateauxNorm(s) \rangle_{\Banach}=\norm{s}_{\Banach}$ and $\norm{\GateauxNorm(s)}_{\Banach'}=1$ when $s\neq0$, the G\^{a}teaux derivative $\GateauxNorm(s)$ peaks at the minimizer $s$, that is, $\langle s,\GateauxNorm(s) \rangle_{\Banach}=\norm{\GateauxNorm(s)}_{\Banach'}\norm{s}_{\Banach}$. Thus,
Lemma~\ref{l:RKBS-opt-rep} can even be seen as a special case of
the representer theorem in Banach spaces \cite[Theorem~1]{MicchelliPontil2004}, more precisely, for the given continuous linear functional $T_1,\ldots,T_N$ on the Banach space $\Banach$, some linear combination of $T_1,\ldots,T_N$ peaks at the minimizer $s$
of the norm $\norm{f}_{\Banach}$ subject to the interpolations $T_1(f)=y_1,\ldots,T_N(f)=y_N$.
Differently from \cite[Theorem~1]{MicchelliPontil2004},
the reproducing properties of RKHSs ensure that the minimizers over RKHSs can be represented by the reproducing kernels.
We conjecture that the reproducing kernels would also help us to obtain the explicit representations of the minimizers over RKBSs rather than the abstract formats.
In Chapter~\ref{char-SVM}, we shall show how to solve the support vector machines in the $p$-norm RKBSs by the reproducing kernels.
\end{itemize}

\subsection*{Regularized Empirical Risks}

Monographs \cite{KimeldorfWahba1970, Wahba1990} first showed the representer theorem in RKHSs
and paper \cite{ScholkopfHerbrichSmola2001} generalized the representer theorem in RKHSs. Papers
\cite{MicchelliPontil2004,ArgyriouMicchelliPontil2009,ArgyriouMicchelliPontil2010} gave the general framework of the function representation for learning methods in Banach spaces. Now we discuss the representer theorem in RKBSs based on the kernel-based approach.

To ensure the uniqueness and existence of empirical learning solutions, the loss functions and
the regularization functions are endowed with additional geometrical properties to be described below.

\begin{assumption}[A-ELR]
Suppose that the regularization function $R:[0,\infty)\to[0,\infty)$ is convex and strictly increasing, and
the loss function $L:\Domain\times\RR\times\RR\to[0,\infty)$ is defined such that $t\mapsto L(\vx,y,t)$ is a convex map
for each fixed $\vx\in\Domain$ and each fixed $y\in\RR$.
\end{assumption}\label{a:A-ELR}

Given any finite data
$$
D:=\left\{\left(\vx_k,y_k\right):k\in\NN_N\right\}\subseteq\Domain\times\RR,
$$
we define the regularized empirical risks in the right-sided RKBS $\Banach$ by
\begin{equation}\label{eq:svm-RKBS}
\svm(f):=\frac{1}{N}\sum_{k\in\NN_N}L\left(\vx_k,y_k,f(\vx_k)\right)+R\left(\norm{f}_{\Banach}\right),\quad\text{for }f\in\Banach.
\end{equation}
Next, we establish the representer theorem in right-sided RKBSs by the same techniques as those used in \cite[Theorem~3.2]{FasshauerHickernellYe2013} (the representer theorem in semi-inner-product RKBSs).

%/////////////////////////////////////////////////////////////////////////////////////////////////////////////////////
\begin{theorem}[{Representer Theorem}]\label{t:RKBS-opt-rep}
Given any pairwise distinct data points $X\subseteq\Domain$ and any associated data values $Y\subseteq\RR$,
the regularized empirical risk $\svm$ is defined as in equation~\eqref{eq:svm-RKBS}.
Let $\Banach$ be a right-sided reproducing kernel Banach space with the right-sided reproducing kernel $K\in\Leb_0(\Domain\times\Domain')$
such that the right-sided kernel set $\Kset_{K}'$ is linearly independent and $\Banach$ is reflexive, strictly convex, and smooth.
If the loss function $L$ and the regularization function $R$ satisfy assumption~(A-ELR),
then the empirical learning problem
\begin{equation}\label{eq:opt-svm}
\min_{f\in\Banach}\svm(f),
\end{equation}
has a unique optimal solution $s$ such that the G\^{a}teaux derivative $\GateauxNorm(s)$, defined by equation \eqref{eq:Gateaux-norm}, has the finite dimensional representation
\[
\GateauxNorm(s)=\sum_{k\in\NN_N}\beta_kK(\vx_k,\cdot),
\]
for some suitable parameters $\beta_1,\ldots,\beta_N\in\RR$ and the norm of $s$ can be written as
\[
\norm{s}_{\Banach}=\sum_{k\in\NN_N}\beta_ks(\vx_k).
\]
\end{theorem}
%/////////////////////////////////////////////////////////////////////////////////////////////////////////////////////
\begin{proof}
We first prove the uniqueness of the minimizer of optimization problem~\eqref{eq:opt-svm}.
Assume to the contrary that there exist two different minimizers $s_1,s_2\in\Banach$ of $\svm$.
Let $s_3:=\frac{1}{2}\left(s_1+s_2\right)$.
Since $\Banach$ is strictly convex, we have that
$$
\norm{s_3}_{\Banach}=\norm{\frac{1}{2}\left(s_1+s_2\right)}_{\Banach}<\frac{1}{2}\norm{s_1}_{\Banach}+\frac{1}{2}\norm{s_2}_{\Banach}.
$$
Using the convexity and strict increasing of $R$, we obtain the inequality that
\[
R\left(\norm{\frac{1}{2}\left(s_1+s_2\right)}_{\Banach}\right)
<R\left(\frac{1}{2}\norm{s_1}_{\Banach}+\frac{1}{2}\norm{s_2}_{\Banach}\right)
\leq\frac{1}{2}R\left(\norm{s_1}_{\Banach}\right)+\frac{1}{2}R\left(\norm{s_2}_{\Banach}\right).
\]
For $f\in\Banach$, define
\[
\risk(f):=\frac{1}{N}\sum_{k\in\NN_N}L\left(\vx_k,y_k,f(\vx_k)\right).
\]
For any $f_1,f_2\in\Banach$ and any $t\in(0,1)$, using the convexity of $L(\vx,y,\cdot)$,
we check that
\[
\risk\left(tf_1+(1-t)f_2\right)\leq
t\risk\left(f_1\right)+(1-t)\risk\left(f_2\right),
\]
which ensures that $\risk$ is a convex function.
The convexity of $\risk$ together with $\svm\left(s_1\right)=\svm\left(s_2\right)$ then shows that
\begin{align*}
\svm\left(s_3\right)
&=\risk\left(\frac{1}{2}\left(s_1+s_2\right)\right)+R\left(\norm{\frac{1}{2}\left(s_1+s_2\right)}_{\Banach}\right)\\
&<\frac{1}{2}\risk\left(s_1\right)+\frac{1}{2}\risk\left(s_2\right)
+\frac{1}{2}R\left(\norm{s_1}_{\Banach}\right)+\frac{1}{2}R\left(\norm{s_2}_{\Banach}\right)\\
&=\frac{1}{2}\svm\left(s_1\right)+\frac{1}{2}\svm\left(s_2\right)\\
&=\svm\left(s_1\right).
\end{align*}
This means that $s_1$ is not a minimizer of $\svm$. Consequently, the assumption that there are two various minimizers is false.

Now, we verify the existence of the minimizer of $\svm$ over the reflexive Banach space $\Banach$. To this end, we check the convexity and continuity of $\svm$.
Since $R$ is convex and strictly increasing, the map $f\to R\left(\norm{f}_{\Banach}\right)$ is convex and continuous. As the above discussion, we have already known that $\risk$ is convex. Since $L(\vx,y,\cdot)$ is convex, the map $L(\vx,y,\cdot)$ is continuous. By Proposition~\ref{p:weakly-conv-RKBS} about the weak convergence of RKBSs, the weak convergence in $\Banach$ implies the pointwise convergence in $\Banach$. Hence, $\risk$ is continuous.
We then conclude that $\svm$ is a convex and continuous map.
Next, we consider the set
\[
\Eset:=\left\{f\in\Banach: \svm(f)\leq\svm(0)\right\}.
\]
We have that $0\in\Eset$ and
\[
\norm{f}_{\Banach}\leq R^{-1}\left(\svm(0)\right),\quad \text{for }f\in\Eset.
\]
In other words, $\Eset$ is a non-empty and bounded subset and thus the theorem of existence of minimizers (\cite[Proposition~6]{EkelandTurnbull1983})
ensures the existence of the optimal solution $s$.

Finally, Lemma~\ref{l:RKBS-opt-rep} guarantees the finite dimensional representation of the G\^{a}teaux derivative $\GateauxNorm(s)$.
We take any $f\in\Banach$ and let
$$
D_f:=\left\{(\vx_k,f(\vx_k)):k\in\NN_N\right\}.
$$
Since the right-sided RKBS $\Banach$ is reflexive, strictly convex and smooth, Lemma~\ref{l:RKBS-opt-rep} ensures that
there is an element $s_f$, at which the G\^{a}teaux derivative $\GateauxNorm\left(s_f\right)$ of the norm of $\Banach$ belongs to $\Span\left\{K(\vx_k,\cdot):k\in\NN_N\right\}$
such that $s_f$ interpolates the values $\left\{f(\vx_k):k\in\NN_N\right\}$ at the data points $X$ and $\norm{s_f}_{\Banach}\leq\norm{f}_{\Banach}$. This implies that
\[
\svm\left(s_f\right)\leq\svm(f).
\]
Therefore, the G\^{a}teaux derivative $\GateauxNorm(s)$ of the minimizer $s$ of $\svm$ over $\Banach$ belongs to $\Span\left\{K(\vx_k,\cdot):k\in\NN_N\right\}$. Thus, there exist the suitable parameters $\beta_1,\ldots,\beta_N\in\RR$ such that
\[
\GateauxNorm(s)=\sum_{k\in\NN_N}\beta_kK(\vx_k,\cdot).
\]
Using equation~\eqref{eq:Gateaux-Der-norm} and the reproducing properties (i)-(ii), we compute the norm
\[
\norm{s}_{\Banach}=\langle s,\GateauxNorm(s) \rangle_{\Banach}
=\sum_{k\in\NN_N}\beta_k\langle s,K(\vx_k,\cdot) \rangle_{\Banach}
=\sum_{k\in\NN_N}\beta_ks(\vx_k).
\]
\end{proof}
%/////////////////////////////////////////////////////////////////////////////////////////////////////////////////////

%////////////////////////////////////////////////////////////////////////////////////////////////////////////////////////
\begin{remark}\label{r:RKBS-opt-rep}
The G\^{a}teaux derivative of the norm of a Hilbert space $\Hilbert$ at $f\neq0$ is given by $\GateauxNorm(f)=\norm{f}_{\Hilbert}^{-1}f$.
Thus, when the regularized empirical risk $\svm$ is defined on a RKHS $\Hilbert$,
Theorem~\ref{t:RKBS-opt-rep} also provides that
$$
s=\norm{s}_{\Hilbert}\GateauxNorm(s)=\sum_{k\in\NN_N}\norm{s}_{\Hilbert}\beta_kK(\vx_k,\cdot).
$$
This means that the representer theorem in RKBSs (Theorem~\ref{t:RKBS-opt-rep}) implies the classical representer theorem in RKHSs \cite[Theorem~5.5]{SteinwartChristmann2008}.
\end{remark}
%////////////////////////////////////////////////////////////////////////////////////////////////////////////////////////

\subsection*{Convex Programming}

Now we study the convex programming for the regularized empirical risks in RKBSs.
Since the RKBS $\Banach$ is reflexive, strictly convex, and smooth, the G\^{a}teaux derivative $\GateauxNorm$ is a one-to-one map from $S_{\Banach}$ onto $S_{\Banach'}$, where $S_{\Banach}$ and $S_{\Banach'}$ are the unit spheres of $\Banach$ and $\Banach'$, respectively.
Combining this with the linear independence of $\left\{K\left(\vx_k,\cdot\right):k\in\NN_N\right\}$,
we define a one-to-one map $\Gamma:\RR^N\to\Banach$ such that
\[
\norm{\Gamma(\vw)}_{\Banach}\GateauxNorm\left(\Gamma(\vw)\right)=\sum_{k\in\NN_N}w_kK(\vx_k,\cdot),\quad
\text{for }\vw:=\left(w_k:k\in\NN_N\right)\in\RR^N.
\]
(Here, if the RKBS $\Banach$ has the \emph{semi-inner product}, then $\sum_{k\in\NN_N}w_kK(\vx_k,\cdot)$ is
identical to the \emph{dual element} of $\Gamma(\vw)$ given in~\cite{ZhangXuZhang2009,FasshauerHickernellYe2013}.)
Thus, Theorem~\ref{t:RKBS-opt-rep} ensures that the empirical learning solution $s$ can be written as
\[
s=\Gamma\left(\norm{s}_{\Banach}\vbeta\right),
\]
and
\[
\norm{s}_{\Banach}^2=\sum_{k\in\NN_N}\norm{s}_{\Banach}\beta_ks(\vx_k),
\]
where $\vbeta=\left(\beta_k:k\in\NN_N\right)$ are the parameters of $\GateauxNorm(s)$.
Let
\[
\vw_{s}:=\norm{s}_{\Banach}\vbeta.
\]
Then empirical learning problem~\eqref{eq:opt-svm} can be transferred to solving the optimal parameters $\vw_{s}$ of
\[
\min_{\vw\in\RR^N}
\left\{\frac{1}{N}\sum_{k\in\NN_N}L\left(\vx_k,y_k,\Gamma(\vw)(\vx_k)\right)+R\left(\norm{\Gamma(\vw)}_{\Banach}\right)\right\}.
\]
This is equivalent to
\begin{equation}\label{eq:svm-coef-gen}
\min_{\vw\in\RR^N}
\left\{\frac{1}{N}\sum_{k\in\NN_N}L\left(\vx_k,y_k,\langle \Gamma(\vw),K(\vx_k,\cdot) \rangle_{\Banach}\right)+\Theta(\vw)\right\},
\end{equation}
where
\[
\Theta(\vw):=R\left(\left(\sum_{k\in\NN_N}w_k\langle \Gamma(\vw),K(\vx_k,\cdot) \rangle_{\Banach}\right)^{1/2}\right).
\]
By the reproducing properties, we also have that
\[
s(\vx)=\langle \Gamma\left(\vw_{s}\right),K(\vx,\cdot) \rangle_{\Banach},\quad\text{for }\vx\in\Domain.
\]

As an example, we look at the hinge loss
$$
L(\vx,y,t):=\max\left\{0,1-yt\right\}, \ \  \mbox{for}\ \ y\in\left\{-1, 1\right\},
$$
and the regularization function
$$
R(r):=\sigma^2r^2, \ \ \mbox{for}\ \  \sigma>0.
$$
In a manner similar to the convex programming of the binary classification in RKHSs \cite[Example~11.3]{SteinwartChristmann2008},
the convex programming of the optimal parameters $\vw_s$ of optimization problem~\eqref{eq:svm-coef-gen} is therefore given by
\begin{equation}\label{eq:svm-primal-opt}
\begin{split}
&\min_{\vw\in\RR^N}\left\{C\sum_{k\in\NN_N}\xi_k+\frac{1}{2}\norm{\Gamma(\vw)}_{\Banach}^2\right\}\\
&\text{subject to} \ \xi_k\geq1-y_k\langle \Gamma(\vw),K(\vx_k,\cdot) \rangle_{\Banach},\quad \xi_k\geq0,\quad \text{for all }k\in\NN_N,
\end{split}
\end{equation}
where
\[
C:=\frac{1}{2N\sigma^2}.
\]
Its corresponding Lagrangian is represented as
\begin{equation}\label{eq:hinge-Lagrangian}
C\sum_{k\in\NN_N}\xi_k+\frac{1}{2}\norm{\Gamma(\vw)}_{\Banach}^2+\sum_{k\in\NN_N}\alpha_k\left(1-y_k\langle \Gamma(\vw),K(\vx_k,\cdot)\rangle_{\Banach}-\xi_k\right)-\sum_{k\in\NN_N}\gamma_k\xi_k,
\end{equation}
where $\left\{\alpha_k,\gamma_k:k\in\NN_N\right\}\subseteq\RR_+$.
Computing the partial derivatives of Lagrangian~\eqref{eq:hinge-Lagrangian} with respect to $\vw$ and $\vxi$,
we obtain optimality conditions
\begin{equation}\label{eq:svm-diff-c}
\norm{\Gamma(\vw)}_{\Banach}\GateauxNorm\left(\Gamma(\vw)\right)-\sum_{k\in\NN_N}\alpha_ky_kK(\vx_k,\cdot)=0,
\end{equation}
and
\begin{equation}\label{eq:svm-diff-xi}
C-\alpha_k-\gamma_k=0,\quad\text{for all }k\in\NN_N.
\end{equation}
Comparing the definition of $\Gamma$, equation~\eqref{eq:svm-diff-c} yields that
\begin{equation}\label{eq:svm-w-alpha-y}
\vw=\left(\alpha_ky_k:k\in\NN_N\right),
\end{equation}
and
\[
\norm{\Gamma(\vw)}_{\Banach}^2=\norm{\Gamma(\vw)}_{\Banach}\langle \Gamma(\vw), \GateauxNorm\left(\Gamma(\vw)\right) \rangle_{\Banach}
=\sum_{k\in\NN_N}\alpha_ky_k\langle \Gamma(\vw),K(\vx_k,\cdot) \rangle_{\Banach}.
\]
Substituting equations~\eqref{eq:svm-diff-xi} and \eqref{eq:svm-w-alpha-y} into primal problem~\eqref{eq:svm-primal-opt}, we obtain the dual problem
\[
\max_{0\leq\alpha_k\leq C,~k\in\NN_N}
\left\{\sum_{k\in\NN_N}\alpha_k-\frac{1}{2}\norm{\Gamma(\left(\alpha_ky_k:k\in\NN_N\right))}_{\Banach}^2\right\}.
\]
When $\Banach$ is a RKHS, we have that
\[
\norm{\Gamma(\left(\alpha_ky_k:k\in\NN_N\right))}_{\Banach}^2
=\sum_{j,k\in\NN_N}\alpha_j\alpha_ky_jy_kK\left(\vx_j,\vx_k\right).
\]

\subsection*{Regularized Infinite-sample Risks}

In this subsection, we consider an infinite-sample learning problem in the two-sided RKBSs.

For the purpose of solving the infinite-sample optimization problems,
the RKBSs need to have certain stronger geometrical properties. Specifically, we suppose that the function space $\Banach$ is a two-sided RKBS with the two-sided reproducing
kernel $K\in\Leb_0(\Domain\times\Domain')$ such that $\vx\mapsto K(\vx,\vy)\in\Leb_{\infty}(\Domain)$ for all $\vy\in\Domain'$ and the map $\vx\mapsto\norm{K(\vx,\cdot)}_{\Banach'}\in\Leb_{\infty}(\Domain)$.
Moreover, the RKBS $\Banach$ is required not only to be reflexive and strictly convex but also to be Fr\'{e}chet differentiable.
We say that the Banach space $\Banach$ is Fr\'{e}chet differentiable if the normed operator $\normopt:f\mapsto\norm{f}_{\Banach}$ is Fr\'{e}chet differentiable for all $f\in\Banach$.
Clearly, the Fr\'{e}chet differentiability implies the G\^{a}teaux differentiability and their derivatives are the same, that is, $\Frechet\normopt=\Gateaux\normopt$.

We consider minimizing the regularized \emph{infinite-sample} risk
\[
\widetilde{\svm}(f):=\int_{\Domain\times\RR}L(\vx,y,f(\vx))\PP(\ud\vx,\ud y)+R\left(\norm{f}_{\Banach}\right),\quad\text{for }f\in\Banach,
\]
where $\PP$ is a probability measure defined on $\Domain\times\RR$.
Let $C_{\mu}:=\mu(\Domain)$.
If
$$
\PP(\vx,y)=C_{\mu}^{-1}\PP(y|\vx)\mu(\vx),
$$
then the regularized risk $\widetilde{\svm}$ can be rewritten as
\begin{equation}\label{eq:g-svm-RKBS}
\widetilde{\svm}(f)=\int_{\Domain}H(\vx,f(\vx))\mu(\ud\vx)+R\left(\norm{f}_{\Banach}\right),\quad\text{for }f\in\Banach,
\end{equation}
where
\begin{equation}\label{eq:H-Fun}
H(\vx,t):=\frac{1}{C_{\mu}}\int_{\RR}L(\vx,y,t)\PP(\ud y|\vx),\quad\text{for }\vx\in\Domain\text{ and }t\in\RR.
\end{equation}
In this article, we consider only the regularized infinite-sample risks written in a form as equation~\eqref{eq:g-svm-RKBS}.
For the representer theorem in RKBSs based on the infinite-sample risks, we need the loss function $L$
and the regularization function $R$ to satisfy additional conditions which we describe below.

\begin{assumption}[A-GLR]
Suppose that the regularization function $R:[0,\infty)\to[0,\infty)$ is convex, strictly increasing, and continuously differentiable, and the
loss function $L:\Domain\times\RR\times\RR\to[0,\infty)$ satisfies that $t\mapsto L(\vx,y,t)$ is a convex map for each fixed $\vx\in\Domain$ and each fixed $y\in\RR$.
Further suppose that the function $H:\Domain\times\RR\to[0,\infty)$ defined in equation~\eqref{eq:H-Fun} satisfies that the map $t\mapsto H(\vx,t)$ is differentiable for any fixed $\vx\in\Domain$ and $\vx\mapsto H(\vx,f(\vx))\in\Leb_1(\Domain),\vx\mapsto H_t(\vx,f(\vx))\in\Leb_{1}(\Domain)$ whenever $f\in\Leb_{\infty}(\Domain)$.
\end{assumption}\label{a:A-GLR}

The convexity of $L$ ensures that
the map $f\mapsto H(\vx,f(\vx))$ is convex for any fixed $\vx\in\Domain$. Hence,
the loss risk
\[
\widetilde{\risk}(f):=\int_{\Domain}H(\vx,f(\vx))\mu(\ud\vx),\quad\text{for }f\in\Banach
\]
is a convex operator.
Because of the differentiability and the integrality of $H$, Corollary~\ref{c:RKBS-FrechetDerivative} ensures that the Fr\'{e}chet derivative of $\widetilde{\risk}$ at $f\in\Banach$
can be represented as
\begin{equation}\label{eq:Frechet-svm}
\Frechet\widetilde{\risk}(f)
=\int_{\Domain}H_t(\vx,f(\vx))K(\vx,\cdot)\mu(\ud\vx),
\end{equation}
where $H_t(\vx,y,t):=\frac{\ud}{\ud t}H(\vx,y,t)$.
Based on the above assumptions, we introduce the representer theorem of the solution of infinite-sample learning problem in RKBSs.

%////////////////////////////////////////////////////////////////////////////////////////////////////////////////////////
\begin{theorem}[{Representer Theorem}]\label{t:RKBS-opt-rep-gen}
The regularized infinite-sample risk $\widetilde{\svm}$ is defined as in equation~\eqref{eq:g-svm-RKBS}.
Let $\Banach$ be a two-sided reproducing kernel Banach space with the two-sided reproducing
kernel $K\in\Leb_0(\Domain\times\Domain')$ such that
$\vx\mapsto K(\vx,\vy)\in\Leb_{\infty}(\Domain)$ for all $\vy\in\Domain'$,
the map $\vx\mapsto\norm{K(\vx,\cdot)}_{\Banach'}\in\Leb_{\infty}(\Domain)$,
and $\Banach$ is reflexive, strictly convex, and Fr\'{e}chet differentiable.
If the loss function $L$ and the regularization function $R$ satisfy assumption~(A-GLR),
then the infinite-sample learning problem
\begin{equation}\label{eq:opt-svm-gen}
\min_{f\in\Banach}\widetilde{\svm}(f),
\end{equation}
has a unique solution $s$ such that the G\^{a}teaux derivative $\GateauxNorm(s)$, defined by equation \eqref{eq:Gateaux-norm}, has the representation
\[
\GateauxNorm(s)=\int_{\Domain}\eta(\vx)K(\vx,\cdot)\mu(\ud\vx),
\]
for some suitable function $\eta\in\Leb_1(\Domain)$ and the norm of $s$ can be written as
\[
\norm{s}_{\Banach}=\int_{\Domain}\eta(\vx)s(\vx)\mu(\ud\vx).
\]
\end{theorem}
%////////////////////////////////////////////////////////////////////////////////////////////////////////////////////////
\begin{proof}
We shall prove this theorem by using techniques similar to those used in the proof of Theorem~\ref{t:RKBS-opt-rep}.

We first verify the uniqueness of the solution of optimization problem~\eqref{eq:opt-svm-gen}.
Assume that there exist two different minimizers $s_1,s_2\in\Banach$ of $\widetilde{\svm}$.
Let $s_3:=\frac{1}{2}\left(s_1+s_2\right)$.
Since $\Banach$ is strictly convex, we have that
$$
\norm{s_3}_{\Banach}=\norm{\frac{1}{2}\left(s_1+s_2\right)}_{\Banach}<\frac{1}{2}\norm{s_1}_{\Banach}+\frac{1}{2}\norm{s_2}_{\Banach}.
$$
According to the convexities and the strict increasing of $\widetilde{\risk}$ and $R$, we obtain that
\begin{align*}
\widetilde{\svm}\left(s_3\right)
&=\widetilde{\risk}\left(\frac{1}{2}\left(s_1+s_2\right)\right)
+R\left(\norm{\frac{1}{2}\left(s_1+s_2\right)}_{\Banach}\right)\\
&<\frac{1}{2}\widetilde{\risk}\left(s_1\right)+\frac{1}{2}\widetilde{\risk}\left(s_2\right)
+\frac{1}{2}R\left(\norm{s_1}_{\Banach}\right)+\frac{1}{2}R\left(\norm{s_2}_{\Banach}\right)\\
&=\frac{1}{2}\widetilde{\svm}\left(s_1\right)+\frac{1}{2}\widetilde{\svm}\left(s_2\right)=\widetilde{\svm}\left(s_1\right).
\end{align*}
This contradicts the assumption that $s_1$ is a minimizer in $\Banach$ of $\widetilde{\svm}$.

Next we prove the existence of the global minimum of $\widetilde{\svm}$.
The convexity of $\widetilde{\risk}$ ensures that $\widetilde{\svm}$ is convex.
The continuity of $R$ ensures that, for any countable sequence $\left\{f_n:n\in\NN\right\}\subseteq\Banach$
and an element $f\in\Banach$ such that $\norm{f_n-f}_{\Banach}\to0$ when $n\to\infty$,
$$
\lim_{n\to\infty}R\left(\norm{f_n}_{\Banach}\right)=R\left(\norm{f}_{\Banach}\right).
$$
By Proposition~\ref{p:weakly-conv-RKBS}, the weak convergence in $\Banach$ implies the pointwise convergence in $\Banach$. Hence,
$$
\lim_{n\to\infty}f_n(\vx)= f(\vx), \ \ \mbox{for all}\ \  \vx\in\Domain.
$$
Since $\vx\mapsto H(\vx,f_n(\vx))\in\Leb_1(\Domain)$ and $\vx\mapsto H(\vx,f(\vx))\in\Leb_1(\Domain)$, we have that
$$
\lim_{n\to\infty}\widetilde{\risk}(f_n)=\widetilde{\risk}(f),
$$
which ensures that
$$
\lim_{n\to\infty}\widetilde{\svm}(f_n)=\widetilde{\svm}(f).
$$
Thus, $\widetilde{\svm}$ is continuous.
Moreover, the set
$$
\widetilde{\Eset}:=\left\{f\in\Banach:\widetilde{\svm}(f)\leq\widetilde{\svm}(0)\right\}
$$
is a non-empty and bounded. Therefore, the existence of the optimal solution $s$ follows from the existence theorem of minimizers.

Finally, we check the representation of $\GateauxNorm(s)$.
Since $\Banach$ is Fr\'{e}chet differentiable, the Fr\'{e}chet and G\^{a}teaux derivatives of the norm of $\Banach$ are the same.
According to equation~\eqref{eq:Frechet-svm}, we obtain the Fr\'{e}chet derivative of $\widetilde{\svm}$ at $f\in\Banach$
\[
\Frechet\widetilde{\svm}(f)=\Frechet\widetilde{\risk}(f)+R_z\left(\norm{f}_{\Banach}\right)\GateauxNorm(f)
=\int_{\Domain}H_t(\vx,f(\vx))K(\vx,\cdot)\mu(\ud\vx)+R_z\left(\norm{f}_{\Banach}\right)\GateauxNorm(f),
\]
where
$$
R_z(z):=\frac{\ud}{\ud z}R(z).
$$
The fact that $s$ is the global minimum solution of $\widetilde{\svm}$ implies that $\Frechet\widetilde{\svm}(s)=0$, because
$$
\widetilde{\svm}(s+h)-\widetilde{\svm}(s)\geq0, \ \  \mbox{for all}\ \ h\in\Banach
$$
and
\[
\widetilde{\svm}(s+h)-\widetilde{\svm}(s)=\langle h,\Frechet \widetilde{\svm} \rangle_{\Banach}+\order(\norm{h}_{\Banach})
\text{ when }\norm{h}_{\Banach}\to0.
\]
Therefore, we have that
\[
\GateauxNorm(s)=\int_{\Domain}\eta(\vx)K(\vx,\cdot)\mu(\ud\vx),
\]
where
\[
\eta(\vx):=-\frac{1}{R_z\left(\norm{s}_{\Banach}\right)}\int_{\Domain}H_t(\vx,s(\vx))K(\vx,\cdot)\mu(\ud\vx),
\quad\text{for }\vx\in\Domain.
\]
Moreover, equation~\eqref{eq:Gateaux-Der-norm}
shows that $\norm{s}_{\Banach}=\langle s,\GateauxNorm(s) \rangle_{\Banach}$;
hence we obtain that
\[
\norm{s}_{\Banach}
=\int_{\Domain}\eta(\vx)\langle s,K(\vx,\cdot) \rangle_{\Banach}\mu(\ud\vx)
=\int_{\Domain}\eta(\vx)s(\vx)\mu(\ud\vx),
\]
by the reproducing properties (i)-(ii).
\end{proof}

To close this section, we comment on the differences of the differentiability conditions of the norms for the empirical and infinite-sample learning problems. The Fr\'{e}chet differentiability of the norms implies the G\^{a}teaux differentiability but the G\^{a}teaux differentiability does not ensure the
Fr\'{e}chet differentiability. Since we need to compute the Fr\'{e}chet derivatives of the infinite-sample regularized risks,
the norms with the stronger conditions of differentiability are required such that the Fr\'{e}chet derivative is well-defined at the global minimizers.

\section{Oracle Inequality}\label{s:OracleInequality}
\sectionmark{Oracle Inequality}

In this section, we investigate the basic statistical analysis of machine learning called
the oracle inequality in RKBSs that relates the minimizations of empirical and infinite-sample risks the same as the argument of RKHSs in \cite[Chapter~6]{SteinwartChristmann2008}.
The oracle inequality shows that the learning ability of RKBSs can be split into the statistical and deterministic parts.
In this section, we focus on the basic oracle inequality by \cite[Proposition~6.22]{SteinwartChristmann2008} the errors of the minimization of empirical and infinite-sample risks over compact sets of the uniform-norm spaces.

In statistical learning, the empirical data $X$ and $Y$
are observations over a probability distribution.
To be more precise, the data $X\times Y\subseteq\Domain\times\Yset$ are seen as
the independent and identically distributed duplications of a probability measure $\PP$ defined on the product $\Domain\times\Yset$ of compact Hausdorff spaces $\Domain$ and $\Yset$,
that is, $\left(\vx_1,y_1\right),\ldots,\left(\vx_N,y_N\right)\sim\text{i.i.d.}\PP(\vx,y)$.
This means that $X$ and $Y$ are random data here.
Hence, we define the product probability measure
$$
\PP^N:=\otimes_{k=1}^N\PP
$$
on the product space
$$
\Domain^N\times\Yset^N:=\otimes_{k=1}^N\Domain\times\Yset.
$$
In order to avoid notational overload, the product space $\Domain^N\times\Yset^N$ is equipped with the universal completion of the product $\sigma$-algebra and the product probability measure $\PP^N$ is the canonical extension.

In this section, we look at the relations of the empirical and infinite-sample risks, that is,
\[
\risk(f):=\frac{1}{N}\sum_{k\in\NN_N}L\left(\vx_k,y_k,f(\vx_k)\right),
\]
and
\[
\widetilde{\risk}(f):=\int_{\Domain\times\Yset}L(\vx,y,f(\vx))\PP(\ud\vx,\ud y),
\]
for $f\in\Linfty(\Domain)$. Differently from Section~\ref{s:RepThm}, $\risk(f)$ is a random element dependent on the random data.

Let $\Banach$ be a right-sided RKBS with the right-sided reproducing kernel $K\in\Leb_0(\Domain\times\Domain')$. Here, we transfer the minimization of regularized empirical risks over the RKBS $\Banach$ to another equivalent minimization problem, more precisely, the minimization of empirical risks over a closed ball $B_M$ of $B$ with the radius $M>0$, that is,
\begin{equation}\label{eq:oracle-empirical-risk-1}
\min_{f\in B_M}\risk(f),
\end{equation}
where
\[
B_M:=\left\{f\in\Banach:\norm{f}_{\Banach}\leq M\right\}.
\]

\cite[Proposition~6.22]{SteinwartChristmann2008} shows the oracle inequality in a compact set of
$\Linfty(\Domain)$. We suppose that the right-sided kernel set $\Kset_{K}'$ is compact in the dual space $\Banach'$ of $\Banach$.
By Proposition~\ref{p:RKBS-compact}, the identity map $I:\Banach\to\Linfty(\Domain)$ is a compact operator; hence $B_M$ is a relatively compact set of $\Linfty(\Domain)$. Let $\Eset_M$ be the closure of $B_M$ with respect to the uniform norm.
Since $\Linfty(\Domain)$ is a Banach space, we have that $\Eset_M\subseteq\Linfty(\Domain)$.
Clearly, $B_M$ is a bounded set in $\Banach$. Thus
$\Eset_M$ is a compact set of $\Linfty(\Domain)$.
Therefore, by \cite[Proposition~6.22]{SteinwartChristmann2008}, we obtain the inequality
\begin{equation}\label{eq:oracle-inequality-0}
\PP^N\left(\widetilde{\risk}(\hat{s})\geq\tilde{r}^{\ast}+\kappa_{\epsilon,\tau}\right)\leq e^{-\tau},
\end{equation}
for any $\epsilon>0$ and $\tau>0$,
where $\hat{s}$ is any minimizer of the empirical risks over $\Eset_M$, that is,
\begin{equation}\label{eq:oracle-empirical-risk-2}
\risk(\hat{s})=\min_{f\in \Eset_M}\risk(f),
\end{equation}
and $\tilde{r}^{\ast}$ is the minimum of infinite-sample risks over $\Eset_M$, that is,
\begin{equation}\label{eq:oracle-infinite-risk}
\tilde{r}^{\ast}:=\inf_{f\in\Eset_M}\widetilde{\risk}(f).
\end{equation}
Moreover, the above error $\kappa_{\epsilon,\tau}$ is given by the formula
\begin{equation}\label{eq:oracle-constant}
\kappa_{\epsilon,\tau}:=\varrho_M\sqrt{\frac{2\tau+2\log N_{\epsilon}}{N}}+4\epsilon\Upsilon_M.
\end{equation}
Here, the positive integer $N_{\epsilon}\in\NN$ is the $\epsilon$-covering number of $\Eset_M$ with respect to the uniform norm, that is, the smallest number of balls with the radius $\epsilon$ of $\linfty(\Domain)$ covering $\Eset_M$ is $N_{\epsilon}$ (see \cite[Definition~6.19]{SteinwartChristmann2008}).
Since $\Eset_M$ is a compact set of $\Linfty(\Domain)$, the $\epsilon$-covering number $N_{\epsilon}$ exists.
Another positive constant $\Upsilon_M$ is the smallest Lipschitz constant of $L(\vx,y,\cdot)$ on $[-C_{\infty},C_{\infty}]$ for the supremum $C_{\infty}$ of the compact set $\Eset_M$ with respect to the uniform norm, that is, $\Upsilon_M$ is the smallest constant such that
\[
\sup_{\vx\in\Domain,y\in\Yset}\abs{L(\vx,y,t_1)-L(\vx,y,t_2)}\leq\Upsilon_M\abs{t_1-t_2}, \ \text{for all} \ t_1,t_2\in[-C_{\infty},C_{\infty}],
\]
(see \cite[Definition~2.18]{SteinwartChristmann2008}).
For the existence of the smallest Lipschitz constant $\Upsilon_M$, we suppose that
$L_t$ is continuous, where
$$
L_t(\vx,y,t):=\frac{\ud}{\ud t}L(\vx,y,t)
$$
is the first derivative of the loss function
$L:\Domain\times\Yset\times\RR\to[0,\infty)$ at the variable $t$.
This means that $L$ is continuously differentiable at $t$.
The compactness of $\Domain\times\Yset\times[-C_{\infty},C_{\infty}]$ and the continuity of $L_t$ ensure that $\Upsilon_M$ exists.
The positive constant $\varrho_M$ is the supremum of $L$ on $\Domain\times\Yset\times[-C_{\infty},C_{\infty}]$.
The existence of $\varrho_M$ is guaranteed by the compactness of $\Domain\times\Yset\times[-C_{\infty},C_{\infty}]$ and the continuity of $L$.
Therefore, we can obtain the oracle inequality in RKBSs.

%////////////////////////////////////////////////////////////////////////////////////////////////////////////////////////
\begin{proposition}\label{p:RKBS-oracle}
Let $\Banach$ be a right-sided reproducing kernel Banach space with the right-sided reproducing kernel $K\in\Leb_0(\Domain\times\Domain')$ such that $\Banach$ is reflexive and the right-sided kernel set $\Kset_{K}'$ is compact in the dual space $\Banach'$ of $\Banach$, let $B_M$ be a closed ball of $\Banach$ with a radius $M>0$, and let $\Eset_M$ be the closure of $B_M$ with respect to the uniform norm.
If the derivative $L_t$ of the loss function $L$ at $t$ is continuous and the map $t\mapsto L(\vx,y,t)$ is convex for all $\vx\in\Domain$ and all $y\in\Yset$,
then
for any $\epsilon>0$ and any $\tau>0$,
\begin{equation}\label{eq:oracle-inequality}
\PP^N\left(\widetilde{\risk}(s)\geq\tilde{r}^{\ast}+\kappa_{\epsilon,\tau}\right)\leq e^{-\tau},
\end{equation}
where $s$ is the minimizer of empirical risks over $B_M$ solved in optimization problem~\eqref{eq:oracle-empirical-risk-1}, $\tilde{r}^{\ast}$ is the minimum of infinite-sample risks over $\Eset_M$ defined in equation~\eqref{eq:oracle-infinite-risk}, and the error $\kappa_{\epsilon,\tau}$ is defined in equation~\eqref{eq:oracle-constant}.
\end{proposition}
%////////////////////////////////////////////////////////////////////////////////////////////////////////////////////////
\begin{proof}
The main idea of the proof is to use inequality~\eqref{eq:oracle-inequality-0} which shows the errors of the minimum of empirical and infinite-sample risks over a compact set of $\Linfty(\Domain)$ given in \cite[Proposition~6.22]{SteinwartChristmann2008}.

First we prove that optimization problem~\eqref{eq:oracle-empirical-risk-1} exists a solution $s$.
Same as the proof of Theorem~\ref{t:RKBS-opt-rep}, the convexity of $L(\vx,y,\cdot)$ insures that $\risk$ is a convex and continuous function on $\Banach$. Moreover, since $\Eset_M$ is a closed ball of a reflexive Banach space $\Banach$, there is an optimal solution $s\in B_M$ to minimize the empirical risks over $B_M$ by the theorem of existence of minimizers.

Next, since $B_M$ is dense in $\Eset_M$ with respect to the uniform norm and $\risk$ is continuous on $\Linfty(\Domain)$, the optimal solution $s$ also minimizes the empirical risks over $\Eset_M$, that is,
\begin{equation}\label{eq:empirical-risks-equal}
\risk(s)=\min_{f\in \Eset_M}\risk(f).
\end{equation}

Finally, by Proposition~\ref{p:RKBS-compact}, the set $\Eset_M$ is a compact set of $\Linfty(\Domain)$.
Since $L_t$ is continuous, the parameters $N_{\epsilon},\Upsilon_M$ of $\kappa_{\epsilon,\tau}$ is well-defined as above discussions.
Thus, the compactness of $\Eset_M$ and the existence of parameters $\kappa_{\epsilon,\tau}$ ensure inequality~\eqref{eq:oracle-inequality-0} is obtained by \cite[Proposition~6.22]{SteinwartChristmann2008}.
Combining inequality~\eqref{eq:oracle-inequality-0} and equation~\eqref{eq:empirical-risks-equal}, the proof of oracle inequality~\eqref{eq:oracle-inequality} is complete.
\end{proof}
%////////////////////////////////////////////////////////////////////////////////////////////////////////////////////////

%////////////////////////////////////////////////////////////////////////////////////////////////////////////////////////
\begin{remark}
In Proposition~\ref{p:RKBS-oracle}, the empirical data $X$ and $Y$ are different from the data in Theorem~\ref{t:RKBS-opt-rep}. They distribute randomly based on the probability $\PP$. Hence, the solution $s$ is randomly given and inequality~\eqref{eq:oracle-inequality} shows the errors is similar to the Chebyshev inequality in probability theory.
\end{remark}
%////////////////////////////////////////////////////////////////////////////////////////////////////////////////////////

\section{Universal Approximation}
\sectionmark{Universal Approximation}

Paper \cite{MicchelliXuZhang2006} provides a nice theorem of machine learning in RKHSs, that is, the universal approximation of RKHSs.
In this section, we investigate the universal approximation of RKBSs.
Let the left-sided and right-sided domains $\Domain$ and $\Domain'$ be the compact Hausdorff spaces.
Suppose that $\Banach$ is the two-sided RKBS with the two-sided reproducing kernel $K$.
By the generalization of the universal approximation of RKHSs,
we shall show that the dual space $\Banach'$ of the RKBS $\Banach$
has
the \emph{universal approximation property}, that is, for any $\epsilon>0$ and any $g\in\Cont(\Domain')$,
there exists a function $s\in\Banach'$ such that $\norm{g-s}_{\infty}\leq\epsilon$.
In other words $\Banach'$ is dense in $\Cont(\Domain')$ with respect to the uniform norm. Here, by the compactness of $\Domain'$, the space $\Cont(\Domain')$ endowed with the uniform norm is a Banach space.
In particular, the universal approximation of the dual space $\Banach'$ ensures that the minimization of empirical risks over $\Cont(\Domain')$
can be transferred to the minimization of empirical risks over $\Banach'$, because $\cup_{M>0}\Eset_M'$ is equal to $\Cont(\Domain')$ where $\Eset_M'$ is the closure of the set
\[
B_M':=\left\{g\in\Banach':\norm{g}_{\Banach'}\leq M\right\},
\]
(see the oracle inequality in Section~\ref{s:OracleInequality}).

Before presenting the proof of the universal approximation of the dual space $\Banach'$, we review a classical result of the integral operator $I_K$. Clearly, $I_K$ is also a compact operator from $\Cont(\Domain)$ into $\Cont(\Domain')$ (see \cite[Theorem~2.21]{Kress1989}).
Let $\DualMeasure(\Domain')$ be the collection of all regular signed Borel measures on $\Domain'$ endowed with their variation norms.
\cite[Example~1.10.6]{Megginson1998} (Riesz-Markov-Kakutani representation theorem) illustrates that $\DualMeasure(\Domain')$ is isometrically equivalent to the dual space of $\Cont(\Domain')$ and
\[
\langle g,\nu' \rangle_{\Cont(\Domain')}=\int_{\Domain'}g(\vy)\nu'(\ud\vy),
\quad\text{for all }g\in\Cont(\Domain')\text{ and all }\nu'\in\DualMeasure(\Domain').
\]
By the compactness of the domains and the continuity of the kernel, we define another linear operator $I_K^{\ast}$ from $\DualMeasure(\Domain')$ into $\DualMeasure(\Domain)$ by
\begin{equation}\label{eq:adjoptIntopt}
\left(I_K^{\ast}\nu'\right)(A):=\int_{A}\mu(\ud\vx)\int_{\Domain'}K(\vx,\vy)\nu'(\ud\vy),
\end{equation}
for all Borel set $A$ in $\Domain$.
Thus, we have that
\[
\begin{split}
\langle I_Kf,\nu' \rangle_{\Cont(\Domain')}
&=\int_{\Domain'}(I_Kf)(\vy)\nu'(\ud\vy)\\
&=\int_{\Domain}f(\vx)\mu(\ud\vx)\int_{\Domain'}K(\vx,\vy)\nu'(\ud\vy)\\
&=\langle f,I_K^{\ast}\nu' \rangle_{\Cont(\Domain)},
\end{split}
\]
for all $f\in\Cont(\Domain)$.
This shows that $I_K^{\ast}$ is the adjoint operator of $I_K$.

Next, we gives the theorem of universal approximation of the dual space $\Banach'$.
%////////////////////////////////////////////////////////////////////////////////////////////////////////////////////////
\begin{proposition}\label{p:RKBS-universial-approx}
Let $\Banach$ be the two-sided reproducing kernel Banach space with the two-sided reproducing kernel $K\in\Cont(\Domain\times\Domain')$ defined on the compact Hausdorff spaces $\Domain$ and $\Domain'$ such that
the map $\vy\mapsto K(\cdot,\vy)$ is continuous on $\Domain'$
and the map $\vx\mapsto\norm{K(\vx,\cdot)}_{\Banach'}\in\Leb_q(\Domain)$.
If the adjoint operator $I_K^{\ast}$ defined in equation~\eqref{eq:adjoptIntopt} is injective, then the dual space $\Banach'$ of $\Banach$ has the universal approximation property.
\end{proposition}
%////////////////////////////////////////////////////////////////////////////////////////////////////////////////////////
\begin{proof}
Similar to the proof of Proposition~\ref{p:RKBS-continue}, we show that
$\Banach'\subseteq\Cont(\Domain')$.
Thus, if we find a subspace of $\Banach'$ which is dense in $\Cont(\Domain')$ with respect to the uniform norm, then the proof is complete.

Moreover, Proposition~\ref{p:RKBS-imbedding-dual} provides that $I_K\left(\Leb_p(\Domain)\right)\subseteq\Banach'$, because $K$ is continuous on the compact spaces $\Domain$ and $\Domain'$.
Here $1\leq p,q\leq\infty$ and $p^{-1}+q^{-1}=1$.
As above discussion, by \cite[Theorem~3.1.17]{Megginson1998}, the injectivity of the adjoint operator $I_K^{\ast}$ ensures that $I_K\left(\Cont(\Domain)\right)$ is dense in $\Cont(\Domain')$ with respect to the uniform norm.
Since $\Omega$ is compact, we also have $\Cont(\Domain)\subseteq\Leb_p(\Domain)$. This shows that $I_K\left(\Cont(\Domain)\right)\subseteq I_K\left(\Leb_p(\Domain)\right)$; hence $I_K\left(\Cont(\Domain)\right)$ is a subspace of $\Banach'$. Therefore, the space $\Banach'$ is also dense in $\Cont(\Domain')$ with respect to the uniform norm.
\end{proof}
%////////////////////////////////////////////////////////////////////////////////////////////////////////////////////////

%////////////////////////////////////////////////////////////////////////////////////////////////////////////////////////
\begin{remark}\label{r:RKBS-universial-approx}
If the RKBS $\Banach$ is reflexive, then the universal approximation property of $\Banach$ can also be checked in the same manner as Proposition~\ref{p:RKBS-universial-approx}.
Thus, we call the reproducing kernel $K$ a left-sided (or right-sided) \emph{universal kernel} if the RKBS $\Banach$ (or its dual space $\Banach'$) has the universal approximation property.
Moreover, if the RKBS $\Banach$ and its dual space $\Banach'$ have the universal approximation property,
then the reproducing kernel $K$ is called a two-sided universal kernel.
Clearly, the definition of the universal kernels of RKHSs is a special case of the universal kernels of RKBSs (see~\cite{MicchelliXuZhang2006}).
In Section~\ref{s:1-RKBS}, we show that the non-reflexive $1$-norm RKBS also has the universal approximation property. But, the dual space of $1$-norm RKBS does not have the universal approximation property. This indicates that the $1$-norm RKBS only has the left-sided universal kernel.
\end{remark}
%////////////////////////////////////////////////////////////////////////////////////////////////////////////////////////

%------------------------------------------------------------------------------------------------------------------------
%------------------------------------------------------------------------------------------------------------------------
\chapter{Generalized Mercer Kernels}\label{char-GMK}
%------------------------------------------------------------------------------------------------------------------------
%------------------------------------------------------------------------------------------------------------------------

This chapter presents the generalized Mercer kernels, the sum of countable symmetric or nonsymmetric expansion terms, which become the reproducing kernels of the $p$-norm RKBSs driven by their given expansion sets.
Comparing with the general RKBSs discussed in Chapter~\ref{char:RKBS}, the $p$-norm RKBSs can be viewed as the generalizations of the Mercer representations of RKHSs. Moreover, the expansion terms of the generalized Mercer kernels can be used to establish the imbedding, compactness, and universal approximation of the $p$-norm RKBSs.

%------------------------------------------------------------------------------------------------------------------------
\section{Constructing Generalized Mercer Kernels}\label{s:Constr-GMK}
%------------------------------------------------------------------------------------------------------------------------
\sectionmark{Constructing Generalized Mercer Kernels}

In this section, we define the generalized Mercer kernels.
First we review what the classical Mercer kernels are.
If $K$ is a continuous symmetric positive definite kernel defined on a compact Hausdorff space $\Domain$,
then the Mercer theorem (\cite[Theorem~4.49]{SteinwartChristmann2008}) ensures that the kernel $K$ has the absolutely and uniformly convergent representation
\begin{equation}\label{eq:pdk-Mercer}
K(\vx,\vy)=\sum_{n\in\NN}\lambda_ne_n(\vx)e_n(\vy),\quad\text{for }\vx,\vy\in\Domain,
\end{equation}
by the countable positive eigenvalues $\left\{\lambda_n: n\in\NN\right\}$
and continuous eigenfunctions $\left\{e_n:n\in\NN\right\}$
of $K$, that is,
\[
\int_{\Domain}K(\vx,\vy)e_n(\vx)\mu(\ud\vx)=\lambda_ne_n(\vy),\ \text{for all} \ n\in\NN.
\]
A kernel $K$ is called a \emph{Mercer kernel} if $K$ can be written as a sum of its eigenvalues multiplying eigenfunctions in the form of equation~\eqref{eq:pdk-Mercer}. If we let
\begin{equation}\label{eq:base-MerKer}
\phi_n:=\lambda_n^{1/2}e_n,\quad\text{for all }n\in\NN,
\end{equation}
then the Mercer kernel $K$ can be rewritten as
\[
K(\vx,\vy)=\sum_{n\in\NN}\phi_n({\color{black}{\vx}})\phi_n(\vy),\quad\text{for }\vx,\vy\in\Domain.
\]

Based on the classical construction we shall define the generalized Mercer kernel by replacing the symmetric expansion terms with nonsymmetric expansion terms.

%/////////////////////////////////////////////////////////////////////////////////////////////////////////////////////
\begin{definition}\label{d:GMK}
Let $\Domain$ and $\Domain'$ be two locally compact Hausdorff spaces equipped with regular Borel measures $\mu$ and $\mu'$, respectively.
A kernel $K\in\Leb_0(\Domain\times\Domain')$ is called a \emph{generalized Mercer kernel} induced by the \emph{left-sided} and \emph{right-sided} expansion sets
\begin{equation}\label{eq:GenMerKer-expan}
\Sset_{K}:=\left\{\phi_n:n\in\NN\right\}\subseteq\Leb_0(\Domain),
\quad
\Sset_{K}':=\left\{\adjphi_n:n\in\NN\right\}\subseteq\Leb_0(\Domain'),
\end{equation}
if the kernel $K$ can be written as the pointwise convergent representation
\begin{equation}\label{eq:GenMerKer}
K(\vx,\vy)=\sum_{n\in\NN}\phi_n(\vx)\adjphi_n(\vy),\quad \text{for }\vx\in\Domain\text{ and }\vy\in\Domain'.
\end{equation}
\end{definition}
%/////////////////////////////////////////////////////////////////////////////////////////////////////////////////////

The adjoint kernel $\adjK$ of the generalized Mercer kernel $K$ is also a generalized Mercer kernel, and $\Sset_{K}$ and $\Sset_{K}'$ are the left-sided and right-sided expansion sets of the adjoint kernel $\adjK$, respectively, that is, $\Sset_{\adjK}=\Sset_{K}'$ and $\Sset_{\adjK}'=\Sset_{K}$.

%////////////////////////////////////////////////////////////////////////////////////////////////////////////////////////
\begin{remark}\label{r:Gen-Mercer-Ker}
In the theory of linear integral equations, if the expansion sets are finite, then $K$ is a separable (degenerate) kernel (see \cite{ChenMicchelliXu2015} and \cite[Section~11]{Kress1989}). Here, we mainly focus on the infinite countable expansion terms of the kernel $K$.

Actually, the expansion sets of the generalized Mercer kernel may not be unique.
To avoid confusions, the expansion sets $\Sset_K$ and $\Sset_K'$ of the generalized Mercer kernel $K$ are FIXED in this article, and we shall give another symbol $W$ to represent this kernel $K$ if $K$ has another expansion sets $\Sset_W$ and $\Sset_W'$. In the following sections, the expansion sets will be used to construct the $p$-norm RKBSs. Differently from RKHSs, the $p$-norm RKBSs induced by a variety of expansion sets of the generalized Mercer kernels may not be unique (see the discussions of equivalent eigenfunctions of positive definite kernels in Section~\ref{s:RKBS-PDK}).

To reduce the complexity of the notation, the index of the expansion sets is FIXED to be the natural numbers $\NN$ in this article. However, it is not difficult for us to extend all the theorems of the generalized Mercer kernel to another countable index $\mathcal{I}$ the same as \cite[Section~4.5]{SteinwartChristmann2008}. For example, the expansion of the generalized Mercer kernel $K$ can be rewritten as
\[
K(\vx,\vy)=\sum_{n\in\mathcal{I}}\phi_n(\vx)\adjphi_n(\vy),\quad \text{for }\vx\in\Domain\text{ and }\vy\in\Domain'.
\]
\end{remark}
%////////////////////////////////////////////////////////////////////////////////////////////////////////////////////////

It is obvious that the classical Mercer kernel is a special case of the generalized Mercer kernels.
A generalized Mercer kernel can also be non-symmetric.
As shown by the corollaries of the Stone-Weierstrass theorem, the continuous kernels defined on the compact domains always have the expansion sums the same as in equation~\eqref{eq:GenMerKer} (see the examples in \cite[Section~11]{Kress1989}).
Furthermore, we give another example of an integral-type kernel
\[
K(\vx,\vy):=\int_{\Gamma}\Psi_l(\vx,\vz)\Psi_r(\vz,\vy)\measure(\ud\vz),
\quad\text{for }\vx\in\Domain\text{ and }\vy\in\Domain',
\]
where $\Psi_l(\vx,\cdot),\Psi_r(\cdot,\vy)\in\Leb_2(\Gamma)$ are defined on a compact Hausdorff space $\Gamma$ equipped with a finite Borel measure $\measure$. Let
\[
\phi_n(\vx):=\int_{\Gamma}\Psi_l(\vx,\vz)\varphi_n(\vz)\measure(\ud\vz),
\quad
\adjphi_n(\vy):=\int_{\Gamma}\Psi_r(\vz,\vy)\varphi_n(\vz)\measure(\ud\vz),
\quad \text{for }n\in\NN,
\]
where $\left\{\varphi_n:n\in\NN\right\}$ is an orthonormal basis of $\Leb_2(\Gamma)$.
Since
\[
\Psi_l(\vx,\vz)=\sum_{n\in\NN}\phi_n(\vx)\varphi_n(\vz),
\quad
\Psi_r(\vz,\vy)=\sum_{n\in\NN}\varphi_n(\vz)\adjphi_n(\vy),
\]
we verify that
\[
K(\vx,\vy)=\int_{\Gamma}\sum_{n\in\NN}\phi_n(\vx)\varphi_n(\vz)\sum_{m\in\NN}\varphi_m(\vz)\adjphi_m(\vy)\measure(\ud\vz)
=\sum_{n\in\NN}\phi_n(\vx)\adjphi_n(\vy).
\]
Thus, this integral-type kernel $K$ is a generalized Mercer kernel if the expansion terms $\phi_n$ and $\adjphi_n$ have infinite countable non-zero elements. We shall discuss a special integral-type kernels induced by positive definite kernels in Section~\ref{s:RKBS-PDK}.

In particular, the generalized Mercer kernel $K$ is called \emph{totally symmetric} if $\Sset_{K}=\Sset_{K}'$.
Here, even though the generalized Mercer kernel $K$ is symmetric, we still do not know whether its expansion sets are the same or not.
The total symmetry indicates that its adjoint kernel is equal to itself, that is, $K(\vx,\vy)=K(\vy,\vx)$ for all $\vx,\vy\in\Domain=\Domain'$.
For any
$\vc:=\left(c_k:k\in\NN_N\right)\in\RR^N$ and any pairwise distinct data points $X:=\left\{\vx_k:k\in\NN_N\right\}\subseteq\Domain$,
we compute the quadratic form
\[
\sum_{j,k\in\NN_N}c_jc_kK(\vx_j,\vx_k)=\sum_{n\in\NN}\sum_{j,k\in\NN_N}c_jc_k\phi_n(\vx_j)\phi_n(\vx_k)
=\sum_{n\in\NN}\left(\sum_{k\in\NN_N}c_k\phi_n(\vx_k)\right)^2\geq0.
\]
Therefore, the totally symmetric generalized Mercer kernel $K$ is a positive definite kernel, and the kernel $K$ is strictly positive definite if and only if $\Kset_{K}'=\Kset_{K}$ are linearly independent. If $\Sset_{K}\subseteq\Cont(\Domain)$ and $\sum_{n\in\NN}\abs{\phi_n(\vx)}$ is uniformly convergent on $\Domain$, then $K\in\Cont(\Domain\times\Domain)$ and its expansion is absolutely and uniformly convergent. This indicates that this totally symmetric generalized Mercer kernel is also a classical Mercer kernel.

In the following sections we shall construct the RKBSs by the generalized Mercer kernels based on the idea of the Mercer representations of RKHSs.

%------------------------------------------------------------------------------------------------------------------------
\section{Constructing $p$-norm Reproducing Kernel Banach Spaces for $1<p<\infty$}\label{s:p-RKBS}
%------------------------------------------------------------------------------------------------------------------------
\sectionmark{Constructing $p$-norm RKBS for $1<p<\infty$}

In this section we mainly focus on how to set up the $p$-norm RKBSs for $1<p<\infty$ such that their reproducing kernels are the given generalized Mercer kernels.
The main idea is to show that these $p$-norm RKBSs and the standard $p$-norm space of countable sequences have the same geometrical structures. Moreover, we verify the imbedding, compactness, and universal approximation of the $p$-norm RKBSs.

The Mercer representation theorem of RKHSs (\cite[Theorem~10.29]{Wendland2005} and \cite[Theorem~4.51]{SteinwartChristmann2008}) guarantees that
the RKHS $\Hilbert_K(\Domain)$ of the classical Mercer kernel
$K$ can be represented by its expansion set $\left\{\phi_n:n\in\NN\right\}$ defined in equation~\eqref{eq:base-MerKer}, that is,
\[
\Hilbert_K(\Domain)
=\left\{f:=\sum_{n\in\NN}a_n\phi_n:~\left(a_n:n\in\NN\right)\in \ltwo\right\},
\]
equipped with the norm
\[
\norm{f}_{\Hilbert_K(\Domain)}=\left(\sum_{n\in\NN}\abs{a_n}^2\right)^{1/2},
\]
where $\ltwo$ is the collection of all countable sequences of scalars with the standard norm $\norm{\cdot}_2$.
The key point of the explicit representation of RKHSs is that $\Hilbert_K(\Domain)$ and $\ltwo$ are isometrically isomorphic and the standard $n$th-coordinate unit vector of countable sequences is the isometrically equivalent element of $\phi_n$ for all $n\in\NN$. This means that the collection of the expansion terms $\phi_n$ can be viewed as a Schauder basis of $\Hilbert_K(\Domain)$ and a biorthogonal system of itself.

To clarify the definitions, we review the Banach space theory of the Schauder basis and its biorthogonal system.
A set $\Eset:=\left\{\phi_n:n\in\NN\right\}$ in a Banach space $\Banach$ is called a Schauder basis of $\Banach$ if for any $f\in\Banach$ there is a unique sequence $\left(a_n:n\in\NN\right)$ of scalars such that $f=\sum_{n\in\NN}a_n\phi_n$.
The Schauder basis $\Eset$ is normalized if $\norm{\phi_n}_{\Banach}=1$ for all $n\in\NN$, and we call the Schauder basis $\Eset$ is unconditional if the expansion $f=\sum_{n\in\NN}a_n\phi_n$ is convergent unconditionally.
A set $\Eset':=\left\{\adjphi_n:n\in\NN\right\}$ is a biorthogonal system of the Schauder basis $\Eset$ if $\langle f,\adjphi_n \rangle_{\Banach}=a_n$ for all $f=\sum_{n\in\NN}a_n\phi_n\in\Banach$, and we call $\adjphi_n$ the $n$th coordinate (biorthogonal) functional. More details can be found in \cite[Sections 4.1 and 4.2]{Megginson1998}.

This provides a new approach to construct the $p$-norm RKBSs and their dual spaces by the given expansion sets of the generalized Mercer kernels such that the $p$-norm RKBSs and the sequence space $\lp$ are isometrically isomorphic. Here, $\lp$ is the collection of all countable sequences of scalars with the standard norm $\norm{\cdot}_{p}$, that is,
\[
\lp:=\left\{\va:=\left(a_n:n\in\NN\right):\left\{a_n:n\in\NN\right\}\subseteq\RR\text{ and }\sum_{n\in\NN}\abs{a_n}^p<\infty\right\},
\]
equipped with the norm
\[
\norm{\va}_{p}:=\left(\sum_{n\in\NN}\abs{a_n}^p\right)^{1/p}.
\]

In addition, the fact that the expansion sets $\left\{\phi_n:n\in\NN\right\}$ of the RKHS $\Hilbert_K(\Domain)$ are linearly independent and satisfy that
\[
\sum_{n\in\NN}\abs{\phi_n(\vx)}^2=K(\vx,\vx)<\infty,\quad\text{for all }\vx\in\Domain.
\]
This indicates that we need additional conditions for completing the setting and proofs of RKBSs.

\begin{assumption}[A-$p$]
Let $1<p,q<\infty$ such that $p^{-1}+q^{-1}=1$.
Suppose that
the left-sided and right-sided expansion sets $\Sset_{K}$ and $\Sset_{K}'$
of
the generalized Mercer kernel $K$
are linearly independent and satisfy that
\[
\tag{C-$p$}
\sum_{n\in\NN}\abs{\phi_n(\vx)}^q<\infty\text{ for all }\vx\in\Domain,
\quad
\sum_{n\in\NN}\abs{\adjphi_n(\vy)}^p<\infty\text{ for all }\vy\in\Domain'.
\]
\end{assumption}\label{a:A-p}
Here, a set $\Eset$ of a linear vector space is called \emph{linearly independent} if, for any $N\in\NN$ and any finite pairwise distinct elements $\phi_1,\ldots,\phi_N\in\Eset$, their linear combination $\sum_{k\in\NN_N}c_k\phi_k=0$ if and only if $c_1=\ldots=c_N=0$.

To simplify the notation, we define the left-sided and right-sided upper-bound functions of the generalized Mercer kernel $K$ as
\[
\Phi_q(\vx):=\sum_{n\in\NN}\abs{\phi_n(\vx)}^q,\quad
\Phi_p'(\vy):=\sum_{n\in\NN}\abs{\adjphi_n(\vy)}^p,
\]
for all $\vx\in\Domain$ and all $\vy\in\Domain'$, respectively. If the expansion sets $\Sset_{K}$ and $\Sset_{K}'$ satisfy conditions (C-$p$), then the functions $\Phi_q$ and $\Phi_p'$ are well-defined pointwise.

%////////////////////////////////////////////////////////////////////////////////////////////////////////////////////////
\begin{remark}
If the expansion terms of degenerate kernels are linearly dependent, then the finite expansion terms can be reduced. Hence, the assumption of the linear independence of the expansion terms of generalized Mercer kernels is meaningful.
Clearly, the generalized Mercer kernels have many kinds of expansion sets. But, conditions (C-$p$) may not be true for all of them. Thus, the assumption of conditions (C-$p$) is still necessary for the constructions of $p$-norm RKBSs.

Since the expansion sets $\Sset_{K}$ and $\Sset_{K}'$ have linearly independent countable elements,
the domains $\Domain$ and $\Domain'$ can NOT be finite sets.
It is obvious that $p$-norm RKBSs and generalized Mercer kernels can still be set up by finite expansion sets by the following processes. But, people are greatly interested in infinite dimensional Banach spaces.
In applications, the domains $\Domain$ and $\Domain'$ are usually chosen to be a variety of manifolds of $d$-dimensional space $\Rd$.
Therefore, we mainly discuss the infinite dimensional RKBSs in this article.
\end{remark}
%////////////////////////////////////////////////////////////////////////////////////////////////////////////////////////

%////////////////////////////////////////////////////////////////////////////////////////////////////////////////////////
\begin{proposition}\label{p:MercerKer-Cp}
If the measurable functions $\phi_n\in\Leb_0(\Domain)$ and $\adjphi_n\in\Leb_0(\Domain')$ for all $n\in\NN$ satisfy conditions~(C-$p$), then
the generalized Mercer kernel $K$ defined in equation~\eqref{eq:GenMerKer}
is well-defined.
\end{proposition}
%////////////////////////////////////////////////////////////////////////////////////////////////////////////////////////
\begin{proof}
Combining the Cauchy-Schwarz inequality and conditions~(C-$p$), we check that
\[
\sum_{n\in\NN}\abs{\phi_n(\vx)\adjphi_n(\vy)}
\leq
\Phi_q(\vx)^{1/q}\Phi_p'(\vy)^{1/p}
<\infty,
\]
for all $\vx\in\Domain$ and all $\vy\in\Domain'$.
Therefore, the kernel $K$ converges pointwise.
\end{proof}
%////////////////////////////////////////////////////////////////////////////////////////////////////////////////////////

In particular, if $\Sset_{K}\subseteq\Cont(\Domain)$ and $\Sset_{K}'\subseteq\Cont(\Domain')$ such that $\sum_{n\in\NN}\abs{\phi_n(\vx)}^q$ and $\sum_{n\in\NN}\abs{\adjphi_n(\vy)}^p$ are uniformly convergent on $\Domain$ and $\Domain'$, respectively, then $K$ is continuous and its expansion is absolutely and uniformly convergent.

Now we describe the representations of the $p$-norm RKBSs.
In this section, if there is not any specific illumination, then the expansion sets $\Sset_K$ and $\Sset_K'$ of the generalized Mercer kernel $K$ satisfy assumption (A-$p$).
Using these expansion sets $\Sset_{K}$ and $\Sset_{K}'$,
we construct the spaces $\Banach_K^p(\Domain)$ and $\Banach_{\adjK}^q(\Domain')$ composing of functions defined on the domains $\Domain$ and $\Domain'$, respectively, that is,
\begin{equation}\label{eq:RKBS-p}
\Banach_K^p(\Domain):=\left\{f:=\sum_{n\in\NN}a_n\phi_n:\left(a_n:n\in\NN\right)\in \lp\right\},
\end{equation}
equipped with the semi-norm
\[
\norm{f}_{\Banach_K^p(\Domain)}:=\left(\sum_{n\in\NN}\abs{a_n}^p\right)^{1/p},
\]
and
\begin{equation}\label{eq:RKBS-q}
\Banach_{\adjK}^q(\Domain'):=\left\{g:=\sum_{n\in\NN}b_n\adjphi_n:\left(b_n:n\in\NN\right)\in \lq\right\},
\end{equation}
equipped with the semi-norm
\[
\norm{g}_{\Banach_{\adjK}^q(\Domain')}:=\left(\sum_{n\in\NN}\abs{b_n}^q\right)^{1/q}.
\]
For the reason that $\adjK(\vy,\vx)=K(\vx,\vy)$ for $\vx\in\Domain$ and $\vy\in\Domain'$,
the spaces $\Banach_{\adjK}^q(\Domain')$ and $\Banach_{K}^p(\Domain)$ are also constructed by the expansion sets $\Sset_{\adjK}=\Sset_{K}'$ and $\Sset_{\adjK}'=\Sset_{K}$ of the adjoint kernel $\adjK$, respectively. Roughly speaking $\Banach_{\adjK}^q(\Domain')$ can be thought as the adjoint format of $\Banach_K^p(\Domain)$. Moreover, the linear independence of $\Sset_{K}$ and $\Sset_{K}'$ will guarantee that the $p$-norm and $q$-norm will be well-defined on $\Banach_K^p(\Domain)$ and $\Banach_{\adjK}^q(\Domain')$, respectively.

%////////////////////////////////////////////////////////////////////////////////////////////////////////////////////////
\begin{remark}\label{r:RKBS-pq}
The generalized Mercer kernel $K$ can have many choices of linearly independent expansion sets.
Obviously, the choice of various linearly independent expansion sets affects the norms of $K(\vx,\cdot)$ and $K(\cdot,\vy)$.
The $p$-norm RKBSs constructed based on various linearly independent expansion sets of $K$ would be different
but they are all isometrically isomorphic by some linear transformations.
In this chapter, we mainly study what kind of kernels can become a two-sided reproducing kernel of some two-sided RKBSs with different normed structures. Same as Remark~\ref{r:Gen-Mercer-Ker}, we think that the expansion sets $\Sset_K$ and $\Sset_K'$ are always fixed when the generalized Mercer kernel $K$ is given.
This means that the spaces $\Banach_{K}^p(\Domain)$ and $\Banach_{\adjK}^q(\Domain')$ are UNIQUE in this article.
In Section~\ref{s:RKBS-PDK}, we look at the relationships of various choices of expansion sets of positive definite kernels.
\end{remark}
%////////////////////////////////////////////////////////////////////////////////////////////////////////////////////////

Next, we verify that the function spaces $\Banach_K^p(\Domain)$ and $\Banach_{\adjK}^q(\Domain')$ are well-defined on $\Domain$ and $\Domain'$ pointwise by conditions (C-$p$).
For any $\vx\in\Domain$, by the Cauchy Schwarz inequality, there holds
\[
\abs{f(\vx)}
\leq\left(\sum_{n\in\NN}\abs{a_n}^p\right)^{1/p}
\left(\sum_{n\in\NN}\abs{\phi_n(\vx)}^q\right)^{1/q}
=\Phi_q(\vx)^{1/q}\norm{f}_{\Banach_K^p(\Domain)},
\]
for all $f:=\sum_{n\in\NN}a_n\phi_n\in\Banach_K^p(\Domain)$.
In the same way, we prove the pointwise situation of $\Banach_{\adjK}^q(\Domain_2)$ as follows. For each $\vy\in\Domain$, we have that
\[
\abs{g(\vy)}
\leq\left(\sum_{n\in\NN}\abs{b_n}^q\right)^{1/q}
\left(\sum_{n\in\NN}\abs{\adjphi_n(\vy)}^p\right)^{1/p}
=
\Phi_p'(\vy)^{1/p}\norm{g}_{\Banach_{\adjK}^q(\Domain')},
\]
for all $g:=\sum_{n\in\NN}b_n\adjphi_n\in\Banach_{\adjK}^q(\Domain')$.
This indicates that the point evaluation functionals $\delta_{\vx}$ and $\delta_{\vy}$ are continuous linear functionals on $\Banach_K^p(\Domain)$ and $\Banach_{\adjK}^q(\Domain_2)$, respectively.
{\color{black}{Since $\Sset_{K}\subseteq\Leb_0(\Domain)$ and $\Sset_{K}'\subseteq\Leb_0(\Domain')$, the pointwise limits guarantee that $\Banach_K^{p}(\Domain)\subseteq\Leb_0(\Domain)$ and $\Banach_{\adjK}^{q}(\Domain')\subseteq\Leb_0(\Domain')$.}}

In particular, if $\Sset_{K}\subseteq\Cont(\Domain)$ such that $\sum_{n\in\NN}\abs{\phi_n(\vx)}^q$ is uniformly convergent on $\Domain$, then
$\Banach_K^{p}(\Domain)\subseteq\Cont(\Domain)$
{\color{black}{,
and if $\Sset_{K}'\subseteq\Cont(\Domain')$ such that
$\sum_{n\in\NN}\abs{\adjphi_n(\vy)}^p$ is uniformly convergent on $\Domain'$, then
$\Banach_{\adjK}^{q}(\Domain')\subseteq\Cont(\Domain')$.}}

\subsection*{Reproducing Properties}

In the following parts we show that $\Banach_K^{p}(\Domain)$ and $\Banach_{\adjK}^{q}(\Domain')$ are the two-sided RKBSs with the two-sided reproducing kernels $K$ and $\adjK$, respectively.

Before presenting the proof, we introduce a useful lemma.
Let
\[
\ve_n:=\left(0,\cdots,0,1,0,\cdots\right)^T,
\]
be the standard $n$th-coordinate unite vector for any $n\in\NN$, or more precisely, its $n$th-coordinate is equal to $1$ but the other coordinates are all equal to $0$.
According to \cite[Example~4.1.3]{Megginson1998},
$\left\{\ve_n:n\in\NN\right\}$ is the \emph{standard unit vector basis} of $\lp$. This means that
$\left\{\ve_n:n\in\NN\right\}$ is an unconditional Schauder basis, and its biorthogonal system is also equal to $\left\{\ve_n:n\in\NN\right\}$.
Suppose that $\left\{\varphi_n:n\in\NN\right\}$ is a countable sequence of linearly independent functions such that
a function space
\[
\Banach_p^0:=\Span\left\{\varphi_n:n\in\NN\right\},
\]
has a well-defined norm
\[
\norm{f}_{\Banach_p^0}:=\norm{\va}_p=\left(\sum_{k\in\NN_N}\abs{a_k}^p\right)^{1/p},
\]
for $f:=\sum_{k\in\NN_N}a_k\varphi_k\in\Banach_p^0$ and $N\in\NN$.
Let $\Banach_p$ be the completion (closure) of $\Banach_p^0$ by the norm $\norm{\cdot}_{\Banach_p^0}$.
Obviously $\Banach_p$ is a Banach space. When the general basis $\left\{\varphi_n:n\in\NN\right\}$ is not required to satisfy any additional condition, then the element of $\Banach_p$ may not be a function but a distribution.

%////////////////////////////////////////////////////////////////////////////////////////////////////////////////////////
\begin{lemma}\label{l:p-RKBS}
The Banach space $\Banach_p$ defined above can be represented explicitly by the form
\[
\Banach_p=\left\{f=\sum_{n\in\NN}a_n\varphi_n:\va:=\left(a_n:n\in\NN\right)\in \lp\right\},
\]
equipped with the norm
\[
\norm{f}_{\Banach_p}:=\norm{\va}_p,
\]
and $\Banach_p$ is isometrically equivalent to $\lp$.
\end{lemma}
%////////////////////////////////////////////////////////////////////////////////////////////////////////////////////////
\begin{proof}
The proof will be completed if we verify that $\Banach_p^0$ is isometrically equivalent to $\lspace_p^0:=\Span\left\{\ve_n:n\in\NN\right\}$ because $\lp$ is the completion of $\lspace_p^0$ by the norm $\norm{\cdot}_{p}$.

Since $\left\{\varphi_n:n\in\NN\right\}$ is linearly independent, we set up an one-to-one linear map $T$ from $\Banach_p^0$ into $\lspace_p^0$ in the form
\[
T:~f:=\sum_{k\in\NN_N}a_k\varphi_n\mapsto\va:=\left(a_k:k\in\NN_N\in\NN\right)=\sum_{k\in\NN_N}a_k\ve_k.
\]
Since
\[
\norm{T\left(f\right)}_{p}=\norm{\va}_{p}
=\norm{f}_{\Banach_p^0},~\text{for all }f:=\sum_{k\in\NN_N}a_n\varphi_k\in\Banach_p^0,
\]
we know that $T$ is an isometrical isomorphism. For any $\va:=\left(a_k:k\in\NN_N\right)\in \lspace_p^0$ there exists an element $f:=\sum_{k\in\NN_N}a_k\varphi_k\in\Banach_p^0$ such that $T\left(f\right)=\sum_{k\in\NN_N}a_k\ve_k=\va$.
Thus, the isometrical isomorphism $T$ is also surjective, which ensures that $\Banach_p^0$ is isometrically equivalent to $\lspace_p^0$.

According to the isometrical isomorphism, $\left\{\varphi_n:n\in\NN\right\}$ is the equivalent element of $\left\{\ve_n:n\in\NN\right\}$. Since $\left\{\ve_n:n\in\NN\right\}$ is the standard unit vector basis of $\lp$, $\left\{\varphi_n:n\in\NN\right\}$ is a Schauder basis of $\Banach_p$ stated in \cite[Propositions~4.1.7 and 4.1.8]{Megginson1998}. The proof is complete.
\end{proof}
%////////////////////////////////////////////////////////////////////////////////////////////////////////////////////////

Replacing the basis $\left\{\varphi_n:n\in\NN\right\}$ of $\Banach_p$ by the expansion sets $\Sset_K$ and $\Sset_K'$, we complete the proof of $\Banach_K^p(\Domain)$ and $\Banach_{\adjK}^q(\Domain')$ by Lemma~\ref{l:p-RKBS}.

%////////////////////////////////////////////////////////////////////////////////////////////////////////////////////////
\begin{proposition}\label{p:RKBS-MecerKer-pq}
Let $1<p,q<\infty$ such that $p^{-1}+q^{-1}=1$.
If a kernel $K\in\Leb_0(\Domain\times\Domain')$ is a generalized Mercer kernel such that
the expansion sets $\Sset_K$ and $\Sset_K'$ of $K$ satisfy assumption (A-$p$),
then $\Banach_K^p(\Domain)$ is a reflexive Banach space and the dual space of $\Banach_K^p(\Domain)$ is isometrically equivalent to the space $\Banach_{\adjK}^q(\Domain')$.
\end{proposition}
%////////////////////////////////////////////////////////////////////////////////////////////////////////////////////////
\begin{proof}
According to the definitions of $\Banach_K^p(\Domain)$ and $\Banach_{\adjK}^q(\Domain')$ defined in equations~\eqref{eq:RKBS-p} and~\eqref{eq:RKBS-q},
we complete the proof by using
Lemma~\ref{l:p-RKBS} that provides the isometrically geometrical structures of function spaces and sequence spaces. Hence $\Banach_K^p(\Domain)$ and $\Banach_{\adjK}^q(\Domain')$ are Banach spaces. Moreover $\Banach_K^p(\Domain)$ and $\Banach_{\adjK}^q(\Domain')$ are isometrically equivalent to $\lp$ and $\lq$, respectively.

Since $\lp$ is reflexive, the isometrical isomorphism guarantees the reflexivity of $\Banach_K^p(\Domain)$.
Moreover, since $1<p,q,<\infty$ and $p^{-1}+q^{-1}=1$, we have that $\left(\lp\right)'\cong \lq$; hence $\Banach_{\adjK}^q(\Domain')$ is isometrically equivalent to the dual space of $\Banach_K^p(\Domain)$.
\end{proof}
%////////////////////////////////////////////////////////////////////////////////////////////////////////////////////////

The explicit forms of the dual bilinear products can be used to prove their two-sided reproducing properties.

%////////////////////////////////////////////////////////////////////////////////////////////////////////////////////////
\begin{theorem}\label{t:RKBS-MercerKer-p}
Let $1<p,q<\infty$ such that $p^{-1}+q^{-1}=1$.
If a kernel $K\in\Leb_0(\Domain\times\Domain')$ is a generalized Mercer kernel such that
the expansion sets $\Sset_K$ and $\Sset_K'$ of $K$ satisfy assumption (A-$p$),
then $\Banach_K^p(\Domain)$ is a two-sided reproducing kernel Banach space with the two-sided reproducing kernels $K$.
\end{theorem}
%////////////////////////////////////////////////////////////////////////////////////////////////////////////////////////
\begin{proof}
We shall prove the reproducing properties of $\Banach_K^p(\Domain)$ and $\Banach_{\adjK}^q(\Domain')$ by their explicit structures.
The representation
$$
K(\vx,\vy)=\sum_{n\in\NN}\phi_n(\vx)\adjphi_n(\vy)
$$
is used in the following proof.

Proposition~\ref{p:RKBS-MecerKer-pq} already shows that $\Banach_K^p(\Domain)$ is a Banach space such that the dual space of $\Banach_K^p(\Domain)$ can be viewed as $\Banach_{\adjK}^q(\Domain')$. Moreover,
$\Banach_K^p(\Domain)\cong \lp$ and $\Banach_{\adjK}^q(\Domain')\cong \lq$. Thus,
each $f:=\sum_{n\in\NN}a_{n}\phi_n\in\Banach_K^p(\Domain)$ and each $g:=\sum_{n\in\NN}b_{n}\adjphi_n\in\Banach_{\adjK}^q(\Domain')$
are the identical elements of $\va:=\left(a_n:n\in\NN\right)\in \lp$ and $\vb:=\left(b_n:n\in\NN\right)\in \lq$, respectively.
Following the isometrical isomorphisms, the dual bilinear product defined on $\Banach_K^p(\Domain)$ and $\Banach_{\adjK}^q(\Domain')$ can be represented in the form
\begin{equation}\label{eq:RKBS-lp-lq-dual}
\langle f,g \rangle_{\Banach_K^p(\Domain)}=\langle \va,\vb \rangle_{\lp}=\sum_{n\in\NN}a_{n}b_{n}.
\end{equation}

We first show the right-sided reproducing properties of $\Banach_K^p(\Domain)$. Take any $\vx\in\Domain$ and any $f\in\Banach_K^p(\Domain)$.
Since conditions (C-$p$) provide that $\sum_{n\in\NN}\abs{\phi_n(\vx)}^q<\infty$, we have that $\left(\phi_n(\vx):n\in\NN\right)\in \lq$; hence
$$
K(\vx,\cdot)=\sum_{n\in\NN}\phi_n(\vx)\adjphi_n\in\Banach_{\adjK}^q(\Domain').
$$
Putting $f$ and $K(\vx,\cdot)$ into equation~\eqref{eq:RKBS-lp-lq-dual}, we obtain that
\[
\langle f,K(\vx,\cdot) \rangle_{\Banach_K^p(\Domain)}
=\sum_{n\in\NN}a_{n}\phi_n(\vx)=f(\vx).
\]
In addition, the left-sided reproducing properties of $\Banach_K^p(\Domain)$ can be proved by the same techniques.
Let $\vy\in\Domain'$ and $g\in\Banach_{\adjK}^q(\Domain')$.
According to conditions (C-$p$) of $\sum_{n\in\NN}\abs{\adjphi_n(\vy)}^p<\infty$ , we find that
$$
K(\cdot,\vy)=\sum_{n\in\NN}\adjphi_n(\vy)\phi_n\in\Banach_K^p(\Domain).
$$
Hence, by equation~\eqref{eq:RKBS-lp-lq-dual}, we have that
\[
\langle K(\cdot,\vy),g \rangle_{\Banach_K^p(\Domain)}
=\sum_{n\in\NN}b_{n}\adjphi_n(\vy)=g(\vy).
\]
Therefore, the Banach space $\Banach_K^p(\Domain)$ has the two-sided reproducing kernel $K$.
\end{proof}
%/////////////////////////////////////////////////////////////////////////////////////////////////////////////////////

%////////////////////////////////////////////////////////////////////////////////////////////////////////////////////////
\begin{remark}\label{r:RKBS-MercerKer-pq}
In~\cite{FasshauerHickernellYe2013,YeThesis2012},
the RKBS $\Banach_{\Phi}^p(\Rd)$ induced by the positive definite function $\Phi$ is isometrically equivalent to $\Leb_q(\Rd;\mu)$, because we would like to compare $\Banach_{\Phi}^p(\Rd)$ and the classical Sobolev space in a straightforward view. In this article we let the RKBS $\Banach_K^p(\Domain)$ be isometrically equivalent to $\lp$ for a convenient way to solve the support vector machines in the $1$-norm Banach spaces.
\end{remark}
%////////////////////////////////////////////////////////////////////////////////////////////////////////////////////////

Because of the reflexivity of the RKBS $\Banach_K^p(\Domain)$, we also obtain the two-sided reproducing properties of its dual space $\Banach_{\adjK}^q(\Domain')$ by the adjoint kernel $\adjK$ of $K$. Therefore, we have the following result.

%////////////////////////////////////////////////////////////////////////////////////////////////////////////////////////
\begin{theorem}\label{t:RKBS-MercerKer-q}
Let $1<p,q<\infty$ such that $p^{-1}+q^{-1}=1$.
If a kernel $K\in\Leb_0(\Domain\times\Domain')$ is a generalized Mercer kernel such that
the expansion sets $\Sset_K$ and $\Sset_K'$ of $K$ satisfy assumption (A-$p$),
then $\Banach_{\adjK}^q(\Domain')$ is a two-sided reproducing kernel Banach space with the two-sided reproducing kernel $\adjK$, the adjoint kernel of $K$.
\end{theorem}
%////////////////////////////////////////////////////////////////////////////////////////////////////////////////////////

\emph{Comparisons:}
Since $\Banach_K^p(\Domain)$ and $\Banach_{\adjK}^q(\Domain')$ are reflexive, they further satisfy the stronger definition of RKBSs given in~\cite{ZhangXuZhang2009}.
Comparing the norms of $\Banach_K^p(\Domain)$ and $\Hilbert_K(\Domain)$, we find that the $p$-norm two-sided RKBS is an extension of the Mercer representations of RKHSs.
Since $\lp$ is uniformly convex and smooth, the isometrical isomorphism ensures that the RKBS $\Banach_K^p(\Domain)$ is also uniformly convex and smooth which indicates that $p$-norm RKBSs have the same geometrical properties as the RKHSs.
The expansion sets $\Sset_{K}$ and $\Sset_{K}'$ are both equivalent to the standard unit vector basis of $\lp$ and $\lq$, respectively.
Hence the left-sided expansion set $\Sset_{K}$ can be viewed as the normalized Schauder basis of $\Banach_K^p(\Domain)$,
which is also unconditional and shrinking.
Moreover, the right-sided expansion set $\Sset_{K}'$ can be seen as the biorthogonal system of $\Sset_{K}$ because
\[
\langle \phi_n,\adjphi_m \rangle_{\Banach_K^p(\Domain)}=\delta_{mn},\quad\text{for all }m,n\in\NN,
\]
where $\delta_{mn}$ is the standard Kronecker delta.
\cite[Theorem~4.1.24]{Megginson1998} guarantees that
\begin{align*}
&f=\sum_{n\in\NN}\langle f,\adjphi_n \rangle_{\Banach_K^p(\Domain)}\phi_n,\quad\text{for }f\in\Banach_K^p(\Domain),
\end{align*}
and
\begin{align*}
&g=\sum_{n\in\NN}\langle g,\phi_n \rangle_{\Banach_{\adjK}^q(\Domain')}\adjphi_n=\sum_{n\in\NN}\langle \phi_n,g \rangle_{\Banach_K^p(\Domain)}\adjphi_n,\quad\text{for }g\in\Banach_{\adjK}^q(\Domain').
\end{align*}
We roughly think that the reproducing properties of RKBSs come from the Schauder biorthogonal system.
If conditions (C-$p$) do not claim, then the function spaces $\Banach_K^p(\Domain)$ and $\Banach_{\adjK}^q(\Domain')$ may still be well-defined pointwise and their isometrical isomorphisms onto $\lp$ and $\lq$ may also be true. But, we can not confirm that the point evaluation functionals belong to the dual spaces of $\Banach_K^p(\Domain)$ or $\Banach_{\adjK}^q(\Domain')$ without conditions (C-$p$) because $K(\vx,\cdot)$ or $K(\cdot,\vy)$ may not be an element of $\Banach_{\adjK}^q(\Domain')$ or $\Banach_K^p(\Domain)$. This shows that the Banach space $\Banach$ may not be a RKBS even though $\Banach$ isometrically equivalent to $\lp$.

%////////////////////////////////////////////////////////////////////////////////////////////////////////////////////////
\begin{corollary}\label{c:RKHS-MercerKer}
If a kernel $K\in\Leb_0(\Domain\times\Domain)$ is a totally symmetric generalized Mercer kernel such that
the expansion sets $\Sset_K=\Sset_K'$ of $K$ satisfy assumption (A-$2$),
then $\Banach_K^2(\Domain)$ is a reproducing kernel Hilbert space with the reproducing kernel $K$.
\end{corollary}
%////////////////////////////////////////////////////////////////////////////////////////////////////////////////////////
\begin{proof}
It is clear that $K(\vx,\vy)=K(\vy,\vx)$ for all $\vx,\vy\in\Domain=\Domain'$.
Since $\Banach_K^2(\Domain)\cong \ltwo$ and $\ltwo$ is a Hilbert space, the RKBS $\Banach_K^2(\Domain)$ is also a Hilbert space.
\end{proof}
%////////////////////////////////////////////////////////////////////////////////////////////////////////////////////////

We employ the dual bilinear product to define the reproducing properties and the inner product is just its special case.
When the generalized Mercer kernel $K$ is not totally symmetric, $\Banach_K^2(\Domain)$ is still a Hilbert space; but it may not be a classical RKHS with the reproducing kernel $K$ and we need to introduce another kernel $K_H(\vx,\vy):=\sum_{n\in\NN}\phi_n(\vx)\phi_n(\vy)$ for the inner-product reproducing properties.

\subsection*{Imbedding}

Now we look at the imbedding of the $p$-norm RKBSs.
Define the \emph{left-sided} and \emph{right-sided} sequence sets of the generalized Mercer kernel $K$, respectively, by
\begin{equation}\label{eq:A-K-phi-psi}
\Aset_{K}:=\left\{\left(\phi_n(\vx):n\in\NN\right):\vx\in\Domain\right\},
\quad
\Aset_{K}':=\left\{\left(\adjphi_n(\vy):n\in\NN\right):\vy\in\Domain'\right\}.
\end{equation}
Clearly $\Aset_{K}'=\Aset_{K'}$.
Conditions (C-$p$) ensure that $\Aset_{K}\subseteq \lq$ and $\Aset_{K}'\subseteq \lp$.

%////////////////////////////////////////////////////////////////////////////////////////////////////////////////////////
\begin{proposition}\label{p:imbedding-RKBS-pq}
Let $1<p,q<\infty$ such that $p^{-1}+q^{-1}=1$
and let $K\in\Leb_0(\Domain\times\Domain')$ be a generalized Mercer kernel such that
the expansion sets $\Sset_K$ and $\Sset_K'$ of $K$ satisfy assumption (A-$p$). If
the measure $\mu(\Domain)$ is finite and the left-sided sequence set $\Aset_{K}$ of $K$ is bounded in $\lq$, then the identity map from $\Banach_K^p(\Domain)$ into $\Leb_q(\Domain)$ is continuous.
In particular, if $\sum_{n\in\NN}\abs{\phi(\vx)}^q$ is uniformly convergent on $\Domain$ and the support of $\mu$ is equal to $\Domain$, then
$\Banach_K^p(\Domain)$ is imbedded into $\Leb_q(\Domain)$.
\end{proposition}
%////////////////////////////////////////////////////////////////////////////////////////////////////////////////////////
\begin{proof}
According to Theorem~\ref{t:RKBS-MercerKer-p}, the space $\Banach_K^p(\Domain)$ is a two-sided RKBS with the two-sided reproducing kernel $K$.
According to Proposition~\ref{p:RKBS-imbedding}, we shall check the condition of $\vx\mapsto\norm{K(\vx,\cdot)}_{\Banach_{\adjK}^q(\Domain')}\in\Leb_q(\Domain)$ to complete the proof.

Next, we compute that
\[
\norm{K(\vx,\cdot)}_{\Banach_{\adjK}^q(\Domain')}
=\Phi_q(\vx)^{1/q},\quad
\text{for }\vx\in\Domain.
\]
Hence, it suffices to verify that $\Phi_q^{1/q}\in\Leb_q(\Domain)$.
Since $\Aset_K$ is bounded in $\lq$,
we obtain the positive constant
\[
\norm{\Phi_q}_{\infty}=\sup_{\vx\in\Domain}\Phi_q(\vx)<\infty.
\]
Thus, the finite measure $\mu(\Domain)$ ensures that
\[
\int_{\Domain}\Phi_q(\vx)\mu(\ud\vx)
\leq \norm{\Phi_q}_{\infty}\mu\left(\Domain\right)<\infty.
\]
Therefore, we conclude that the identity map from $\Banach_K^p(\Domain)$ into $\Leb_q(\Domain)$ is continuous.

Finally, the uniform convergence of $\sum_{n\in\NN}\abs{\phi(\vx)}^q$ ensures that $\Banach_{K}^p(\Domain)\subseteq\Cont(\Domain)$.
Moreover, since $\supp(\mu)=\Domain$, the $p$-norm RKBS $\Banach_K^p(\Domain)$ satisfies the $\mu$-measure zero condition.
Thus, the imbedding of $\Banach_K^p(\Domain)$ into $\Leb_q(\Domain)$ is verified by Corollary~\ref{c:RKBS-imbedding}.
\end{proof}
%////////////////////////////////////////////////////////////////////////////////////////////////////////////////////////

We look at the right-sided integral operator
\[
I_{K}'(\xi)(\vx):=\int_{\Domain'}K(\vx,\vy)\xi(\vy)\mu'(\ud\vy),
\]
for $\xi\in\Leb_q(\Domain')$ and $\vx\in\Domain$. Clearly $I_K'=I_{\adjK}$.
In the following theorem we shall check that $I_{K}'$ is also a continuous operator from $\Leb_q(\Domain')$ into $\Banach_K^p(\Domain)$.

%////////////////////////////////////////////////////////////////////////////////////////////////////////////////////////
\begin{proposition}\label{p:imbedding-intopt-RKBS-pq}
Let $1<p,q<\infty$ such that $p^{-1}+q^{-1}=1$ and let
$K\in\Leb_0(\Domain\times\Domain')$ be a generalized Mercer kernel such that
the expansion sets $\Sset_K$ and $\Sset_K'$ of $K$ satisfy assumption (A-$p$).
If the measure $\mu'(\Domain')$ is finite and the right-sided sequence set $\Aset_K'$ of $K$ is bounded in $\lp$, then
the right-sided integral operator $I_{K}'$ maps $\Leb_q(\Domain')$ into $\Banach_K^p(\Domain)$ continuously and
\[
\int_{\Domain}g(\vy)\xi(\vy)\mu'(\ud\vy)=
\langle I_{K}'(\xi),g \rangle_{\Banach_K^p(\Domain)},
\]
for all $\xi\in\Leb_q(\Domain')$ and all $g\in\Banach_{\adjK}^q(\Domain')$.
In particular, if $\sum_{n\in\NN}\abs{\phi(\vy)}^p$ is uniformly convergent on $\Domain'$ and the support of $\mu'$ is equal to $\Domain'$,
then the range of $I_K'$ is dense in $\Banach_K^p(\Domain)$.
\end{proposition}
%////////////////////////////////////////////////////////////////////////////////////////////////////////////////////////
\begin{proof}
The main technique of the proof is to view $\Banach_K^p(\Domain)$ as the dual space of a two-sided RKBS so that
Proposition~\ref{p:RKBS-imbedding-dual} about the imbedding of RKBSs can be applied.

Proposition~\ref{p:RKBS-MecerKer-pq} and
Theorem~\ref{t:RKBS-MercerKer-q} ensure that $\Banach_{\adjK}^q(\Domain')$ is the two-sided RKBS with the two-sided reproducing kernel $\adjK$ and $\left(\Banach_{\adjK}^q(\Domain')\right)'\cong\Banach_{K}^p(\Domain)$.
Because of the bounded conditions of $\Aset_K'$ in $\lp$, the positive constant
\[
\norm{\Phi_p'}_{\infty}=\sup_{\vy\in\Domain'}\Phi_p'(\vy)<\infty,
\]
is well-defined.
Since $\mu'(\Domain')<\infty$,
we have that
\[
\int_{\Domain'}\Phi_p'(\vy)\mu'(\ud\vy)
\leq \norm{\Phi_p'}_{\infty}\mu'(\Domain')<\infty.
\]
Hence, $\Phi_p^{\prime1/p}\in\Leb_p(\Domain')$.
We already know that
\[
\abs{K(\vx,\vy)}
\leq\sum_{n\in\NN}\abs{\phi_n(\vx)\phi_n'(\vy)}
\leq\Phi_q(\vx)^{1/q}\Phi_p'(\vy)^{1/p},
\]
for all $\vx\in\Domain$ and all $\vy\in\Domain'$.
Thus,
\[
\int_{\Domain'}\abs{K(\vx,\vy)}^p\mu'(\ud\vy)
\leq\Phi_q(\vx)^{p/q}\int_{\Domain'}\Phi_p'(\vy)\mu'(\ud\vy)
<\infty.
\]
This shows that $K(\vx,\cdot)\in\Leb_p(\Domain')$ for all $\vx\in\Domain$.
In addition, since
\[
\norm{K(\cdot,\vy)}_{\Banach_K^{p}(\Domain)}
=\Phi_p'(\vy)^{1/p},\quad
\text{for }\vy\in\Domain',
\]
we determine that $\vy\mapsto\norm{K(\cdot,\vy)}_{\Banach_K^{p}(\Domain)}\in\Leb_p(\Domain')$.
Therefore, by Proposition~\ref{p:RKBS-imbedding-dual}, the integral operator $I_{K}'$ mapping $\Leb_q(\Domain')$ into $\Banach_{K}^p(\Domain)$ is continuous, and
\[
\int_{\Domain'}g(\vy)\xi(\vy)\mu'(\ud\vy)
=\langle g,I_{K}'(\xi) \rangle_{\Banach_{\adjK}^q(\Domain')}
=\langle I_{K}'(\xi),g \rangle_{\Banach_{K}^p(\Domain)},
\]
for all $\xi\in\Leb_q(\Domain')$ and all $g\in\Banach_{\adjK}^q(\Domain')$.

Finally, the uniform convergence of $\sum_{n\in\NN}\abs{\adjphi(\vy)}^p$ ensures that $\Banach_{\adjK}^q(\Domain')\subseteq\Cont(\Domain')$.
Combining this continuity with the equality of $\supp(\mu')=\Domain'$, we have that the $q$-norm RKBS $\Banach_{\adjK}^q(\Domain')$ satisfies the $\mu'$-measure zero condition. Moreover, since $\Banach_{\adjK}^q(\Domain')$ and $\Leb_p(\Domain)$ are reflexive,
Corollary~\ref{c:RKBS-imbedding-dual} yields the density of the range of $I_{K}'$ in $\Banach_K^p(\Domain)$.
\end{proof}
%////////////////////////////////////////////////////////////////////////////////////////////////////////////////////////

When $\Domain=\Domain'$, we could find the relationships between the two-sided RKBS $\Banach_{K}^p(\Domain)$ and the Lebesgue space $\Leb_q(\Domain)$ for the conjugate exponents of $p$ and $q$.

\subsection*{Compactness}

Moreover, we check whether the identity map from the RKBS $\Banach_{K}^p(\Domain)$ into the sup-norm space $\Linfty(\Domain)$ is a compact operator. In other words, the unit ball of $\Banach_{K}^p(\Domain)$
\[
B_{\Banach_{K}^p(\Domain)}:=\left\{f\in\Banach_{K}^p(\Domain):\norm{f}_{\Banach_{K}^p(\Domain)}\leq1\right\},
\]
is a relatively compact set of $\Linfty(\Domain)$.
However, the compactness may not be true for a $p$-norm RKBS. We require another sufficient conditions of the generalized Mercer kernel $K$ to guarantee the compactness.
We search these conditions by the compactness of RKBSs stated in Proposition~\ref{p:RKBS-compact}.

%////////////////////////////////////////////////////////////////////////////////////////////////////////////////////////
\begin{proposition}\label{p:compact-RKBS-pq}
Let $1<p,q<\infty$ such that $p^{-1}+q^{-1}=1$
and let $K\in\Leb_0(\Domain\times\Domain')$ be a generalized Mercer kernel such that
the expansion sets $\Sset_K$ and $\Sset_K'$ of $K$ satisfy assumption (A-$p$).
If the left-sided sequence set $\Aset_{K}$ of $K$ is compact in the space $\lq$, then the identity map from $\Banach_{K}^p(\Domain)$ into $\Linfty(\Domain)$ is compact.
\end{proposition}
%////////////////////////////////////////////////////////////////////////////////////////////////////////////////////////
\begin{proof}
Using Proposition~\ref{p:RKBS-compact}, the proof will be completed if we verify that the right-sided kernel set $\Kset_K'$ of $K$ is compact in the dual space $\Banach_{\adjK}^q(\Domain')$ of the RKBS $\Banach_{K}^p(\Domain)$.

By the proof of Proposition~\ref{p:RKBS-MecerKer-pq}, we have that $\Banach_{\adjK}^q(\Domain')\cong \lq$; hence $\Aset_K$ in $\lq$ is the identical element of $\Kset_K'$ in $\Banach_{\adjK}^q(\Domain')$.
Therefore, the compactness of $\Aset_K$ in $\lq$ implies the compactness of $\Kset_K'$ in $\Banach_{\adjK}^q(\Domain')$.
\end{proof}
%////////////////////////////////////////////////////////////////////////////////////////////////////////////////////////

We know that any compact set of a normed space is closed and bounded.
The bounded condition of $\Aset_K$ can be checked by $\norm{\Phi_q}_{\infty}<\infty$.
In other hands, it is not a big deal to let $\Aset_K$ become closed.
If $\Aset_K$ is not closed, then we extend the domain $\Domain$ to include the limit elements of the closure of the complement of $\Aset_K$.
For example, suppose that there exists a Cauchy sequence $\left\{\left(\phi_n(\vx_m):n\in\NN\right):m\in\NN\right\}$ in $\lp$ but
its limit $\left(\gamma_n:n\in\NN\right)$ is not in $\Aset_K$. We add a new point $\tilde{\vx}$ to $\Domain$, that is,
$\tilde{\Domain}:=\Domain\cup\left\{\tilde{\vx}\right\}$, such that the left-sided expansion set $\Sset_K$ can be extended onto $\tilde{\Domain}$ by the form
\[
\phi_n(\tilde{\vx}):=\gamma_n,\quad\text{for all }n\in\NN.
\]
Therefore, $\Aset_K$ extended onto $\tilde{\Domain}$ includes the limit $\left(\phi_n(\tilde{\vx}):n\in\NN\right)$.
Obviously $\Sset_K$ extended onto $\tilde{\Domain}$ preserves its linear independence;
hence $\Sset_K$ extended onto $\tilde{\Domain}$ still satisfies assumption (A-$p$).
Since
\[
\sum_{n\in\NN}\abs{\phi_n(\tilde{\vx})}^q
=\sum_{n\in\NN}\abs{\gamma_n}^q
=\lim_{m\to\infty}\sum_{n\in\NN}\abs{\phi_n(\vx_m)}^q
\leq \norm{\Phi_q}_{\infty},
\]
the bounded condition still preserves for the extended case.
This means that we always let $\Aset_K$ become closed and bounded. In particular, if the expansion set $\Sset_K\subseteq\Cont(\Domain)$ and $\sum_{n\in\NN}\abs{\phi_n(\vx)}^q$ is uniformly convergent on $\Domain$, which ensures that the map $\Phi_q$ is continuous on $\Domain$, then $\Aset_K$ is always closed whenever $\Domain$ is compact.

But a closed and bounded set of an infinite dimensional normed space may not be compact, for example, the unit sphere of an infinite dimensional normed space is not compact.
Even though $\Aset_K$ may not be compact,
we employ many techniques to check the compactness of $\Aset_K$ with additional conditions.
If there exist a continuous linear operator $T: \lr\to \lq$ and a bounded set $\Eset$ in $\lr$ such that
$T(\Eset)=\Aset_K$ where $1<q<r<\infty$, then Pitt's theorem provides that $\Aset_K$ is relatively compact.
In another way, a special compact linear operator $T:\lq\to \lq$ induced by a sequence $\left(\alpha_n:n\in\NN\right)\in \czero$, which indicates that
\[
T(\vb):=\left(\alpha_nb_n:n\in\NN\right),\quad\text{for }\vb=\left(b_n:n\in\NN\right)\in \lq,
\]
can be used to check the relative compactness of $\Aset_K$.
If we find a positive sequence $\left(\alpha_n:n\in\NN\right)\in \czero$ such that
\[
\left(\alpha_n^{-1}\phi_n(\vx):n\in\NN\right)\in \lq,
\]
and
\[
\sup_{\vx\in\Domain}\sum_{n\in\NN}\frac{\abs{\phi_n(\vx)}^q}{\alpha_n^q}<\infty,
\]
then $\Aset_K$ is relatively compact,
because $\left\{\left(\alpha_n^{-1}\phi_n(\vx):n\in\NN\right):\vx\in\Domain\right\}$ is bounded in $\lq$ and
\[
\left\{T\left(\left(\alpha_n^{-1}\phi_n(\vx):n\in\NN\right)\right):\vx\in\Domain\right\}=\Aset_K.
\]

\subsection*{Universal Approximation}
By the same method of \cite[Corollary~5]{MicchelliXuZhang2006},
the universal approximation of the $p$-norm RKBSs and $q$-norm RKBSs can be shown by the density of $\Span\left\{\Sset_K\right\}$ and $\Span\left\{\Sset_K'\right\}$ directly.

%////////////////////////////////////////////////////////////////////////////////////////////////////////////////////////
\begin{proposition}\label{p:universial-approx-RKBS-pq}
Let $1<p,q<\infty$ such that $p^{-1}+q^{-1}=1$
and let $K\in\Leb_0(\Domain\times\Domain')$ be a generalized Mercer kernel such that
the expansion sets $\Sset_K$ and $\Sset_K'$ of $K$ satisfy assumption (A-$p$).
If $\sum_{n\in\NN}\abs{\phi_n(\vx)}^q$ is uniformly convergent on $\Domain$, the left-sided domain $\Domain$ is a compact Hausdorff space,
and $\Span\left\{\Sset_K\right\}$ is dense in $\Cont(\Domain)$,
then the reproducing kernel Banach spaces $\Banach_{K}^p(\Domain)$ has the universal approximation property.
\end{proposition}
%////////////////////////////////////////////////////////////////////////////////////////////////////////////////////////
\begin{proof}
Now we mainly verify that $\Banach_{K}^p(\Domain)$ is dense in $\Cont(\Domain)$.

By the continuity of $\Sset_K$, the uniform convergence of $\sum_{n\in\NN}\abs{\phi_n(\vx)}^q$ provides that $\Banach_{K}^p(\Domain)\subseteq\Cont(\Domain)$.
It is clear that $\Span\left\{\Sset_K\right\}\subseteq\Banach_{K}^p(\Domain)$.
Therefore, the density of $\Span\left\{\Sset_K\right\}$ in $\Cont(\Domain)$ ensures the density of $\Banach_{K}^p(\Domain)$ in $\Cont(\Domain)$.
\end{proof}
%////////////////////////////////////////////////////////////////////////////////////////////////////////////////////////

By Proposition~\ref{p:universial-approx-RKBS-pq}, the generalized Mercer kernel $K$ can be a right-sided, left-sided, or two-side universal kernel.

\subsection*{Uniqueness}

Based on the theoretical results of the $p$-norm RKBSs, we discuss the uniqueness of reproducing kernels and RKBSs. In the first step, we take a reproducing kernel to look at its RKBS.
It is well-known that the RKHS induced by a given positive definite kernel is unique.
But, if the expansion sets set up by the eigenvalues and eigenfunctions of a positive definite kernel satisfy conditions (C-$p$) for all $1<p<\infty$, then this positive definite kernel can be a reproducing kernel of various $p$-norm RKBSs (see the examples of Min Kernel, Gaussian Kernel, and power series kernels given in Chapter~\ref{char-PDK}).

Next, we consider the reproducing kernel of a given Banach space.
There could be many different kinds of function spaces which are isometrically equivalent to the dual spaces of the $p$-norm RKBS,
because our reproducing properties are defined by the dual bilinear products of the Banach space depending on its dual space.
This indicates that the reproducing kernel can be changed when the isometrically isomorphic function space of the dual space of the $p$-norm RKBS is chosen differently.

An example, we look at two generalized Mercer kernels
\[
K(x,y):=\sum_{n\in\NN}\frac{2}{n^4\pi^4}\sin(n\pi x)\sin(n\pi y),\quad\text{for }x,y\in[0,1],
\]
and
\[
W(x,y):=\sum_{n\in\NN}\frac{2}{n^4\pi^4}\sin(n\pi x)\cos(n\pi y),\quad\text{for }x,y\in[0,1].
\]
It is obvious that $K\neq W$. Moreover, their expansion sets $\Sset_{K}=\left\{\phi_n:n\in\NN\right\},\Sset_K'=\left\{\adjphi_n:n\in\NN\right\}$ and $\Sset_{W}=\left\{\varphi_n:n\in\NN\right\},\Sset_W'=\left\{\adjvarphi_n:n\in\NN\right\}$ can be defined as
\[
\phi_n(x):=\frac{\sqrt{2}}{n^2\pi^2}\sin(n\pi x),
\quad
\adjphi_n(y):=\frac{\sqrt{2}}{n^2\pi^2}\sin(n\pi y),
\quad\text{for }n\in\NN,
\]
and
\[
\varphi_n(x):=\frac{\sqrt{2}}{n^2\pi^2}\sin(n\pi x),
\quad
\adjvarphi_n(y):=\frac{\sqrt{2}}{n^2\pi^2}\cos(n\pi y),
\quad\text{for }n\in\NN.
\]
Since
\[
\sum_{n\in\NN}\frac{2^{p/2}}{n^{2p}\pi^{2p}}\abs{\sin(n\pi x)}^p
\leq\left(\frac{\sqrt{2}}{n^2\pi^2}\abs{\sin(n\pi x)}\right)^{p}
\leq\frac{2^{(p+2)/2}}{\pi^{2p}}<\infty,
\]
and
\[
\sum_{n\in\NN}\frac{2^{q/2}}{n^{2q}\pi^{2q}}\abs{\cos(n\pi y)}^q
\leq\left(\frac{\sqrt{2}}{n^2\pi^2}\abs{\cos(n\pi y)}\right)^{q}
\leq\frac{2^{(q+2)/2}}{\pi^{2q}}<\infty,
\]
for all $x,y\in[0,1]$,
we determine that the expansion sets $\Sset_{K},\Sset_K'$ and $\Sset_{W},\Sset_W'$ are linearly independent and both satisfy conditions (C-$p$) for all $1<p<\infty$.
By Theorem~\ref{t:RKBS-MercerKer-p}, the spaces $\Banach_K^p([0,1])$ and $\Banach_W^p([0,1])$ induced by $\Sset_K$ and $\Sset_W$ are the two-sided RKBSs with the two-sided reproducing kernels $K$ and $W$, respectively.
Since $\Sset_K=\Sset_W$, we have that $\Banach_K^p([0,1])\equiv\Banach_W^p([0,1])$.
Furthermore, we find that
$$
\Banach_{\adjK}^q([0,1])\cap\Banach_{W'}^q([0,1])=\{0\}
$$
even though
$$
\Banach_{\adjK}^q([0,1])\cong\left(\Banach_K^p([0,1])\right)'=\left(\Banach_W^p([0,1])\right)'\cong\Banach_{W'}^q([0,1]),
$$
where $q$ is the conjugate exponent of $p$. This example shows that the $p$-norm RKBS $\Banach_K^p([0,1])=\Banach_W^p([0,1])$ has two different reproducing kernels $K$ and $W$
associated with different choices of isometrically equivalent function spaces $\Banach_{\adjK}^q([0,1])$ and $\Banach_{W'}^q([0,1])$ of the its dual space.

%////////////////////////////////////////////////////////////////////////////////////////////////////////////////////////
\begin{remark}\label{r:RKBS-pq-unique}
Since the generalized Mercer kernel $K$ is the two-sided reproducing kernel of $\Banach_{K}^p(\Domain)$,
Proposition~\ref{p:RKBS-uniquess-1} guarantees that the isometrically isomorphic function space $\Fun$ of the dual space $\left(\Banach_{K}^p(\Domain)\right)'$ is uniquely equal to $\Banach_{\adjK}^q(\Domain')$. Same as Remark~\ref{r:RKBS-Fun},
the dual space $\left(\Banach_{K}^p(\Domain)\right)'$ and the function space $\Banach_{\adjK}^q(\Domain')$ can be viewed as the same.
\end{remark}
%////////////////////////////////////////////////////////////////////////////////////////////////////////////////////////

\subsection*{Generalization of $p$-norm Reproducing Kernel Banach Spaces}

In the above discussion, we have already shown the constructions of the $p$-norm RKBSs of which the reproducing kernels are the given generalized Mercer kernels. But
people may have another issue whether there exists a generalized Mercer kernel which is a reproducing kernel for the given Banach space. Finally, we show what kinds of Banach spaces would possess the reproducing properties and their reproducing kernels are the generalized Mercer kernels.

Here, we consider a general Banach space $\Banach$ composed of functions
$f\in\Leb_0(\Domain)$ such that the dual space $\Banach'$ of $\Banach$ is isometrically equivalent to a normed space $\Fun$ consisting of functions $g\in\Leb_0(\Domain')$. Then the dual space $\Banach'\cong\Fun$ can be seen as a function space defined on $\Domain'$.
We hope that this Banach space $\Banach$ has similar structures of the $p$-norm RKBS $\Banach_K^p(\Domain)$.
So we suppose that $\Banach$ is reflexive, and $\Banach$ has an unconditionally normalized Schauder basis $\Eset:=\left\{\phi_n:n\in\NN\right\}$.
Then the biorthogonal system $\Eset':=\left\{\adjphi_n:n\in\NN\right\}$ of the Schauder basis $\Eset$ is well-defined.
The reflexivity of $\Banach$ also guarantees that its biorthogonal system $\Eset'$ is also an unconditionally Schauder basis of the dual space $\Banach'$ (see \cite[Theorem~4.4.1 and Corollary~4.4.16]{Megginson1998}).

Let $P_n$ be the $n$th natural projection depending on the Schauder basis $\Eset$, that is, $P_n(f):=\sum_{k\in\NN_n}a_k\phi_k$ for all $f:=\sum_{n\in\NN}a_n\phi_n\in\Banach$. Then \cite[Corollary~4.1.17]{Megginson1998} ensures that $\sup_{n\in\NN}\norm{P_n}<\infty$. Using the normalization of $\Eset$, we obtain the upper bounds of the Schauder basis $\Eset$ and the biorthogonal system $\Eset'$ as
\begin{equation}\label{eq:RKBS-MercerKer-Gen-2}
\begin{split}
&\sup_{n\in\NN}\norm{\phi_n}_{\Banach}=1<\infty,\\
&\sup_{n\in\NN}\norm{\adjphi_n}_{\Banach'}\leq 2\sup_{n\in\NN}\norm{P_n}<\infty,
\end{split}
\end{equation}
because
\[
\abs{\langle f,\adjphi_n \rangle_{\Banach}}\norm{\phi_n}_{\Banach}
=\norm{a_n\phi_n}_{\Banach}
=\norm{P_{n}(f)-P_{n-1}(f)}_{\Banach}
\leq2\sup_{n\in\NN}\norm{P_n}\norm{f}_{\Banach},
\]
for all $f\in\Banach$.
By inequality~\eqref{eq:RKBS-MercerKer-Gen-2}, we conclude that
\begin{equation}\label{eq:RKBS-MercerKer-Gen-4}
\begin{split}
&\sup_{n\in\NN}\abs{a_n}=\sup_{n\in\NN}\abs{\langle f,\adjphi_n\rangle_{\Banach}}
\leq C\norm{f}_{\Banach},\\
&\sup_{n\in\NN}\abs{b_n}=\sup_{n\in\NN}\abs{\langle \phi_n,g \rangle_{\Banach}}
\leq \norm{g}_{\Banach'}
\end{split}
\end{equation}
for all $f:=\sum_{n\in\NN}a_{n}\phi_n\in\Banach$ and all
$g:=\sum_{n\in\NN}b_{n}\adjphi_n\in\Banach'$, where the positive constant $C:=2\sup_{n\in\NN}\norm{P_n}$.
Analogous to conditions (C-$p$), we further require that the sets $\Eset$ and $\Eset'$ satisfy that
\[
\tag{C-GEM}
\sum_{n\in\NN}\abs{\phi_n(\vx)}<\infty\text{ for all }\vx\in\Domain,
\quad
\sum_{n\in\NN}\abs{\adjphi_n(\vy)}<\infty\text{ for all }\vy\in\Domain'.
\]
Conditions (C-GEN) ensure that
\[
\sum_{n\in\NN}\abs{\phi_n(\vx)\adjphi_n(\vy)}
\leq\left(\sum_{n\in\NN}\abs{\phi_n(\vx)}\right)\left(\sum_{n\in\NN}\abs{\adjphi_n(\vy)}\right)<\infty,
\]
for all $\vx\in\Domain$ and all $\vy\in\Domain'$; hence
the generalized Mercer kernel
\begin{equation}\label{eq:RKBS-MercerKer-Gen-5}
K(\vx,\vy):=\sum_{n\in\NN}\phi_n(\vx)\adjphi_n(\vy),
\end{equation}
is well-defined pointwise on the domain $\Domain\times\Domain'$. This means that
the sets $\Eset$ and $\Eset'$ can be viewed as the left-sided and right-sided expansion sets of $K$, respectively.
For any $\vx\in\Domain$ and $\vy\in\Domain'$,
by conditions (C-GEN) and inequality~\eqref{eq:RKBS-MercerKer-Gen-4}, we have that
\[
\abs{f(\vx)}\leq
\sup_{n\in\NN}\abs{a_n}\sum_{n\in\NN}\abs{\phi_n(\vx)}
\leq C\sum_{n\in\NN}\abs{\phi_n(\vx)}\norm{f}_{\Banach},
\]
and
\[
\abs{g(\vy)}\leq
\sup_{n\in\NN}\abs{b_n}\sum_{n\in\NN}\abs{\adjphi_n(\vy)}
\leq \sum_{n\in\NN}\abs{\adjphi_n(\vy)}\norm{g}_{\Banach'},
\]
for all $f\in\Banach$ and all
$g\in\Banach'$.
Therefore, the point evaluation functionals $\delta_{\vx}$ and $\delta_{\vy}$ are continuous linear functionals on $\Banach$ and $\Banach'$, respectively.
Next, we verify that this Banach space $\Banach$ is a two-sided RKBS and the generalized Mercer kernel $K$ is the two-sided reproducing kernel of $\Banach$.

%////////////////////////////////////////////////////////////////////////////////////////////////////////////////////////
\begin{theorem}\label{t:RKBS-MercerKer-Gen}
Let $\Banach$ be a reflexive Banach space composed of functions
$f\in\Leb_0(\Domain)$ such that
$\Banach$ has an unconditionally normalized Schauder basis $\Eset$,
the dual space $\Banach'$ of $\Banach$ is isometrically equivalent to a normed space $\Fun$ consisting of functions $g\in\Leb_0(\Domain')$,
and the biorthogonal system $\Eset'$ of $\Eset$ is chosen from $\Fun$. If the sets $\Eset$ and $\Eset'$ satisfy conditions (C-GEN),
then $\Banach$ is a two-sided reproducing kernel Banach space and the two-sided reproducing kernel of $\Banach$ is a generalized Mercer kernel $K$ set up by $\Eset$ and $\Eset'$ defined in equation~\eqref{eq:RKBS-MercerKer-Gen-5}.
\end{theorem}
%////////////////////////////////////////////////////////////////////////////////////////////////////////////////////////
\begin{proof}
We shall prove the two-sided reproducing properties of $\Banach$ by the characterization of the Schauder basis $\Eset$ and its biorthogonal system $\Eset'$.

First we check its right-sided reproducing properties. Take any $f\in\Banach$ and any $\vx\in\Domain$.
By inequality~\eqref{eq:RKBS-MercerKer-Gen-2} and conditions (C-GEN), we have that
\[
\sum_{n\in\NN}\norm{\phi_n(\vx)\adjphi_n}_{\Banach'}
=\sum_{n\in\NN}\abs{\phi_n(\vx)}\norm{\adjphi_n}_{\Banach'}
\leq\left(\sup_{n\in\NN}\norm{\adjphi_n}_{\Banach'}\right)
\left(\sum_{n\in\NN}\abs{\phi_n(\vx)}\right)<\infty;
\]
hence
\[
\sum_{n\in\NN}\phi_n(\vx)\adjphi_n\text{ converges in }\Banach',
\]
which ensures that $K(\vx,\cdot)\in\Banach'$.
Since $\Eset$ is the Schauder basis of $\Banach$, there exists a unique sequence $\left(a_n:n\in\NN\right)$ such that $f=\sum_{n\in\NN}a_n\phi_n$.
This ensures that
\[
\langle f,K(\vx,\cdot) \rangle_{\Banach}
=\sum_{n\in\NN}\phi_n(\vx)\langle f,\adjphi_n \rangle_{\Banach}
=\sum_{n\in\NN}a_n\phi_n(\vx)=f(\vx).
\]

In the same manner, we verify the left-sided reproducing properties. We pick any $g\in\Banach'$ and any $\vy\in\Domain'$.
Because of conditions (C-GEN), we have that
\[
\sum_{n\in\NN}\norm{\adjphi_n(\vy)\phi_n}_{\Banach}
=\sum_{n\in\NN}\abs{\adjphi_n(\vy)}\norm{\phi_n}_{\Banach}
=\sum_{n\in\NN}\abs{\adjphi_n(\vy)}<\infty,
\]
which ensures that $K(\cdot,\vy)\in\Banach$. Moreover, since $\Banach$ is reflexive, the biorthogonal system $\Eset'$ is also the Schauder basis of $\Banach'$; hence $g$ has the expansion $g=\sum_{n\in\NN}b_n\adjphi_n$.
Therefore, we compute
\[
\langle K(\cdot,\vy),g \rangle_{\Banach}
=\sum_{n\in\NN}b_n\langle K(\cdot,\vy),\adjphi_n \rangle_{\Banach}
=\sum_{n\in\NN}b_n\adjphi_n(\vy)=g(\vy).
\]
\end{proof}
%/////////////////////////////////////////////////////////////////////////////////////////////////////////////////////

%------------------------------------------------------------------------------------------------------------------------
\section{Constructing $1$-norm Reproducing Kernel Banach Spaces}\label{s:1-RKBS}
%------------------------------------------------------------------------------------------------------------------------
\sectionmark{Constructing $1$-norm RKBS}

Following the previous discussions, we introduce Banach spaces and their dual spaces by generalized Mercer kernels in the traditional $1$-norm framework. Even though the infinite dimensional $1$-norm Banach space is always non-reflexive, we also obtain the reproducing properties, imbedding, compactness, and universal approximation of these infinite dimensional $1$-norm Banach spaces.

In the same way of Section~\ref{s:p-RKBS}, the following additional conditions of the expansion sets are needed for the constructions of the $1$-norm RKBSs.

\begin{assumption}[A-$1*$]
Suppose that the left-sided and right-sided expansion sets $\Sset_{K}$ and $\Sset_{K}'$
of
the generalized Mercer kernel $K$
are linearly independent and satisfy that
\[
\tag{C-$1*$}
\sum_{n\in\NN}\abs{\phi_n(\vx)}<\infty\text{ for all }\vx\in\Domain,
\quad
\sum_{n\in\NN}\abs{\adjphi_n(\vy)}<\infty\text{ for all }\vy\in\Domain'.
\]
\end{assumption}\label{a:A-1*}
Similar to the definitions of $\Phi_p$ and $\Phi_q'$ in Section~\ref{s:p-RKBS}, we define
another left-sided and right-sided upper-bound functions of the generalized Mercer kernel $K$
\[
\Phi_1(\vx):=\sum_{n\in\NN}\abs{\phi_n(\vx)},\quad
\Phi_1'(\vy):=\sum_{n\in\NN}\abs{\adjphi_n(\vy)},
\]
for all $\vx\in\Domain$ and all $\vy\in\Domain'$, respectively.

Now we compare conditions (C-$1*$) and conditions (C-$p$) for $1<p<\infty$.
Conditions (C-$1*$) implies conditions (C-$p$)
because
\begin{equation}\label{A1-to-Ap-1}
\sum_{n\in\NN}\abs{\phi_n(\vx)}^q\leq
\Phi_1(\vx)^q<\infty,
\quad\text{for }\vx\in\Domain,
\end{equation}
and
\begin{equation}\label{A1-to-Ap-2}
\sum_{n\in\NN}\abs{\adjphi_n(\vy)}^p\leq
\Phi_1'(\vy)^p<\infty,
\quad
\text{for }
\vy\in\Domain',
\end{equation}
for all $1<p,q<\infty$.
This indicates that conditions
(C-$1*$) is stronger than conditions (C-$p$). But, conditions (C-$p_1$) and (C-$p_2$) do not have any relationships when $1<p_1<p_2<\infty$.

We study the constructions of the 1-norm and $\infty$-norm RKBSs.
In this section we suppose that the expansion sets $\Sset_{K}$ and $\Sset_{K}'$
of the generalized Mercer kernel $K$ always satisfy assumption~(A-$1*$).
We employ the expansion sets $\Sset_{K}$ and $\Sset_{K}'$ to define the spaces $\Banach_K^1(\Domain)$ and $\Banach_{\adjK}^{\infty}(\Domain')$ consisting of functions on the domains $\Domain$ and $\Domain'$, respectively. Then we define
\begin{equation}\label{eq:RKBS-1}
\Banach_K^1(\Domain):=\left\{f:=\sum_{n\in\NN}a_n\phi_n:\left(a_n:n\in\NN\right)\in \lone\right\},
\end{equation}
equipped with the semi-norm
\[
\norm{f}_{\Banach_K^1(\Domain)}:=\sum_{n\in\NN}\abs{a_n},
\]
and
\begin{equation}\label{eq:RKBS-infty}
\Banach_{\adjK}^{\infty}(\Domain'):=\left\{g:=\sum_{n\in\NN}b_n\adjphi_n:\left(b_n:n\in\NN\right)\in \czero\right\},
\end{equation}
equipped with the semi-norm
\[
\norm{g}_{\Banach_{\adjK}^{\infty}(\Domain')}:=\sup_{n\in\NN}\abs{b_n}.
\]
Here, the spaces $\lone$ and $\linfty$ are the collection of all countable sequences of scalars with the standard norms $\norm{\cdot}_{1}$ and $\norm{\cdot}_{\infty}$, respectively, and the space $\czero$ is the subspace of $\linfty$ with all countable sequence of scalars that converge to $0$, that is,
\[
\czero:=\left\{\va:=\left(a_n:n\in\NN\right):\left\{a_n:n\in\NN\right\}\subseteq\RR\text{ and }\lim_{n\to\infty}\abs{a_n}=0\right\},
\]
equipped with the norm
\[
\norm{\va}_{\infty}:=\sup_{n\in\NN}\abs{a_n}.
\]
Moreover, the linear independence of $\Sset_{K}$ and $\Sset_{K}'$ will guarantee that the $1$-norm and $\infty$-norm will be well-defined on $\Banach_K^1(\Domain)$ and $\Banach_{\adjK}^{\infty}(\Domain')$, respectively.

We continue to check that the function spaces $\Banach_K^1(\Domain)$ and $\Banach_{\adjK}^{\infty}(\Domain')$ are both well-defined pointwise. For each $f:=\sum_{n\in\NN}a_n\phi_n\in\Banach_K^1(\Domain)$ and each $g:=\sum_{n\in\NN}b_n\adjphi_n\in\Banach_{\adjK}^{\infty}(\Domain')$, we prove, by the Cauchy Schwarz inequality, that
\[
\abs{f(\vx)}\leq\sum_{n\in\NN}\abs{a_n\phi_n(\vx)}
\leq\Phi_1(\vx)\norm{f}_{\Banach_K^{1}(\Domain)},
\]
for all $\vx\in\Domain$,
and
\[
\abs{g(\vy)}\leq\sum_{n\in\NN}\abs{b_n\adjphi_n(\vy)}
\leq\Phi_1'(\vy)\norm{g}_{\Banach_{\adjK}^{\infty}(\Domain')},
\]
for all $\vy\in\Domain'$.
{\color{black}{Since $\Sset_{K}\subseteq\Leb_0(\Domain)$ and $\Sset_{K}'\subseteq\Leb_0(\Domain')$, the pointwise limits guarantee that $\Banach_K^{1}(\Domain)\subseteq\Leb_0(\Domain)$ and $\Banach_{\adjK}^{\infty}(\Domain')\subseteq\Leb_0(\Domain')$.}}
In particular, if $\Sset_{K}\subseteq\Cont(\Domain)$ such that $\sum_{n\in\NN}\abs{\adjphi_n(\vy)}$ is uniformly convergent on $\Domain$, then
$\Banach_K^{1}(\Domain)\subseteq\Cont(\Domain)$,
and if $\Sset_{K}'\subseteq\Cont(\Domain')$ such that
$\sum_{n\in\NN}\abs{\adjphi_n(\vy)}$ is uniformly convergent on $\Domain'$, then
$\Banach_{\adjK}^{\infty}(\Domain')\subseteq\Cont(\Domain')$.

Assumption (A-$1*$) ensures that the spaces
$\Banach_K^p(\Domain)$ defined in equations~\eqref{eq:RKBS-p} is well-defined for all $1<p<\infty$.
Since $\lone$ is imbedded into $\lp$ by the identity map, we have that
$\Banach_K^1(\Domain)$ is imbedded into $\Banach_K^p(\Domain)$ by the identity map, that is,
\[
\Banach_K^1(\Domain)
\subseteq\Banach_K^p(\Domain),
\]
and
\[
\norm{f}_{\Banach_K^p(\Domain)}\leq\norm{f}_{\Banach_K^1(\Domain)},
\]
for all $f\in\Banach_K^1(\Domain)$.
By the same way, the space $\Banach_{\adjK}^q(\Domain')$ defined in equations~\eqref{eq:RKBS-q} is also well-defined for all $1<q<\infty$.
Since $\lq$ is imbedded into $\czero$ by the identity map, we have that $\Banach_{\adjK}^q(\Domain')$ is imbedded into $\Banach_{\adjK}^{\infty}(\Domain')$ by the identity map, that is,
\[
\Banach_{\adjK}^q(\Domain')\subseteq\Banach_{\adjK}^{\infty}(\Domain'),
\]
and
\[
\norm{g}_{\Banach_{\adjK}^{\infty}(\Domain')}\leq\norm{g}_{\Banach_{\adjK}^q(\Domain')},
\]
for all $g\in\Banach_{\adjK}^{\infty}(\Domain')$.
Stronger conditions (C-$1*$) guarantee the imbedding properties of $\Banach_K^1(\Domain)$ and $\Banach_{\adjK}^q(\Domain')$ which will be used to solve the $1$-norm support vector machines.

\subsection*{Reproducing Properties}

It is easy to check that Lemma~\ref{l:p-RKBS} also covers the case of $\Banach_1$. In the following proof, we also need another lemma of $\Banach_{\infty}$. By the same techniques of Lemma~\ref{l:p-RKBS}, we use the linearly independent functions $\left\{\varphi_n:n\in\NN\right\}$ to set up the function space
\[
\Banach_{\infty}^0:=\Span\left\{\varphi_n:n\in\NN\right\},
\]
equipped with the
norm
\[
\norm{f}_{\Banach_{\infty}^0}:=\norm{\va}_{\infty}=\sup_{k\in\NN_N}\abs{a_k},
\]
for $f:=\sum_{k\in\NN_N}a_k\varphi_k\in\Banach_{\infty}^0$ and $N\in\NN$,
and let $\Banach_{\infty}$ be the completion (closure) of $\Banach_{\infty}^0$ by the norm $\norm{\cdot}_{\Banach_{\infty}^0}$.

The space $\lspace_{\infty,0}$ is the completion of $\czero:=\Span\left\{\ve_n:n\in\NN\right\}$ by the norm $\norm{\cdot}_{\infty}$; hence
we conclude that $\Banach_{\infty}^0$ and $\czero$ are isometrically isomorphisms using the equivalent norms between $\Banach_{\infty}^0$ and $\czero$,
which indicates that $\left\{\varphi_n:n\in\NN\right\}$ is a Schauder basis of $\Banach_{\infty}$.
Therefore, we obtain the following lemma.
%////////////////////////////////////////////////////////////////////////////////////////////////////////////////////////
\begin{lemma}\label{l:infty-RKBS}
The Banach space $\Banach_{\infty}$ defined above can be represented explicitly by the form
\[
\Banach_{\infty}=\left\{f=\sum_{n\in\NN}a_n\varphi_n:\va:=\left(a_n:n\in\NN\right)\in \czero\right\},
\]
equipped with the norm
\[
\norm{f}_{\Banach_{\infty}}:=\norm{\va}_{\infty},
\]
and $\Banach_{\infty}$ is isometrically equivalent to $\czero$.
\end{lemma}
%////////////////////////////////////////////////////////////////////////////////////////////////////////////////////////
Here we notice that $\Banach_{\infty}$ is just isometrically imbedded into $\linfty$ but they can not be equivalent.
Next, we consider their reproducing properties
by the same techniques used in the proof of Proposition~\ref{p:RKBS-MecerKer-pq} and Theorem~\ref{t:RKBS-MercerKer-p}.

%////////////////////////////////////////////////////////////////////////////////////////////////////////////////////////
\begin{proposition}\label{p:RKBS-MecerKer-1}
If a kernel $K\in\Leb_0(\Domain\times\Domain')$ is a generalized Mercer kernel such that
the expansion sets $\Sset_K$ and $\Sset_K'$ of $K$ satisfy assumption (A-$1*$),
then
$\Banach_K^1(\Domain)$ and $\Banach_{\adjK}^{\infty}(\Domain')$ are Banach spaces and $\Banach_K^1(\Domain)$ is isometrically equivalent to the dual space of $\Banach_{\adjK}^{\infty}(\Domain')$.
\end{proposition}
%////////////////////////////////////////////////////////////////////////////////////////////////////////////////////////
\begin{proof}
We employ the same techniques used in Proposition~\ref{p:RKBS-MecerKer-pq} to complete the proof.
The isometrical isomorphisms of function spaces and sequence spaces given in Lemmas~\ref{l:p-RKBS} and~\ref{l:infty-RKBS} imply that
$\Banach_K^1(\Domain)\cong\lone$ and
$\Banach_{\adjK}^{\infty}(\Domain')\cong\czero$ by combining equations~\eqref{eq:RKBS-1} and~\eqref{eq:RKBS-infty}.

Since $\lone$ and $\linfty$ are Banach spaces with the dual property such as $\left(\czero\right)'\cong \lone$,
the conclusions are proved by the isometrical isomorphisms.
\end{proof}
%////////////////////////////////////////////////////////////////////////////////////////////////////////////////////////

%////////////////////////////////////////////////////////////////////////////////////////////////////////////////////////
\begin{theorem}\label{t:RKBS-MercerKer-1}
If a kernel $K\in\Leb_0(\Domain\times\Domain')$ is a generalized Mercer kernel such that
the expansion sets $\Sset_K$ and $\Sset_K'$ of $K$ satisfy assumption (A-$1*$),
then $\Banach_{K}^{1}(\Domain)$ is a right-sided reproducing kernel Banach space with the right-sided reproducing kernel $K$.
\end{theorem}
%////////////////////////////////////////////////////////////////////////////////////////////////////////////////////////
\begin{proof}
Proposition~\ref{p:RKBS-MecerKer-1} has concluded that $\Banach_{K}^{1}(\Domain)$ is a Banach space. Now we verify the right-sided reproducing properties of $\Banach_{K}^{1}(\Domain)$ by the kernel $K$.

By the proof of Proposition~\ref{p:RKBS-MecerKer-1}, we have that $\left(\Banach_K^1(\Domain)\right)'\cong\left(\lone\right)'\cong \linfty$ and $\Banach_{\adjK}^{\infty}(\Domain')\cong \czero$. Since $\czero$ is isometrically imbedded into $\linfty$, the normed space $\Banach_{\adjK}^{\infty}(\Domain')$ is also isometrically imbedded into the dual space of $\Banach_{K}^{1}(\Domain)$. This ensures that the dual bilinear product defined on $\Banach_{K}^{1}(\Domain)$ and $\Banach_{\adjK}^{\infty}(\Domain')$ can be represented as
\[
\langle f,g \rangle_{\Banach_K^1(\Domain)}
=\langle \va,\vb \rangle_{\lone}
=\sum_{n\in\NN}a_{n}b_{n},
\]
for each $f:=\sum_{n\in\NN}a_{n}\phi_n\in\Banach_K^1(\Domain)$ and each
$g:=\sum_{n\in\NN}b_{n}\adjphi_n\in\Banach_{\adjK}^{\infty}(\Domain')$.

Let $\vx\in\Domain$ and $f\in\Banach_K^1(\Domain)$. Conditions~(C-$1*$) ensure $\sum_{n\in\NN}\abs{\phi_n(\vx)}<\infty$; hence $\lim_{n\to\infty}\abs{\phi_n(\vx)}=0$. Therefore, by the expansion of $K$, we have that
\[
K(\vx,\cdot)=\sum_{n\in\NN}\phi_n(\vx)\adjphi_n\in\Banach_{\adjK}^{\infty}(\Domain'),
\]
and
\[
\langle f,K(\vx,\cdot) \rangle_{\Banach_{K}^1(\Domain)}
=\sum_{n\in\NN}a_n\phi_n(\vx)=f(\vx).
\]
\end{proof}
%////////////////////////////////////////////////////////////////////////////////////////////////////////////////////////

Even though $\Banach_{K}^{1}(\Domain)$ is the right-sided RKBS, its dual space $\Banach_{\adjK}^{\infty}(\Domain')$ can be viewed as the two-sided RKBS.

%////////////////////////////////////////////////////////////////////////////////////////////////////////////////////////
\begin{theorem}\label{t:RKBS-MercerKer-infty}
If a kernel $K\in\Leb_0(\Domain\times\Domain')$ is a generalized Mercer kernel such that
the expansion sets $\Sset_K$ and $\Sset_K'$ of $K$ satisfy assumption (A-$1*$),
then
$\Banach_{\adjK}^{\infty}(\Domain')$ is a two-sided reproducing kernel Banach space with the two-sided reproducing kernel $\adjK$, the adjoint kernel of $K$.
\end{theorem}
%////////////////////////////////////////////////////////////////////////////////////////////////////////////////////////
\begin{proof}
The reproducing properties will be proved by using the representation
\[
\adjK(\vy,\vx)=K(\vx,\vy)=\sum_{n\in\NN}\phi_n(\vx)\adjphi_n(\vy),
\]
and the dual bilinear product defined on $\Banach_{\adjK}^\infty(\Domain')$ and $\Banach_{K}^1(\Domain)$ in the form
\begin{equation}\label{eq:dual-RKBS-infty}
\langle g,f \rangle_{\Banach_{\adjK}^\infty(\Domain')}
=\langle \vb,\va \rangle_{l_{\infty,0}}
=\sum_{n\in\NN}b_{n}a_{n},
\end{equation}
for each $g:=\sum_{n\in\NN}b_n\adjphi_n\in\Banach_{\adjK}^\infty(\Domain')$ and each
$f:=\sum_{n\in\NN}a_n\phi_n\in\Banach_{K}^1(\Domain)$.

Since $\Banach_{\adjK}^\infty(\Domain')\cong \czero$ and $\left(\czero\right)'\cong \lone\cong\Banach_{K}^1(\Domain)$ stated in the proof of Proposition~\ref{p:RKBS-MecerKer-1}, the dual bilinear product given in equation~\eqref{eq:dual-RKBS-infty} is well-defined.
For any $\vy\in\Domain'$ and any $\vx\in\Domain$,
Conditions~(C-$1*$) ensure that $\sum_{n\in\NN}\abs{\adjphi(\vy)}<\infty$ and
$\sum_{n\in\NN}\abs{\phi(\vx)}<\infty$; hence
\[
\adjK(\vy,\cdot)=K(\cdot,\vy)\in\Banach_K^1(\Domain),
\quad
\adjK(\cdot,\vx)=K(\vx,\cdot)\in\Banach_{\adjK}^\infty(\Domain').
\]
Based on the dual bilinear product given in equation~\eqref{eq:dual-RKBS-infty}, we check that
\[
\langle g,\adjK(\vy,\cdot) \rangle_{\Banach_{\adjK}^\infty(\Domain')}
=\sum_{n\in\NN}b_n\adjphi_n(\vy)=g(\vy),
\]
and
\[
\langle \adjK(\cdot,\vx),f \rangle_{\Banach_{\adjK}^\infty(\Domain')}
=\sum_{n\in\NN}a_n\phi_n(\vx)=f(\vx).
\]
for all $g\in\Banach_{\adjK}^\infty(\Domain')$ and all
$f\in\Banach_{K}^1(\Domain)$.
The proof is complete.
\end{proof}
%////////////////////////////////////////////////////////////////////////////////////////////////////////////////////////

\emph{Comparison:}
Now we compare the geometrical structures $\Banach_K^{1}(\Domain)$ and $\Banach_K^{p}(\Domain)$ for $1<p<\infty$.
Since $\Banach_K^{1}(\Domain)$ and $\Banach_{\adjK}^{\infty}(\Domain')$ are NOT reflexive, they do not fit the strong definition of RKBSs given in~\cite{ZhangXuZhang2009}.
Even though $\Banach_{\adjK}^{\infty}(\Domain')$ is a two-sided RKBS, its dual space $\Banach_K^{1}(\Domain)$ is just a right-sided RKBS.
The expansion sets $\Sset_{K}$ and $\Sset_{K}'$ are the Schauder bases of $\Banach_K^{1}(\Domain)$ and $\Banach_{\adjK}^{\infty}(\Domain')$, respectively.
The reproducing properties come from these Schauder bases and their biorthogonal systems.
Actually $\linfty$ is isometrically imbedded into $\Leb_{\infty}(\Domain')$ when the support of $\mu'$ is equal to $\Domain'$. This indicates that the dual space of $\Banach_K^{1}(\Domain)$ is also isometrically imbedded into $\Leb_{\infty}(\Domain')$.
However, $\Sset_{K}'$ is not a Schauder basis of the whole dual space $\left(\Banach_K^{1}(\Domain)\right)'$ and $\Sset_{K}'$ is just a basis sequence of $\left(\Banach_K^{1}(\Domain)\right)'$ because $\left(\Banach_K^{1}(\Domain)\right)'\cong \linfty$ and there is no Schauder basis of the whole space $\linfty$.
This guarantees that there exists a function $g\in\left(\Banach_K^{1}(\Domain)\right)'$ which can not be represented as any expansion $\sum_{n\in\NN}b_n\adjphi_n$.
Thus this function $g$ becomes a counter example of the left-sided reproducing properties of $\Banach_K^{1}(\Domain)$, or more precisely, there is a $\vy\in\Domain'$ such that $\langle K(\cdot,\vy),g \rangle_{\Banach_K^{1}(\Domain)}\neq g(\vy)$.
More details of the relationships between the Schauder bases and the $\infty$-norm spaces can be found in~\cite[Chapter~4]{Megginson1998}.

\subsection*{Imbedding}

Now we discuss the imbedding of $\Banach_K^{1}(\Domain)$.
Conditions (C-$1*$) ensure that the sequence sets $\Aset_K$ and $\Aset_K'$ defined in equation~\eqref{eq:A-K-phi-psi} can be well endowed with the $\lone$-norm.
Let
\[
\Phi_{\infty}(\vx):=\sup_{n\in\NN}\abs{\phi_n(\vx)},\quad
\text{for }\vx\in\Domain.
\]

%////////////////////////////////////////////////////////////////////////////////////////////////////////////////////////
\begin{proposition}\label{p:imbedding-RKBS-1}
Let $1\leq q\leq\infty$ and let $K\in\Leb_0(\Domain\times\Domain')$ be a generalized Mercer kernel such that
the expansion sets $\Sset_K$ and $\Sset_K'$ of $K$ satisfy assumption (A-$1*$).
If the left-sided domain $\Domain$ is compact and the left-sided sequence set $\Aset_{K}$ of $K$ is bounded in $\lone$, then the identity map from $\Banach_{K}^{1}(\Domain)$ into $\Leb_q(\Domain)$ is continuous.
In particular, if $\sum_{n\in\NN}\abs{\phi(\vx)}$ is uniformly convergent on $\Domain$ and the support of $\mu$ is equal to $\Domain$,
then $\Banach_{K}^{1}(\Domain)$ is imbedded into $\Leb_q(\Domain)$.
\end{proposition}
%////////////////////////////////////////////////////////////////////////////////////////////////////////////////////////
\begin{proof}
According to Theorem~\ref{t:RKBS-MercerKer-1}, the space $\Banach_{K}^{1}(\Domain)$ is the right-sided RKBS with the right-sided reproducing kernel $K$.
Using the bounded condition of $\Aset_K$, we obtain a positive constant
\[
\norm{\Phi_1}_{\infty}=\sup_{\vx\in\Domain}\Phi_1(\vx)<\infty.
\]
Since $\Domain$ is compact, we have that
\[
\int_{\Domain}\abs{\Phi_{\infty}(\vx)}^q\mu(\ud\vx)
\leq\int_{\Domain}\left(\Phi_1(\vx)\right)^q\mu(\ud\vx)
\leq\norm{\Phi_1}_{\infty}^q\mu(\Domain)<\infty,
\]
when $1\leq q<\infty$,
and
\[
\sup_{\vx\in\Domain}\Phi_{\infty}(\vx)
\leq\sup_{\vx\in\Domain}\Phi_1(\vx)
\leq\norm{\Phi_1}_{\infty}<\infty,
\]
which ensures that $\Phi_{\infty}\in\Leb_q(\Domain)$.
Moreover, since
\[
\norm{K(\vx,\cdot)}_{\Banach_{\adjK}^{\infty}(\Domain')}
=\Phi_{\infty}(\vx),\quad
\text{for }\vx\in\Domain,
\]
we conclude that $\vx\mapsto\norm{K(\vx,\cdot)}_{\Banach_{\adjK}^{\infty}(\Domain')}\in\Leb_q(\Domain)$.
Therefore, Proposition~\ref{p:RKBS-imbedding} assures that the identity map is a continuous linear operator from $\Banach_{K}^{1}(\Domain)$ into $\Leb_q(\Domain)$.

Finally, the uniform convergence of $\sum_{n\in\NN}\abs{\phi(\vx)}$ ensures that $\Banach_{K}^1(\Domain)\subseteq\Cont(\Domain)$.
Moreover, since $\supp(\mu)=\Domain$, the $1$-norm RKBS $\Banach_K^1(\Domain)$ satisfies the $\mu$-measure zero condition.
Thus, the imbedding of $\Banach_K^1(\Domain)$ into $\Leb_q(\Domain)$ is verified by Corollary~\ref{c:RKBS-imbedding}.
\end{proof}
%////////////////////////////////////////////////////////////////////////////////////////////////////////////////////////

%////////////////////////////////////////////////////////////////////////////////////////////////////////////////////////
\begin{proposition}\label{p:imbedding-intopt-RKBS-1-right}
Let $1\leq q\leq\infty$
and let $K\in\Leb_0(\Domain\times\Domain')$ be a generalized Mercer kernel such that
the expansion sets $\Sset_K$ and $\Sset_K'$ of $K$ satisfy assumption (A-$1*$).
If the measure $\mu'(\Domain')$ is finite and the right-sided sequence set $\Aset_K'$ of $K$ is bounded in $\lone$, then the right-sided integral operator $I_{K}'$ maps $\Leb_q(\Domain')$ into $\Banach_K^1(\Domain)$ continuously
and
\[
\int_{\Domain'}g(\vy)\xi(\vy)\mu'(\ud\vy)
=\langle I_{K}'(\xi),g \rangle_{\Banach_K^1(\Domain)},
\]
for all $\xi\in\Leb_q(\Domain')$ and all $g\in\Banach_{\adjK}^{\infty}(\Domain')$.
In particular, if $\sum_{n\in\NN}\abs{\phi(\vy)}$ is uniformly convergent on $\Domain'$ and the support of $\mu'$ is equal to $\Domain'$,
then the range of $I_K'$ is weakly* dense in $\Banach_K^1(\Domain)$.
\end{proposition}
%////////////////////////////////////////////////////////////////////////////////////////////////////////////////////////
\begin{proof}
The key point of the proof is to consider $\Banach_K^1(\Domain)$ as the dual space of a two-sided RKBS so that we can use
Proposition~\ref{p:RKBS-imbedding-dual} to draw the conclusion.

Let $p$ be the conjugate exponent of $q$.
Theorem~\ref{t:RKBS-MercerKer-infty} ensures that $\Banach_{\adjK}^{\infty}(\Domain')$ is the two-sided RKBS with the two-sided reproducing kernel $\adjK$ and $\left(\Banach_{K}^{\infty}(\Domain')\right)'\cong\Banach_{K}^{1}(\Domain)$.

Now we show that $\vy\mapsto\norm{K(\cdot,\vy)}_{\Banach_K^{1}(\Domain)}$ belongs to $\Leb_p(\Domain')$.
Since $\Aset_K'$ is bounded in $\lone$, we obtain that
\[
\norm{\Phi_1'}_{\infty}=\sup_{\vy\in\Domain'}\Phi_1'(\vy)
<\infty.
\]
Moreover, since $\mu'(\Domain')<\infty$,
we have that
\[
\int_{\Domain'}\Phi_1'(\vy)^p\mu'(\ud\vy)
\leq\norm{\Phi_1'}_{\infty}^p\mu'(\Domain')<\infty,
\]
when $1\leq p<\infty$.
Thus $\Phi_1'\in\Leb_p(\Domain')$.
This shows that
\[
\int_{\Domain'}\abs{K(\vx,\vy)}\mu'(\ud\vy)
\leq\Phi_1(\vx)\int_{\Domain'}\Phi_1'(\vy)\mu'(\ud\vy)
<\infty;
\]
hence $K(\vx,\cdot)\in\Leb_1(\Domain')$ for all $\vx\in\Domain$.
In another way, we also have that $\vy\mapsto\norm{K(\cdot,\vy)}_{\Banach_K^{1}(\Domain)}\in\Leb_p(\Domain')$, because
\[
\norm{K(\cdot,\vy)}_{\Banach_K^{1}(\Domain)}
=\Phi_1'(\vy),\quad
\text{for }\vy\in\Domain'.
\]

In conclusion, Proposition~\ref{p:RKBS-imbedding-dual} ensures that
the integral operator $I_{K}'$ maps $\Leb_q(\Domain')$ into $\Banach_{K}^{1}(\Domain)$ continuously, and
\[
\int_{\Domain'}g(\vy)\xi(\vy)\mu'(\ud\vy)=\langle g,I_{K}'(\xi) \rangle_{\Banach_{\adjK}^{\infty}(\Domain')},
=\langle I_{K}'(\xi),g \rangle_{\Banach_{K}^{1}(\Domain)},
\]
for all $\xi\in\Leb_q(\Domain')$ and all $g\in\Banach_{\adjK}^{\infty}(\Domain')$.

Finally, the uniform convergence of $\sum_{n\in\NN}\abs{\adjphi(\vy)}$ ensures that $\Banach_{\adjK}^{\infty}(\Domain')\subseteq\Cont(\Domain')$.
Combining this continuity with the equality $\supp(\mu')=\Domain'$, we have that $\Banach_{\adjK}^{\infty}(\Domain')$ satisfies the $\mu'$-measure zero condition;
hence the weak* density of the range of $I_{K}'$ in $\Banach_K^1(\Domain)$ is verified by Corollary~\ref{c:RKBS-imbedding-dual}.
\end{proof}
%////////////////////////////////////////////////////////////////////////////////////////////////////////////////////////

In Section~\ref{s:p-RKBS} we only check the imbedding relationships for fixed $1<p,q<\infty$. However, we obtain all the imbedding in the $1$-norm RKBSs for all $1\leq p,q\leq\infty$.

\subsection*{Compactness}

Because $\Banach_{\adjK}^{\infty}(\Domain')\cong \czero$, we determine that $\Aset_{K}$ in $\czero$ is the identical element of $\Kset_K'$ in $\Banach_{\adjK}^{\infty}(\Domain')$.
Then, by the same techniques of Proposition~\ref{p:compact-RKBS-pq} we check the compactness of $\Banach_{K}^1(\Domain)$ in $\Linfty(\Domain)$ using Proposition~\ref{p:RKBS-compact}.

%////////////////////////////////////////////////////////////////////////////////////////////////////////////////////////
\begin{proposition}\label{p:compact-RKBS-1}
Let $K\in\Leb_0(\Domain\times\Domain')$ be a generalized Mercer kernel such that
the expansion sets $\Sset_K$ and $\Sset_K'$ of $K$ satisfy assumption (A-$1*$).
If the left-sided sequence set $\Aset_{K}$ is compact in $\linfty$, then the identity map from  $\Banach_{K}^1(\Domain)$ into $\Linfty(\Domain)$ is compact.
\end{proposition}
%////////////////////////////////////////////////////////////////////////////////////////////////////////////////////////

Same as the discussions at the end of Section~\ref{s:p-RKBS}, the set $\Aset_{K}$ can become closed when the domain $\Domain$ is complemented by the convergence of $\Aset_{K}$,
and the set $\Aset_{K}$ is relatively compact if there is
a sequence $\left(\alpha_n:n\in\NN\right)\in \czero$ such that
\[
\left(\alpha_n^{-1}\phi_n(\vx):n\in\NN\right)\in \linfty,
\]
and
\[
\sup_{\vx\in\Domain,n\in\NN}\frac{\abs{\phi_n(\vx)}}{\alpha_n}<\infty.
\]

\subsection*{Universal Approximation}
By Theorem~\ref{t:RKBS-MercerKer-1}, we know that $\Banach_{K}^1(\Domain)$ is only a right-sided RKBS.
But, we still show that $\Banach_{K}^1(\Domain)$ has the universal approximation the same as Proposition~\ref{p:universial-approx-RKBS-pq}.

%////////////////////////////////////////////////////////////////////////////////////////////////////////////////////////
\begin{proposition}\label{p:universial-approx-RKBS-1}
Let $K\in\Leb_0(\Domain\times\Domain')$ be a generalized Mercer kernel such that
the expansion sets $\Sset_K$ and $\Sset_K'$ of $K$ satisfy assumption (A-$1*$).
If $\sum_{n\in\NN}\abs{\phi_n(\vx)}$ is uniformly convergent on $\Domain$, the left-sided domain $\Domain$ is a compact Hausdorff space,
and $\Span\left\{\Sset_K\right\}$ is dense in $\Cont(\Domain)$,
then the reproducing kernel Banach spaces $\Banach_{K}^1(\Domain)$ has the universal approximation property.
\end{proposition}
%////////////////////////////////////////////////////////////////////////////////////////////////////////////////////////
\begin{proof}
First we show the continuity.
Since $\sum_{n\in\NN}\abs{\phi_n(\vx)}$ is uniformly convergent,
the continuity of $\Sset_K$ provides that $\Banach_{K}^1(\Domain)\subseteq\Cont(\Domain)$.

Next, since $\Span\left\{\Sset_K\right\}\subseteq\Banach_{K}^1(\Domain)$,
the density of $\Span\left\{\Sset_K\right\}$ in $\Cont(\Domain)$ ensures the density of $\Banach_{K}^1(\Domain)$ in $\Cont(\Domain)$.
\end{proof}
%////////////////////////////////////////////////////////////////////////////////////////////////////////////////////////

%////////////////////////////////////////////////////////////////////////////////////////////////////////////////////////
\begin{remark}\label{r:universial-approx-RKBS-1}
By the same method of Proposition~\ref{p:universial-approx-RKBS-1}, if $\sum_{n\in\NN}\abs{\adjphi_n(\vy)}$ is uniformly convergent on $\Domain'$, the right-sided domain $\Domain'$ is a compact Hausdorff space, and $\Span\left\{\Sset_K'\right\}$ is dense in $\Cont(\Domain')$, then the $\infty$-norm RKBS $\Banach_{\adjK}^{\infty}(\Domain')$ has the universal approximation property.
We already know that the two-sided RKBS $\Banach_{\adjK}^{\infty}(\Domain')$ is not reflexive.
This shows that the adjoint kernel $\adjK$ can be a two-sided universal kernel.
But, the dual space of $\Banach_{K}^1(\Domain)$ is not isometrically equivalent to any subspace of $\Cont(\Domain)$.
Hence, the dual space of $\Banach_{K}^1(\Domain)$ does not have the universal approximation property. This indicates that the $1$-norm RKBS $\Banach_{K}^1(\Domain)$ has a left-sided universal kernel only while there exists no right-sided universal kernel.
\end{remark}
%////////////////////////////////////////////////////////////////////////////////////////////////////////////////////////

%------------------------------------------------------------------------------------------------------------------------
%------------------------------------------------------------------------------------------------------------------------
\chapter{Positive Definite Kernels}\label{char-PDK}
%------------------------------------------------------------------------------------------------------------------------
%------------------------------------------------------------------------------------------------------------------------

In this chapter we present special results of positive definite kernels. Specifically, we demonstrate that the generalized Mercer kernels can cover many important examples of positive definite kernels defined on regular domains or manifold surfaces. They include min kernels, Gaussian kernels, and power series kernels.

%------------------------------------------------------------------------------------------------------------------------
\section{Definition of Positive Definite Kernels}\label{Def_PDK}
%------------------------------------------------------------------------------------------------------------------------
%\sectionmark{Constructing RKBS by PDK}

We first recall the definition of the positive definite kernels.

%/////////////////////////////////////////////////////////////////////////////////////////////////////////////////////
\begin{definition}\label{d:PDK}
A kernel $K:\Domain\times\Domain\to\RR$ is called \emph{positive definite} on a domain $\Domain$ if, for all $N\in\NN$ and
all sets of pairwise distinct points $X:=\left\{\vx_k:k\in\NN_N\right\}\subseteq\Domain$, the quadratic form
\[
\sum_{j,k\in\NN_N}c_jc_kK(\vx_j,\vx_k)\geq0,\quad \text{for all }\vc:=\left(c_k:k\in\NN_N\right)\in\RR^N.
\]
Moreover, the kernel $K$ is called strictly positive definite if the quadratic form is positive when $\vc\neq\v0$.
\end{definition}
%/////////////////////////////////////////////////////////////////////////////////////////////////////////////////////

This definition of the positive definite kernel is used through out this chapter. The positive {\it semi-definite} kernels given in the text books~\cite{Wendland2005,Fasshauer2007} have the same meaning as the positive definite kernel given Definition~\ref{d:PDK}.

%------------------------------------------------------------------------------------------------------------------------
\section{Constructing Reproducing Kernel Banach Spaces by Eigenvalues and Eigenfunctions}\label{s:RKBS-PDK}
%------------------------------------------------------------------------------------------------------------------------
\sectionmark{Constructing RKBS by PDK}

It has been shown in \cite{YeThesis2012,FasshauerHickernellYe2013} that the positive definite functions (translation invariant positive definite kernels) can become reproducing kernels of RKBSs. However, the question whether a positive definite kernel can be viewed as a reproducing kernel of a RKBS without an inner product needs to be answered. In this section we shall answer this question.

We suppose that the domain $\Domain$ is always a \emph{compact Hausdorff space} and the support of the regular Borel measure $\mu$ is $\Domain$ in this section.
Further suppose that $K\in\Cont(\Domain\times\Domain)$ is the symmetric positive definite kernel.
Therefore, the Mercer theorem guarantees that $K$ has countable eigenvalues $\Lambda_K:=\left\{\lambda_n:n\in\NN\right\}\subseteq\RR_+$ and
eigenfunctions $\Eset_{K}:=\left\{e_n:n\in\NN\right\}\subseteq\Cont(\Domain)$, that is, $I_Ke_n=\lambda_ne_n$ for $n\in\NN$.
Furthermore, $\Eset_{K}$ is an orthonormal set of $\Leb_2(\Domain)$ and
$K$ possesses
the absolutely and uniformly convergent representation
\begin{equation}\label{eq:PDK-expansion}
K(\vx,\vy)=\sum_{n\in\NN}\lambda_ne_n(\vx)e_n(\vy),\quad\text{for }\vx,\vy\in\Domain.
\end{equation}
For convenience, the elements of $\Eset_{K}$ are always chosen to be orthonormal in $\Leb_2(\Domain)$.
This means that $K$ is a (generalized) Mercer kernel.

%/////////////////////////////////////////////////////////////////////////////////////////////////////////////////////
\begin{remark}
By the compactness of the domain $\Domain$, we can confirm that the symmetric positive definite kernels $K$ always has the countable eigenvalues and eigenfunctions.
In this section, we use the infinitely countable eigenvalues and eigenfunctions to construct infinite dimensional RKBSs. To this end, we suppose that
the eigenvalues $\Lambda_K$ and eigenfunctions $\Eset_{K}$ of $K$ always compose of \emph{infinitely countable} elements.
If $K$ is strictly positive definite, then $\Eset_K$ is an orthonormal basis of $\Leb_2(\Domain)$.
In particular, it is easy to check that all the following results of RKBSs driven by positive definite kernels are also true for finite eigenvalues and eigenfunctions.
\end{remark}
%/////////////////////////////////////////////////////////////////////////////////////////////////////////////////////

By the symmetric condition, the adjoint kernel $\adjK$ of $K$ is equal to $K$ itself and its right-sided and left-sided kernel sets are the same, that is,
\[
\Kset_K'=\Kset_K=\left\{K(\vx,\cdot):\vx\in\Domain\right\}.
\]
If we choose the left-sided and right-sided expansion sets $\Sset_K$ and $\Sset_K'$ of $K$ as
\[
\phi_n=\adjphi_n:=\lambda_n^{1/2}e_n,\quad\text{for }n\in\NN,
\]
then $\Sset_K=\Sset_K'$ are linearly independent, because the eigenfunctions $\Eset_K$ is an orthonormal set of $\Leb_2(\Domain)$.
This means that we only need to check conditions~(C-$1*$) of $\Sset_K=\Sset_K'$ for assumption~(A-$1*$). For this purpose, we transfer the assumption~(A-$1*$) to an equivalent assumption for positive definite kernels as follows.

\begin{assumption}[A-PDK]
Suppose that the positive eigenvalues $\Lambda_K$ and the continuous eigenfunctions $\Eset_{K}$ of
the positive definite kernel $K$ satisfy that
\[\tag{C-PDK}
\sum_{n\in\NN}\lambda_n^{1/2}\abs{e_n(\vx)}<\infty,\quad\text{for all }\vx\in\Domain.
\]
\end{assumption}\label{a:A-PDK}

In this section, we always assume that the eigenvalues $\Lambda_K$ and the eigenfunctions $\Eset_{K}$ of $K$ satisfy assumption (A-PDK).
As the discussions presented in Sections~\ref{s:p-RKBS} and~\ref{s:1-RKBS}, if the eigenvalues $\Lambda_K$ and the eigenfunctions $\Eset_{K}$ satisfy conditions~(C-PDK),
then they gratify conditions (C-$1*$) and (C-$p$) for all $1<p<\infty$.
Using these eigenvalues $\Lambda_K$ and the eigenfunctions $\Eset_{K}$ of $K$, we set up the normed spaces
\[
\Banach_K^p(\Domain):=\left\{f:=\sum_{n\in\NN}a_n\lambda_n^{1/2}e_n:\left(a_n:n\in\NN\right)\in \lp\right\},\quad
\text{when }1\leq p<\infty,
\]
equipped with the norm
\[
\norm{f}_{\Banach_K^p(\Domain)}:=\left(\sum_{n\in\NN}\abs{a_n}^p\right)^{1/p},
\]
and
\[
\Banach_K^{\infty}(\Domain):=\left\{f:=\sum_{n\in\NN}a_n\lambda_n^{1/2}e_n:\left(a_n:n\in\NN\right)\in l_{\infty,0}\right\},
\]
equipped with the norm
\[
\norm{f}_{\Banach_K^{\infty}(\Domain)}:=\sup_{n\in\NN}\abs{a_n}.
\]
Because of the imbedding of the sequence spaces, we have that
\[
\Banach_K^{p_1}(\Domain)\subseteq\Banach_K^{p_2}(\Domain),
\]
and
\[
\norm{f}_{\Banach_K^{p_2}(\Domain)}\leq\norm{f}_{\Banach_K^{p_1}(\Domain)},\quad
\text{for }f\in\Banach_K^{p_1}(\Domain),
\]
when $1\leq p_1\leq p_2\leq\infty$. According to Propositions~\ref{p:RKBS-MecerKer-pq} and~\ref{p:RKBS-MecerKer-1}, we know that
the dual space of $\Banach_K^{p}(\Domain)$ is isometrically equivalent to $\Banach_K^{q}(\Domain)$ when $1<p,q<\infty$ and $p^{-1}+q^{-1}=1$, and the dual space of $\Banach_K^{\infty}(\Domain)$ is isometrically equivalent to $\Banach_K^{1}(\Domain)$. However, $\Banach_K^{\infty}(\Domain)$ is isometrically imbedded into the dual space of $\Banach_K^{1}(\Domain)$.

%////////////////////////////////////////////////////////////////////////////////////////////////////////////////////////
\begin{theorem}\label{t-RKBS-PDK}
If a kernel $K\in\Cont(\Domain\times\Domain)$ is a symmetric positive definite kernel such that
the eigenvalues $\Lambda_K$ and eigenfunctions $\Eset_{K}$ of $K$ satisfy
assumption (A-PDK),
then $\Banach_K^1(\Domain)$
is a right-sided reproducing kernel Banach space with the right-sided reproducing kernel $K$ and
$\Banach_K^p(\Domain)$ for $1<p\leq\infty$
is a two-sided reproducing kernel Banach space with the two-sided reproducing kernel $K$.
\end{theorem}
%////////////////////////////////////////////////////////////////////////////////////////////////////////////////////////
\begin{proof}
As the discussion given above, assumption (A-PDK) is equivalent to assumption (A-$1*$) for positive definite kernels.
Therefore, Theorems~\ref{t:RKBS-MercerKer-p}, \ref{t:RKBS-MercerKer-1}, and \ref{t:RKBS-MercerKer-infty}
ensure the reproducing conclusions.
\end{proof}
%////////////////////////////////////////////////////////////////////////////////////////////////////////////////////////

If $\sum_{n\in\NN}\lambda_n^{1/2}\abs{e_n(\vx)}$ is uniformly convergent on $\Domain$, then
$\Banach_K^p(\Domain)\subseteq\Cont(\Domain)$ for $1\leq p\leq\infty$ because $\Eset_K\subseteq\Cont(\Cont)$.

Comparing the representations of RKHSs, we find that
$\Banach_K^2(\Domain)=\Hilbert_K(\Domain)$.
It is well-known that there exists a unique RKHS for the given positive definite kernel.
Theorem~\ref{t-RKBS-PDK} shows, however, that a positive definite kernel can be a reproducing kernel of a variety of RKBSs endowed with different norms.
Since the RKHSs are unique, the support vector machine solutions in RKHSs are unique.
However, the support vector machine solutions in RKBSs will have more colorful formats because of various geometrical structures of RKBSs (see Theorem~\ref{t:RKBS-MercerKer-svm-rep-pq}).

Next, we consider two types of positive definite kernels which can always be the reproducing kernels of the $p$-norm RKBSs for all $1\leq p\leq\infty$. The first type of kernels are identified by their eigenvalues and
eigenfunctions.

%////////////////////////////////////////////////////////////////////////////////////////////////////////////////////////
\begin{proposition}\label{p:PDK-EigenValVec}
If a kernel $K\in\Cont(\Domain\times\Domain)$ is a symmetric positive definite kernel
such that
the eigenvalues $\Lambda_K$ and eigenfunctions $\Eset_{K}$ of $K$ satisfy that
\[
\sum_{n\in\NN}\lambda_n^{1/2}<\infty,\quad \sup_{n\in\NN}\norm{e_n}_{\infty}<\infty,
\]
then $\Banach_K^1(\Domain)$
is a right-sided reproducing kernel Banach space with the right-sided reproducing kernel $K$ and
$\Banach_K^p(\Domain)$ for $1<p\leq\infty$
is a two-sided reproducing kernel Banach space with the two-sided reproducing kernel $K$. Moreover, $\Banach_K^p(\Domain)$ for $1\leq p\leq\infty$ is included in $\Cont(\Domain)$.
\end{proposition}
%////////////////////////////////////////////////////////////////////////////////////////////////////////////////////////
\begin{proof}
If conditions~(C-PDK) of $\Lambda_K$ and $\Eset_K$ are verified, then we obtain the reproducing properties of
$\Banach_K^p(\Domain)$ for all $1\leq p\leq\infty$ using Theorem~\ref{t-RKBS-PDK}.
Actually, we have that
\[
\sum_{n\in\NN}\lambda_n^{1/2}\abs{e_n(\vx)}\leq
\sup_{n\in\NN}\norm{e_n}_{\infty}\sum_{n\in\NN}\lambda_n^{1/2}
<\infty,\quad
\text{for }\vx\in\Domain.
\]

For the continuity of $\Banach_K^p(\Domain)$, we only need to check the uniform convergence of series $\sum_{n\in\NN}\lambda_n^{1/2}\abs{e_n(\vx)}$ on $\Domain$. Because of the convergence of $\sum_{n=1}\lambda_n^{1/2}$ and the uniform boundedness of $\left\{\abs{e_n(\vx)}:n\in\NN\right\}$ on $\Domain$,
the Abel uniform convergence test provides that $\sum_{n\in\NN}\lambda_n^{1/2}\abs{e_n(\vx)}$ is uniformly convergent on $\Domain$.
\end{proof}
%////////////////////////////////////////////////////////////////////////////////////////////////////////////////////////

The second type of kernels are defined by an integral of a symmetric positive definite kernel.

%////////////////////////////////////////////////////////////////////////////////////////////////////////////////////////
\begin{proposition}\label{p:integral-typeKer}
Let $\Psi\in\Cont(\Domain\times\Domain)$ be a symmetric positive definite kernel.
If a kernel $K$ can be represented in the integral form of the kernel $\Psi$ such as
\[
K(\vx,\vy)=\int_{\Domain}\Psi(\vx,\vz)\Psi(\vz,\vy)\mu(\ud\vz),\quad \text{for }\vx,\vy\in\Domain,
\]
then
$\Banach_K^1(\Domain)$
is a right-sided reproducing kernel Banach space with the right-sided reproducing kernel $K$ and
$\Banach_K^p(\Domain)$ for $1<p\leq\infty$
is a two-sided reproducing kernel Banach space with the two-sided reproducing kernel $K$. Moreover, $\Banach_K^p(\Domain)$ for $1\leq p\leq\infty$ is included in $\Cont(\Domain)$.
\end{proposition}
%////////////////////////////////////////////////////////////////////////////////////////////////////////////////////////
\begin{proof}
The key technique used in the proof of this result is the computation of the eigenvalues $\Lambda_K$ and eigenfunctions $\Eset_K$ of $K$ by
the positive eigenvalues $\Lambda_{\Psi}:=\left\{\rho_n:n\in\NN\right\}$ and the continuous eigenfunctions $\Eset_{\Psi}:=\left\{\varphi_n:n\in\NN\right\}$ of $\Psi$.

The collection $\Eset_{\Psi}$ of eigenfunctions is an orthonormal set of $\Leb_2(\Domain)$ and $\Psi$ has
the uniformly convergent representation
\begin{equation}\label{eq:integral-typeKer-0}
\Psi(\vx,\vy)=\sum_{n\in\NN}\rho_n\varphi_n(\vx)\varphi_n(\vy),\quad\text{for }\vx,\vy\in\Domain.
\end{equation}
This ensures that the integral-type kernel $K$ has the uniformly convergent representation
\begin{align*}
K(\vx,\vy)=&
\int_{\Domain}\Psi(\vx,\vz)\Psi(\vz,\vy)\mu(\ud\vz)\\
=&\int_{\Domain}\left(\sum_{m\in\NN}\rho_m\varphi_m(\vx)\varphi_n(\vz)\right)
\left(\sum_{n\in\NN}\rho_n\varphi_n(\vz)\varphi_n(\vy)\right)\mu(\ud\vz)\\
=&\sum_{m,n\in\NN}\rho_m\rho_n\varphi_m(\vx)\varphi_n(\vy)\int_{\Domain}\varphi_m(\vz)\varphi_m(\vz)\mu(\ud\vz)\\
=&\sum_{n\in\NN}\rho_n^2\varphi_n(\vx)\varphi_n(\vy),
\end{align*}
for $\vx,\vy\in\Domain$.
Hence, $K$ is a positive definite kernel and its eigenvalues $\lambda_n$ and eigenfunctions $e_n$ can be computed by $\rho_n$ and $\varphi_n$, that is,
\[
\lambda_n=\rho_n^2,\quad e_n=\varphi_n,\quad\text{for all }n\in\NN.
\]

Next, we verify conditions (C-PDK) of $\Lambda_K$ and $\Eset_K$.
Since the symmetric kernel $\Psi$ defined on the compact Hausdorff space $\Domain$ is continuous and positive definite, we have that
\[
\sum_{n\in\NN}\rho_n<\infty,\quad \norm{\Psi}_{\infty}<\infty.
\]
In addition,
by the Cauchy Schwarz inequality, we find for each $\vx\in\Domain$ that
\begin{equation}\label{eq:integral-typeKer-1}
\sum_{n\in\NN}\lambda_n^{1/2}\abs{e_n(\vx)}
=\sum_{n\in\NN}\rho_n\abs{\varphi_n(\vx)}
\leq\left(\sum_{n\in\NN}\rho_n\right)^{1/2}\left(\sum_{n\in\NN}\rho_n\abs{\varphi_n(\vx)}^2\right)^{1/2}.
\end{equation}
Moreover, we have that
\begin{equation}\label{eq:integral-typeKer-2}
\left(\sum_{n\in\NN}\rho_n\right)^{1/2}\sqrt{\Psi(\vx,\vx)}
\leq\left(\sum_{n\in\NN}\rho_n\right)^{1/2}\norm{\Psi}_{\infty}^{1/2}.
\end{equation}
Combining inequalities~\eqref{eq:integral-typeKer-0},~\eqref{eq:integral-typeKer-1}, and~\eqref{eq:integral-typeKer-2}, we conclude that
\[
\sum_{n\in\NN}\lambda_n^{1/2}\abs{e_n(\vx)}<\infty,\quad
\text{for all }\vx\in\Domain.
\]
Therefore, Theorem~\ref{t-RKBS-PDK} guarantees
the reproducing properties of
$\Banach_K^p(\Domain)$ for all $1\leq p\leq\infty$.

Since $\sum_{n\in\NN}\rho_n\abs{\varphi_n(\vx)}^2$ is uniformly convergent on $\Domain$,
inequality~\eqref{eq:integral-typeKer-1} ensures that $\sum_{n\in\NN}\lambda_n^{1/2}\abs{e_n(\vx)}$ is also uniformly convergent on $\Domain$.
This yields $\Banach_K^p(\Domain)\subseteq\Cont(\Domain)$ for all $1\leq p\leq\infty$.
\end{proof}
%////////////////////////////////////////////////////////////////////////////////////////////////////////////////////////

%/////////////////////////////////////////////////////////////////////////////////////////////////////////////////////
\begin{remark}
In this article, we mainly look at the RKBSs induced by the positive definite kernels defined on compact domains.
Roughly speaking, this chapter focuses on the constructions of RKBSs by the Fourier series.
In the previous paper~\cite{FasshauerHickernellYe2013}, it has been shown how to construct the RKBSs by the positive definite functions defined on the whole space $\Rd$. It is clear that the positive definite function is a translation invariant kernel.
The key technique of the proof in~\cite{FasshauerHickernellYe2013} is the Fourier transform (see \cite[Theorems~4.1 and~4.4]{FasshauerHickernellYe2013}).
The positivity of the Fourier transforms is guaranteed by the Bochner theorem~\cite[Theorem~6.6]{Wendland2005}.
Here, the positive eigenvalues of positive definite kernels can be viewed as the equivalent element of the positive measures of positive definite functions.
In this article, we already show that the generalized Mercer kernel with the negative eigenvalues can become a reproducing kernel of a RKBS.
So, we conjecture that a non-positive definite function could be used to construct a RKBS by the Fourier transform.
\end{remark}
%/////////////////////////////////////////////////////////////////////////////////////////////////////////////////////

According to the proof of Propositions~\ref{p:PDK-EigenValVec} and \ref{p:integral-typeKer}, the sequence set
\[
\Aset_K=\Aset_K'=\left\{\left(\lambda_n^{1/2}e_n(\vx):n\in\NN\right):\vx\in\Domain\right\},
\]
of the positive definite kernel $K$ given in Propositions~\ref{p:PDK-EigenValVec} and \ref{p:integral-typeKer} is both bounded and closed in $\lone$.

\emph{Comparisons:}
The $p$-norm RKBSs induced by the positive definite kernels can be viewed as the special case of the $p$-norm RKBSs discussed in Chapter~\ref{char-GMK}.
By Theorem~\ref{t-RKBS-PDK}, we show that many well-known positive definite kernels can become the reproducing kernels of these $p$-norm RKBSs for all $1\leq p\leq\infty$ such as the min kernel, the Gaussian kernel, and the power series kernel (see Sections~\ref{s:min},~\ref{s:gaussian} and~\ref{s:powerseries}).
Because of the symmetry of the positive definite kernels, the dualities of these $p$-norm RKBSs given here are consistent with themselves.
Proposition~\ref{p:integral-typeKer} shows that
the reproducing properties can be checked by the integral-formats of the positive definite kernels.
The eigenvalues and eigenfunctions of the positive definite kernels ensure the imbedding, compactness, and universal approximation of these $p$-norm RKBSs.
In Section~\ref{s:SVM-Sp-RKBS}, we shall find that the support vector machine solutions in the $p$-norm RKBSs can be represented by the eigenvalues and eigenfunctions.

\subsection*{Imbedding}

The imbedding of RKBSs stated in Propositions~\ref{p:RKBS-imbedding} and \ref{p:RKBS-imbedding-dual} ensures the imbedding of $\Banach_K^p(\Domain)$.
Suppose that the sequence set $\Aset_K$ is bounded in $\lone$. We define a positive constant
\[
C_1:=\sup_{\vx\in\Domain}\sum_{n\in\NN}\lambda_n^{1/2}\abs{e_n(\vx)}<\infty.
\]
Here, if $\sum_{n\in\NN}\lambda_n^{1/2}\abs{e_n(\vx)}$ is uniformly convergent on the compact Hausdorff space $\Domain$, then $\vx\mapsto\sum_{n\in\NN}\lambda_n^{1/2}\abs{e_n(\vx)}$ belongs to $\Cont(\Domain)$; hence the compactness of $\Domain$ implies that $\Aset_K$ is bounded and closed in $\lone$.

Let $1\leq p,q\leq\infty$ such that $p^{-1}+q^{-1}=1$.
For $\vx\in\Domain$,
we have that
\[
\Phi_q(\vx)^{1/q}=\norm{K(\vx,\cdot)}_{\Banach_K^q(\Domain)}
=\left(\sum_{n\in\NN}\lambda_n^{q/2}\abs{e_n(\vx)}^q\right)^{1/q},
\]
when $1\leq q<\infty$,
and
\[
\Phi_{\infty}(\vx)=\norm{K(\vx,\cdot)}_{\Banach_K^{\infty}(\Domain)}
=\sup_{n\in\NN}\lambda_n^{1/2}\abs{e_n(\vx)}.
\]
Moreover, we find that
\[
\int_{\Domain}\abs{\Phi_1(\vx)}^{\gamma}\mu(\ud\vx)
=\int_{\Domain}\left(\sum_{n\in\NN}\lambda_n^{1/2}\abs{e_n(\vx)}\right)^{\gamma}\mu(\ud\vx)
\leq C_{1}^{\gamma}\mu(\Domain)<\infty,
\]
when $1\leq\gamma<\infty$,
and
\[
\sup_{\vx\in\Domain}\abs{\Phi_1(\vx)}
=\sup_{\vx\in\Domain}\sum_{n\in\NN}\lambda_n^{1/2}\abs{e_n(\vx)}
= C_{1}<\infty.
\]
This ensures that $\Phi_1\in\Leb_{\gamma}(\Domain)$.
Since
\[
0\leq \Phi_q(\vx)^{1/q}\leq \Phi_1(\vx),\quad\text{for }\vx\in\Domain,
\]
we have that
\[
\Phi_q^{1/q}\in\Leb_{\gamma}(\Domain),\quad \text{for all }1\leq\gamma\leq\infty.
\]
This shows that
\begin{equation}\label{eq:imbedding-RKBS-PDK-0}
\vx\mapsto\norm{K(\vx,\cdot)}_{\Banach_K^q(\Domain)}\in\Leb_{\gamma}(\Domain),\quad \text{for all }1\leq\gamma\leq\infty.
\end{equation}

%////////////////////////////////////////////////////////////////////////////////////////////////////////////////////////
\begin{proposition}\label{p:imbedding-RKBS-PDK}
Let $1\leq p,\gamma\leq\infty$
and let $K\in\Cont(\Domain\times\Domain)$ be a symmetric positive definite kernel such that
the eigenvalues $\Lambda_K$ and eigenfunctions $\Eset_{K}$ of $K$ satisfy
assumption (A-PDK).
If the sequence set $\Aset_K$ of $K$ is bounded in $\lone$,
then the identity map from $\Banach_K^p(\Domain)$ into
$\Leb_{\gamma}(\Domain)$ is continuous. In particular, if $\sum_{n\in\NN}\lambda_n^{1/2}\abs{e_n(\vx)}$ is uniformly convergent on $\Domain$,
then $\Banach_K^p(\Domain)$ is imbedded into
$\Leb_{\gamma}(\Domain)$.
\end{proposition}
%////////////////////////////////////////////////////////////////////////////////////////////////////////////////////////
\begin{proof}
According to Theorem~\ref{t-RKBS-PDK}, the space
$\Banach_K^1(\Domain)$
is a right-sided RKBS and the space
$\Banach_K^p(\Domain)$ for $1<p\leq\infty$
is a two-sided RKBS.
Thus, inclusion relation~\eqref{eq:imbedding-RKBS-PDK-0} ensures the continuity of the identity map by Proposition~\ref{p:RKBS-imbedding}.

For the final statement, the uniform convergence of $\sum_{n\in\NN}\lambda_n^{1/2}\abs{e_n(\vx)}$ guarantees that $\Banach_K^p(\Domain)\subseteq\Cont(\Domain)$.
Moreover, since $\supp(\mu)=\Domain$, we have that $\Banach_K^p(\Domain)$ satisfies the $\mu$-measure zero condition; hence the imbedding of $\Banach_K^p(\Domain)$ in
$\Leb_{\gamma}(\Domain)$ is verified by Corollary~\ref{c:RKBS-imbedding}.
\end{proof}
%////////////////////////////////////////////////////////////////////////////////////////////////////////////////////////

When the kernel $K$ is symmetric, then the left-sided integral operator $I_K$ and the right-sided integral operator $I_K'$ are the same.
Since $K\in\Cont(\Domain\times\Domain)$ is defined on a compact space $\Domain$, the integral operator $I_K$ is well-defined on $\Leb_{\beta}(\Domain)$ for any $1\leq\beta\leq\infty$.
We now show that the integral operator $I_K$ maps $\Leb_{\beta}(\Domain)$ continuously into $\Banach_K^p(\Domain)$.

%////////////////////////////////////////////////////////////////////////////////////////////////////////////////////////
\begin{proposition}\label{p:RKBS-imbedding-dual-PDK}
Let $1\leq p<\infty$ and $1<q\leq\infty$ such that $p^{-1}+q^{-1}=1$, let $1\leq\beta\leq\infty$,
and let $K\in\Cont(\Domain\times\Domain)$ be a symmetric positive definite kernel such that
the eigenvalues $\Lambda_K$ and eigenfunctions $\Eset_{K}$ of $K$ satisfy
assumption (A-PDK).
If the sequence set $\Aset_K$ of $K$ is bounded in $\lone$,
then the integral operator $I_K$ maps $\Leb_{\beta}(\Domain)$ into $\Banach_K^p(\Domain)$ continuously such that
\begin{equation}\label{eq:imbeding-RKBS-dual-PDK}
\int_{\Domain}g(\vy)\xi(\vy)\mu(\ud\vy)=\langle g,I_K\xi \rangle_{\Banach_K^q(\Domain)}=\langle I_K\xi,g \rangle_{\Banach_K^p(\Domain)},
\end{equation}
for all $\xi\in\Leb_{\beta}(\Domain)$ and all $g\in\Banach_K^q(\Domain)$.
In particular, if $\sum_{n\in\NN}\lambda_n^{1/2}\abs{e_n(\vx)}$ is uniformly convergent on $\Domain$, then the range of $I_K$ is (weakly*) dense in $\Banach_K^p(\Domain)$ when $1<p<\infty$ and $1<\beta<\infty$ ($p=1$ or $\beta=1,\infty$).
\end{proposition}
%////////////////////////////////////////////////////////////////////////////////////////////////////////////////////////
\begin{proof}
Theorem~\ref{t-RKBS-PDK} provides that the space
$\Banach_K^q(\Domain)$
is a two-sided RKBS and the dual space of $\Banach_K^q(\Domain)$ is isometrically equivalent to $\Banach_K^p(\Domain)$.
Moreover, the uniform convergence of $\sum_{n\in\NN}\lambda_n^{1/2}\abs{e_n(\vx)}$ and the equality of $\supp(\mu)=\Domain$ ensures that $\Banach_K^p(\Domain)$ satisfies the $\mu$-measure zero condition.
Let $\gamma$ be the conjugate exponent of $\beta$.
Based on equation~\eqref{eq:imbedding-RKBS-PDK-0}, we employ Proposition~\ref{p:RKBS-imbedding-dual} and Corollary~\ref{c:RKBS-imbedding-dual} to complete the proof.
\end{proof}
%////////////////////////////////////////////////////////////////////////////////////////////////////////////////////////

Since $\Domain$ is compact, we know that $\Leb_{\gamma_2}(\Domain)$ is imbedded into $\Leb_{\gamma_1}(\Domain)$ by the identity map when $1\leq\gamma_1\leq\gamma_2\leq\infty$.
On the other hand, $\Banach_K^{p_1}(\Domain)$ is imbedded into $\Banach_K^{p_2}(\Domain)$ by the identity map when $1\leq p_1\leq p_2\leq\infty$.

If $1<p_1<p_2<\infty$, then $\Banach_K^{p_1}(\Domain)$ and $\Banach_K^{p_2}(\Domain)$ have various normed structures while
equation~\eqref{eq:imbeding-RKBS-dual-PDK} provides that
\[
\langle f,I_K\zeta \rangle_{\Banach_K^{p_1}(\Domain)}=
\int_{\Domain}f(\vx)\zeta(\vx)\mu(\ud\vx)=\langle f,I_K\zeta \rangle_{\Banach_K^{p_2}(\Domain)},
\]
for all $f\in\Banach_K^{p_1}(\Domain)\subseteq\Banach_K^{p_2}(\Domain)$ and all $\zeta\in\Leb_{\gamma}(\Domain)$. This means that the $p$-norm RKBSs with the same reproducing kernel could be seen as the identical elements in the integral sense.

\subsection*{Compactness}

Below, we consider the compactness of the $p$-norm RKBS.

%////////////////////////////////////////////////////////////////////////////////////////////////////////////////////////
\begin{proposition}\label{p:PDK-compact-RKBS}
Let $1\leq p,q\leq\infty$ such that $p^{-1}+q^{-1}=1$
and let $K\in\Cont(\Domain\times\Domain)$ be a symmetric positive definite kernel such that
the eigenvalues $\Lambda_K$ and eigenfunctions $\Eset_{K}$ of $K$ satisfy
assumption (A-PDK).
If the sequence set $\Aset_K$ of $K$ is compact in $\lq$,
then the identity map from $\Banach_{K}^p(\Domain)$ into $\Linfty(\Domain)$ is compact.
\end{proposition}
%////////////////////////////////////////////////////////////////////////////////////////////////////////////////////////
\begin{proof}
According to Proposition~\ref{p:RKBS-compact}, the compactness is proved if we verify that the kernel set $\Kset_K$ of $K$ is compact in $\Banach_{K}^q(\Domain)$.
Because $\Aset_K$ in $\lq$ is the identical element of $\Kset_K$ in $\Banach_{K}^q(\Domain)$ when $1\leq q<\infty$, and $\Aset_K$ in $\czero$ is the identical element of $\Kset_K$ in $\Banach_{K}^{\infty}(\Domain)$, the compactness of $\Aset_K$ in $\lq$ implies the compactness of $\Kset_K$ in $\Banach_{K}^q(\Domain)$.
\end{proof}
%////////////////////////////////////////////////////////////////////////////////////////////////////////////////////////

One technique to determine the relative compactness of $\Aset_K$ is to find
a positive sequence $\left(\alpha_n:n\in\NN\right)\in \czero$
such that
\[
\left(\alpha_n^{-1}\lambda_n^{1/2}\abs{e_n(\vx)}:n\in\NN\right)\in \lq,
\]
and
\[
\sup_{\vx\in\Domain}\sum_{n\in\NN}\frac{\lambda_n^{q/2}\abs{e_n(\vx)}^q}{\alpha_n^q}<\infty,\quad
\text{when }1\leq q<\infty,
\quad
\sup_{\vx\in\Domain,n\in\NN}\frac{\lambda_n^{1/2}\abs{e_n(\vx)}}{\alpha_n}<\infty.
\]

\subsection*{Universal Approximation}

Now we investigate the universal approximation of the $p$-norm RKBS $\Banach_{K}^p(\Domain)$ by Propositions~\ref{p:universial-approx-RKBS-pq} and~\ref{p:universial-approx-RKBS-1}.

%////////////////////////////////////////////////////////////////////////////////////////////////////////////////////////
\begin{proposition}\label{p:universial-approx-RKBS-PDK}
Let $1\leq p\leq\infty$ and let $K\in\Cont(\Domain\times\Domain)$ be a symmetric strictly positive definite kernel such that
the eigenvalues $\Lambda_K$ and eigenfunctions $\Eset_{K}$ of $K$ satisfy
assumption (A-PDK).
If $\sum_{n\in\NN}\lambda_n^{1/2}\abs{e_n(\vx)}$ is uniformly convergent on $\Domain$,
then the reproducing kernel Banach spaces $\Banach_{K}^p(\Domain)$ has the universal approximation property.
\end{proposition}
%////////////////////////////////////////////////////////////////////////////////////////////////////////////////////////
\begin{proof}
If the sufficient conditions of Propositions~\ref{p:universial-approx-RKBS-pq} and~\ref{p:universial-approx-RKBS-1} are checked, then the proof is complete. The expansion terms $\phi_n$ are composed of the eigenvalues $\Lambda_K$ and eigenfunctions $\Eset_{K}$ of $K$, that is, $\phi_n=\lambda_n^{1/2}e_n$ for all $n\in\NN$.

First the uniform convergence of $\sum_{n\in\NN}\lambda_n^{1/2}\abs{e_n(\vx)}$ ensures that $\sum_{n\in\NN}\abs{\phi_n(\vx)}^q$ is uniformly convergent for all $1\leq q<\infty$.

Next,
the compactness of $\Domain$ provides that $\Cont(\Domain)\subseteq\Leb_2(\Domain)$.
Moreover, the continuous eigenfunctions $\Eset_{K}$ of the strictly positive definite kernel $K$ become an orthonormal basis of $\Leb_2(\Domain)$; Since the support of $\mu$ is $\Domain$, we have that $\Span\left\{\Eset_{K}\right\}$ is dense in $\Cont(\Domain)$. This ensures that $\Span\left\{\Sset_{K}\right\}$ is dense in $\Cont(\Domain)$.

Therefore, the universal approximation property of $\Banach_{K}^p(\Domain)$ has been obtained by Propositions~\ref{p:universial-approx-RKBS-pq} and~\ref{p:universial-approx-RKBS-1}.
\end{proof}
%////////////////////////////////////////////////////////////////////////////////////////////////////////////////////////

As the discussions in Remarks~\ref{r:universial-approx-RKBS-1} and~\ref{r:RKBS-universial-approx},
the strictly positive definite kernel $K$ is the two-sided universal kernel for the $p$-norm RKBS $\Banach_{K}^p(\Domain)$ when $1<p\leq\infty$. But, the strictly positive definite kernel $K$ is only the left-sided universal kernel for the $1$-norm RKBS $\Banach_{K}^1(\Domain)$.

\subsection*{Equivalent Eigenfunctions}

In the above statements, the eigenvalues and eigenfunctions of the same positive definite kernel are fixed.
It is well-known that the RKHS or the $2$-norm RKBS is unique for an arbitrary positive definite kernel.
But the $p$-norm RKBSs induced by a variety of expansion sets could be different from their norms.
Now we discuss the affections of the $p$-norm RKBSs by various choices of eigenfunctions of the same positive definite kernel, and show that the $p$-norm RKBSs driven by the equivalent eigenfunctions are equivalent by the identity map.

According to the spectral theorem for compact and self-adjoint operators, the eigenvalues $\Lambda_K$ of the positive definite kernel $K$ has zero as the only accumulation point and each eigenvalue $\lambda_n$ has finite multiplicity, or more precisely, the space
\[
\Space_n:=\left\{f\in\Leb_2(\Domain):I_K(f)=\lambda_nf\right\}
\]
is finite dimensional for all $n\in\NN$.
Let the dimension of $\Space_n$ be $N_n$.
Obviously, if the dimension of $\Space_n$ is equal to one for all $n\in\NN$, then the choice of eigenfunctions is just different from the sign.
We are more interested in the space $\Space_n$ with the non-single basis.
For convenience, we reorder the eigenvalues $\Lambda_K$ such that they become a strictly decreasing sequence, that is,
\[
\lambda_1>\lambda_2>\cdots>\lambda_n>\cdots>0.
\]
Then the eigenfunctions $\Eset_K$ of the positive definite kernel $K$ compose of an orthonormal basis
$\Eset_K^n:=\left\{e_{n,k}:k\in\NN_{N_n}\right\}$ of the space $\Space_n$, that is,
\[
\Eset_K=\underset{n\in\NN}{\cup}\Eset_K^n,
\]
where $\NN_{N_n}:=\left\{1,\ldots,N_n\right\}$.
Then equation~\eqref{eq:PDK-expansion} can be rewritten as
\[
K(\vx,\vy)=\sum_{n\in\NN}\sum_{k\in\NN_{N_n}}\lambda_ne_{n,k}(\vx)e_{n,k}(\vy),\quad
\text{for }\vx,\vy\in\Domain.
\]

It is obvious that the eigenfunctions $\Eset_K$ is unique if $N_1=\cdots=N_n=\cdots=1$.
But, if $N_n\neq1$ for some $n\in\NN$, then we have another choice of eigenfunctions $\Eset_W$; hence we obtain another
orthonormal basis $\Eset_{W}^n:=\left\{\varphi_{n,k}:k\in\NN_{N_n}\right\}$ of $\Space_n$ such that
\[
\Eset_{W}:=\underset{n\in\NN}{\cup}\Eset_{W}^n.
\]
To avoid the confusion, we use another symbol to denote the representation
\[
W(\vx,\vy):=\sum_{n\in\NN}\sum_{k\in\NN_{N_n}}\lambda_n\varphi_{n,k}(\vx)\varphi_{n,k}(\vy),\quad
\text{for }\vx,\vy\in\Domain.
\]
Actually we have that $K=W$, and the eigenvalues $\Lambda_W$ of $W$ are equal to $\Lambda_K$.

Since $\Eset_{K}^n$ and $\Eset_{W}^n$ are bases of $\Space_n$, we let
\[
\varphi_{n,j}:=\sum_{k\in\NN_{N_n}}u_{j,k}^{n}e_{n,k},\quad\text{for }j\in\NN_n\text{ and }n\in\NN.
\]
Then
\begin{align*}
\int_{\Domain}\varphi_{n,i}(\vx)\varphi_{n,j}(\vx)\mu\left(\ud\vx\right)&
=\sum_{k,l\in\NN_{N_n}}u_{i,k}^nu_{j,l}^n\int_{\Domain}e_{n,k}(\vx)e_{n,l}(\vy)\mu\left(\ud\vx\right),\\
\delta_{ij}&=\sum_{k,l\in\NN_{N_n}}u_{i,k}^nu_{j,l}^n\delta_{kl},
\end{align*}
for $i,j,k,l\in\NN_n$, where $\delta_{ij}$ is the Kronecker delta function of $i$ and $j$, that is, $\delta_{ij}=1$ when $i=j$ and $\delta_{ij}=0$ when $i\neq j$. This yields that
\[
\vU_n:=\left(u_{j,k}^n:j,k\in\NN_{N_n}\right),
\]
is a unitary matrix, that is, $\vU_n^T\vU_n=\vI$.
Let
\[
\ve_n:=\left(e_{n,k}:k\in\NN_{N_n}\right),
\quad
\vvarphi_n:=\left(\varphi_{n,k}:k\in\NN_{N_n}\right),\quad\text{for }n\in\NN.
\]
Then
\[
\vvarphi_n(\vx)=\vU_n\ve(\vx),\quad \ve_n(\vx)=\vU_n^T\vvarphi_n(\vx),
\quad
\text{for }\vx\in\Domain\text{ and }n\in\NN;
\]
hence
\[
\norm{\vU_n}_{\infty}^{-1}\norm{\ve_n(\vx)}_1
=
\norm{\vU_n^T}_1^{-1}\norm{\ve_n(\vx)}_1\leq\norm{\vvarphi_n(\vx)}_1\leq\norm{\vU_n}_1\norm{\ve_n(\vx)}_1.
\]

Moreover, we say that these two kinds of the eigenfunctions $\Eset_K$ and $\Eset_W$ are \emph{equivalent} if
the positive constant
\begin{equation}\label{eq:equivalent-eigenfuns}
C_U:=\sup_{n\in\NN}\left(\norm{\vU_n}_{\infty}+\norm{\vU_n}_1\right)<\infty,
\end{equation}
is well-defined. In the following discussions, we suppose that $\Eset_K$ and $\Eset_W$ are equivalent.

The equivalence of the eigenfunctions $\Eset_K$ and $\Eset_W$ indicates that
\begin{equation}\label{eq:PDK-equiv-eigenfun-cond}
C_U^{-1}\sum_{n\in\NN}\lambda_n^{1/2}\norm{\ve_n(\vx)}_1
\leq\sum_{n\in\NN}\lambda_n^{1/2}\norm{\vvarphi_n(\vx)}_1
\leq
C_U\sum_{n\in\NN}\lambda_n^{1/2}\norm{\ve_n(\vx)}_1,
\end{equation}
for all $\vx\in\Domain$.
Inequality~\eqref{eq:PDK-equiv-eigenfun-cond} provides that $\Lambda_K$ and $\Eset_K$ of $K$ satisfy conditions (C-PDK)
if and only if $\Lambda_{W}$ and $\Eset_{W}$ of $W$ satisfy conditions (C-PDK).
Therefore, if assumption (A-PDK) is well-posed for any choice of $\Lambda_K,\Eset_K$ or $\Lambda_W,\Eset_W$,
then the RKBSs $\Banach_{K}^p(\Domain)$ and $\Banach_{W}^p(\Domain)$ are both well-defined for all $1\leq p\leq\infty$ by Theorem~\ref{t-RKBS-PDK}.

Finally, we verify that the RKBSs $\Banach_{K}^p(\Domain)$ and $\Banach_{W}^p(\Domain)$ are equivalent by the identity map, that is, the identity map is an isomorphism from $\Banach_{K}^p(\Domain)$ onto $\Banach_{W}^p(\Domain)$.

Here, we need another condition to complete the proof.
We further suppose that $\Aset_K$ or $\Aset_{W}$ is bounded in $\lone$. By inequality~\eqref{eq:PDK-equiv-eigenfun-cond} $\Aset_K$ is bounded in $\lone$ if and only if $\Aset_{W}$ is bounded in $\lone$.
According to Proposition~\ref{p:imbedding-RKBS-PDK}, both RKBSs $\Banach_{K}^p(\Domain)$ and $\Banach_{W}^p(\Domain)$ can be seen as a subspace of $\Leb_2(\Domain)$.

Let $f\in\Leb_2(\Domain)$.
Then $f$ can be represented as the expansions
\[
f=\sum_{n\in\NN}\sum_{k\in\NN_{N_n}}a_{n,k}\lambda_n^{1/2}e_{n,k},
\]
and
\[
f=\sum_{n\in\NN}\sum_{k\in\NN_{N_n}}b_{n,k}\lambda_n^{1/2}\varphi_{n,k},
\]
where
\[
a_{n,k}:=\lambda_n^{-1/2}\int_{\Domain}f(\vx)e_{n,k}(\vx)\mu\left(\ud\vx\right),
\quad
b_{n,k}:=\lambda_n^{-1/2}\int_{\Domain}f(\vx)\varphi_{n,k}(\vx)\mu\left(\ud\vx\right);
\]
hence $f\in\Banach_K^p(\Domain)$ or $f\in\Banach_{W}^p(\Domain)$ if and only if
\[
\sum_{n\in\NN}\norm{\va_n}_p^p<\infty\text{ or }
\sum_{n\in\NN}\norm{\vb_n}_p^p<\infty,\quad\text{when }1\leq p<\infty,
\]
and
\[
\lim_{n\to\infty}\norm{\va_n}_{\infty}=0
\text{ or }
\lim_{n\to\infty}\norm{\vb_n}_{\infty}=0,
\]
where $\va_n:=\left(a_{n,k}:k\in\NN_{N_n}\right)$ and $\vb_n:=\left(b_{n,k}:k\in\NN_{N_n}\right)$.
We further find the relationships of the above coefficients as follows:
\[
\vb_n=\vU_n\va_n,
\]
and
\[
C_U^{-1}\norm{\va_n}_{p}\leq
\norm{\vU_n^T}_{p}^{-1}\norm{\va_n}_{p}\leq\norm{\vb_n}_{p}\leq\norm{\vU_n}_{p}\norm{\va_n}_{p}\leq C_U\norm{\va_n}_{p},
\]
for all $n\in\NN$.
This indicates that $f\in\Banach_K^p(\Domain)$ if and only if $f\in\Banach_{W}^p(\Domain)$, and their norms are equivalent, that is,
\[
C_U^{-1}\norm{f}_{\Banach_K^p(\Domain)}\leq\norm{f}_{\Banach_{W}^p(\Domain)}
\leq C_U\norm{f}_{\Banach_K^p(\Domain)}.
\]

Therefore, we conclude the following result.
%////////////////////////////////////////////////////////////////////////////////////////////////////////////////////////
\begin{proposition}\label{p:equivalent-eigenfuns}
Let $K\in\Cont(\Domain\times\Domain)$ be a symmetric positive definite kernel such that $K$ has
the equivalent eigenfunctions $\Eset_K$ and $\Eset_W$ satisfying
assumption (A-PDK). If the sequence set $\Aset_K$ of $K$ or the sequence set $\Aset_W$ of $W$ is bounded in $\lone$, then the reproducing kernel Banach spaces $\Banach_{K}^p(\Domain)$ and $\Banach_{W}^p(\Domain)$ are equivalent by the identity map.
\end{proposition}
%////////////////////////////////////////////////////////////////////////////////////////////////////////////////////////

In particular, if $1<p,q<\infty$ such that $p^{-1}+q^{-1}=1$, then
the dual bilinear products of $\Banach_K^p(\Domain)$ and $\Banach_{W}^p(\Domain)$ are identical, that is,
\[
\langle f,g \rangle_{\Banach_K^p(\Domain)}
=
\langle f,g \rangle_{\Banach_{W}^p(\Domain)},
\]
for all $f\in\Banach_K^p(\Domain)=\Banach_{W}^p(\Domain)$
and all $g\in\Banach_K^q(\Domain)=\Banach_{W}^q(\Domain)$.

%------------------------------------------------------------------------------------------------------------------------
\section{Min Kernels}\label{s:min}
%------------------------------------------------------------------------------------------------------------------------
\sectionmark{Min Kernels}

In this section, we investigate the use of the min kernel in constructing the $p$-norm RKBSs.
We first recall
the min kernel with the homogeneous boundary condition on $[0,1]$ defined by
\[
\Psi_1(x,y):=\min\{x,y\}-xy,\quad 0\leq x,y\leq1.
\]
The min kernel $\Psi_1$ has the absolutely and uniformly convergent representation
\[
\Psi_1(x,y)=\sum_{n\in\NN}\rho_n\varphi_n(x)\varphi_n(y),
\]
expanded by its positive eigenvalues and continuous eigenfunctions
\[
\rho_{n}:=\frac{1}{n^2\pi^2},\quad \varphi_{n}(x):=\sqrt{2}\sin(n\pi x),\quad\text{for }n\in\NN.
\]
Associated with the min kernel $\Psi_1$ we define the integral-type kernel
$$
K_1(x,y):=\int_0^1\Psi_1(x,z)\Psi_1(z,y)\ud z.
$$
It can be computed that
\begin{align*}
K_1(x,y)&=\sum_{n\in\NN}\rho_{n}^2\varphi_{n}(x)\varphi_{n}(y)\\
&=\begin{cases}
-\frac{1}{6}x^3+\frac{1}{6}x^3y+\frac{1}{6}xy^3-\frac{1}{2}xy^2+\frac{1}{3}xy,&0\leq x\leq y\leq 1,\\
-\frac{1}{6}y^3+\frac{1}{6}xy^3+\frac{1}{6}x^3y-\frac{1}{2}x^2y+\frac{1}{3}xy,&0\leq y\leq x \leq1.
\end{cases}
\end{align*}
This ensures that the eigenvalues and eigenfunctions of $K_1$ can be represented, respectively, as
\[
\lambda_{1,n}:=\rho_{n}^2,\quad
e_{1,n}:=\varphi_{n},\quad\text{for }n\in\NN.
\]
Clearly, $\Psi_1$ is a continuous function and it is positive definite on $[0,1]$.

Next, we study the integral-type min kernel $K_d$ defined on the $d$-dimensional cubic $[0,1]^d$.
Here, the measure $\mu$ is the Lebesgue measure defined on $[0,1]^d$.
Let the product kernel
\[
\Psi_d(\vx,\vy):=\prod_{k\in\NN_d}\Psi_1\left(x_k,y_k\right),
\]
for $\vx:=(x_k:k\in\NN_d),\vy:=(y_k:k\in\NN_d)\in[0,1]^d$.
Associated with $\Psi_d$, we have the $d$-dimensional integral-type kernel
\[
K_d(\vx,\vy)=\int_{[0,1]^d}\Psi_d(\vx,\vz)\Psi_d(\vz,\vy)\ud\vz,
\quad\text{for }\vx,\vy\in[0,1]^d.
\]
Moreover, its eigenvalues and eigenfunctions can be computed by $\Lambda_{\Phi}$ and $\Eset_{\Phi}$, that is,
\[
\lambda_{d,\vn}:=\prod_{k\in\NN_d}\lambda_{1,n_k}=\prod_{k\in\NN_d}\rho_{n_k}^2,
\quad
e_{d,\vn}(\vx):=\prod_{k\in\NN_d}e_{1,n_k}(x_k)=\prod_{k\in\NN_d}\varphi_{n_k}(x_k),
\]
for $\vn:=(n_k:k\in\NN_d)\in\NN^d$. Here, $\NN_d:=\left\{1,2,\ldots,d\right\}$ is a finite set of natural numbers and $\NN^d=\otimes_{k=1}^d\NN$ is the tensor product of natural numbers.
We find that
\begin{equation}\label{eq:int-min-kernel}
K_d(\vx,\vy)=\prod_{k\in\NN_d}K_1(x_k,y_k),\quad \text{for }\vx,\vy\in[0,1]^d.
\end{equation}
If the expansion set $\Sset_{K_d}=\left\{\phi_{d,\vn}:\vn\in\NN^d\right\}$ of $K_d$ is chosen to be the collection of
\[
\phi_{d,\vn}(\vx):=\lambda_{d,\vn}^{1/2}e_{d,\vn}(\vx)=\prod_{k\in\NN_d}\rho_{n_k}\varphi_{n_k}(x_k),\quad
\text{for }\vn\in\NN^d,
\]
then we have that
\[
K_d(\vx,\vy)=\sum_{\vn\in\NN^d}\phi_{d,\vn}(\vx)\phi_{d,\vn}(\vy),
\quad\text{for }\vx,\vy\in[0,1]^d.
\]
Using the expansion sets $\Sset_{K_d}$ we construct the normed spaces
\[
\Banach_{K_d}^p\left([0,1]^d\right):=\left\{f:=\sum_{\vn\in\NN^d}a_{\vn}\phi_{d,\vn}:\left(a_{\vn}:\vn\in\NN^d\right)\in \lp\right\},\quad
\text{when }1\leq p<\infty,
\]
equipped with the norm
\[
\norm{f}_{\Banach_{K_d}^p\left([0,1]^d\right)}:=\left(\sum_{\vn\in\NN^d}\abs{a_{\vn}}^p\right)^{1/p},
\]
and
\[
\Banach_{K_d}^{\infty}\left([0,1]^d\right):=
\left\{f:=\sum_{\vn\in\NN^d}a_{\vn}\phi_{d,\vn}:\left(a_{\vn}:\vn\in\NN^d\right)\in \czero\right\},
\]
equipped with the norm
\[
\norm{f}_{\Banach_{K_d}^{\infty}\left([0,1]^d\right)}:=\sup_{\vn\in\NN^d}\abs{a_{\vn}}.
\]

%////////////////////////////////////////////////////////////////////////////////////////////////////////////////////////
\begin{remark}\label{r:Multi-Kernel-Expan}
To complete the constructions of the kernels defined on $d$-dimensional spaces, we repeat the sum-product techniques as follows:
\[
\prod_{k\in\NN_d}\left(\sum_{n\in\NN}a_{k,n}\right)
=\sum_{n_1,\ldots,n_d\in\NN}\prod_{k\in\NN_d}a_{k,n_k},
\]
if $\sum_{n\in\NN}a_{k,n}$ are convergent for all positive sequences $\left(a_{k,n}:n\in\NN\right)$ and $k\in\NN_d$.
\end{remark}
%////////////////////////////////////////////////////////////////////////////////////////////////////////////////////////

Using the techniques given in Remark~\ref{r:Multi-Kernel-Expan}, it is easy to check that
\[
\sum_{\vn\in\NN^d}\lambda_{d,\vn}^{1/2}
=\sum_{\vn\in\NN^d}\prod_{k\in\NN_d}\frac{1}{n_k^2\pi^2}
=\prod_{k\in\NN_d}\left(\sum_{n_k\in\NN}\frac{1}{n_k^2\pi^2}\right)
=\left(\sum_{n\in\NN}\frac{1}{n^2\pi^2}\right)^d<\infty,
\]
and
\[
\sup_{\vn\in\NN^d}\norm{e_{d,\vn}}_{\infty}
=\sup_{\vn\in\NN^d,\vx\in[0,1]^d}\prod_{k\in\NN_d}\abs{\sqrt{2}\sin(n_k\pi x_k)}=2^{d/2}<\infty.
\]
By Proposition~\ref{p:PDK-EigenValVec}, we have the following result.

%////////////////////////////////////////////////////////////////////////////////////////////////////////////////////////
\begin{corollary}\label{c:int-min-RKBS}
The integral-min kernel $K_d$ defined in equation~\eqref{eq:int-min-kernel} is the right-sided reproducing kernel of the right-sided reproducing kernel Banach space $\Banach_{K_d}^{1}\left([0,1]^d\right)$
and
the two-sided reproducing kernel of the two-sided reproducing kernel Banach space $\Banach_{K_d}^{p}\left([0,1]^d\right)$ for $1<p\leq\infty$.
Moreover, $\Banach_{K_d}^{p}\left([0,1]^d\right)\subseteq\Cont\left([0,1]^d\right)$ for $1\leq p\leq\infty$.
\end{corollary}
%////////////////////////////////////////////////////////////////////////////////////////////////////////////////////////

According to the facts that
\[
\sum_{\vn\in\NN^d}\abs{\phi_{d,\vn}(\vx)}
\leq\sup_{\vx\in[0,1]^d,\vn\in\NN^d}\abs{e_{d,\vn}(\vx)}\sum_{\vn\in\NN^d}\lambda_{d,\vn}^{1/2}
=\sum_{\vn\in\NN^d}\prod_{k\in\NN_d}\frac{1}{n_k^2\pi^2}<\infty,
\]
and the continuity of the map $\vx\mapsto\sum_{\vn\in\NN^d}\lambda_{d,\vn}^{1/2}\abs{e_{d,\vn}(\vx)}$ on $[0,1]^d$,
the sequence set
\[
\Aset_{K_d}:=\left\{\left(\phi_{d,\vn}(\vx):\vn\in\NN^d\right):\vx\in[0,1]^d\right\},
\]
is bounded and closed in $\lone$.
By Propositions~\ref{p:imbedding-RKBS-PDK} and~\ref{p:RKBS-imbedding-dual-PDK}, we also obtain the imbedding of
$\Banach_{K_d}^{p}\left([0,1]^d\right)$.
Let $\alpha_{\vn}:=\prod_{k\in\NN_d}n_k^{-1/2}$ for all $\vn\in\NN^d$. Then $\left(\alpha_{\vn}:\vn\in\NN^d\right)\in \czero$ and
\[
\sum_{\vn\in\NN^d}\frac{\abs{\phi_{d,\vn}(\vx)}}{\alpha_n}
\leq\sup_{\vx\in[0,1]^d,\vn\in\NN^d}\abs{e_{d,\vn}(\vx)}\sum_{\vn\in\NN^d}\frac{\lambda_{d,\vn}^{1/2}}{\alpha_n}
=\sum_{\vn\in\NN^d}\prod_{k\in\NN_d}\frac{1}{n_k^{3/2}\pi^2}<\infty.
\]
This ensures that $\Aset_{K_d}$ is relatively compact in $\lone$. Hence, $\Aset_{K_d}$ is relatively compact in $\lq$ for all $1\leq q\leq\infty$.
Because of the compactness of $\Aset_{K_d}$, Proposition~\ref{p:PDK-compact-RKBS} implies that
the identity map from $\Banach_{K}^p\left([0,1]^d\right)$ into $\Linfty\left([0,1]^d\right)$ is compact.
Since the integral-min kernel $K$ is strictly positive definite on $(0,1)^d$ and the measure of the boundary of $[0,1]^d$ is equal to zero, Proposition~\ref{p:universial-approx-RKBS-PDK} ensures that $\Banach_{K}^p\left([0,1]^d\right)$ has the universal approximation property.

%------------------------------------------------------------------------------------------------------------------------
\section{Gaussian Kernels}\label{s:gaussian}
%------------------------------------------------------------------------------------------------------------------------
\sectionmark{Gaussian Kernels}

There is a long history of the
Gaussian kernel
\[
G_{\theta}(x,y):=e^{-\theta^2\abs{x-y}},\quad
\text{for }x,y\in\RR,
\]
where $\theta$ is a positive shape parameter.  The kernel was first introduced by K. F. Gaussian in his book~\cite{Gauss1809} in 1809.
Recently, the Gaussian kernels became a vigorous numerical tool for high dimensional approximation and machine learning. Now we show that the Gaussian kernel can also be employed to construct the $p$-norm RKBSs.

Let the $\vtheta$-norm on $d$-dimensional real space $\Rd$
be
\[
\norm{\vx-\vy}_{\vtheta}:=\left(\sum_{k\in\NN_d}\theta_k^2\abs{x_k-y_k}^2\right)^{1/2},
\]
for $\vx:=(x_k:k\in\NN_d),\vy:=(y_k:k\in\NN_d)\in\Rd$, where $\vtheta:=\left(\theta_k:k\in\NN_d\right)$ is a vector of positive shape parameters.
We write the Gaussian kernel with shape parameters $\vtheta\in\RR_{+}^d$ as
\[
G_{\vtheta}(\vx,\vy):=e^{-\norm{\vx-\vy}_{\vtheta}^2},
\quad\text{for }\vx,\vy\in\Rd.
\]
Moreover, the $d$-dimensional Gaussian kernel $G_{\vtheta}$ can be written as the product of univariate Gaussian kernels $G_{\theta_k}$ for $k\in\NN_d$, that is,
\[
G_{\vtheta}(\vx,\vy)=\prod_{k\in\NN_d}G_{\theta_k}\left(x_k,y_k\right),
\quad\text{for }\vx,\vy\in\Rd.
\]

It is well-known that the Gaussian kernel is a strictly positive definite kernel on $\Rd$.
Since the space $\Rd$ is not compact, there is a question whether the theorems in Section~\ref{s:RKBS-PDK} are applicable to the Gaussian kernels. The reason of the condition of the compactness of $\Domain$ in Section~\ref{s:RKBS-PDK} is that there are two conditions needed for the constructions and the imbedding. One is to guarantee the integral operator $I_K$ is a compact operator from $\Leb_2(\Domain)$ to $\Leb_2(\Domain)$. Thus, the continuous positive definite kernel $K$ has countable positive eigenvalues and continuous eigenfunctions by the Mercer theorem. The other one is to ensure $\mu\left(\Domain\right)<\infty$.
By the theoretical results given in papers \cite{FasshauerHickernellWozniakowski2012,FasshauerHickernellWozniakowski2012MCQMC}
and \cite[Section~4.4]{SteinwartChristmann2008}, we also obtain these two important conditions for the Gaussian kernels even though $\Rd$ is not compact.

For the probability density function
\[
p_d(\vx):=\frac{1}{\pi^{d/2}}e^{-\norm{\vx}_2^2},\quad \text{for }\vx\in\Rd,
\]
we define a probability measure $\mu_d$ on $\Rd$
\[
\mu_d(\ud\vx):=p_d(\vx)\ud\vx,
\]
such that $\Rd$ endowed with $\mu_d$ becomes a probability space.
For convenience we denote by $\Domain_d$ the \emph{probability measurable space} $\left(\Rd,\Borel_{\Rd},\mu_d\right)$.
Clearly, we have that
\[
\mu_d(\Domain_d)=\int_{\Domain_d}\mu_d(\ud\vx)=\int_{\Rd}p_d(\vx)\ud\vx=1.
\]
It is clear that $\supp\left(\mu_d\right)=\Domain_d$.
We verify that the integral operator $I_{K_{\vtheta}}$ is a compact operator from $\Leb_2\left(\Domain_d\right)$
to $\Leb_2\left(\Domain_d\right)$.

Let us look at the eigenvalues and eigenfunctions of the univariate Gaussian kernel $G_{\theta}$.
Papers \cite{FasshauerHickernellWozniakowski2012,FasshauerHickernellWozniakowski2012MCQMC}
showed that the eigenvalues and eigenfunctions of $G_{\theta}$ are represented, respectively, as
\begin{equation}\label{eq:Gaussian-1}
\rho_{\theta,n}:=\left(1-w_{\theta}\right)w_{\theta}^{n-1},
\end{equation}
and
\[
e_{\theta,n}(x):=
\left(\frac{\left(1+4\theta^2\right)^{1/4}}{2^{n-1}(n-1)!}\right)^{1/2}
e^{-u_{\theta}x^2}
H_{n-1}\left(\left(1+4\theta^2\right)^{1/4}x\right),
\]
for $n\in\NN$, where
\[
w_{\theta}:=\frac{2\theta^2}{1+\left(1+4\theta^2\right)^{1/2}+2\theta^2},
\]
and
\[
u_{\theta}:=\frac{2\theta^2}{1+\left(1+4\theta^2\right)^{1/2}}.
\]
Here, $H_{n-1}$ is the Hermite polynomial of degree $n-1$, that is,
\[
H_{n-1}(x):=(-1)^{n-1}e^{x^2}\frac{\ud^{n-1}}{\ud x^{n-1}}\left(e^{-x^2}\right),
\]
so that
\[
\int_{\Domain_1}H_{n-1}(x)^2\mu_1(\ud x)
=\frac{1}{\pi^{1/2}}\int_{\RR}H_{n-1}(x)^2e^{-x^2}\ud x
=2^{n-1}(n-1)!.
\]
Since $w_{\theta}\in(0,1)$, we have that
\[
\sum_{n\in\NN}\rho_{\theta,n}
=(1-w_{\theta})\sum_{n\in\NN}w_{\theta}^{n-1}
=(1-w_{\theta})\frac{1}{1-w_{\theta}}=1.
\]
Because of the fact that
\[
\int_{\Domain_1}e_{\theta,m}(x)e_{\theta,n}(x)\mu_1(\ud x)
=\delta_{mn},
\]
for all $m,n\in\NN$,
the univariate Gaussian kernel $G_{\theta}$ also has the absolutely and uniformly convergent representation
\[
G_{\theta}(x,y)
=\sum_{n\in\NN}\rho_{\theta,n}e_{\theta,n}(x)e_{\theta,n}(y),
\quad\text{for }x,y\in\Domain_1.
\]
Now we compute the integral-type kernel
\begin{equation}\label{eq:Gaussian-2}
\begin{split}
K_{\theta}(x,y)&=\int_{\Domain_1}G_{\theta}(x,z)G_{\theta}(z,y)\mu_1\left(\ud z\right)\\
&=\int_{\RR}G_{\theta}(x,z)G_{\theta}(z,y)p_1(z)\ud z\\
&=\frac{1}{\sqrt{\pi}}\int_{\RR}e^{-\theta^2(x-z)^2-\theta^2(z-y)^2-z^2}\ud z\\
&=\frac{1}{\sqrt{\pi}}e^{-\theta^2\left(x^2+y^2\right)}
\int_{\RR}e^{2\theta^2(x+y)z-\left(2\theta^2+1\right)z^2}\ud z\\
&=\left(2\theta^2+1\right)^{-1/2}e^{-\vartheta(\theta)^2\left(x-y\right)^2-\alpha(\theta)^2xy}\\
&=\left(2\theta^2+1\right)^{-1/2}G_{\vartheta(\theta)}(x,y)\Psi_{\alpha(\theta)}(x,y),
\end{split}
\end{equation}
for $x,y\in\RR$, where
\[
\vartheta(\theta):=\left(\frac{\theta^4+\theta^2}{2\theta^2+1}\right)^{1/2},
\quad
\alpha(\theta):=\left(\frac{2\theta^2}{2\theta^2+1}\right)^{1/2},
\]
and
\[
\Psi_{\alpha}(x,y):=e^{-\alpha^2xy}.
\]

Next, we look at the integral-type kernel $K_{\vtheta}$ induced by the $d$-dimensional Gaussian kernel $G_{\vtheta}$, that is,
\begin{equation}\label{eq:int-Gaussian-kernel}
K_{\vtheta}(\vx,\vy):=
\int_{\Domain_d}G_{\vtheta}(\vx,\vz)G_{\vtheta}(\vz,\vy)\mu_d\left(\ud\vz\right),
\quad\text{for }\vx,\vy\in\Domain_d.
\end{equation}
Then the integral-type kernel $K_{\vtheta}$ can also be rewritten as
\begin{equation}\label{eq:Gaussian-3}
K_{\vtheta}(\vx,\vy)
=\prod_{k\in\NN_d}\int_{\Domain_1}G_{\theta_k}\left(x_k,z_k\right)G_{\theta_k}\left(z_k,y_k\right)\mu_1\left(\ud z_k\right)
=\prod_{k\in\NN_d}K_{\theta_k}\left(x_k,y_k\right),
\end{equation}
for $\vx:=\left(x_k:k\in\NN_d\right),\vy:=\left(y_k:k\in\NN_d\right)\in\Domain_d$.
Let
\[
\vvartheta(\vtheta):=\left(\vartheta(\theta_k):k\in\NN_d\right),
\quad
\valpha(\vtheta):=\left(\alpha(\theta_k):k\in\NN_d\right)
\]
and
\[
\Psi_{\valpha}(\vx,\vy):=
\prod_{k\in\NN_d}\Psi_{\alpha_k}\left(x_k,y_k\right)
=
\prod_{k\in\NN_d}e^{-\alpha_k^2x_ky_k}=e^{-(\vx,\vy)_{\valpha}},
\]
where the $\valpha$-inner product is given by
\[
(\vx,\vy)_{\valpha}:=\sum_{k\in\NN_d}\alpha_k^2x_ky_k.
\]
Combining equations~\eqref{eq:Gaussian-2} and~\eqref{eq:Gaussian-3}, we have that
\[
K_{\vtheta}(\vx,\vy)=C(\vtheta)G_{\vvartheta(\vtheta)}(\vx,\vy)\Psi_{\vtheta}(\vx,\vy),
\]
where
\[
C(\vtheta):=\prod_{k\in\NN_d}\left(2\theta_k^2+1\right)^{-1/2}.
\]

By the proof of Proposition~\ref{p:integral-typeKer}, we obtain the eigenvalues and eigenfunctions of $K_{\vtheta}$
\[
\lambda_{\vtheta,\vn}:=\prod_{k\in\NN_d}\lambda_{\theta_k,n_k}=\prod_{k\in\NN_d}\rho_{\theta_k,n_k}^2,
\quad
e_{\vtheta,\vn}(\vx):=\prod_{k\in\NN_d}e_{\theta_k,n_k}\left(x_k\right),
\]
for $\vn:=\left(n_k:k\in\NN_d\right)\in\NN^d$. Hence, the expansion elements of $K_{\vtheta}$ have the form
\[
\phi_{\vtheta,\vn}(\vx):=\lambda_{\vtheta,\vn}^{1/2}e_{\vtheta,\vn}(\vx)
=\prod_{k\in\NN_d}\rho_{\theta_k,n_k}e_{\theta_k,n_k}\left(x_k\right).
\]
Based on the expansion set
$\Sset_{K_{\vtheta}}:=\left\{\phi_{\vtheta,\vn}:\vn\in\NN^d\right\}$, we define the $p$-norm spaces as
\[
\Banach_{K_{\vtheta}}^p\left(\Domain_d\right):=\left\{f:=\sum_{\vn\in\NN^d}a_{\vn}\phi_{\vtheta,\vn}:\left(a_{\vn}:\vn\in\NN^d\right)\in \lp\right\},\quad
\text{when }1\leq p<\infty,
\]
equipped with the norm
\[
\norm{f}_{\Banach_{K_{\vtheta}}^p\left(\Domain_d\right)}:=\left(\sum_{\vn\in\NN^d}\abs{a_{\vn}}^p\right)^{1/p},
\]
and
\[
\Banach_{K_{\vtheta}}^{\infty}\left(\Domain_d\right):=
\left\{f:=\sum_{\vn\in\NN^d}a_{\vn}\phi_{\vtheta,\vn}:\left(a_{\vn}:\vn\in\NN^d\right)\in \czero\right\},
\]
equipped with the norm
\[
\norm{f}_{\Banach_{K_{\vtheta}}^{\infty}\left(\Domain_d\right)}:=\sup_{\vn\in\NN^d}\abs{a_{\vn}}.
\]
Therefore, Proposition~\ref{p:integral-typeKer} yields the reproducing properties of $\Banach_{K_{\vtheta}}^p\left(\Domain_d\right)$.

%////////////////////////////////////////////////////////////////////////////////////////////////////////////////////////
\begin{corollary}\label{c:Gaussian-RKBS}
The integral-Gaussian kernel $K_{\vtheta}$ defined in equation~\eqref{eq:int-Gaussian-kernel} is the right-sided reproducing kernel of the right-sided reproducing kernel Banach space $\Banach_{K_{\vtheta}}^{1}\left(\Domain_d\right)$
and
the two-sided reproducing kernel of the two-sided reproducing kernel Banach space $\Banach_{K_{\vtheta}}^{p}\left(\Domain_d\right)$ for $1<p\leq\infty$.
Moreover, $\Banach_{K_{\vtheta}}^{p}\left(\Domain_d\right)\subseteq\Cont\left(\Domain_d\right)$ for $1\leq p\leq\infty$.
\end{corollary}
%////////////////////////////////////////////////////////////////////////////////////////////////////////////////////////

Let
\[
\rho_{\vtheta,\vn}:=\prod_{k\in\NN_d}\rho_{\theta_k,n_k},
\quad\text{for }\vn\in\NN^d.
\]
Then the $d$-dimensional Gaussian kernel $G_{\vtheta}$ can be represented as
\[
G_{\vtheta}(\vx,\vy)=
\prod_{k\in\NN_d}\left(\sum_{n_k\in\NN}\rho_{\theta_k,n_k}e_{\theta_k,n_k}\left(x_k\right)e_{\theta_k,n_k}\left(y_k\right)\right)
=\sum_{\vn\in\NN^d}\rho_{\vtheta,\vn}e_{\vtheta,\vn}(\vx)e_{\vtheta,\vn}(\vy).
\]
Equation~\eqref{eq:Gaussian-1} implies that
\[
\sum_{\vn\in\NN^d}\rho_{\vtheta,\vn}
=\prod_{k\in\NN_d}\left(\sum_{n_k\in\NN}\rho_{\theta_k,n_k}\right)
=1.
\]
Using the same techniques of Theorem~\ref{p:integral-typeKer}, we determine that
\[
\sum_{\vn\in\NN^d}\abs{\phi_{\vtheta,\vn}(\vx)}
\leq\left(\sum_{\vn\in\NN^d}\rho_{\vtheta,\vn}\right)^{1/2}
\left(\sum_{\vn\in\NN^d}\rho_{\vtheta,\vn}\abs{e_{\vtheta,\vn}(\vx)}^2\right)^{1/2}
=\left(G_{\vtheta}(\vx,\vx)\right)^{1/2},
\]
for all $\vx\in\Domain_d$. This shows that the sequence set
\[
\Aset_{K_{\vtheta}}:=\left\{\left(\phi_{\vtheta,\vn}(\vx):\vn\in\NN^d\right):\vx\in\Domain_d\right\}
\]
is bounded in $\lone$.
By Propositions~\ref{p:imbedding-RKBS-PDK} and~\ref{p:RKBS-imbedding-dual-PDK}, we obtain the imbedding of the
$p$-norm RKBSs $\Banach_{K_{\vtheta}}^p\left(\Domain_d\right)$.

In addition, we check the compactness of $\Banach_{K_{\vtheta}}^p\left(\Domain_d\right)$. Take a sequence
\[
\left(w_{\vtheta}^{(\vn-1)/4}:\vn\in\NN^d\right),
\]
where
\[
w_{\vtheta}:=\prod_{k\in\NN_d}w_{\theta_k},\quad
w_{\vtheta}^{(\vn-1)/4}:=\prod_{k\in\NN_d}w_{\theta_k}^{(n_k-1)/4}.
\]
Then we have that
\[
\left(w_{\vtheta}^{(\vn-1)/4}:\vn\in\NN^d\right)\in \czero,
\]
because $w_{\theta_k}\in(0,1)$ for all $k\in\NN_d$.
By the techniques given in Remark~\ref{r:Multi-Kernel-Expan}, equation~\eqref{eq:Gaussian-1} ensures that
\begin{align*}
\sum_{\vn\in\NN^d}\frac{\rho_{\vtheta,\vn}}{w_{\vtheta}^{(\vn-1)/2}}&
=\prod_{k\in\NN_d}\left(\sum_{n_k\in\NN}\frac{\rho_{\theta_k,n_k}}{w_{\theta_k}^{(n_k-1)/2}}\right)\\
&=\prod_{k\in\NN_d}\left(1-w_{\theta_k}\right)\sum_{n_k\in\NN}\left(w_{\theta_k}^{1/2}\right)^{n_k-1}\\
&=\prod_{k\in\NN_d}\frac{1-w_{\theta_k}}{1-w_{\theta_k}^{1/2}}.
\end{align*}
Therefore,
\begin{align*}
\sum_{\vn\in\NN^d}\frac{\rho_{\vtheta,\vn}\abs{e_{\vtheta,\vn}(\vx)}}{w_{\vtheta}^{(\vn-1)/4}}
\leq&
\left(\sum_{\vn\in\NN^d}\frac{\rho_{\vtheta,\vn}}{w_{\vtheta}^{(\vn-1)/2}}\right)^{1/2}
\left(\sum_{\vn\in\NN^d}\rho_{\vtheta,\vn}\abs{e_{\vtheta,\vn}(\vx)}^2\right)^{1/2}\\
=&\left(G_{\vtheta}(\vx,\vx)\right)^{1/2}\prod_{k\in\NN_d}\frac{1-w_{\theta_k}}{1-w_{\theta_k}^{1/2}}
<\infty,
\end{align*}
for all $\vx\in\Domain_d$.
This ensures that $\Aset_{K_{\vtheta}}$ is relatively compact in $\lone$.
The continuity of the map $\vx\mapsto\rho_{\vtheta,\vn}\abs{e_{\vtheta,\vn}(\vx)}$ ensures that
$\Aset_{K_{\vtheta}}$ is closed in $\lone$.
So we obtain the compactness of $\Aset_{K_{\vtheta}}$.
Proposition~\ref{p:PDK-compact-RKBS} provides that the identity map is a compact operator from
$\Banach_{K_{\vtheta}}^p\left(\Domain_d\right)$ into $\Linfty\left(\Domain_d\right)$.

Clearly, the integral-Gaussian kernel $K_{\vtheta}$ is a strictly positive definite kernel on $\Domain_d$.
Since the domain $\Domain_d$ is not compact, the uniform norm is not well-defined on $\Cont(\Domain_d)$.
But, for any compact domain $\Domain$ in $\Domain_d$, the $p$-norm RKBS $\Banach_{K_{\vtheta}}^p\left(\Domain\right)$ has the universal approximation property by Proposition~\ref{p:universial-approx-RKBS-PDK}.

\subsection*{Gaussian Kernels defined on $d$-dimensional space $\Rd$}

There is another interesting question whether the Gaussian kernel can still become a reproducing kernel of some RKBSs when the measure of its domain is not finite, that is, $\mu(\Domain)=\infty$.
To answer this question, we consider another symmetric expansion set $\tilde{\Sset}_{G_{\vtheta}}=\tilde{\Sset}_{G_{\vtheta}}'$ of the $d$-dimensional Gaussian kernel $G_{\vtheta}$ defined on $\Rd$.
According to \cite[Theorem~4.38]{SteinwartChristmann2008} about the expansions of the Gaussian kernel defined on $\RR$, the univariate Gaussian kernel $G_{\theta}$ has the absolutely and uniformly convergent representation
\[
G_{\theta}(x,y)=\sum_{n\in\NN_0}\tilde{\phi}_{\theta,n}(x)\tilde{\phi}_{\theta,n}(y),
\quad\text{for }x,y\in\RR,
\]
where
\[
\tilde{\phi}_{\theta,n}(x):=\left(\frac{2^n}{n!}\right)^{1/2}\left(\theta x\right)^ne^{-\theta^2x^2},
\quad\text{for }n\in\NN_0=\NN\cup\{0\}.
\]
Here, the $d$-dimensional space $\Rd$ is a standard Euler space and the measure $\mu$ is the Lebesgue measure defined on $\Rd$.
This indicates that
\begin{align*}
G_{\vtheta}(\vx,\vy)=&\prod_{k\in\NN_d}G_{\theta_k}\left(x_k,y_k\right)
=\sum_{k\in\NN_d}\left(\sum_{n_k\in\NN_0}\tilde{\phi}_{\theta_k,n_k}(x_k)\tilde{\phi}_{\theta_k,n_k}(y_k)\right)\\
=&\sum_{\vn\in\NN_0^d}\left(\prod_{k\in\NN_d}\tilde{\phi}_{\theta_k,n_k}(x_k)\tilde{\phi}_{\theta_k,n_k}(y_k)\right)
=\sum_{\vn\in\NN_0^d}\tilde{\phi}_{\vtheta,\vn}(\vx)\tilde{\phi}_{\vtheta,\vn}(\vy),
\end{align*}
where
\begin{equation}\label{eq:Gaussian-exp-element}
\tilde{\phi}_{\vtheta,\vn}(\vx):=\prod_{k\in\NN_d}\tilde{\phi}_{\theta_k,n_k}(x_k)
=\left(\frac{2^{\norm{\vn}_1}}{\vn!}\right)^{1/2}\vtheta^{\vn}\vx^{\vn}e^{-\norm{\vx}_{\vtheta}^2}.
\end{equation}
Using the power series
\[
\sum_{n\in\NN_0}\frac{\abs{z}^{n}}{\sqrt{n!}}
\leq\left(\sum_{n\in\NN_0}\tau^{2n}\right)
\left(\sum_{n\in\NN_0}\frac{\left(\tau^{-1}z\right)^{2n}}{n!}\right)
=\frac{1}{1-\tau^2}e^{\tau^{-2}z^2},
\]
for any $0<\tau<1$,
we have that
\begin{align*}
\sum_{\vn\in\NN_0^d}\abs{\tilde{\phi}_{\vtheta,\vn}(\vx)}
\leq&
\left(\sum_{\vn\in\NN_0^d}\left(\prod_{k\in\NN_d}\sqrt{\frac{2^{n_k}}{n_k!}}\theta_k^{n_k}\abs{x_k}^{n_k}e^{-\theta_k^2x_k^2}\right)\right)\\
\leq&\prod_{k\in\NN_d}\left(\sum_{n_k\in\NN_0}\frac{\left(\sqrt{2}\theta_k\abs{x_k}\right)^{n_k}}{\sqrt{n_k!}}\right)e^{-\theta_k^2x_k^2}\\
\leq&\prod_{k\in\NN_d}\frac{e^{\left(2\tau^{-2}-1\right)\theta_k^2x_k^2}}{1-\tau^2}
<\infty,
\end{align*}
for each $\vx\in\Rd$.
Let $\Domain$ be an arbitrary connected region of $\Rd$.
Then the expansion set $\tilde{\Sset}_{G_{\vtheta}}:=\left\{\tilde{\phi}_{\vtheta,\vn}|_{\Domain}:\vn\in\NN_0^d\right\}$ satisfies conditions (C-$1*$).
Combining this with the linear independence of $\tilde{\Sset}_{G_{\vtheta}}$, we ensure that the expansion sets $\tilde{\Sset}_{G_{\vtheta}}=\tilde{\Sset}_{G_{\vtheta}}'$ of the Gaussian kernel $G_{\vtheta}|_{\Domain\times\Domain}$ satisfy assumption (A-$1*$).
Same as the discussion of Theorem~\ref{t-RKBS-PDK},
the Gaussian kernel $G_{\vtheta}|_{\Domain\times\Domain}$ is a reproducing kernel and its
RKBSs $\Banach_{G_{\vtheta}}^p\left(\Domain\right)$ for all $1\leq p\leq\infty$ are well-defined by $\tilde{\Sset}_{G_{\vtheta}}$, that is,
\[
\Banach_{G_{\vtheta}}^p\left(\Domain\right):=
\left\{f:=\sum_{\vn\in\NN_0^d}a_{\vn}\tilde{\phi}_{\vtheta,\vn}|_{\Domain}:\left(a_{\vn}:\vn\in\NN^d\right)\in \lp\right\},\quad
\text{when }1\leq p<\infty,
\]
equipped with the norm
\[
\norm{f}_{\Banach_{G_{\vtheta}}^p\left(\Domain\right)}:=\left(\sum_{\vn\in\NN_0^d}\abs{a_{\vn}}^p\right)^{1/p},
\]
and
\[
\Banach_{G_{\vtheta}}^{\infty}\left(\Domain\right):=
\left\{f:=\sum_{\vn\in\NN_0^d}a_{\vn}\tilde{\phi}_{\vtheta,\vn}|_{\Domain}:\left(a_{\vn}:\vn\in\NN^d\right)\in \czero\right\},
\]
equipped with the norm
\[
\norm{f}_{\Banach_{G_{\vtheta}}^{\infty}\left(\Domain\right)}:=\sup_{\vn\in\NN_0^d}\abs{a_{\vn}}.
\]

If the domain $\Domain$ is compact in $\Rd$, then
the sequence set
\[
\tilde{\Aset}_{G_{\vtheta}}:=\left\{\left(\tilde{\phi}_{\vtheta,\vn}(\vx):\vn\in\NN_0^d\right):\vx\in\Domain\right\},
\]
is compact in $\lone$. Hence, we obtain
the imbedding properties and compactness of $\Banach_{G_{\vtheta}}^p\left(\Domain\right)$ by the same manners of Propositions~\ref{p:imbedding-RKBS-PDK},~\ref{p:RKBS-imbedding-dual-PDK} and~\ref{p:PDK-compact-RKBS}.
By Proposition~\ref{p:universial-approx-RKBS-PDK}, the $p$-norm RKBS $\Banach_{G_{\vtheta}}^p\left(\Domain\right)$ has the universal approximation property.

%------------------------------------------------------------------------------------------------------------------------
\section{Power Series Kernels}\label{s:powerseries}
%------------------------------------------------------------------------------------------------------------------------
\sectionmark{Power Series Kernels}

This example of power series kernels show that we use non-eigenvalues and non-eigenfunctions of positive definite kernels to construct its $p$-norm RKBSs differently from RKHSs.

Let
\[
\vx^{\vn}:=\prod_{k\in\NN_d}x_k^{n_k},\quad
\vn!:=\prod_{k\in\NN_d}n_k!,
\]
for $\vx:=\left(x_k:k\in\NN_d\right)\in\Rd$ and $\vn:=\left(n_k:k\in\NN_d\right)\in\NN_0^d$.
According to \cite[Remark 1]{Zwicknagl2009}, we know that the \emph{power series kernel} $K$
\[
K(\vx,\vy):=\sum_{\vn\in\NN^d_0}w_{\vn}\frac{\vx^{\vn}}{\vn!}\frac{\vy^{\vn}}{\vn!},
\quad\text{for }\vx,\vy\in (-1,1)^d,
\]
is a positive definite kernel if the positive sequence $\left(w_{\vn}:\vn\in\NN_0^d\right)$ satisfies
\[
\sum_{\vn\in\NN_0^d}\frac{w_{\vn}}{\vn!^2}<\infty.
\]
Here, the measure $\mu$ is the Lebesgue measure defined on $(-1,1)^d$.
We observe that the polynomial expansion elements of the power series kernel $K$
\begin{equation}\label{eq:power-kernel-1}
\phi_{\vn}(\vx):=w_{\vn}^{1/2}\frac{\vx^{\vn}}{\vn!},
\end{equation}
are NOT the eigenvalues and eigenfunctions of $K$.
However, we still use the expansion set $\Sset_{K}:=\left\{\phi_{\vn}:\vn\in\NN_0^d\right\}$ to construct the $p$-norm RKBSs.

Next we consider a special case of power series kernels.
A power series kernel $K$ is called \emph{nonlinear and factorizable} if
there exists an analytic function
$$
\eta(z):=\sum_{n\in\NN_0}c_nz^n
$$
with positive coefficients $\left\{c_n:n\in\NN_0\right\}$
such that
\begin{equation}\label{eq:non-fact-power-kernel}
K(\vx,\vy)=\prod_{k\in\NN_d}\eta(x_ky_k).
\end{equation}
Hence,
\[
K(\vx,\vy)=\prod_{k\in\NN_d}\left(\sum_{n_k\in\NN_0}c_{n_k}x_k^{n_k}y_k^{n_k}\right)
=\sum_{\vn\in\NN_0^d}\left(\prod_{k\in\NN_d}c_{n_k}\right)\vx^{\vn}\vy^{\vn}.
\]
This indicates that the positive weights of the nonlinear factorizable power series kernel $K$ are given by
\[
w_{\vn}:=\prod_{k\in\NN_d}c_{n_k}\left(n_k!\right)^2,\quad\text{for }\vn\in\NN_0^d.
\]
Hence, its expansion elements defined in equation~\eqref{eq:power-kernel-1} can be rewritten as
\begin{equation}\label{eq:non-fact-power-kernel-expan-element}
\phi_{\vn}(\vx):=\left(\prod_{k\in\NN_d}c_{n_k}^{1/2}\right)\vx^{\vn}.
\end{equation}
Moreover, if
\[
\sum_{n\in\NN_0}c_n^{1/2}<\infty,
\]
then
\begin{equation}\label{eq:power-kernel-2}
\begin{split}
\sum_{\vn\in\NN_0^d}\abs{\phi_{\vn}(\vx)}
=\sum_{\vn\in\NN_0^d}\left(\prod_{k\in\NN_d}c_{n_k}^{1/2}\right)\abs{\vx^{\vn}}
\leq\left(\sum_{n\in\NN_0}c_n^{1/2}\right)^d<\infty,
\end{split}
\end{equation}
for $\vx\in(-1,1)^d$ because
\[
\sum_{\vn\in\NN_0^d}\left(\prod_{k\in\NN_d}c_{n_k}^{1/2}\right)
=\prod_{k\in\NN_d}\left(\sum_{n_k\in\NN_0}c_{n_k}^{1/2}\right)
=\left(\sum_{n\in\NN_0}c_n^{1/2}\right)^d.
\]
Since the standard polynomial functions $\left\{\vx^{\vn}:\vn\in\NN_0^d\right\}$ are linearly independent, the expansion set $\Sset_K=\Sset_K'$ of the power series kernel $K$ satisfies assumption (A-$1*$).
Thus, we employ
the expansion set
$\Sset_{K}$ to build the $p$-norm spaces
\[
\Banach_{K}^p\left((-1,1)^d\right):=
\left\{f:=\sum_{\vn\in\NN_0^d}a_{\vn}\phi_{\vn}:\left(a_{\vn}:\vn\in\NN_0^d\right)\in \lp\right\},\quad
\text{when }1\leq p<\infty,
\]
equipped with the norm
\[
\norm{f}_{\Banach_{K}^p\left((-1,1)^d\right)}:=\left(\sum_{\vn\in\NN_0^d}\abs{a_{\vn}}^p\right)^{1/p},
\]
and
\[
\Banach_{K}^{\infty}\left((-1,1)^d\right):=
\left\{f:=\sum_{\vn\in\NN_0^d}a_{\vn}\phi_{\vn}:\left(a_{\vn}:\vn\in\NN_0^d\right)\in \czero\right\},
\]
equipped with the norm
\[
\norm{f}_{\Banach_{K}^{\infty}\left((-1,1)^d\right)}:=\sup_{\vn\in\NN_0^d}\abs{a_{\vn}}.
\]
In the same manner as in Theorem~\ref{t-RKBS-PDK}, we obtain the following result.

\begin{corollary}\label{c:Power-Fac-RKBS}
If the positive expansion coefficients $\left\{c_n:n\in\NN_0\right\}$ of the nonlinear factorizable power series kernel $K$ satisfy that
\begin{equation}\label{eq:power-kernel-3}
\sum_{n\in\NN_0}c_n^{1/2}<\infty,
\end{equation}
then $K$ is the right-sided reproducing kernel of the right-sided reproducing kernel Banach space $\Banach_{K}^{1}\left((-1,1)^d\right)$
and
the two-sided reproducing kernel of the two-sided reproducing kernel Banach space $\Banach_{K}^{p}\left((-1,1)^d\right)$ for $1<p\leq\infty$.
Moreover, $\Banach_{K}^{p}\left((-1,1)^d\right)\subseteq\Cont\left((-1,1)^d\right)$ for $1\leq p\leq\infty$.
\end{corollary}
%////////////////////////////////////////////////////////////////////////////////////////////////////////////////////////
\begin{proof}
Since $K$ satisfies assumption (A-$1*$),
Theorems~\ref{t:RKBS-MercerKer-p}, \ref{t:RKBS-MercerKer-1} and \ref{t:RKBS-MercerKer-infty}
assure the reproducing properties of $\Banach_{K}^{p}\left((-1,1)^d\right)$ for all $1\leq p\leq\infty$.
By the Abel uniform convergence test,
the power series $\sum_{\vn\in\NN_0^d}\abs{\phi_{\vn}(\vx)}$ is uniformly convergent on $(-1,1)^d$. It follows that $\Banach_{K}^{p}\left((-1,1)^d\right)$ is included in $\Cont\left((-1,1)^d\right)$.
\end{proof}
%////////////////////////////////////////////////////////////////////////////////////////////////////////////////////////

Inequality~\eqref{eq:power-kernel-2} implies that the sequence set
\[
\Aset_K:=\left\{\left(\phi_{\vn}(\vx):\vn\in\NN_0^d\right):\vx\in(-1,1)^d\right\}
\]
is bounded in $\lone$. Therefore, we obtain the imbedding properties of $\Banach_{K}^{p}\left((-1,1)^d\right)$ the same as in Propositions~\ref{p:imbedding-RKBS-PDK} and~\ref{p:RKBS-imbedding-dual-PDK}.
Moreover, by Proposition~\ref{p:universial-approx-RKBS-PDK}, the $p$-norm RKBS $\Banach_{K}^{p}\left(\Domain\right)$ defined on any compact domain $\Domain$ of $(-1,1)^d$ (a bounded and closed region in $(-1,1)^d$) has the universal approximation property.

We give below examples of nonlinear factorizable power series kernels with the condition given in equation~\eqref{eq:power-kernel-3}, for example,
\begin{align*}
K(\vx,\vy):=\prod_{k\in\NN_d}e^{x_ky_k},\quad&\text{when }\eta(z):=e^z,\\
K(\vx,\vy):=\prod_{k\in\NN_d}\frac{1}{1-\theta x_ky_k},\quad&\text{when }\eta(z):=\frac{1}{1-\theta z},~0<\theta<1,\\
K(\vx,\vy):=\prod_{k\in\NN_d}I_0\left(2x_ky_k^{1/2}\right),\quad&\text{when }\eta(z):=I_0(2z^{1/2}),
\end{align*}
where $I_0$ is the modified Bessel function of the first kind of order $0$.

%------------------------------------------------------------------------------------------------------------------------
%------------------------------------------------------------------------------------------------------------------------
\chapter{Support Vector Machines}\label{char-SVM}
%------------------------------------------------------------------------------------------------------------------------
%------------------------------------------------------------------------------------------------------------------------

We develop in this chapter support vector machines on the $p$-norm RKBSs for $1<p<\infty$ and the $1$-norm RKBSs induced by the generalized Mercer kernels.

%------------------------------------------------------------------------------------------------------------------------
\section{Background of Support Vector Machines}\label{s:Background-svm}
%------------------------------------------------------------------------------------------------------------------------
\sectionmark{Background of Support Vector Machines}

We begin with a review of the traditional support vector machines on a reproducing Hilbert space, which solves a binary classification problem with the labels $\left\{\pm1\right\}$. For the classical machine learning, the original support vector machines~\cite{BoserGuyonVapnik1992,CortesVapnik1995} were generalized by the portrait method for the pattern recognitions invented earlier in~\cite{VapnikLerner1963}.
The field of learning methods has witnessed intense activities in the study of the support vector machines.

Let the training data $D:=\left\{\left(\vx_k,y_k\right):k\in\NN_N\right\}$ be composed of input data points $X:=\left\{\vx_k:k\in\NN_N\right\}\subseteq\Domain\subseteq\Rd$ and output data values $Y:=\left\{y_k:k\in\NN_N\right\}\subseteq\left\{\pm1\right\}$.
We begin to find an affine linear hyperplane described by $\left(\vw_s,b_s\right)\in\Rd\times\RR$ that separates the training data $D$ into the two groups for $y_k=1$ and $y_k=-1$.
The generalized portrait algorithm finds a perfectly separating hyperplane
\[
s(\vx):=\vw_s^T\vx+b_s,
\]
that has the maximal \emph{margin} to maximize the distance of the hyperplane to the points in $D$. This separating hyperplane gives us a decision rule to predict the labels at unknown locations, that is,
\[
r(\vx):=\Sign\left(s(\vx)\right),\quad\text{for }\vx\in\Rd.
\]
The $\left(\vw_s,b_s\right)$ is solved by the hard margin support vector machine
\begin{equation}\label{eq:hard-margin-svm}
\begin{split}
&\min_{\left(\vw,b\right)\in\RR^d\times\RR}\norm{\vw}_2^2\\
\text{s.t. }&y_k\left(\vw^T\vx_k+b\right)\geq1,\quad\text{for all }k\in\NN_N.
\end{split}
\end{equation}
It is also called the maximal margin classifier.
To relax the constraints of optimization problem~\eqref{eq:hard-margin-svm} for some slack variables $\xi_k\geq0$, we find $\left(\vw_s,b_s\right)$ by the soft margin support vector machine
\begin{equation}\label{eq:soft-margin-svm}
\begin{split}
&\min_{\left(\vw,b\right)\in\RR^d\times\RR}
\left\{C\sum_{k\in\NN_N}\xi_k+\frac{1}{2}\norm{\vw}_2^2\right\}\\
\text{s.t. }&y_k\left(\vw^T\vx_k+b\right)\geq1-\xi_k,\quad \xi_k\geq0, \quad\text{for all }k\in\NN_N,\\
\end{split}
\end{equation}
where the constant $C$ is a free positive parameter for balancing the margins and the errors.
Let
$L(\vx,y,t):=\max\left\{0,1-yt\right\}$ be the hinge loss function and $\sigma:=(2NC)^{-1/2}$. Optimization problem~\eqref{eq:soft-margin-svm} is then equivalent to
\begin{equation}\label{eq:svm-hint-loss-vec}
\min_{\left(\vw,b\right)\in\RR^d\times\RR}
\left\{\frac{1}{N}\sum_{k\in\NN_N}L\left(\vx_k,y_k,\vw^T\vx_k+b\right)+\sigma^2\norm{\vw}_2^2\right\}.
\end{equation}

The support vector machines are also considered in a Hilbert space $\Hilbert$ of functions $f:\Domain\to\RR$. Such a Hilbert space is also called the feature space. A \emph{feature map} $F:\Domain\to\Hilbert$ is driven to generalize the portrait algorithm to the mapped data set $\left\{\left(F(\vx_k),y_k\right):k\in\NN_N\right\}$.
If we replace the vector $\vw$ and the doc inner product $\vw^T\vx_k$, respectively, by the function $f$ and the inner product $\left(f,F(\vx_k)\right)_{\Hilbert}$, then we extend optimization problem~\eqref{eq:svm-hint-loss-vec} to
\begin{equation}\label{eq:svm-hint-loss-RKHS}
\min_{f\in\Hilbert}
\left\{\frac{1}{N}\sum_{k\in\NN_N}L\left(\vx_k,y_k,\left(f,F(\vx_k)\right)_{\Hilbert}\right)+\sigma^2\norm{f}_{\Hilbert}^2\right\}.
\end{equation}
Since the offset term $b$ of the hyperplane has neither a known theoretical nor an empirical advantage for the nonlinear feature map, the offset term $b$ is absorbed by the function $f$ (more details may be found in \cite[Chapter 11]{SteinwartChristmann2008}). We also hope that the feature map $F$ satisfies $\left(f,F(\vx)\right)_{\Hilbert}=f(\vx)$, or more precisely, $F(\vx)$ can be viewed as the identical element of the point evaluation functional $\delta_{\vx}$. Combining this with $\Hilbert'\cong\Hilbert$, we have that $\delta_{\vx}\in\Hilbert'$. Moreover, if we let
\[
K(\vx,\vy):=\left(F(\vy),F(\vx)\right)_{\Hilbert}=\left(\delta_{\vx},\delta_{\vy}\right)_{\Hilbert'},\quad
\text{for }\vx,\vy\in\Domain,
\]
then we check that
\[
F(\vx)=K(\vx,\cdot), \quad\text{for }\vx\in\Domain,
\]
and the space $\Hilbert$ is a RKHS with the reproducing kernel $K$.
The representer theorem (\cite[Theorem~5.5]{SteinwartChristmann2008}) ensures that
the optimal solution $s$ of the support vector machine~\eqref{eq:svm-hint-loss-RKHS} can be represented as
\[
s(\vx)=\sum_{k\in\NN_N}c_kK(\vx_k,\vx),\quad\text{for }\vx\in\Domain,
\]
for some suitable parameters $c_1,\ldots,c_N\in\RR$.

In addition, we solve the general support vector machines in a RKHS $\Hilbert$ of functions $f$
defined on a probability space $\Domain$ equipped with a probability measure $\mu$.
%Let $\Leb_0(\Domain)$ be the collection of all measurable functions defined on $\Domain$.
Given the hinge loss function $L$,
our learning goal is to find a target function that approximately achieves the smallest possible risks
\[
\inf_{f\in\Leb_0(\Domain)}\int_{\Domain\times\RR}L(\vx,y,f(\vx))\PP\left(\ud\vx,\ud y\right),
\]
where $\PP$ is a probability measure generating input and output data. The probability $\PP$ can be usually represented by a conditional probability $\PP(\cdot|\vx)$ on $\RR$, that is, $\PP(\vx,y)=\PP(y|\vx)\mu(\vx)$.
It is often too difficult to solve the target function directly.
We then turn to estimating the target function by a support vector machine solution $s$ of a regularized infinite-sample or empirical minimization over the certain RKHS $\Hilbert$, that is,
\[
s=\underset{f\in\Hilbert}{\text{argmin}}
\left\{\int_{\Domain\times\RR}L(\vx,y,f(\vx))\PP\left(\ud\vx,\ud y\right)+\sigma^2\norm{f}_{\Hilbert}^2\right\},
\]
or
\[
s=\underset{f\in\Hilbert}{\text{argmin}}
\left\{\frac{1}{N}\sum_{k\in\NN_N}L\left(\vx_k,y_k,f(\vx_k)\right)+\sigma^2\norm{f}_{\Hilbert}^2\right\},
\]
where the training data $D:=\left\{\left(\vx_k,y_k\right):k\in\NN_N\right\}\subseteq\Domain\times\RR$ are the independent random duplicates of $\PP(\vx,y)$. By the representer theorem (\cite[Theorem~5.5 and~5.8]{SteinwartChristmann2008}), the support vector solution $s$ can be written as a representation by the reproducing kernel $K$ of the RKHS $\Hilbert$. The oracle inequalities for support vector machines provide the approximation errors and the learning rates of support vector machine solutions (see \cite[Chapter~6]{SteinwartChristmann2008}).

The recent papers~\cite{ZhangXuZhang2009,FasshauerHickernellYe2013} discussed the support vector machines in some semi-inner-product Banach space $\Banach$ composed of functions $f:\Domain\to\RR$. The inner product $(\cdot,\cdot)_{\Hilbert}$ of the Hilbert space $\Hilbert$ is extended into the semi-inner product $[\cdot,\cdot]_{\Banach}$ of the semi-inner-product Banach space $\Banach$ in order that optimization problem~\eqref{eq:svm-hint-loss-RKHS} is generalized to
\begin{equation}\label{eq:svm-hint-loss-RKBS-semi}
\min_{f\in\Banach}
\left\{\frac{1}{N}\sum_{k\in\NN_N}L\left(\vx_k,y_k,\left[f,F(\vx_k)\right]_{\Banach}\right)+\sigma^2\norm{f}_{\Banach}^2\right\},
\end{equation}
where $F$ becomes the related feature map from $\Domain$ into $\Banach$ such that $[f,F(\vx)]_{\Banach}=f(\vx)$. This indicates that the dual element $\left(F(\vx)\right)^{\ast}$ of $F(\vx)$ is identical element of $\delta_{\vx}\in\Banach'$, because
$$
\langle f,\left(F(\vx)\right)^{\ast}\rangle_{\Banach}=[f,F(\vx)]_{\Banach}=f(\vx)=\langle f,\delta_{\vx} \rangle_{\Banach}.
$$
For the construction of its reproducing kernel,
we suppose that the dual space $\Banach'$ of $\Banach$ is isometrically isomorphic onto a normed space $\Fun$ composed of functions $g:\Domain\to\RR$.
Since the semi-inner product is nonsymmetric, we need another adjoint feature map $F^{\ast}:\Domain\to\Banach'$ in order that
the dual element $\left(F^{\ast}(\vy)\right)^{\ast}$ of $F^{\ast}(\vy)$ is the identical element of $\delta_{\vy}\in\Banach''$,
that is,
$$
\langle g,\left(F^{\ast}(\vy)\right)^{\ast} \rangle_{\Banach'}=[g,F^{\ast}(\vy)]_{\Banach'}=g(\vy)=\langle g,\delta_{\vy} \rangle_{\Banach'}.
$$
If the above feature maps $F$ and $F^{\ast}$ both exist, then we define the reproducing kernel
\[
K(\vx,\vy):=\left[\left(F^{\ast}(\vy)\right)^{\ast},F(\vx)\right]_{\Banach}
=\left[\left(F(\vx)\right)^{\ast},F^{\ast}(\vy)\right]_{\Banach'}
,\quad\text{for }\vx,\vy\in\Domain,
\]
in order that
\[
F(\vx)=\left(K(\vx,\cdot)\right)^{\ast},\quad F^{\ast}(\vy)=\left(K(\cdot,\vy)\right)^{\ast},\quad
\text{for }\vx,\vy\in\Domain,
\]
which ensures that the space $\Banach$ is a semi-inner-product RKBS with the semi-inner-product reproducing kernel $K$.
By the representer theorem of empirical data in semi-inner-product RKBSs (\cite[Theorem~4.2]{FasshauerHickernellYe2013}), we know that the dual element $s^{\ast}$ of the support vector machine solution $s$ of optimization problem~\eqref{eq:svm-hint-loss-RKBS-semi} is the linear combination of $K(\vx_1,\cdot),\ldots,K(\vx_N,\cdot)$. When the RKBS $\Banach$ is constructed by the positive definite function, we recover the finite dimensional formula of the support vector machine solution $s$ by its dual element $s^{\ast}$ (see the details given in \cite{FasshauerHickernellYe2013}).

In the previous construction of RKBSs by using the semi-inner products, it was required that the RKBSs are reflexive, uniformly convex and Fr\'{e}chet differentiable so that the dual elements are all well-defined. In the following sections, we improve the original theoretical results of RKBSs and employ the dual bilinear product $\langle \cdot,\cdot \rangle_{\Banach}$ to replace the semi-inner product $[\cdot,\cdot]_{\Banach}$ in order to obtain the support vector machines in a general Banach space $\Banach$.

%------------------------------------------------------------------------------------------------------------------------
\section{Support Vector Machines in $p$-norm Reproducing Kernel Banach Spaces for $1<p<\infty$}\label{s:SVM-p-RKBS}
%------------------------------------------------------------------------------------------------------------------------
\sectionmark{Solving Support Vector Machines in $p$-norm RKBS for $1<p<\infty$}

The Riesz representation theorem~\cite{Riesz1907,Riesz1909} provides the finite dimensional representations of the support vector machine solutions in RKHSs for their numerical implementations.
In this section,
we shall show that the infinite dimensional support vector machines in the $p$-norm RKBSs can be equivalently transferred into the finite dimensional convex optimization problems by Theorem~\ref{t:RKBS-opt-rep} (representer theorem for learning methods in RKBSs).
This ensures that we employ the finite suitable parameters to reconstruct the support vector machine solutions in the $p$-norm RKBSs.

Let $1<p,q<\infty$ such that $p^{-1}+q^{-1}=1$.
By Theorems~\ref{t:RKBS-MercerKer-p} and \ref{t:RKBS-MercerKer-1}, the two-sided
RKBSs $\Banach_K^p(\Domain)$ and $\Banach_{K'}^q(\Domain')$ are well-defined by the left-sided and right-sided expansion sets $\Sset_{K}=\left\{\phi_n:n\in\NN\right\}$ and $\Sset_{K}'=\left\{\adjphi_n:n\in\NN\right\}$ of the generalized Mercer kernel $K\in\Leb_0(\Domain\times\Domain')$.
When the loss function $L$ and the regularization function $R$ satisfy assumption~(A-ELR) given in Section~\ref{s:RepThm}, the representer theorem (Theorem~\ref{t:RKBS-opt-rep}) guarantees the global minimizer of the regularized empirical risks over $\Banach_K^p(\Domain)$
\begin{equation}\label{eq:svm-p-RKBS}
\svm_p(f):=
\frac{1}{N}\sum_{k\in\NN_N}L(\vx_k,y_k,f(\vx_k))+R\left(\norm{f}_{\Banach_{K}^p(\Domain)}\right),
\quad\text{for }f\in\Banach_{K}^p(\Domain),
\end{equation}
for the given pairwise distinct data points $X=\left\{\vx_k:k\in\NN_N\right\}\subseteq\Domain$ and the related observing values $Y=\left\{y_k:k\in\NN_N\right\}\subseteq\RR$.
Now we study whether the support vector machine solutions can also be represented by the expansion sets with the finite suitable parameters.

%/////////////////////////////////////////////////////////////////////////////////////////////////////////////////////
\begin{remark}
In this section, the support vector machines can be computed by not only the hinge loss function but also general loss functions such as least square loss, $p$th power absolute distance loss, logistic loss, Huber loss, pinball loss, and $\epsilon$-insensitive loss.
The regularization function can also be endowed with various forms, for example, $R(r):=\sigma r^{m}$ with $\sigma>0$ and $m\geq1$.
\end{remark}
%/////////////////////////////////////////////////////////////////////////////////////////////////////////////////////

In the solving processes, Theorem~\ref{t:RKBS-opt-rep} is employed. Hence, we need additional condition of the linear independence of the right-sided kernel set $\Kset_{K}'$ of $K$.
Special kernels always ensure the linear independence everywhere, for example, the continuous symmetric positive definite kernels.
In Section~\ref{s:RepThm}, we require $\Kset_{K}'$ to be linearly independent such that arbitrary data points $X$ are well-posed in Theorem~\ref{t:RKBS-opt-rep}. But, if the data points $X$ are already given in practical implementations, then we only need to check the linear independence of $\left\{K(\vx_k,\cdot):k\in\NN_N\right\}$ that guarantees the theoretical results of Theorem~\ref{t:RKBS-opt-rep}.

%/////////////////////////////////////////////////////////////////////////////////////////////////////////////////////
\begin{theorem}\label{t:RKBS-MercerKer-svm-rep-pq}
Given any pairwise distinct data points $X\subseteq\Domain$ and any associated data values $Y\subseteq\RR$,
the regularized empirical risk $\svm_p$ is defined as in equation~\eqref{eq:svm-p-RKBS}.
Let $1<p,q<\infty$ such that $p^{-1}+q^{-1}=1$,
and let $K\in\Leb_0(\Domain\times\Domain')$ be a generalized Mercer kernel such that the expansion sets $\Sset_K$ and $\Sset_K'$ of $K$ satisfy assumption (A-$p$) and
the right-sided kernel set $\Kset_K'$ is linearly independent.
If the loss function $L$ and the regularization function $R$ satisfy assumption~(A-ELR),
then
the support vector machines
\[
\min_{f\in\Banach_{K}^p(\Domain)}\svm_p(f),
\]
has a unique optimal solution $s_p$, which has the representation
\begin{equation}\label{eq:svm-coef-pq}
s_p=\sum_{n\in\NN}\left(\sum_{j\in\NN_N}c_{p,j}\phi_n(\vx_j)\abs{\sum_{k\in\NN_N}c_{p,k}\phi_n(\vx_k)}^{q-2}\right)\phi_n,
\end{equation}
for some suitable parameters $c_{p,1},\ldots,c_{p,N}\in\RR$ and whose norm can be written as
\[
\norm{s_p}_{\Banach_{K}^p(\Domain)}
=\left(\sum_{n\in\NN}\abs{\sum_{k\in\NN_N}c_{p,k}\phi_n(\vx_k)}^{q}\right)^{1/p}.
\]
\end{theorem}
%/////////////////////////////////////////////////////////////////////////////////////////////////////////////////////
\begin{proof}
The main idea of the proof is to recover $s_p$ by using the G\^{a}teaux derivative $\GateauxNorm(s_p)$ of the $p$-norm of the two-sided RKBS $\Banach_{K}^p(\Domain)$.

The isometrical isomorphisms given in Proposition~\ref{p:RKBS-MecerKer-pq} preserve
the geometrical structures of $\Banach_{K}^p(\Domain)\cong \lp$ and $\Banach_{\adjK}^q(\Domain')\cong \lq$.
This ensures that the two-sided RKBS $\Banach_{K}^p(\Domain)$ is reflexive, strictly convex, and smooth, because $1<p<\infty$ implies the reflexivity, strict convexity, and smoothness of $\lp$.
Thus, Theorem~\ref{t:RKBS-opt-rep} can be applied to ensure that the support vector machine in $\Banach_{K}^p(\Domain)$ has a unique optimal solution
\[
s_p=\sum_{n\in\NN}a_{p,n}\phi_n\in\Banach_{K}^p(\Domain),
\]
such that its G\^{a}teaux derivative $\GateauxNorm(s_p)$ is a linear combination of
$K(\vx_1,\cdot),\ldots,K(\vx_N,\cdot)$, that is,
\begin{equation}\label{eq:sp-1}
\GateauxNorm(s_p)=\sum_{k\in\NN_N}\beta_kK(\vx_k,\cdot)
\in\Banach_{\adjK}^q(\Domain'),
\end{equation}
for the suitable parameters $\beta_1,\ldots,\beta_N\in\RR$, and its norm can be written as
\begin{equation}\label{eq:sp-1-1}
\norm{s}_{\Banach_{K}^p(\Domain)}=\sum_{k\in\NN_N}\beta_ks_p(\vx_k).
\end{equation}

We complete the proof of the finite dimensional representation of $s_p$ given in equation~\eqref{eq:svm-coef-pq} if there exist the coefficients $c_{p,1},\ldots,c_{p,N}$ such that
\begin{equation}\label{eq:sp-coef}
a_{p,n}=\sum_{j\in\NN_N}c_{p,j}\phi_n(\vx_j)\abs{\sum_{k\in\NN_N}c_{p,k}\phi_n(\vx_k)}^{q-2}.
\end{equation}
When $s_p=0$, the conclusion is clearly true.
Now we suppose that $s_p$ is not a trivial solution.
Since $\va_p:=\left(a_{p,n}:n\in\NN\right)$ is an isometrically equivalent element of $s_p$, we use the G\^{a}teaux derivative $\GateauxNorm(\va_p)$ of the map $\va\mapsto\norm{\va}_p$ at $\va_p$ to compute the G\^{a}teaux derivative $\GateauxNorm(s_p)$ of the norm of $\Banach_{K}^p(\Domain)$ at $s_p$.
Specifically, because $\va_p\neq0$ and $1<p<\infty$ (see the examples in \cite{James1947,Giles1967}), we have that
\[
\GateauxNorm(\va_p)=\left(\frac{a_{p,n}\abs{a_{p,n}}^{p-2}}{\norm{\va_p}_{p}^{p-1}}: n\in\NN\right)\in \lq.
\]
Then we obtain the following formula of the G\^{a}teaux derivative
\begin{equation}\label{eq:sp-2}
\GateauxNorm(s_p)=\sum_{n\in\NN}\frac{a_{p,n}\abs{a_{p,n}}^{p-2}}{\norm{\va_{p}}_{p}^{p-1}}\adjphi_n\in\Banach_{\adjK}^q(\Domain'),
\end{equation}
where $\va_p:=\left(a_{p,n}:n\in\NN\right)$.
Using the expansion
$$
K(\vx,\vy)=\sum_{n\in\NN}\phi_n(\vx)\adjphi_n(\vy),
$$
equation~\eqref{eq:sp-1} can be rewritten as
\begin{equation}\label{eq:sp-3}
\GateauxNorm(s_p)=\sum_{k\in\NN_N}\beta_k\sum_{n\in\NN}\phi_n(\vx_k)\adjphi_n
=\sum_{n\in\NN}\left(\sum_{k\in\NN_N}\beta_k\phi_n(\vx_k)\right)\adjphi_n.
\end{equation}
Comparing equations~\eqref{eq:sp-2} and~\eqref{eq:sp-3}, we have that
\[
\frac{a_{p,n}\abs{a_{p,n}}^{p-2}}{\norm{\va_{p}}_{p}^{p-1}}=\left(\sum_{k\in\NN_N}\beta_k\phi_n(\vx_k)\right),
\quad
\text{for all }n\in\NN.
\]
Hence, we choose the parameters
\begin{equation}\label{eq:sp-4-0}
c_{p,k}:=\norm{\va_{p}}_{p}^{p-1}\beta_k,
\quad
\text{for all }k\in\NN_N,
\end{equation}
to obtain equation~\eqref{eq:sp-coef}. In this way, the representation of $s_p$ is achieved.

Combining equations~\eqref{eq:sp-1-1} and~\eqref{eq:sp-4-0}
we find that
\begin{equation}\label{eq:sp-4}
\norm{s_p}_{\Banach_{K}^p(\Domain)}=\sum_{k\in\NN_N}\beta_ks_p(\vx_k)
=\sum_{k\in\NN_N}\norm{\va_{p}}_{p}^{1-p}c_{p,k}s_p(\vx_k).
\end{equation}
Substituting the representation of $s_p$ given in equation~\eqref{eq:svm-coef-pq} and $\norm{s_p}_{\Banach_{K}^p(\Domain)}=\norm{\va_p}_p$ into equation~\eqref{eq:sp-4}, we obtain the desired formula for
$\norm{s_p}_{\Banach_{K}^p(\Domain)}$.
\end{proof}
%/////////////////////////////////////////////////////////////////////////////////////////////////////////////////////

%////////////////////////////////////////////////////////////////////////////////////////////////////////////////////////
\begin{remark}\label{r:RKBS-MercerKer-svm-rep-pq}
The parameters of the support vector machine solution $s_p$ may not be equal to the coefficients of its G\^{a}teaux derivative $\GateauxNorm(s_p)$.
According to Theorems~\ref{t:RKBS-opt-rep} and~\ref{t:RKBS-MercerKer-svm-rep-pq} about the representations of the support vector machine solutions in RKBSs, the coefficients of $\GateauxNorm(s_p)$ are always linear but the coefficients of $s_p$ could be nonlinear.
\end{remark}
%////////////////////////////////////////////////////////////////////////////////////////////////////////////////////////

\subsection*{Fixed-point Algorithm}

To close this section, we develop a fixed-point algorithm for solving the optimization problem studied in this section. Fixed-point iteration for solving convex optimization problems that come from the proximal algorithm was proposed in \cite{MicchelliShenXu2011,ParikhBoyd2014}. Specifically, we apply the techniques of \cite[Corollary~5.5]{FasshauerHickernellYe2013} to compute the finite parameters $c_{p,1},\ldots,c_{p,N}$ of the support vector machine solution $s_p$ by a fixed-point iteration method.
Choosing any parameter vector $\vc=\left(c_k:k\in\NN_N\right)\in\RR^N$, we set up a function $f_{\vc}\in\Banach_{K}^p(\Domain)$ depending on equation~\eqref{eq:svm-coef-pq}, that is,
\[
f_{\vc}(\vx)=\sum_{n\in\NN}\left(\sum_{j\in\NN_N}c_{j}\phi_n(\vx_j)
\abs{\sum_{k\in\NN_N}c_{k}\phi_n(\vx_k)}^{q-2}\right)\phi_n(\vx),
\quad\text{for }\vx\in\Domain.
\]
Clearly, $s_p=f_{\vc_p}$.
Define the vector-valued function
\[
\veta:=\left(\eta_k:k\in\NN_N\right): \RR^N\to\RR^N,
\]
where
\[
\eta_k(\vc):=f_{\vc}(\vx_k),
\quad\text{for }k\in\NN_N.
\]
Thus, the regularized empirical risk of $f_{\vc}$ can be computed by
\[
T_p(\vc):=\svm_p\left(f_{\vc}\right)
=\frac{1}{N}\sum_{k\in\NN_N}L(\vx_k,y_k,\eta_k(\vc))+R\left(\left(\veta(\vc)^T\vc\right)^{1/p}\right),
\]
and the parameters $\vc_p$ of the support vector machine solution $s_p$ is the global minimum solution of
\[
\min_{\vc\in\RR^N}T_p(\vc).
\]
If $R\in\Cont^{1}[0,\infty)$ and $t\mapsto L(\vx,y,t)\in\Cont^{1}(\RR)$ for any fixed $\vx\in\Domain$ and any $y\in\RR$, then
the gradient of $T_p$ has the form
\[
\nabla T_p(\vc)=\frac{1}{N}
\nabla\veta(\vc)^T\vl_t\left(\vc\right)+
\alpha(\vc)
\left(\veta(\vc)+\nabla\veta(\vc)^T\vc\right),
\]
where
\[
\veta_t(\vc):=\left(L_t\left(\vx_k,y_k,\eta_{k}(\vc)\right): k\in\NN_N\right),
\]
and
\[
\alpha(\vc):=
\begin{cases}
p^{-1}\left(\veta(\vc)^T\vc\right)^{-1/q}
R_r\left(\left(\veta(\vc)^T\vc\right)^{1/p}\right),
&\text{when }\vc\in\RR^N\backslash\{\v0\},\\
0,&\text{when }\vc=\v0.
\end{cases}
\]
Here, the notations of $L_t$ and $R_r$ represent the differentiations of $L$ and $R$ at $t$ and $r$, respectively, that is,
$L_t(\vx,y,t):=\frac{\ud}{\ud t}L(\vx,y,t)$ and $R_r(r):=\frac{\ud}{\ud r}R(r)$,
and the element
$$
\nabla\veta(\vc):=\left(\frac{\partial}{\partial c_k}\eta_j(\vc):j,k\in\NN_N\right)
$$
is a Jacobian matrix with the entries
\[
\frac{\partial}{\partial c_k}\eta_j(\vc)
=(q-1)\sum_{n\in\NN}\phi_n(\vx_j)\phi_n(\vx_k)\abs{\sum_{l\in\NN_N}c_l\phi_n(\vx_l)}^{q-2}.
\]
Since the global minimizer $\vc_p$ is a stationary point of $\nabla T_p$, the parameter vector $\vc_p$ is a fixed-point of the function
\begin{equation}\label{eq:fix-point-coef-opt-rep-pq}
F_p(\vc):=\vc+\nabla T_p(\vc),\quad\text{for }\vc\in\RR^N.
\end{equation}

%/////////////////////////////////////////////////////////////////////////////////////////////////////////////////////
\begin{corollary}\label{c:coef-RKBS-MercerKer-svm-rep-pq}
If the loss function $L$ and the regularization function $R$ satisfy that assumption (A-ELR), $t\mapsto L(\vx,y,t)\in\Cont^{1}(\RR)$ for any fixed $\vx\in\Domain$ and any $y\in\RR$, and $R\in\Cont^{1}[0,\infty)$, then the parameter vector $\vc_p=\left(c_{p,k}:k\in\NN_N\right)$ of the support vector machine solution $s_p$ defined in equation~\eqref{eq:svm-coef-pq} is a fixed-point of the map $F_p$ defined in equation~\eqref{eq:fix-point-coef-opt-rep-pq}, that is, $F_p\left(\vc_p\right)=\vc_p$.
\end{corollary}
%/////////////////////////////////////////////////////////////////////////////////////////////////////////////////////

Corollary~\ref{c:coef-RKBS-MercerKer-svm-rep-pq} ensures that the finite coefficients of the support vector machine solutions in RKBSs
can be computed by the fixed-point iteration algorithm (the proximal algorithm).
This indicates that the support vector machine in RKBSs is of easy implementation the same as the classical methods in RKHSs.

%------------------------------------------------------------------------------------------------------------------------
\section{Support Vector Machines in $1$-norm Reproducing Kernel Banach Spaces}\label{s:SVM-1-RKBS}
%------------------------------------------------------------------------------------------------------------------------
\sectionmark{Solving Support Vector Machines in $1$-norm RKBS}

Now we show that the support vector machines in the $1$-norm RKBS can be approximated by the support vector machines in the $p$-norm RKBSs as $p\to1$.

In this section, we suppose that the left-sided and right-sided expansion sets $\Sset_{K}=\left\{\phi_n:n\in\NN\right\}$ and $\Sset_{K}'=\left\{\adjphi_n:n\in\NN\right\}$ of the generalized Mercer kernel $K\in\Leb_0(\Domain\times\Domain')$ satisfy assumption (A-$1*$) and the right-sided kernel set $\Kset_K'$ of $K$ is linearly independent. As the discussions in inequalities (\ref{A1-to-Ap-1})-(\ref{A1-to-Ap-2}), we know that assumption (A-$1*$) implies assumption (A-$p$) for all $1<p<\infty$.
Theorems~\ref{t:RKBS-MercerKer-p} and~\ref{t:RKBS-MercerKer-1} show that the generalized Mercer kernel $K$ is the reproducing kernel of the right-sided RKBS $\Banach_{K}^{1}(\Domain)$ and the two-sided RKBS $\Banach_{K}^{p}(\Domain)$ for $1<p<\infty$ induced by $\Sset_K$ (see Sections~\ref{s:p-RKBS} and~\ref{s:1-RKBS} for more details).

For the constructions of the support vector machines in the $1$-norm RKBSs,
we moreover suppose that the loss function $L$ and the regularization function $R$ satisfy assumption~(A-ELR) given in Section~\ref{s:RepThm}.
For the given pairwise distinct data points
$X\subseteq\Domain$ and the associated data values $Y\subseteq\RR$,
we set up the regularized empirical risks in $\Banach_{K}^{1}(\Domain)$
\begin{equation}\label{eq:svm-1-RKBS}
\svm_1(f):=
\frac{1}{N}\sum_{k\in\NN_N}L(\vx_k,y_k,f(\vx_k))+R\left(\norm{f}_{\Banach_{K}^1(\Domain)}\right),
\quad\text{for }f\in\Banach_{K}^1(\Domain).
\end{equation}
We wish to know whether there is a global minimum of $\svm_1$ over $\Banach_{K}^1(\Domain)$
\begin{equation}\label{eq:svm-1-RKBS-opt}
\min_{f\in\Banach_{K}^1(\Domain)}\svm_1(f).
\end{equation}

Since $\Banach_{K}^1(\Domain)$ is not reflexive, strictly convex or smooth, we can not solve the support vector machines in $\Banach_{K}^1(\Domain)$ directly by Theorem~\ref{t:RKBS-opt-rep}.
According to Theorem~\ref{t:RKBS-MercerKer-svm-rep-pq}, however, we solve the uniquely global minimizer $s_p$ of the support vector machines in $\Banach_{K}^p(\Domain)$ for $1<p\leq2$, or more precisely,
\begin{equation}\label{eq:svm-p-RKBS-opt}
s_p=\underset{f\in\Banach_{K}^p(\Domain)}{\text{argmin}}\svm_p(f)
=\underset{f\in\Banach_{K}^p(\Domain)}{\text{argmin}}
\left\{\frac{1}{N}\sum_{k\in\NN_N}L(\vx_k,y_k,f(\vx_k))+R\left(\norm{f}_{\Banach_{K}^p(\Domain)}\right)\right\}.
\end{equation}
The parameters of the infinite basis of
\[
s_p=\sum_{n\in\NN}a_{p,n}\phi_n\in\Banach_{K}^p(\Domain),
\]
can be represented in the form
\[
a_{p,n}:=\sum_{j\in\NN_N}c_{p,j}\phi_n(\vx_j)\abs{\sum_{k\in\NN_N}c_{p,k}\phi_n(\vx_k)}^{q-2},
\quad\text{for all }n\in\NN,
\]
by the finite suitable parameters $c_{p,1},\ldots,c_{p,N}\in\RR$, where $p^{-1}+q^{-1}=1$.
Note that $2\leq q<\infty$ since $1<p\leq 2$.

Now we verify that $s_p$ can be used to approximate $s_1$ when $p\to1$.
Before we show this in Theorem~\ref{t:opt-svm-s1}, we establish four preliminary lemmas.
We first prove that $s_p\in\Banach_{K}^p(\Domain)$ for all $1<p\leq 2$.

%/////////////////////////////////////////////////////////////////////////////////////////////////////////////////////
\begin{lemma}\label{l:sp-in-1-RKBS}
If the support vector machines in $\Banach_{K}^1(\Domain)$ and $\Banach_{K}^p(\Domain)$ are defined in equations~\eqref{eq:svm-1-RKBS-opt} and~\eqref{eq:svm-p-RKBS-opt}, respectively, then
the support vector machine solution $s_p$ belongs to $\Banach_{K}^1(\Domain)$ for each $1<p\leq2$.
\end{lemma}
%/////////////////////////////////////////////////////////////////////////////////////////////////////////////////////
\begin{proof}
It suffices to prove that series
\begin{equation}\label{eq:sp-in-B1-0}
\sum_{n\in\NN}\abs{a_{p,n}}
=\sum_{n\in\NN}\abs{\sum_{k\in\NN_N}c_{p,k}\phi_n(\vx_k)}^{q-1}
\end{equation}
is convergent.
By the Cauchy-Schwarz inequality, we have for $n\in\NN$ that
\begin{equation}\label{eq:sp-in-B1-1}
\abs{\sum_{k\in\NN_N}c_{p,k}\phi_n(\vx_k)}^{q-1}
\leq \left(\sup_{j\in\NN_N}\abs{c_{p,j}}^{q-1}\right)
\left(\sum_{k\in\NN_N}\abs{\phi_n(\vx_k)}\right)^{q-1}.
\end{equation}
We introduce the vector $\vc_{p}:=\left(c_{p,k}:k\in\NN_N\right)$. Since
\[
\sup_{j\in\NN_N}\abs{c_{p,j}}^{q-1}\leq\norm{\vc_p}_{\infty}^{q-1},
\]
and
\[
\left(\sum_{k\in\NN_N}\abs{\phi_n(\vx_k)}\right)^{q-1}
\leq N^{q-2}\sum_{k\in\NN_N}\abs{\phi_n(\vx_k)}^{q-1},
\quad\text{for }n\in\NN,
\]
from inequality~\eqref{eq:sp-in-B1-1}, we obtain that
\begin{equation}\label{eq:sp-in-B1-2}
\abs{\sum_{k\in\NN_N}c_{p,k}\phi_n(\vx_k)}^{q-1}
\leq \norm{\vc_p}_{\infty}^{q-1}N^{q-2}\sum_{k\in\NN_N}\abs{\phi_n(\vx_k)}^{q-1},
\quad\text{for }n\in\NN.
\end{equation}
Moreover, conditions (C-$1*$) ensure that
\begin{equation}\label{eq:sp-in-B1-3}
\sum_{n\in\NN}\abs{\phi_n(\vx_k)}^{q-1}\leq C_{X}^{q-1},\quad
\text{for }k\in\NN_N,
\end{equation}
where the positive constant $C_X$ is given by
\[
C_{X}:=\sup_{k\in\NN_N}\sum_{n\in\NN}\abs{\phi_n(\vx_k)}.
\]
Combining inequalities~\eqref{eq:sp-in-B1-2} and~\eqref{eq:sp-in-B1-3},
we obtain that
\begin{equation}\label{eq:sp-in-B1-4}
\sum_{n\in\NN}\abs{\sum_{k\in\NN_N}c_{p,k}\phi_n(\vx_k)}^{q-1}
\leq C_{X}^{q-1}N^{q-2}\norm{\vc_p}_{\infty}^{q-1}<\infty,
\end{equation}
proving convergence of series \eqref{eq:sp-in-B1-0}.
\end{proof}
%/////////////////////////////////////////////////////////////////////////////////////////////////////////////////////

Since $s_p$ is the global minimum solution of the regularized empirical risks over $\Banach_{K}^p(\Domain)$, we have that
\[
\svm_p\left(s_p\right)=\min_{f\in\Banach_{K}^p(\Domain)}\svm_p(f).
\]
Next, we show that the minimum of $1$-norm regularized empirical risks is the upper bound of $\svm_p\left(s_p\right)$.

%/////////////////////////////////////////////////////////////////////////////////////////////////////////////////////
\begin{lemma}\label{l:bound-1-svm}
If the support vector machines in $\Banach_{K}^1(\Domain)$ and $\Banach_{K}^p(\Domain)$ are defined in equations~\eqref{eq:svm-1-RKBS-opt} and~\eqref{eq:svm-p-RKBS-opt}, respectively,
then
\[
\lim_{p\to1}\svm_p(s_p)\leq\min_{f\in\Banach_{K}^1(\Domain)}\svm_1(f).
\]
\end{lemma}
%/////////////////////////////////////////////////////////////////////////////////////////////////////////////////////
\begin{proof}
Suppose that $1\leq p_1\leq p_2\leq2$. The proof of Propositions~\ref{p:RKBS-MecerKer-pq} and~\ref{p:RKBS-MecerKer-1} show that $\Banach_{K}^{p_1}(\Domain)\cong \lspace_{p_1}$ and
$\Banach_{K}^{p_2}(\Domain)\cong \lspace_{p_2}$. Hence, by the imbedding of various norms $\norm{\cdot}_p$, we have that
\[
\Banach_{K}^{p_1}(\Domain)\subseteq\Banach_{K}^{p_2}(\Domain),
\]
and
\[
\norm{f}_{\Banach_{K}^{p_2}(\Domain)}=\norm{\va}_{p_2}
\leq\norm{\va}_{p_1}=\norm{f}_{\Banach_{K}^{p_1}(\Domain)},
\quad\text{for }f=\sum_{n\in\NN}a_n\phi_n\in\Banach_{K}^{p_1}(\Domain).
\]
This ensures that
\begin{equation}\label{eq:bound-1-svm-1}
\min_{f\in\Banach_{K}^{p_2}(\Domain)}\svm_{p_2}(f)\leq\min_{f\in\Banach_{K}^{p_1}(\Domain)}\svm_{p_2}(f).
\end{equation}
Moreover, since $R$ is strictly increasing, we observe that
\[
R\left(\norm{f}_{\Banach_{K}^{p_2}(\Domain)}\right)\leq R\left(\norm{f}_{\Banach_{K}^{p_1}(\Domain)}\right),
\quad \text{for }f\in\Banach_{K}^{p_1}(\Domain),
\]
and then
\[
\svm_{p_2}(f)\leq\svm_{p_1}(f),\quad
\text{for }f\in\Banach_{K}^{p_1}(\Domain).
\]
Thus, we conclude that
\begin{equation}\label{eq:bound-1-svm-2}
\min_{f\in\Banach_{K}^{p_1}(\Domain)}\svm_{p_2}(f)\leq\min_{f\in\Banach_{K}^{p_1}(\Domain)}\svm_{p_1}(f).
\end{equation}
Combining inequalities~\eqref{eq:bound-1-svm-1} and~\eqref{eq:bound-1-svm-2} we have that
\begin{equation}\label{eq:bound-1-svm-3}
\min_{f\in\Banach_{K}^{p_2}(\Domain)}\svm_{p_2}(f)\leq\min_{f\in\Banach_{K}^{p_1}(\Domain)}\svm_{p_1}(f).
\end{equation}
This gives that
\[
\svm_p\left(s_p\right)=\min_{f\in\Banach_{K}^p(\Domain)}\svm_p(f)
\leq\min_{f\in\Banach_{K}^1(\Domain)}\svm_1(f),
\quad\text{for all }1<p\leq2.
\]
Moreover, the function $p\mapsto \svm_p\left(s_p\right)$ is decreasing on $(1,2]$.
Therefore, the limit of $\lim_{p\to1}\svm_p\left(s_p\right)$ exists and $\min_{f\in\Banach_{K}^1(\Domain)}\svm_1(f)$ is its upper bound.
\end{proof}
%/////////////////////////////////////////////////////////////////////////////////////////////////////////////////////

Lemma~\ref{l:sp-in-1-RKBS} has already shown that $\left\{s_{p}:1<p\leq2\right\}\subseteq\Banach_{K}^{1}(\Domain)$.
Now, we show that $\left\{s_{p}:1<p\leq2\right\}$ can be used to catch a global minimizer $s_1$ of $\svm_1$ over $\Banach_{K}^{1}(\Domain)$.
By Proposition~\ref{p:RKBS-MecerKer-1}, $\Banach_{K}^1(\Domain)$ is isometrically equivalent to the dual space of $\Banach_{\adjK}^{\infty}(\Domain')$. Thus, we obtain the weak* topology of $\Banach_{K}^1(\Domain)$ to describe the relationships between $\left\{s_{p}:1<p\leq2\right\}$ and $s_1$.

%/////////////////////////////////////////////////////////////////////////////////////////////////////////////////////
\begin{lemma}\label{l:opt-svm-s1}
Let the support vector machines in $\Banach_{K}^1(\Domain)$ and $\Banach_{K}^p(\Domain)$ be defined in equations~\eqref{eq:svm-1-RKBS-opt} and~\eqref{eq:svm-p-RKBS-opt}, respectively.
If there exist an element $s_1\in\Banach_{K}^{1}(\Domain)$ and a countable sequence $s_{p_m}$, with $\lim_{m\to\infty}p_m=1$, of the support vector machine solutions $\left\{s_{p}:1<p\leq2\right\}$ such that
\[
s_{p_m}\overset{\text{weak*}-\Banach_{\adjK}^{\infty}(\Domain')}{\longrightarrow}s_1,\quad \text{when }p_m\to1,
\]
then
\[
\svm_1\left(s_1\right)=\min_{f\in\Banach_{K}^1(\Domain)}\svm_1(f),
\]
and
\[
\lim_{p_m\to1}\svm_{p_m}\left(s_{p_m}\right)=\svm_1\left(s_1\right).
\]
\end{lemma}
%/////////////////////////////////////////////////////////////////////////////////////////////////////////////////////
\begin{proof}
Theorem~\ref{t:RKBS-MercerKer-infty} shows that $\Banach_{\adjK}^{\infty}(\Domain')$ is a two-sided RKBS and its two-sided reproducing kernel is the adjoint kernel $\adjK$ of $K$.
The main technique used in the proof is to employ the two-sided reproducing properties of $\Banach_{\adjK}^{\infty}(\Domain')$.

According to Proposition~\ref{p:weaklystar-conv-RKBS} the weak* convergence implies the pointwise convergence in $\Banach_{K}^1(\Domain)$. Hence,
\[
\lim_{m\to\infty}s_{p_m}(\vx)=s_1(\vx),\quad\text{for all }\vx\in\Domain.
\]
Since for any fixed $\vx\in\Domain$ and $y\in\RR$, $L(\vx,y,\cdot)$ is convex, it is continuous.
It follows that
\begin{equation}\label{eq:opt-svm-s1-1}
\lim_{m\to\infty}\frac{1}{N}\sum_{k\in\NN_N}L\left(\vx_k,y_k,s_{p_m}(\vx_k)\right)
=\frac{1}{N}\sum_{k\in\NN_N}L\left(\vx_k,y_k,s_{1}(\vx_k)\right).
\end{equation}
Since $R$ is a strictly increasing function and
\[
\norm{s_1}_{\Banach_{K}^1(\Domain)}
=\norm{\text{weak*}-\lim_{m\to\infty}s_{p_m}}_{\Banach_{K}^1(\Domain)}
\leq\liminf_{m\to\infty}\norm{s_{p_m}}_{\Banach_{K}^1(\Domain)},
\]
by the weakly* lower semi-continuity (\cite[Theorem~2.6.14]{Megginson1998}),
we have that
\begin{equation}\label{eq:opt-svm-s1-2}
R\left(\norm{s_1}_{\Banach_{K}^1(\Domain)}\right)
\leq R\left(\norm{s_{p_m}}_{\Banach_{K}^1(\Domain)}\right).
\end{equation}
Combining equation~\eqref{eq:opt-svm-s1-1} and inequality~\eqref{eq:opt-svm-s1-2}, we obtain that
\begin{equation}\label{eq:opt-svm-s1-3}
\svm_1\left(s_1\right)
\leq\liminf_{m\to\infty}
\svm_{p_m}\left(s_{p_m}\right).
\end{equation}
Since $\lim_{m\to\infty}p_m=1$,
Lemma~\ref{l:bound-1-svm} leads to the inequality
\begin{equation}\label{eq:opt-svm-s1-4}
\lim_{m\to\infty}\svm_{p_m}\left(s_{p_m}\right)\leq\min_{f\in\Banach_{K}^1(\Domain)}\svm_1(f).
\end{equation}
Comparing inequalities~\eqref{eq:opt-svm-s1-3} and~\eqref{eq:opt-svm-s1-4}, we conclude that
\[
\svm_1\left(s_1\right)\leq\lim_{m\to\infty}\svm_{p_m}\left(s_{p_m}\right)
\leq\min_{f\in\Banach_{K}^1(\Domain)}\svm_1(f),
\]
which implies that
\[
\svm_1\left(s_1\right)=\lim_{m\to\infty}\svm_{p_m}\left(s_{p_m}\right)
=\min_{f\in\Banach_{K}^1(\Domain)}\svm_1(f).
\]
\end{proof}
%/////////////////////////////////////////////////////////////////////////////////////////////////////////////////////

It is desirable to know the upper bound of $\norm{s_p}_{\Banach_{K}^{1}(\Domain)}$ as $p\to1$. To this end, we let
\[
C_0:=\frac{1}{N}\sum_{k\in\NN_N}L(\vx_k,y_k,0)+R(0)<\infty.
\]
Since
\[
R\left(\norm{f}_{\Banach_{K}^{p}(\Domain)}\right)\leq
\svm_{p}\left(s_p\right)
\leq\svm_p(0)=C_0,\quad\text{for }1<p\leq2,
\]
we have that
\begin{equation}\label{eq:sp-bound-1-norm-0}
\norm{s_p}_{\Banach_{K}^{p}(\Domain)}
\leq R^{-1}\left(C_0\right),
\quad\text{for }1<p\leq2,
\end{equation}
where $R^{-1}$ is the inverse function of the strictly increasing function $R$. Thus, the set
$\left\{\norm{s_p}_{\Banach_{K}^{p}(\Domain)}:1<p\leq2\right\}$ is bounded.
Lemma~\ref{l:sp-in-1-RKBS} shows that $\left\{s_p:1<p\leq2\right\}$ is a set of $\Banach_{K}^1(\Domain)$.

%/////////////////////////////////////////////////////////////////////////////////////////////////////////////////////
\begin{lemma}\label{l:sp-bound-1-norm}
If the support vector machines in $\Banach_{K}^1(\Domain)$ and $\Banach_{K}^p(\Domain)$ are defined in equations~\eqref{eq:svm-1-RKBS-opt} and~\eqref{eq:svm-p-RKBS-opt}, respectively,
then
\[
\limsup_{p\to1}\norm{s_p}_{\Banach_{K}^{1}(\Domain)}<+\infty.
\]
\end{lemma}
%/////////////////////////////////////////////////////////////////////////////////////////////////////////////////////
\begin{proof}
We verify this result by using the definition of the norms of bounded linear functionals.

According to Proposition~\ref{p:RKBS-MecerKer-1} and Theorem~\ref{t:RKBS-MercerKer-infty}, the space $\Banach_{\adjK}^{\infty}(\Domain')$ is a two-sided RKBS with the two-sided reproducing kernel $\adjK$ and $\Banach_{K}^{1}(\Domain)$ is isometrically equivalent to the dual space of $\Banach_{\adjK}^{\infty}(\Domain')$.
Let $1<p\leq2$ and $2\leq q<\infty$ such that $p^{-1}+q^{-1}=1$.
Moreover, Proposition~\ref{p:RKBS-MecerKer-pq} and Theorem~\ref{t:RKBS-MercerKer-q} provide that the space $\Banach_{\adjK}^{q}(\Domain')$ is a two-sided RKBS with the two-sided reproducing kernel $\adjK$ and $\Banach_{K}^{p}(\Domain)$ is isometrically equivalent to the dual space of $\Banach_{\adjK}^{q}(\Domain')$.

For $g\in\Span\left\{\Kset_{\adjK}\right\}=\Span\left\{\Kset_K'\right\}$, we express it in terms of the kernel $K'$ as
$$
g=\sum_{k\in\NN_M}\alpha_k\adjK(\cdot,\vx_k).
$$
Using the left-sided reproducing properties of $\Banach_{\adjK}^{\infty}(\Domain')$ and $\Banach_{\adjK}^{q}(\Domain')$, we have that
\[
\langle g,s_p \rangle_{\Banach_{\adjK}^{\infty}(\Domain')}
=\sum_{k\in\NN_M}\alpha_k\langle \adjK(\cdot,\vx_k),s_p \rangle_{\Banach_{\adjK}^{\infty}(\Domain')}
=\sum_{k\in\NN_M}\alpha_ks_p\left(\vx_k\right),
\]
and
\[
\langle g,s_p \rangle_{\Banach_{\adjK}^{q}(\Domain')}
=\sum_{k\in\NN_M}\alpha_k\langle \adjK(\cdot,\vx_k),s_p \rangle_{\Banach_{\adjK}^{q}(\Domain')}
=\sum_{k\in\NN_M}\alpha_ks_p\left(\vx_k\right).
\]
It follows that
\begin{equation}\label{eq:sp-bound-1-norm-1}
\abs{\langle g,s_p \rangle_{\Banach_{\adjK}^{\infty}(\Domain')}}
=\abs{\langle g,s_p \rangle_{\Banach_{\adjK}^{q}(\Domain')}}
\leq\norm{s_p}_{\Banach_{K}^{p}(\Domain)}
\norm{g}_{\Banach_{\adjK}^{q}(\Domain')}.
\end{equation}
Substituting inequality \eqref{eq:sp-bound-1-norm-0} into inequality~\eqref{eq:sp-bound-1-norm-1}
yields that
\begin{equation}\label{eq:sp-bound-1-norm-2}
\abs{\langle g,s_p \rangle_{\Banach_{\adjK}^{\infty}(\Domain')}}
\leq R^{-1}\left(C_0\right)
\norm{g}_{\Banach_{\adjK}^{q}(\Domain')},
\end{equation}
for $1<p\leq2$ and $2\leq q<\infty$ such that $p^{-1}+q^{-1}=1$.
Using the convergence of $\lq$-norm to $\linfty$-norm, we have that
\[
\lim_{q\to\infty}\norm{g}_{\Banach_{\adjK}^{q}(\Domain')}
=\lim_{q\to\infty}\norm{\vb}_q
=\norm{\vb}_{\infty}
=\norm{g}_{\Banach_{\adjK}^{\infty}(\Domain')},
\]
where $\vb:=\left(b_n:n\in\NN\right)$ are the coefficients of $g$ for the Schauder bases $\Sset_{\adjK}=\Sset_K'$.
Hence, we conclude that
\[
\limsup_{p\to1}\abs{\langle g,s_p \rangle_{\Banach_{\adjK}^{\infty}(\Domain')}}
\leq
R^{-1}\left(C_0\right)\lim_{q\to\infty}
\norm{g}_{\Banach_{\adjK}^{q}(\Domain')}
=
R^{-1}\left(C_0\right)
\norm{g}_{\Banach_{\adjK}^{\infty}(\Domain')}.
\]

Moreover, Proposition~\ref{p:RKBS-dense}
ensures that $\Span\left\{\Kset_{\adjK}\right\}$ is dense in $\Banach_{\adjK}^{\infty}(\Domain')$.
By continuous extensions, we also have that
\[
\limsup_{p\to1}\abs{\langle g,s_p \rangle_{\Banach_{\adjK}^{\infty}(\Domain')}}
\leq
R^{-1}\left(C_0\right)
\norm{g}_{\Banach_{\adjK}^{\infty}(\Domain')},
\]
for all $g\in\Banach_{\adjK}^{\infty}(\Domain')$.
Therefore, we obtain that
\[
\limsup_{p\to1}\norm{s_p}_{\Banach_K^1(\Domain)}\leq
R^{-1}\left(C_0\right)<\infty,
\]
proving the desired result.
\end{proof}
%/////////////////////////////////////////////////////////////////////////////////////////////////////////////////////

Finally, we show that the support vector machine in $\Banach_{K}^{1}(\Domain)$ is solvable and its solution can be approximated by the solutions of the support vector machines in $\Banach_{K}^{p}(\Domain)$ as $p\to1$.

%/////////////////////////////////////////////////////////////////////////////////////////////////////////////////////
\begin{theorem}\label{t:opt-svm-s1}
Given any pairwise distinct data points $X\subseteq\Domain$ and any associated data values $Y\subseteq\RR$,
the regularized empirical risk $\svm_1$ is defined as in equation~\eqref{eq:svm-1-RKBS}.
Let $K\in\Leb_0(\Domain\times\Domain')$ be a generalized Mercer kernel such that the expansion sets $\Sset_K$ and $\Sset_K'$ of $K$ satisfy assumption (A-$1*$)
and the right-sided kernel set $\Kset_K'$ is linearly independent.
If the loss function $L$ and the regularization function $R$ satisfy assumption~(A-ELR),
then
the support vector machine
\[
\min_{f\in\Banach_{K}^1(\Domain)}\svm_1(f)
\]
has a global minimum solution $s_1$ and there exists
a countable sequence $s_{p_m}$, with $\lim_{m\to\infty}p_m=1$, of the support vector machine solutions $\left\{s_p:1<p\leq2\right\}$ defined in equation~\eqref{eq:svm-p-RKBS-opt}
such that
\[
\lim_{p_m\to1}s_{p_m}(\vx)=s_1(\vx),\quad \text{for all }\vx\in\Domain,
\]
and
\[
\lim_{p_m\to1}\svm_{p_m}\left(s_{p_m}\right)=\svm_1\left(s_1\right).
\]
\end{theorem}
%/////////////////////////////////////////////////////////////////////////////////////////////////////////////////////
\begin{proof}
Proposition~\ref{p:RKBS-MecerKer-1} ensures that $\Banach_{K}^1(\Domain)$ is isometrically equivalent to the dual space of $\Banach_{\adjK}^{\infty}(\Domain')$.
Then we employ the Banach-Alaoglu theorem (\cite[Theorem~2.6.18]{Megginson1998}) to find a countable sequence $s_{p_m}$, with $\lim_{m\to\infty}p_m=1$, of
$\left\{s_{p}:1<p\leq2\right\}$ such that it is a weakly* convergent sequence.

According to Lemma~\ref{l:sp-bound-1-norm}, there exists a countable sequence $s_{i_n}$, with $\lim_{n\to\infty}i_n=1$, of
$\left\{s_{p}:1<p\leq2\right\}$ such that the set
\[
\left\{\norm{s_{i_n}}_{\Banach_{K}^1(\Domain)}:\lim_{n\to\infty}i_n=1\right\}
\]
is bounded and
$\left\{i_n:\lim_{n\to\infty}i_n=1\right\}$ is an increasing sequence.
Thus, the Banach-Alaoglu theorem guarantees that the countable sequence $\left\{s_{i_n}:\lim_{n\to\infty}i_n=1\right\}$ is relatively weakly* compact.
This ensures that there exists weakly* convergent sequence $\left\{s_{p_m}:\lim_{m\to\infty}p_m=1\right\}$ of $\left\{s_{i_n}:\lim_{n\to\infty}i_n=1\in\NN\right\}$.
Since $\Banach_{K}^1(\Domain)$ is a Banach space, we find an element $s_1\in\Banach_{K}^1(\Domain)$ such that
\[
s_{p_m}\overset{\text{weak*}-\Banach_{\adjK}^{\infty}(\Domain')}{\longrightarrow}s_1,\quad \text{when }p_m\to1.
\]
By Lemma~\ref{l:opt-svm-s1}, we conclude that
\[
\svm_1\left(s_1\right)=\min_{f\in\Banach_{K}^1(\Domain)}\svm_1(f),
\]
and
\[
\lim_{p_m\to1}\svm_{p_m}\left(s_{p_m}\right)=\svm_1\left(s_1\right).
\]
Moreover, since $\Banach_{\adjK}^{\infty}(\Domain')$ is a two-sided RKBS, Proposition~\ref{p:weaklystar-conv-RKBS} ensures that $s_{p_m}$ is also convergent to $s_1$ pointwise as $p_m\to1$.
\end{proof}
%/////////////////////////////////////////////////////////////////////////////////////////////////////////////////////

The minimizer $s_1$ of the regularized empirical risk over $\Banach_{K}^1(\Domain)$ may not be unique because $\Banach_{K}^1(\Domain)$ is not strictly convex. If the support vector machine solution $s_1$ is unique, then we apply any convergent sequence
$\left\{s_{p_m}:\lim_{m\to\infty}p_m=1\right\}$ to approximate $s_1$ when $m\to\infty$.

%------------------------------------------------------------------------------------------------------------------------
\section{Sparse Sampling in $1$-norm Reproducing Kernel Banach Spaces}\label{s:sparse}
%------------------------------------------------------------------------------------------------------------------------
\sectionmark{Sparse Sampling in $1$-norm RKBS}

Fast numerical algorithms for sparse sampling have obtained a central achievement in the signal and image processing.
This engages people to investigate whether the support vector machines possess the sparsity (see~\cite{Steinwart2003}).
Following the discussions in Section~\ref{s:SVM-1-RKBS}, we show that the support vector machines in the $1$-norm RKBSs can be equivalently transferred into the classical $\lone$-sparse approximation.

Let $\Domain$ be a compact Hausdorff space and $K\in\Cont(\Domain\times\Domain)$ be a symmetric positive definite kernel.
Suppose that the positive eigenvalues $\Lambda_K$ and continuous eigenfunctions $\Eset_{K}$ of $K$ satisfy
assumptions (A-PDK). Then Theorem~\ref{t-RKBS-PDK} ensures that $K$ is a reproducing kernel of the $1$-norm RKBS $\Banach_K^1(\Domain)$ induced by $\Lambda_K$ and $\Eset_{K}$.
Further suppose that the eigenfunctions $\Eset_{K}$ is an orthonormal basis of $\Leb_2(\Domain)$.
Here we choose the least square loss
\[
L(\vx,y,t):=\frac{1}{2}\left(y-t\right)^2,\quad
\text{for }\vx\in\Domain,~y\in\RR,~t\in\RR,
\]
and the linear regularization function
\[
R(r):=\sigma r,\quad
\text{for }r\in[0,\infty),
\]
where $\sigma$ is a positive parameter.

Let $X:=\left\{\vx_k:k\in\NN_N\right\}$ be a set of pairwise distinct data points of $\Domain$ and $\vy_X:=\left(y_k:k\in\NN_N\right)\in\RR^N$ compose of the data values $Y:=\left\{y_k:k\in\NN_N\right\}\subseteq\RR$.
Then we construct the support vector machines in $\Banach_K^1(\Domain)$ as follows
\[
\min_{f\in\Banach_K^1(\Domain)}\svm_1(f)
=\min_{f\in\Banach_K^1(\Domain)}
\left\{\frac{1}{N}\sum_{k\in\NN_N}L(\vx_k,y_k,f(\vx_k))+R\left(\norm{f}_{\Banach_{K}^1(\Domain)}\right)\right\}.
\]

First we look at the function $f\in\Leb_2(\Domain)$ that can be represented as $f=\sum_{n\in\NN}\xi_ne_n$ with the coefficients $\vxi:=\left(\xi_n:n\in\NN\right)\in\ltwo$. Thus, we have that
\[
\begin{pmatrix}
f(\vx_1)\\
\vdots\\
f(\vx_N)
\end{pmatrix}
=
\begin{pmatrix}
\sum_{n\in\NN}\xi_ne_n(\vx_1)\\
\vdots\\
\sum_{n\in\NN}\xi_ne_n(\vx_N)
\end{pmatrix}
=
\vE_X\vxi,
\]
where
\[
\vE_X:=
\begin{pmatrix}
e_1(\vx_1)&\cdots&e_n(\vx_1)&\cdots\\
\vdots&\ddots&\vdots&\ddots\\
e_1(\vx_N)&\cdots&e_n(\vx_N)&\cdots
\end{pmatrix}.
\]

Next, we study the function $f=\sum_{n\in\NN}\xi_ne_n\in\Banach_K^1(\Domain)$.
Since
\[
\left(\sum_{n\in\NN}\abs{\xi_n}^2\right)^{1/2}
\leq\sum_{n\in\NN}\abs{\xi_n}
\leq\left(\sup_{n\in\NN}\lambda_n^{1/2}\right)\left(\sum_{n\in\NN}\frac{\abs{\xi_n}}{\lambda_n^{1/2}}\right),
\]
we conclude that $\Banach_K^1(\Domain)\subseteq\Leb_2(\Domain)$ and $\vxi\in\lone$.
This ensures that $f\in\Banach_K^1(\Domain)$ if and only if $\vD_{\lambda}^{1/2}\vxi\in\lone$, where
\[
\vD_{\lambda}^{1/2}:=
\begin{pmatrix}
\lambda_1^{1/2}&\cdots&0&\cdots\\
\vdots&\ddots&\vdots&\ddots\\
0&\cdots&\lambda_n^{1/2}&\cdots\\
\vdots&\ddots&\vdots&\ddots
\end{pmatrix}
\in\RR^{\NN\times\NN}.
\]

Therefore, for any $f:=\sum_{n\in\NN}\xi_ne_n\in\Banach_K^1(\Domain)$, we have that
\begin{equation}\label{eq:svm-sparse-1}
L\left(\vx_k,y_k,f(\vx_k)\right)=\frac{1}{2}\sum_{k\in\NN_N}\left(y_k-f(\vx_k)\right)^2
=\frac{1}{2}\norm{\vy_X-\vE_X\vxi}_2^2,
\end{equation}
and
\begin{equation}\label{eq:svm-sparse-2}
R\left(\norm{f}_{\Banach_K^1(\Domain)}\right)=
\sigma\sum_{n\in\NN}\frac{\abs{\xi_n}}{\sqrt{\lambda_n}}
=\sigma\norm{\vD_{\lambda}^{1/2}\vxi}_1.
\end{equation}
Combining equations~\eqref{eq:svm-sparse-1} and \eqref{eq:svm-sparse-2}, the support vector machine in $\Banach_K^1(\Domain)$ is equivalently to
the classical \emph{$\lone$-sparse regularization}, that is,
\begin{equation}\label{eq:svm-sparse-main}
\min_{f\in\Banach_K^1(\Domain)}\svm_1(f)
=
\min_{\vxi\in\lone}
\left\{\frac{1}{2N}\norm{\vy_X-\vE_X\vxi}_2^2+\sigma\norm{\vD_{\lambda}^{1/2}\vxi}_1\right\}.
\end{equation}

Recently, the sparse sampling inspires the numerical practical developments (see
\cite{DonohoElad2003,CandesRombergTao2006,BrucksteinDonohoElad2009}).
Now we review some fast algorithms which solve the sparse regularization~\eqref{eq:svm-sparse-main} in compressed sensing and large-scale computation.
\begin{itemize}
\item A simple strategy to attack the minimization problem~\eqref{eq:svm-sparse-main} is the \emph{iteratively reweighted least squares algorithms} (IRLS) in~\cite{Karlovitz1970,RaoDelgado1999,RaoEnganCotterPalmerDelgado2003}. The key technique of IRLS is that the $\lone$-norm is viewed as an adaptively weighted version of the squared $\ltwo$-norm. The main iterative steps of the IRLS include regularized least squares -- weight update.
\item But the IRLS may lose much of its appeal when facing the large-scale problems. An alternative method is the \emph{iterative shrinkage-thresholding algorithms} (ISTA) in~\cite{ChambolleDeVoreLeeLucier1998,FigueiredoNowak2003,DaubechiesDefriseMol2004}.
    The primary idea of ISTA is that the updated entries are evaluated by the shrinkage operator at the preliminary entries. In the optimization literature, the ISTA can be traced back to the proximal forward-backward iterative schemes similar to the splitting methods (see~\cite{BruckRonald1997,Passty1979}).
    The general iterative steps of the ISTA include back projection -- shrinkage -- line search -- solution and residual update.
    The ISTA is simple to implement and the ISTA is particularly useful in the large-scale problems, for example, the big size of $X$.
    Moreover, there are many other methods to accelerate the ISTA, for example, the two-step iterative shrinkage-thresholding algorithm (TwIST) in \cite{BioucasFigueiredo2007} and the fast iterative shrinkage-thresholding algorithm (FISTA) in \cite{BeckTeboulle2009}.
\item Recent papers~\cite{MicchelliShenXu2011,LiMicchelliShenXu2012,MicchelliShenXuZeng2013,LiShenXuZhang2015} provide a novel framework of the \emph{proximity algorithm} to study the total variation model. The main point is that a proximity operator can be constructed by the fixed-point methodology. This leads to develop efficient numerical algorithms via various fixed-point iterations. The key steps of the proximity algorithm include proximity subgradient -- fixed-point update. Paper~\cite{ParikhBoyd2014} introduces a library of implementations of proximity operators to present the efficient evaluations of the smooth and nonsmooth optimization problems including parallel and distributed algorithms. Therefore, the fixed-point iterations can accelerate computational processes of support vector machines in $p$-norm RKBSs by the proximity methods.
\end{itemize}

The connection of support vector machines and sparse sampling shows that the fast algorithms for sparse approximation can be applied in machine learning. Almost all sparse models only focus on the least-square loss. But we know that there are many other kinds of loss functions in machine learning.
Based on the constructions of the $1$-norm RKBSs, we propose a new research field of machine learning that call \emph{sparse learning methods}. In our current research proposal, we shall develop efficient numerical algorithms to learn the nonlinear structures of the big data by the sparse learning method.

%------------------------------------------------------------------------------------------------------------------------
\section{Support Vector Machines in Special Reproducing Kernel Banach Spaces}\label{s:SVM-Sp-RKBS}
%------------------------------------------------------------------------------------------------------------------------
\sectionmark{Solving Support Vector Machines in Special RKBS}

In this section, we study the support vector machines in the special $p$-norm RKBSs.
We suppose that the
expansion sets $\Sset_K$ and $\Sset_K'$ of the generalized Mercer kernel $K\in\Leb_0(\Domain\times\Domain')$ satisfy assumption (A-$1*$) which also implies assumption (A-$p$) for all $1<p<\infty$ (see inequalities (\ref{A1-to-Ap-1})-(\ref{A1-to-Ap-2})).
Let
\begin{equation}\label{eq:SVM-typ-pm}
p_m:=\frac{2m}{2m-1},\quad q_m:=2m,\quad\text{for }m\in\NN.
\end{equation}
Thus, Theorem~\ref{p:RKBS-MecerKer-pq} guarantees that
the $p_m$-norm RKBSs $\Banach_K^{p_m}(\Domain)$ induced by $\Sset_K$ are well-defined for all $m\in\NN$.
Now we show that the support vector machine solutions in $\Banach_K^{p_m}(\Domain)$ can be simplified for convenient coding and computation.
To be more precise, the support vector machine solutions in $\Banach_K^{p_m}(\Domain)$ can be represented as an linear combination of the special kernel basis.

Further suppose that the loss function $L$ and the regularization function $R$ satisfy assumption~(A-ELR). Given the pairwise distinct data points  $X\subseteq\Domain$ and the associated data values $Y\subseteq\RR$,
the regularized empirical risk
\begin{equation}\label{eq:PDK-svm}
\svm_{p_m}(f):=
\frac{1}{N}\sum_{k\in\NN_N}L(\vx_k,y_k,f(\vx_k))+R\left(\norm{f}_{\Banach_{K}^{p_m}(\Domain)}\right),
\end{equation}
is well-defined for any $f\in\Banach_{K}^{p_m}(\Domain)$.
By Theorem~\ref{t:RKBS-MercerKer-svm-rep-pq}, we solve
the unique minimizer $s_{p_m}$ of $\svm_{p_m}$ over $\Banach_{K}^{p_m}(\Domain)$.
Now we simplify the support vector machine solution $s_{p_m}$ in the following way.

%////////////////////////////////////////////////////////////////////////////////////////////////////////////////////////
\begin{theorem}\label{t:SVM-opt-rep-typ}
Given any pairwise distinct data points $X\subseteq\Domain$ and any associated data values $Y\subseteq\RR$,
the regularized empirical risk $\svm_{p_m}$ is defined as in equation~\eqref{eq:PDK-svm}.
Let $p_m$ be defined in equation~\eqref{eq:SVM-typ-pm} and let
$K\in\Leb_0(\Domain\times\Domain')$ be a generalized Mercer kernel such that the expansion sets $\Sset_K$ and $\Sset_K'$ of $K$ satisfy assumption (A-$1*$)
and the right-sided kernel set $\Kset_K'$ is linearly independent.
If the loss function $L$ and the regularization function $R$ satisfy assumption~(A-ELR),
then the unique global solution $s_{p_m}$ of the support vector machine
\[
\min_{f\in\Banach_K^{p_m}(\Domain)}\svm_{p_m}(f),
\]
has the finite dimensional representation
\[
s_{p_m}(\vx)=\sum_{k_1,\ldots,k_{2m-1}\in\NN_N}\alpha_{k_1,\ldots,k_{2m-1}}
K^{\ast(2m-1)}\left(\vx,\vx_{k_1},\cdots,\vx_{k_{2m-1}}\right), \quad\text{for } \vx\in\Domain,
\]
and the norm of $s_{p_m}$ can be written as
\begin{align*}
\norm{s_{p_m}}_{\Banach_K^{p_m}(\Domain)}
=\left(\sum_{k_0,\ldots,k_{2m-1}\in\NN_N}\beta_{k_0,\ldots,k_{2m-1}}
K^{\ast(2m-1)}\left(\vx_{k_0},\cdots,\vx_{k_{2m-1}}\right)\right)^{1-\frac{1}{2m}},
\end{align*}
where
the kernel $K^{\ast(2m-1)}$ is computed by
\begin{equation}\label{eq:svm-kernel-2m-1}
K^{\ast(2m-1)}\left(\vx,\vy_1,\cdots,\vy_{2m-1}\right)
:=\sum_{n\in\NN}\phi_n(\vx)\prod_{j\in\NN_{2m-1}}\phi_n(\vy_j),
\end{equation}
for $\vx,\vy_1,\ldots,\vy_{2m-1}\in\Domain$
and
the coefficients
\begin{align*}
\alpha_{k_1,\ldots,k_{2m-1}}:=\prod_{j\in\NN_{2m-1}}c_{k_j},
\quad
\beta_{k_0,k_1,\ldots,k_{2m-1}}:=c_{k_0}\prod_{j\in\NN_{2m-1}}c_{k_j},
\end{align*}
for $k_0,\ldots,k_{2m-1}\in\NN_N$ are computed by some suitable parameters $c_1,\ldots,c_{N}\in\RR$.
\end{theorem}
%////////////////////////////////////////////////////////////////////////////////////////////////////////////////////////
\begin{proof}
Theorem~\ref{t:RKBS-MercerKer-svm-rep-pq} ensures the support vector machine solution
\begin{equation}\label{eq:PDK-SVM-opt-rep-1}
s_{p_m}(\vx)=
\sum_{n\in\NN}\left(\sum_{j\in\NN_N}c_{j}\phi_n(\vx_j)
\abs{\sum_{k\in\NN_N}c_{k}\phi_n(\vx_k)}^{2m-2}\right)\phi_n(\vx)
\end{equation}
and its norm
\begin{equation}\label{eq:PDK-SVM-opt-rep-2}
\norm{s_{p_m}}_{\Banach_{K}^{p_m}(\Domain)}
=\left(\sum_{n\in\NN}\abs{\sum_{k\in\NN_N}c_{k}\phi_n(\vx_k)}^{2m}\right)^{\frac{2m-1}{2m}},
\end{equation}
for some suitable parameters $c_{1},\ldots,c_{N}\in\RR$.
Expanding equations~\eqref{eq:PDK-SVM-opt-rep-1} and~\eqref{eq:PDK-SVM-opt-rep-2}, we have that
\begin{align*}
s_{p_m}(\vx)=&
\sum_{n\in\NN}\sum_{k_1,\ldots,k_{2m-1}\in\NN_N}
\phi_n(\vx)\prod_{j\in\NN_{2m-1}}c_{k_j}\phi_n(\vx_{k_j})
\\
=&\sum_{k_1,\ldots,k_{2m-1}\in\NN_N}
\prod_{j\in\NN_{2m-1}}c_{k_j}\sum_{n\in\NN}
\phi_n(\vx)\prod_{l\in\NN_{2m-1}}\phi_n(\vx_{k_l}),
\end{align*}
and
\begin{align*}
\norm{s_{p_m}}_{\Banach_{K}^{p_m}(\Domain)}^{\frac{2m}{2m-1}}=&
\sum_{n\in\NN}\sum_{k_0,k_1,\ldots,k_{2m-1}\in\NN_N}
c_{k_0}\phi_n(\vx_{k_0})\prod_{j\in\NN_{2m-1}}c_{k_j}\phi_n(\vx_{k_j})
\\
=&\sum_{k_0,k_1,\ldots,k_{2m-1}\in\NN_N}
c_{k_0}\prod_{j\in\NN_{2m-1}}c_{k_j}
\sum_{n\in\NN}\phi_n(\vx_{k_0})\prod_{l\in\NN_{2m-1}}\phi_n(\vx_{k_l}).
\end{align*}
Here $\NN_{2m-1}:=\left\{1,2,\ldots,2m-1\right\}$.
The proof is complete.
\end{proof}
%////////////////////////////////////////////////////////////////////////////////////////////////////////////////////////

Here, we think that Theorem~\ref{t:SVM-opt-rep-typ} is the special case of Theorem~\ref{t:RKBS-MercerKer-svm-rep-pq}.
In particular, if the support vector machine $\min_{f\in\Banach_K^{1}(\Domain)}\svm_{1}(f)$
has a unique optimal solution $s_1$, then Theorem~\ref{t:opt-svm-s1} confirms that $s_{p_m}$ converges to $s_1$ pointwisely when $m\to\infty$.

%////////////////////////////////////////////////////////////////////////////////////////////////////////////////////////
\begin{remark}
If $K$ is totally symmetric, then
\[
K^{\ast1}(\vx,\vy)=\sum_{n\in\NN}\phi_n(\vx)\phi_n(\vy)=K(\vx,\vy),
\]
and $\Banach_K^{2}(\Domain)$
is also a RKHS with the reproducing kernel $K$. This indicates that $s_2$ is consistent with the classical support vector machine solutions in RKHSs.
Therefore, Theorem~\ref{t:SVM-opt-rep-typ} also covers the classical results of the support vector machines in RKHSs. Moreover, if the expansion set $\Kset_K$ are composed of the eigenvalues $\Lambda_K$ and the eigenfunctions $\Eset_{K}$ of the positive definite kernel $K$, then the kernel $K^{\ast(2m-1)}$ can be rewritten as
\[
K^{\ast(2m-1)}\left(\vx,\vy_1,\cdots,\vy_{2m-1}\right)
=\sum_{n\in\NN}\lambda_n^me_n(\vx)\prod_{j\in\NN_{2m-1}}e_n(\vy_j).
\]
\end{remark}
%////////////////////////////////////////////////////////////////////////////////////////////////////////////////////////

\emph{Comparisons:}
Let us look at two special examples of $m=1,2$ such that $p_1=2$ and $p_2=4/3$. Here, we suppose that $K$ is a totally symmetric generalized Mercer kernel. Then we compare the support vector machine solutions in the RKBSs $\Banach_K^{2}(\Domain)$ and $\Banach_K^{4/3}(\Domain)$.
Obviously $\Banach_K^{2}(\Domain)$ is also a Hilbert space while $\Banach_K^{4/3}(\Domain)$ is a Banach space but not a Hilbert space.
According to Theorem~\ref{t:SVM-opt-rep-typ}, the support vector machine solution $s_2$ composes of the kernel basis $\left\{K^{\ast1}\left(\cdot,\vx_k\right):k\in\NN_N\right\}=\left\{K\left(\cdot,\vx_k\right):k\in\NN_N\right\}$.
Moreover, we obtain a novel formula of support vector machine solutions $s_{4/3}$ in the RKBS $\Banach_K^{4/3}(\Domain)$, that is, the representation of $s_{4/3}$ is a liner combination of $\left\{K^{\ast3}\left(\cdot,\vx_{k_1},\vx_{k_2},\vx_{k_3}\right):k_1,k_2,k_3\in\NN_N\right\}$.
We further find that the norm of $s_2$ in $\Banach_K^{2}(\Domain)$ is obtained by the matrix
\begin{equation}\label{eq:matrix-A2}
\vA_2:=\left(K^{\ast1}\left(\vx_{k_0},\vx_{k_1}\right):k_0,k_1\in\NN_N\right)\in\RR^{N\times N},
\end{equation}
while the norm of $s_{4/3}$ in $\Banach_K^{4/3}(\Domain)$ is computed by the tensor
\begin{equation}\label{eq:matrix-A4}
\vA_4:=\left(K^{\ast3}\left(\vx_{k_0},\vx_{k_1},\vx_{k_2},\vx_{k_3}\right):k_0,k_1,k_2,k_3\in\NN_N\right)\in\RR^{N\times N\times N\times N}.
\end{equation}
It is well-known that $\vA_2$ is a positive definite matrix. Hence, the tensor $\vA_4$ would also be called positive definite here.
This shows that $s_{4/3}$ has more kernel-based information than $s_2$ even though both RKBSs $\Banach_K^{2}(\Domain)$ and $\Banach_K^{4/3}(\Domain)$ have the same reproducing kernel $K$. However, the coefficients of $s_2$ and $s_{4/3}$ are both solved by the $N$-dimensional convex optimizations.
This offers the novel learning tools to solve the practical implementations, for example, binary classifications.

Paper~\cite{Ye2014RKBS} shows that the kernel $K^{\ast(2m-1)}$ can be rewritten as the classical kernel-based form to let all $\vy_1,\ldots,\vy_{2m-1}$ mass at one point $\vw$
when the positive definite kernel $K$ is the linear transition of Gaussian functions or Mat\'ern functions.
For convenience, we suppose that $\Domain$ is a domain of $\Rd$ here.
Based on the construction of $K^{\ast(2m-1)}$ defined in equation~\eqref{eq:svm-kernel-2m-1}, the mass point $\vw$ is defined as
\[
\vw\left(\vy_1,\ldots,\vy_{2m-1}\right):=\left(w\left(y_{1,k},\ldots,y_{2m-1,k}\right):k\in\NN_d\right),
\]
where
\[
w\left(y_1,\ldots,y_{2m-1}\right):=\prod_{j\in\NN_{2m-1}}y_{j}.
\]
Roughly speaking, the map $\vw$ can be viewed as the $d$-dimensional version of the map $w$.
Thus, we study a question whether there is a kernel $K_m\in\Leb_0(\Domain\times\Domain)$ such that
\begin{equation}\label{eq:svm-typ-ker}
K^{\ast(2m-1)}\left(\vx,\vy_1,\ldots,\vy_{2m-1}\right)=K_{m}\left(\vx,\vw\left(\vy_1,\ldots,\vy_{2m-1}\right)\right).
\end{equation}
In the following examples, we answer this question.

%////////////////////////////////////////////////////////////////////////////////////////////////////////////////////////
\begin{example}\label{exa:svm-power-series-kernel}
First we look at the nonlinear factorizable power series kernel $K$ defined on $(-1,1)^d$ (see Section \ref{s:powerseries}).
According to equation~\eqref{eq:non-fact-power-kernel}, the kernel $K$ is set up by the analytic function $\eta$, that is,
\[
K(\vx,\vy)=\prod_{k\in\NN_d}\eta(x_ky_k),
\]
for $\vx:=\left(x_k:k\in\NN_d\right),
\vy:=\left(y_k:k\in\NN_d\right)\in
(-1,1)^d$. Naturally, we guess that the kernel $K_m$ could also be represented by an analytic function $\eta_m$, that is,
\begin{equation}\label{eq:svm-typ-power-1}
K_{m}(\vx,\vy)=\prod_{k\in\NN_d}\eta_{m}\left(x_ky_k\right).
\end{equation}

Now we search the possible analytic function $\eta_m$.
When the analytic function $\eta$ is rewritten as the expansion $\eta(z)=\sum_{n\in\NN_0}a_nz^n$, then equation~\eqref{eq:non-fact-power-kernel-expan-element} yields that the expansion element $\phi_{\vn}$ of $K$ can be represented as
\[
\phi_{\vn}(\vx)=\left(\prod_{k\in\NN_d}a_{n_k}^{1/2}\right)\vx^{\vn},
\]
for $\vn:=\left(n_k:k\in\NN_d\right)\in\NN_0^d$.
Therefore, we have that
\begin{align*}
K^{\ast(2m-1)}\left(\vx,\vy_1,\ldots,\vy_{2m-1}\right)
=&\sum_{\vn\in\NN_0^d}\left(\prod_{k\in\NN_d}a_{n_k}^{1/2}\right)\vx^{\vn}
\prod_{j\in\NN_{2m-1}}\left(\prod_{k\in\NN_d}a_{n_k}^{1/2}\right)\vy_j^{\vn}\\
=&
\prod_{k\in\NN_d}\left(\sum_{n_k\in\NN}a_{n_k}^mx_k^{n_k}w_k\left(\vy_1,\ldots,\vy_{2m-1}\right)^{n_k}\right).
\end{align*}
This leads us to
\[
\eta_{m}(z):=\sum_{n\in\NN_0}a_{n}^mz^n,
\]
such that
\begin{equation}\label{eq:svm-typ-power-2}
K^{\ast(2m-1)}\left(\vx,\vy_1,\ldots,\vy_{2m-1}\right)=
\prod_{k\in\NN_d}\eta_{m}\left(x_kw_k\left(\vy_1,\ldots,\vy_{2m-1}\right)\right).
\end{equation}
Combining equations~\eqref{eq:svm-typ-power-1} and~\eqref{eq:svm-typ-power-2}, we confirm equation~\eqref{eq:svm-typ-ker}.
Clearly $\eta_1=\eta$ and $\eta_{m}$ is an analytic function. This ensures that $K_m$ is also a nonlinear factorizable power series kernel.
\end{example}
%////////////////////////////////////////////////////////////////////////////////////////////////////////////////////////

%////////////////////////////////////////////////////////////////////////////////////////////////////////////////////////
\begin{example}\label{exa:svm-Gaussian-kernel}
In this example, we look at the Gaussian kernel $G_{\vtheta}$ with shape parameters $\vtheta\in\RR_{+}^d$, that is,
\[
G_{\vtheta}(\vx,\vy):=e^{-\norm{\vx-\vy}_{\vtheta}^2},
\quad\text{for }\vx,\vy\in\Rd,
\]
where
\[
\norm{\vx-\vy}_{\vtheta}:=\norm{\vTheta(\vx)-\vTheta(\vy)}_2,
\]
and
\[
\vTheta(\vz):=\left(\theta_kz_k:k\in\NN_d\right),\quad\text{for }\vz\in\Rd,
\]
(see Section \ref{s:gaussian}).
By equation~\eqref{eq:Gaussian-exp-element}, the expansion element $\tilde{\phi}_{\vtheta,\vn}$ can be written as
\[
\tilde{\phi}_{\vtheta,\vn}(\vx)
=\left(\frac{2^{\norm{\vn}_1}}{\vn!}\right)^{1/2}\vTheta(\vx)^{\vn}e^{-\norm{\vTheta(\vx)}_{2}^2},
\]
for $\vn\in\NN_0^d$. Thus, we have that
\begin{equation}\label{eq:svm-typ-Gaussain-1}
\begin{split}
&G_{\vtheta}^{\ast(2m-1)}\left(\vx,\vy_1,\ldots,\vy_{2m-1}\right)\\
=&\sum_{\vn\in\NN_0^d}\left(\frac{2^{\norm{\vn}_1}}{\vn!}\right)^{m}
\vTheta(\vx)^{\vn}e^{-\norm{\vTheta(\vx)}_2^2}\prod_{j\in\NN_{2m-1}}\vTheta(\vy_j)^{\vn}
e^{-\norm{\vTheta(\vy_j)}_{2}^2}.
\end{split}
\end{equation}
Moreover, we check that
\begin{equation}\label{eq:svm-typ-Gaussain-2}
\vw\left(\vTheta(\vy_1),\ldots,\vTheta(\vy_{2m-1})\right)^{\vn}
=\prod_{j\in\NN_{2m-1}}\vTheta(\vy_j)^{\vn},
\end{equation}
and
\begin{equation}\label{eq:svm-typ-Gaussain-3}
w\left(\gamma(\vy_1),\ldots,\gamma(\vy_{2m-1})\right)
=\prod_{j\in\NN_{2m-1}}
e^{-\norm{\vTheta(\vy_j)}_{2}^2},
\end{equation}
where
\[
\gamma(\vz):=e^{-\norm{\vTheta(\vz)}_{2}^2}.
\]

Next, we check that
\[
\sum_{n\in\NN_0}\left(\frac{2^n}{n!}\right)^{1/2}<\infty;
\]
hence the construction of the nonlinear factorizable power series kernel
\[
K(\vx,\vy):=\prod_{k\in\NN_d}\eta(x_ky_k),\quad
\text{for }\vx,\vy\in(-1,1)^d,
\]
can be well-defined by the analytic function
\[
\eta(z):=\sum_{n\in\NN_0}\frac{2^n}{n!}z^n,\quad
\text{for }z\in(-1,1),
\]
(see equation~\eqref{eq:non-fact-power-kernel}).
By the discussions in Example~\ref{exa:svm-power-series-kernel}, we obtain another
nonlinear factorizable power series kernel $K_m$ the same as equation~\eqref{eq:svm-typ-power-1}
such that $K^{\ast(2m-1)}$ and $K_m$ can be equivalently transferred by equation~\eqref{eq:svm-typ-ker}.

If the Gaussian kernel $G_{\vtheta}$ is considered in the domain $(-1,1)^d$, then $G_{\vtheta}$ can be rewritten as
\[
G_{\vtheta}(\vx,\vy)=K\left(\vTheta(\vx),\vTheta(\vy)\right)R\left(\gamma(\vx),\gamma(\vy)\right),
\]
where the kernel $R$ is the tensor product kernel, that is, $R(s,t):=st$ for $s,t\in\RR$.
Putting equations \eqref{eq:svm-typ-ker} and~(\ref{eq:svm-typ-Gaussain-2})-(\ref{eq:svm-typ-Gaussain-3}) into equation~\eqref{eq:svm-typ-Gaussain-1},
we have that
\begin{align*}
&G_{\vtheta}^{\ast(2m-1)}\left(\vx,\vy_1,\ldots,\vy_{2m-1}\right)\\
=&K^{\ast(2m-1)}\left(\vx_{\vtheta},\vTheta(\vy_1),\ldots,\vTheta(\vy_{2m-1})\right)
R^{\ast(2m-1)}\left(\gamma(\vx),\gamma(\vy_1),\ldots,\gamma(\vy_{2m-1})\right)\\
=&K_m\left(\vx_{\vtheta},\vw\left(\vTheta(\vy_1),\ldots,\vTheta(\vy_{2m-1})\right)\right)
R\left(\gamma(\vx),w\left(\gamma(\vy_1),\ldots,\gamma(\vy_{2m-1})\right)\right).
\end{align*}
Here, the kernel
$$
R^{\ast(2m-1)}\left(s,t_1,\ldots,t_{2m-1}\right):=s\prod_{j\in\NN_{2m-1}}t_j
$$
can be thought as the simplified version of equation~\eqref{eq:svm-kernel-2m-1}.
\end{example}
%////////////////////////////////////////////////////////////////////////////////////////////////////////////////////////

Example~\ref{exa:svm-power-series-kernel} illustrates that the kernel basis $K^{\ast(2m-1)}\left(\cdot,\vy_1,\ldots,\vy_{2m-1}\right)$ can be represented by the positive definite kernel $K_m(\cdot,\vw)$ centered at the mass point $\vw$ of $\vy_1,\ldots,\vy_{2m-1}$.
Sometimes, the kernel $K_m$ can not be obtained directly while we may still
introduce the equivalent transformations of equation~\eqref{eq:svm-typ-ker} similar to Example~\ref{exa:svm-Gaussian-kernel}.
However, the kernel $K_m$ will be a case by case.

Finally, we investigate the connections of \emph{Green functions} and reproducing kernels.

Duchon~\cite{Duchon1977} established the connections of Green functions and thin-plate
splines for minimizing rotation-invariant semi-norms in $\Leb_2$-based Sobolev spaces.
Papers~\cite{FasshauerYe2011Dist,FasshauerYe2011DiffBound} guarantee a reproducing kernel and its RKHS can be computed via a Green function and a (generalized) Sobolev space induced by a vector differential operators and a possible vector boundary operator consisting of finitely or countably many elements.
A special case of differential operators is a variety of linear combinations of distributional derivatives.
The theory of \cite{FasshauerYe2011Dist,FasshauerYe2011DiffBound} covers many well-known kernels such as min kernels, Gaussian kernels, polyharmonic splines, and Sobolev splines.

This indicates that the generalized Mercer kernel $K$ can be a Green function.
Moreover, we conjecture that the kernel $K^{\ast(2m-1)}$ could also be computed by the Green function.

%////////////////////////////////////////////////////////////////////////////////////////////////////////////////////////
\begin{example}\label{exa:Green}
Let us look at a simple example of Green functions that can be transferred into the kernel $K^{\ast(2m-1)}$.
Let $G_{\tau}$ defined on a circle $\TT$ be a Green function of the differential operator $\left(-\Delta\right)^{\tau}$, that is,
\[
\left(-\Delta\right)^{\tau}G_{\tau}=\delta_0,\quad\text{in }\TT,
\]
where $\Delta:=\ud^2/\ud t^2$ is a Laplace differential operator, $\tau\geq1$, and $\delta_0$ is a Dirac delta function at the origin.

For convenience, the eigenfunctions of $G_{\tau}$ could be complex-valued in this example.
In this article, we mainly focus on the real-valued functions and we \emph{only} consider the complex-valued functions in this example. But, it is not hard to extend all theoretical results discussed here to the complex-valued kernels.

Now we solve
the eigenvalues $\Lambda_{G_{\tau}}:=\left\{\lambda_{\tau,n}:n\in\ZZ\right\}$ and the eigenfunctions $\Eset_{G_{\tau}}:=\left\{e_{\tau,n}:n\in\ZZ\right\}$ of $G_{\tau}$ by
the differential equations
\[
\left(-\Delta\right)^{\tau}e_{\tau,n}=\lambda_{\tau,n}^{-1}e_{\tau,n},
\quad\text{for }n\in\ZZ.
\]
This ensures that
\[
\lambda_{\tau,n}:=\frac{1}{n^{2\tau}},\quad
e_{\tau,n}(z):=e^{\ii nz},\quad\text{for }z\in\TT\text{ and }n\in\ZZ,
\]
where $\ii:=\sqrt{-1}$ is denoted to be the imaginary unit. (More detail of the eigenvalues and the eigenfunctions of the Green functions can be found in~\cite{Duffy2001}.)
Obviously, the eigenvalues $\Lambda_{G_{\tau}}$ are positive and the eigenfunctions $\Eset_{G_{\tau}}$ is an orthonormal basis of $\Leb_2(\TT)$.
Thus, the Green kernel $K_{\tau}(x,y):=G_{\tau}(x-y)$ is a positive definite kernel. Moreover, the Green kernel $K_{\tau}$ has the absolutely and uniformly convergent representation
\[
K_{\tau}(x,y)=G_{\tau}(x-y)=\sum_{n\in\ZZ}\lambda_{\tau,n}e_{\tau,n}(x)\overline{e_{\tau,n}(y)},
\quad \text{for }x,y\in\TT.
\]
This indicates that the expansion element $\phi_{\tau,n}$ of $K_{\tau}$ is given by
\[
\phi_{\tau,n}:=\lambda_{\tau,n}^{1/2}e_{\tau,n},\quad
\text{for }n\in\NN.
\]
By expanding the modulus of equation~\eqref{eq:svm-kernel-2m-1}, we have that
\begin{equation}\label{eq:svm-typ-Green-1}
K_{\tau}^{(2m-1)}\left(x,y_1,\ldots,y_{2m-1}\right)
=\sum_{n\in\ZZ}\phi_{\tau,n}(x)\overline{\phi_{\tau,n}(y_1)}\phi_{\tau,n}(y_2)\ldots\overline{\phi_{\tau,n}(y_{2m-1})},
\end{equation}
for $x,y_1,\ldots,y_{2m-1}\in\TT$.

Let
\[
\zeta\left(y_1,\ldots,y_{2m-1}\right):=y_{1}-y_{2}+\cdots-y_{2m-2}+y_{2m-1}.
\]
Combining the tensor products of equation~\eqref{eq:svm-typ-Green-1}, we have that
\[
\begin{split}
K_{\tau}^{(2m-1)}\left(x,y_1,\ldots,y_{2m-1}\right)
=&\sum_{n\in\ZZ}\phi_{\tau,n}(x)\overline{\phi_{\tau,n}\left(\zeta\left(y_1,\ldots,y_{2m-1}\right)\right)}\\
=&G_{m\tau}\left(x-\zeta\left(y_1,\ldots,y_{2m-1}\right)\right)\\
=&K_{m\tau}\left(x,\zeta\left(y_1,\ldots,y_{2m-1}\right)\right).
\end{split}
\]
Therefore, the kernel $K_{\tau}^{(2m-1)}$ can be computed by the Green function $G_{m\tau}$.
\end{example}
%////////////////////////////////////////////////////////////////////////////////////////////////////////////////////////

This example promises that the Green functions can be used to construct the support vector machine solutions in RKBSs.

{\bf Reduction of Computational Complexity.}
Recently the investigation of the fast computational algorithms of machine learning in RKHSs became a popular research subject.
One of the techniques of the simplification of support vector machines in the RKHS $\Banach_K^2(\Domain)$ is the eigen-decomposition of the positive definite matrix $\vA_2$ defined in equation~\eqref{eq:matrix-A2}, that is, $\vA_2=\vU_2\vD_2\vU_2^T$ where $\vD_2$ and $\vU_2$ are the unitary and diagonal matrices composed of nonnegative eigenvalues and orthonormal eigenvectors of $\vA_2$, respectively.
If we write $\vD_2=\diag\left(\mu_1,\ldots,\mu_N\right)$
and $\vU_2=\left(\vv_1,\ldots,\vv_N\right)$,
then we have that
\[
\vA_2=\sum_{k\in\NN_N}\mu_k\vv_k\vv_k^T.
\]
Thus, we obtain the decompositions of the key elements of the support vector machines in the RKHS $\Banach_K^2(\Domain)$ to compute the coefficients $\vc$ of the optimal solution $s_2$, that is,
\[
\norm{s_2}_{\Banach_K^2(\Domain)}^2
=\vc^T\vA_2\vc
=\sum_{k\in\NN_N}\mu_k\left(\vv_k^T\vc\right)^2,
\]
and
\[
\left(s_2(\vx_k):k\in\NN_N\right)
=\sum_{k\in\NN_N}\mu_k\left(\vv_k^T\vc\right)\vv_k.
\]
Moreover, one of our current researches also focuses on the fast computational algorithms to solve the learning problems in RKBSs.
For example, if we can decompose the positive definite tensor $\vA_4$ by its eigenvalues and eigenvectors, that is,
\[
\vA_4=\sum_{k\in\NN_N}\mu_k\vv_k\otimes\vv_k\otimes\vv_k\otimes\vv_k,
\]
where $\otimes$ is the related tensor product (see \cite{KoldaBader2009,LathauwerMoorVandewalle2000}),
then the eigen-decomposition of $\vA_4$ shows that
\[
\norm{s_{4/3}}_{\Banach_K^{4/3}(\Domain)}^{4/3}
=\sum_{k\in\NN_N}\mu_k\left(\vv_k^T\vc\right)^4,
\]
and
\[
\left(s_{4/3}(\vx_k):k\in\NN_N\right)
=\sum_{k\in\NN_N}\mu_k\left(\vv_k^T\vc\right)^3\vv_k.
\]
Therefore, the fast computational algorithms of the support vector machines in the $p_m$-norm RKBSs can be obtained by the decomposition of the positive definite tensor. If the tensor decomposition is done, then the computational costs of the support vector machines in the RKHSs and the $p_m$-norm RKBSs are different in the polynomial time.
This indicates that the computational complexity of these fast algorithms is mainly evaluated by the computational costs of the tensor decomposition.
Now the decomposition of positive definite matrices induced by positive definite kernels
is still an open problem.
In particular, paper~\cite{FasshauerMcCourt2012} provides a fast and stable algorithm to decompose the positive definite matrices induced by the Gaussian kernels approximating by the eigenvalues and eigenfunctions of the Gaussian kernels, that is,
$\vA_2\approx\hat{\vU}_2\hat{\vD}_2\hat{\vU}_2^T$ where the approximate unitary and diagonal matrices $\hat{\vU}_2,\hat{\vD}_2$
are estimated by the eigenvalues and eigenfunctions of the Gaussian kernels.
Moreover, we have another fast algorithms for the computation of the key elements of machine learning in RKHSs such as a stable Gaussian decomposition in \cite{FornbergLarssonFlyer2011} and a multilevel circulant matrix in \cite{SongXu2010}.
At our current research project, the approximate tensor decomposition will be investigated by the expansion sets of the generalized Mercer kernels to reduce the computational complexity of the support vector machines in the $p_m$-norm RKBSs even though the computational process of the $p_m$-norm support vector machine solutions cover more information of the kernel basis.

%------------------------------------------------------------------------------------------------------------------------
%------------------------------------------------------------------------------------------------------------------------
\chapter{Concluding Remarks}
%------------------------------------------------------------------------------------------------------------------------
%------------------------------------------------------------------------------------------------------------------------

Recently, developing machine learning method in Banach spaces
became a popular research subject.
The technique point of this research is that general Banach spaces may not possess any inner product to express the orthogonality in the learning algorithms.
Papers~\cite{Zhou2003, MicchelliPontil2004,ZhangXuZhang2009,FasshauerHickernellYe2013}
began to develop the learning analysis in the special Banach spaces.
The notion of the reproducing kernel Banach space was first introduced in \cite{ZhangXuZhang2009}.
However, the theory of the machine learning in the Banach spaces are still partially understood.
Many problems remain to answer,
such as what kernels are the reproducing kernels of RKBSs, or what RKBSs have the $1$-norm structures. In this article,
we develop the general and specialized theory of the primary learning methods -- the reproducing properties in Banach spaces.

\subsection*{Contributions}

First we complete the theoretical analysis of general RKBSs in Chapter~\ref{char:RKBS}. In this article, we redefine the reproducing properties by the dual bilinear products of Banach spaces such that the RKBSs can be extended into non-reflexive Banach spaces.
This indicates that the RKBSs can be endowed with the $1$-norm structures.
The advanced properties of RKBSs can also be verified such as density, continuity, separability, implicit representation, imbedding, compactness, representer theorem for learning methods, oracle inequality, and universal approximation.

Moreover, we develop a new concept of generalized Mercer kernels to introduce the $p$-norm RKBSs for $1\leq p\leq\infty$ that
preserve the same beautiful format as the Mercer representation of RKHSs and further possess more geometrical structures than RKHSs including sparsity (see Chapter~\ref{char-GMK}).
The primary technique is that the reproducing properties of the $p$-norm RKBS can be set up by the Schauder basis and the biorthogonal system driven from some generalized Mercer kernel which becomes the reproducing kernel of this $p$-norm RKBS.
This reproducing construction ensures that the support vector machines are still well-computable in the $p$-norm RKBS.
In Chapter~\ref{char-PDK}, we show that a large class of positive definite kernels can be used to construct the $p$-norm RKBSs such as min kernels, Gaussian kernels, and power series kernels.

The theory of the $p$-norm RKBSs given here pushes the abstract learning analysis in Banach spaces to the simple implementations for the learning methods in Banach spaces.
In particular, we show that the support vector machine solutions in the $p$-norm RKBSs, which do not have any inner product, can still be represented as a linear combinations of a kernel basis, and their coefficients can be solved by a finite dimensional convex optimization problem. This shows that the learning methods in the $p$-norm RKBSs are still well-posed in practical applications, for the numerical examples of the binary classification in Figure~\ref{fig:Learning}. In particular, the numerical example shows an imaginative ability of the learning method in RKBSs at the areas of no training data.
Surprisingly, we find that the support vector machines in the $1$-norm RKBSs are strongly connected to the sparse sampling.
This leads us to developing the novel learning tools called sparse learning methods.
More details are mentioned in Chapter~\ref{char-SVM}.

\begin{figure}[h]
\begin{minipage}{0.49\textwidth}
\center
{\small Error of testing data $=$ $10.16\%$}
\includegraphics[width=\textwidth, height=0.86\textwidth]{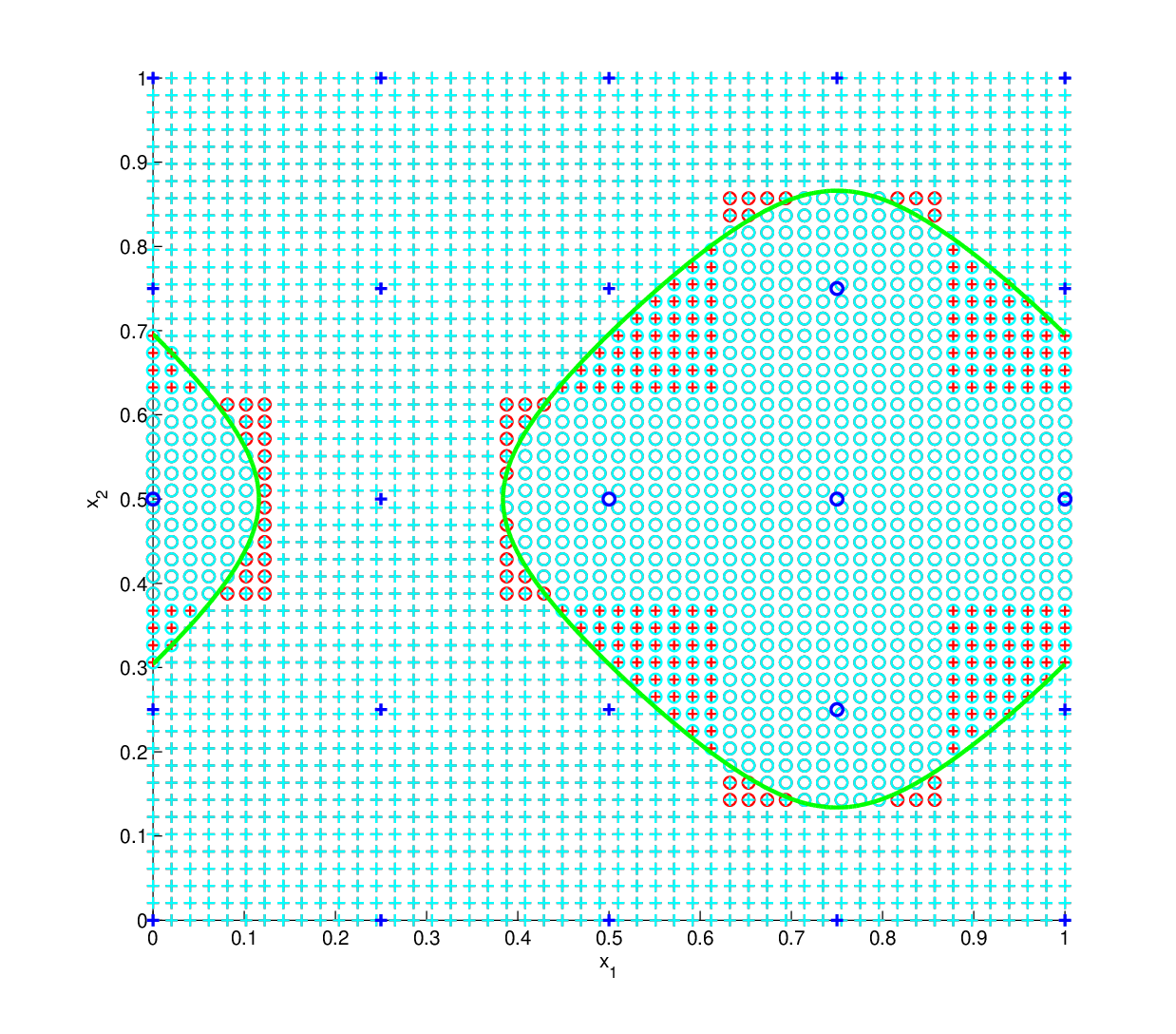}
{\small Error of testing data $=$ $14.72\%$}
\includegraphics[width=\textwidth, height=0.86\textwidth]{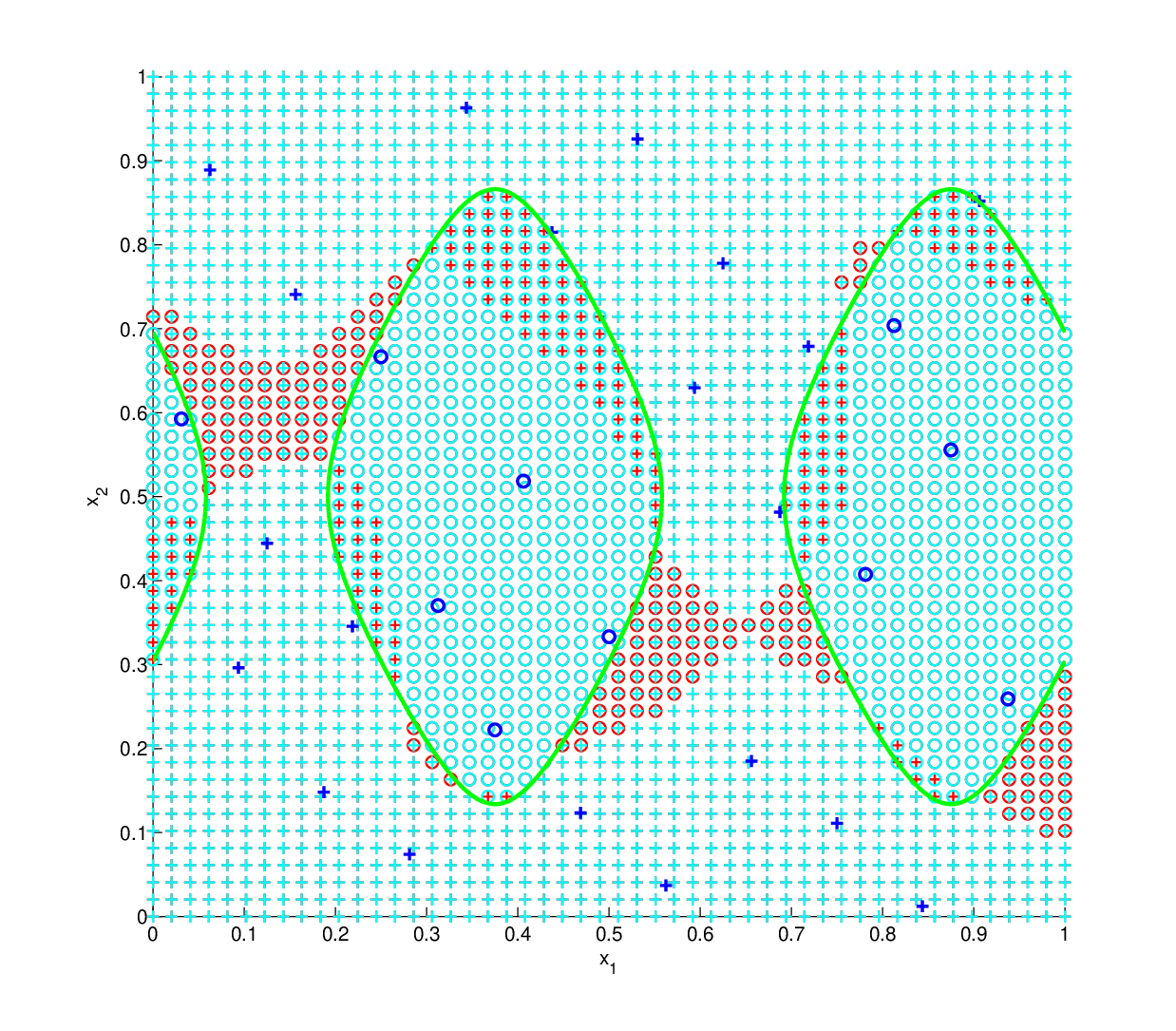}
{\small RKHS $\Banach_G^{2}(\Domain)=\Hilbert_G(\Domain)$}
\end{minipage}
\begin{minipage}{0.49\textwidth}
\center
{\small Error of testing data $=$ $9.80\%$}
\includegraphics[width=\textwidth, height=0.86\textwidth]{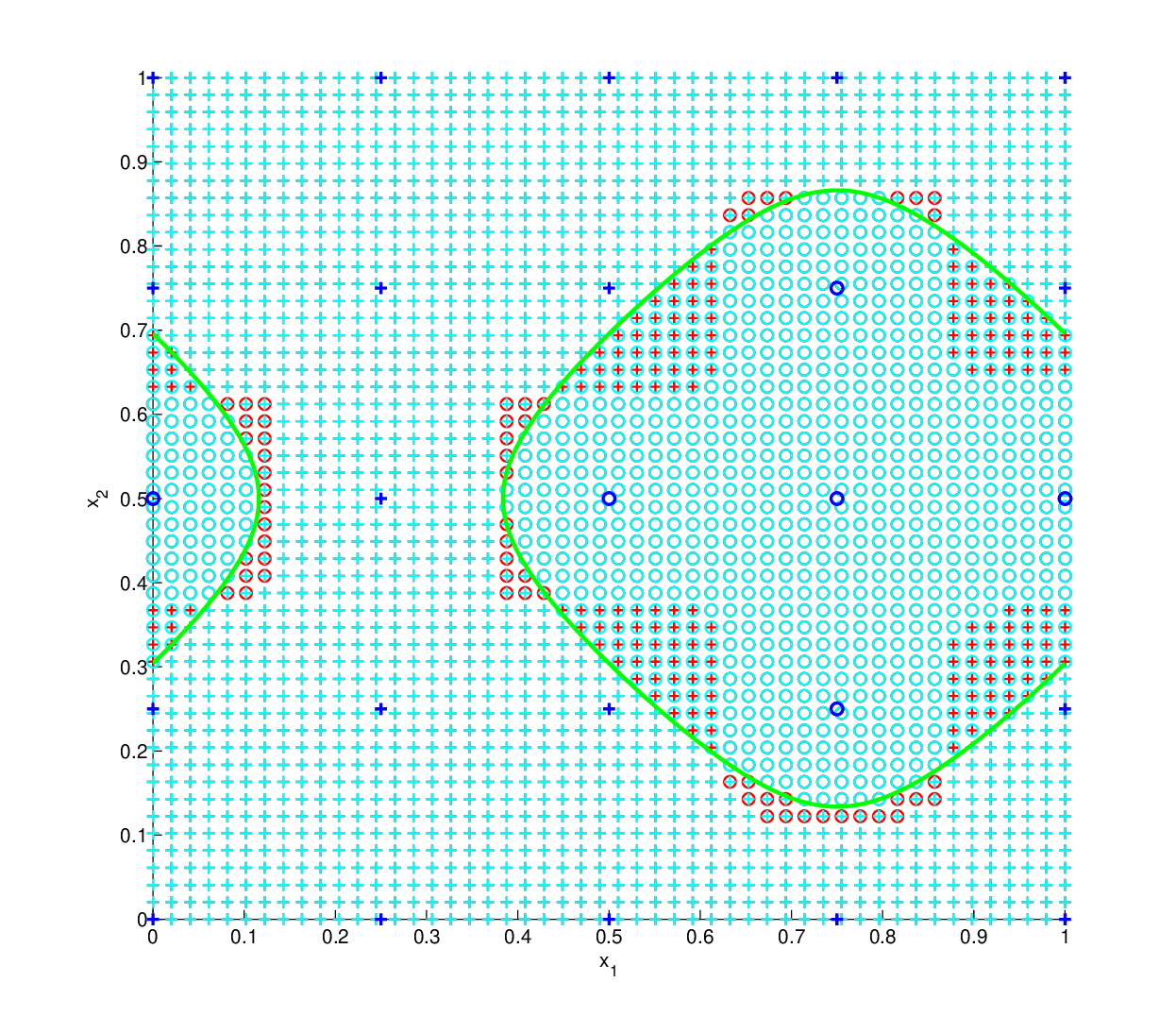}
{\small Error of testing data $=$ $12.32\%$}
\includegraphics[width=\textwidth, height=0.86\textwidth]{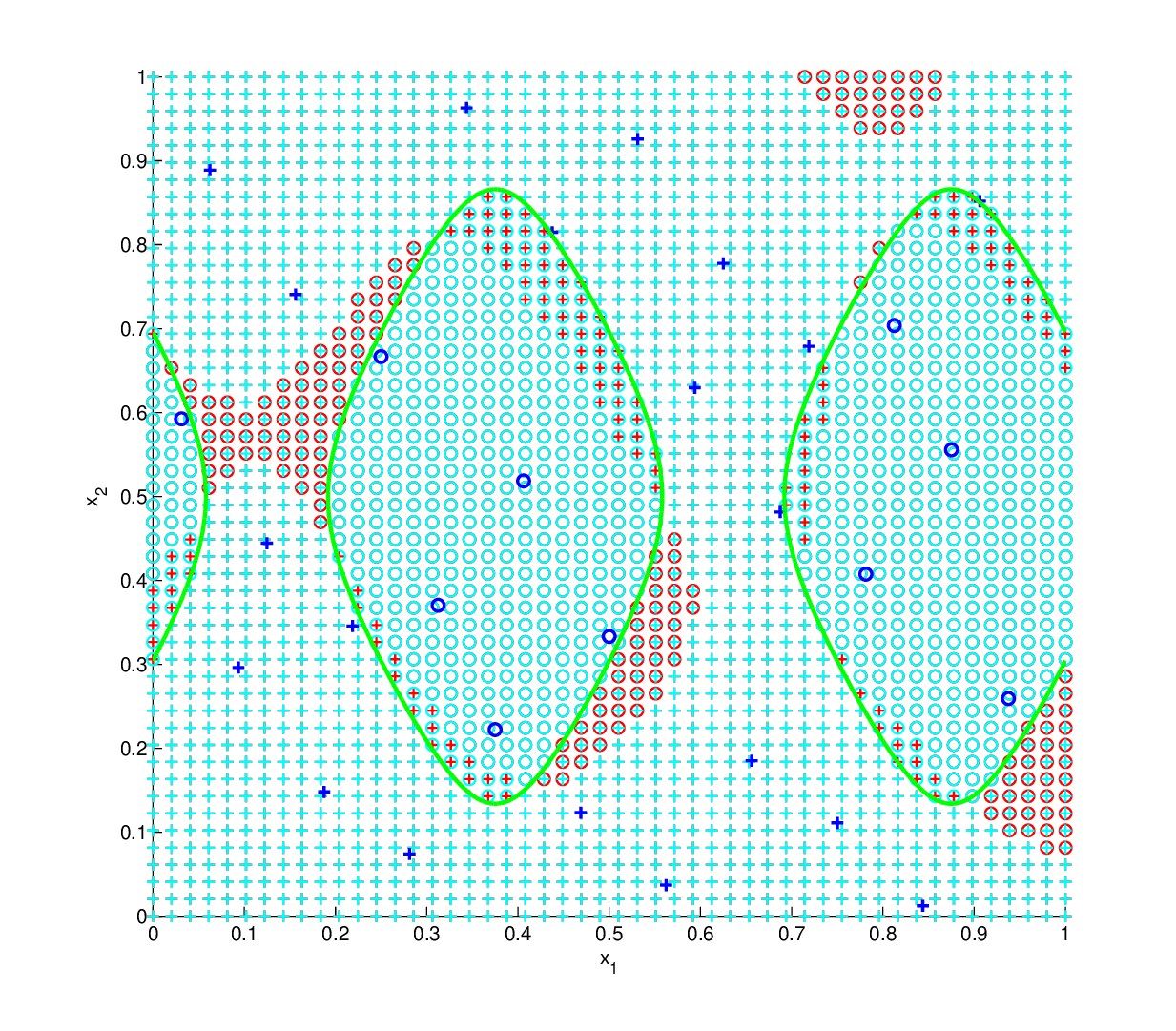}
{\small RKBS $\Banach_G^{4/3}(\Domain)$}
\end{minipage}
\caption{\small The binary classification in the domain $\Domain=[0,1]^2$ discussed in Section~\ref{s:Background-svm}:
The classes are coded as a binary variable (cross$=+1$ and circle$=-1$).
The left and right panels represent the classification results of the support vector machine solutions $s_{2}$ and $s_{4/3}$ induced by the Gaussian kernel $G$ in Example~\ref{exa:svm-Gaussian-kernel}, respectively. The boldface (blue) symbols denote the training data, the lightface (cyan) symbols denote the accurate testing data, and the mixture lightface (red and cyan) symbols denote the false testing data. The training data is the observed data to discover potentially
a predictive relationship such as the discrete data $X$ and $Y$. The testing data are observed data used to assess the strength and
utility of a predictive relationship.
The training data of the top panels are the uniform data and the training data of the bottom panels are the halton data.
All testing data are the uniform data.
The exact decision boundaries are the (green) solid lines.}\label{fig:Learning}
\end{figure}

%////////////////////////////////////////////////////////////////////////////////////////////////////////////////////////
\begin{remark}
In this article, we start from the generalized Mercer kernels to construct the $p$-norm RKBSs
such that the structures and the notations are consistent with the classical versions of RKHSs in books~\cite{Wahba1990,Wendland2005,Fasshauer2007,SteinwartChristmann2008}.
For simplifying the complexities, the expansion sets of the generalized Mercer kernels are always fixed (see Remark~\ref{r:Gen-Mercer-Ker}).
Actually, we also use the expansion sets to construct the $p$-norm RKBSs and the reproducing kernels.
Suppose that the linearly independent sets $\Sset$ and $\Sset'$ satisfy conditions (C-$p$) and (C-$1*$). Here $\Sset:=\left\{\phi_n:n\in\NN\right\}$ and $\Sset':=\left\{\adjphi_n:n\in\NN\right\}$ are thought as the left-sided and right-sided expansion sets, respectively.
By Proposition~\ref{p:MercerKer-Cp}, the generalized Mercer kernel $K(\vx,\vy)=\sum_{n\in\NN}\phi_n(\vx)\adjphi_n(\vy)$ is well-defined by the expansion sets $\Sset$ and $\Sset'$. Obviously, when we start from the expansion sets, then the generalized Mercer kernels and the $p$-norm RKBSs are uniquely determined.
\end{remark}
%////////////////////////////////////////////////////////////////////////////////////////////////////////////////////////

%-----------------------------------------------------------------------------
% Beginning of preface.tex
%-----------------------------------------------------------------------------
%
% AMS-LaTeX 1.2 sample file for a monograph, based on amsbook.cls.
% This is a data file input by chapter.tex.
%%%%%%%%%%%%%%%%%%%%%%%%%%%%%%%%%%%%%%%%%%%%%%%%%%%%%%%%%%%%%%%%%%%%%%%%

\chapter*{Acknowledgments}

This research is supported in part by the United States National Science Foundation under grant DMS-1522332, by Guangdong Provincial Government of China through the ``Computational Science Innovative Research Team'' program and by the Natural Science Foundation of China under grants 11471013 and 91530117.

\aufm{Yuesheng Xu}

\bigskip

{\color{black}{I}} would like to express {\color{black}{my}} gratitude to the Philip T. Church postdoctoral fellowship of Syracuse University
and the HKBU FRG grant of Hong Kong Baptist University for their supports of the research of machine learning in reproducing kernel Banach spaces.
The current research of kernel-based methods for machine learning is supported by the grant of the ``Thousand Talents Program'' {\color{black}{for junior scholars}} of China ({\color{black}{C83024}}), by the grant of the Natural Science Foundation of China (11601162), and by the grant of South China Normal University ({\color{black}{671082, S80835, and S81031}}).

\aufm{Qi Ye {\color{black}{(corresponding author)}}}

%-----------------------------------------------------------------------------
% End of preface.tex
%-----------------------------------------------------------------------------

%\appendix
%%    Include appendix "chapters" here.
%\include{RKHS}
%
%\include{Banach}
%
%-----------------------------------------------------------------------
% Beginning of index.tex
%-----------------------------------------------------------------------
%
%  AMS-LaTeX sample file for a monograph index, as used by AMS document
%  classes for book series.  This is a data file input by chapter.tex.
%  It was copied from a .ind file prepared for another purpose, and is
%  given here only for the purpose of showing the output format of an
%  index in this series.
%
% *** DO NOT CREATE YOUR INDEX FILE BY HAND, or use this as a model. ***
%  Put \index{...} entries in your source files.  Put \makeindex and
%  \printindex in your driver file, as shown in the template, to produce
%  a .idx file.  When the .idx file is processed by the makeindex
%  program, it will produce the .ind file with correct page numbers.
%
%%%%%%%%%%%%%%%%%%%%%%%%%%%%%%%%%%%%%%%%%%%%%%%%%%%%%%%%%%%%%%%%%%%%%%%%

\begin{theindex}

\item $\NN$ is the collection of all positive integers,
\item $\NN_0:=\NN\cup\{0\}$,
\item $\NN_N:=\left\{1,\ldots,N\right\}$ for $N\in\NN$,
\item $\NN^d:=\left\{\left(n_1,\cdots,n_d\right):n_1,\ldots,n_d\in\NN\right\}$,
\item $\ZZ$ is the collection of all integers,
\item $1/\infty:=0$,

\indexspace

\item $\RR$ is the collections of all real numbers,
\item $\RR^d$ is the $d$-dimensional real space,
\item $\CC$ is the collections of all complex numbers,
\item $\TT$ is a circle,
\item $\Domain$ is a locally compact Hausdorff space equipped with a regular Borel measure $\mu$,
\item $\supp(\mu)$ is the support of the Borel measure $\mu$,

\indexspace

\item $\lp^n$ is the normed space formed by the $n$-dimensional space $\RR^n$ with the standard norm $\norm{\cdot}_{p}$ for $1\leq p\leq\infty$,
\item $\lp$ is the collection of all countable sequences of scalars with the standard norm $\norm{\cdot}_{p}$ for $1\leq p\leq\infty$,
\item $\czero$ is the subspace of $\linfty$ with all countable sequence of scalars that converge to $0$,

\indexspace

\item $\Leb_0(\Domain)$ is the collection of all measurable functions defined on the domain $\Domain$,
\item $\Leb_p(\Domain)$ is the standard $p$-norm Lebesgue space defined on $\Domain$ for $1\leq p\leq\infty$,
\item $\Linfty(\Domain):=\left\{f\in\Leb_0(\Domain): \sup_{\vx\in\Domain}\abs{f(\vx)}<\infty\right\}$,
\item $\Cont(\Domain)$ is the collection of all continuous functions defined on the domain $\Domain$, when $\Domain$ is compact then $\Cont(\Domain)$ is endowed with uniform norm.
\item $\DualMeasure(\Domain)$ is the collection of all regular signed Borel measures on $\Domain$ endowed with their variation norms,

\indexspace

\item $\Hilbert$ is an inner-product space,
\item $(\cdot,\cdot)_{\Hilbert}$ is the inner product of $\Hilbert$,
\item $\Banach$ is a normed space,
\item $\Banach'$ is the dual space of $\Banach$,
\item $\norm{\cdot}_{\Banach}$ is the norm of $\Banach$,
\item $\langle \cdot,\cdot \rangle_{\Banach}$ is the dual bilinear product defined on $\Banach$ and $\Banach'$,
\item $\cong$ denotes that two normed spaces $\Banach_1$ and $\Banach_2$ are isometrically isomorphic, that is, $\Banach_1\cong\Banach_2$,

\indexspace

\item $\Span\left\{\Eset\right\}$ is the collection of all finite linear combinations of elements of a set $\Eset$ of a linear vector space,

\indexspace

\item $K:\Domain\times\Domain'\to\RR$ is a kernel, for example, reproducing kernel, generalized Mercer kernel and positive definite kernel,
\item $\adjK$ is the adjoint kernel of $K$, that is, $\adjK(\vy,\vx)=K(\vx,\vy)$,
\item $K^{\ast(2m-1)}$ is defined in equation~\eqref{eq:svm-kernel-2m-1},

\indexspace

\item $\Kset_{K}:=\left\{K(\cdot,\vy): \vy\in\Domain'\right\}$ is the left-sided kernel set of the reproducing kernel $K$,
\item $\Kset_{K}':=\left\{K(\vx,\cdot): \vx\in\Domain\right\}$ is the right-sided kernel set of the reproducing kernel $K$,
\item $\Sset_{K}:=\left\{\phi_n:n\in\NN\right\}$ is the left-sided expansion set of the generalized Mercer kernel $K$,
\item $\Sset_{K}':=\left\{\adjphi_n:n\in\NN\right\}$ is the right-sided expansion set of the generalized Mercer kernel $K$,
\item $\Aset_{K}:=\left\{\left(\phi_n(\vx):n\in\NN\right):\vx\in\Domain\right\}$ is the left-sided sequence set of the generalized Mercer kernel $K$,
\item $\Aset_{K}':=\left\{\left(\adjphi_n(\vy):n\in\NN\right):\vy\in\Domain'\right\}$ is the right-sided sequence set of the generalized Mercer kernel $K$,
\item $\Lambda_K:=\left\{\lambda_n: n\in\NN\right\}$ is the eigenvalues of the positive definite kernel $K$,
\item $\Eset_K:=\left\{e_n:n\in\NN\right\}$ is the eigenfunctions of the positive definite kernel $K$,

\indexspace

\item $\Banach_K^p(\Domain)$ is the $p$-norm space induced by $\Sset_{K}$ for $1\leq p\leq\infty$,
\item $\Banach_{\adjK}^q(\Domain')$ is the $q$-norm space induced by $\Sset_{K}'=\Sset_{\adjK}$ for $1\leq q\leq\infty$,

\indexspace

\item $L$ is a loss function,
\item $R$ is a regularization function,
\item $R_r$ is $\frac{\ud}{\ud r}R$,
\item $H_t$ is $\frac{\ud}{\ud t}H$,
\item $\risk$ is a $L$-risk,
\item $\svm$ is a regularized empirical risk over RKBSs,
\item $\widetilde{\svm}$ is a regularized infinite-sample risk over RKBSs,
\item $\svm_p$ is a regularized empirical risk over $\Banach_K^p(\Domain)$,

\indexspace

\item $\normopt$ is the norm operator of $\Banach$, that is, $\normopt(f):=\norm{f}_{\Banach}$ for $f\in\Banach$,
\item $\Gateaux$ is the G\^{a}teaux derivative,
\item $\Frechet$ is the Fr\'{e}chet derivative,
\item $\Gateaux\normopt$ is the standard G\^{a}teaux derivative of the normed operator $\normopt$,
\item $\GateauxNorm$ is the redefined G\^{a}teaux derivative of the norm of $\Banach$, that is, $\GateauxNorm(f)=\Gateaux\normopt(f)$ when $f\neq0$, and $\GateauxNorm(f)=0$ when $f=0$,

\indexspace

\item $\delta_{\vx}$ is a point evaluation functional at $\vx$, that is, $\delta_{\vx}(f):=f(\vx)$,
\item $I_K$ is an left-sided integral operator of the kernel $K$, that is, $I_K(\zeta):=\int_{\Domain}K(\vx,\cdot)\zeta(\vx)\mu(\ud\vx)$,
\item $I_K'$ is an right-sided integral operator of the kernel $K$, that is, $I_K'(\xi):=\int_{\Domain'}K(\cdot,\vy)\xi(\vy)\mu'(\ud\vy)$,
\item $I_K^{\ast}$ is an adjoint operator of the integral operator $I_K$ for the space of continous functions, that is, $\langle f,I_K^{\ast}\nu \rangle_{\Cont(\Domain)}=\langle I_Kf,\nu \rangle_{\Cont(\Domain')}$,
\item $I$ is the identity map,
\item $\Phi$ is the map $\vx\mapsto\norm{K(\vx,\cdot)}_{\Banach'}$, that is, $\Phi(\vx):=\norm{K(\vx,\cdot)}_{\Banach'}$,
\item $\Phi'$ is the map $\vy\mapsto\norm{K(\cdot,\vy)}_{\Banach}$, that is, $\Phi'(\vy):=\norm{K(\cdot,\vy)}_{\Banach}$,
\item $\Phi_q$ is a left-sided upper-bound function of the generalized Mercer kernel $K$, that is, $\Phi_{\infty}(\vx):=\sup_{n\in\NN}\abs{\phi_n(\vx)}$ or $\Phi_q(\vx):=\sum_{n\in\NN}\abs{\phi_n(\vx)}^q$ when $1\leq q<\infty$,
\item $\Phi_p'$ is a right-sided upper-bound function of the generalized Mercer kernel $K$, that is, $\Phi_{\infty}'(\vx):=\sup_{n\in\NN}\abs{\phi_n'(\vx)}$ or $\Phi_p'(\vy):=\sum_{n\in\NN}\abs{\adjphi_n(\vy)}^p$ when $1\leq p<\infty$,

\indexspace

\item $w$ is the map massing $y_1,\ldots,y_{2m-1}\in\RR$ by the tensor product, that is, $w\left(y_1,\cdots,y_{2m-1}\right):=\prod_{j\in\NN_{2m-1}}y_{j}$.
\item $\zeta$ is the map massing $y_1,\ldots,y_{2m-1}\in\RR$ by the sum, that is, $\zeta\left(y_1,\cdots,y_{2m-1}\right):=\sum_{j\in\NN_{2m-1}}(-1)^{j+1}y_{j}$
\item $\vw$ is the map massing $\vy_1,\ldots,\vy_{2m-1}\in\Rd$ by the tensor product, that is,
$\vw\left(\vy_1,\cdots,\vy_{2m-1}\right):=\left(w\left(y_{1,k},\cdots,y_{2m-1,k}\right):k\in\NN_d\right)$.

\indexspace

\item The linear independence of $\Eset$: $\sum_{k\in\NN_N}c_k\phi_k=0$ if and only if $c_1=\ldots=c_N=0$, for any $N\in\NN$ and any finite pairwise distinct elements $\phi_1,\ldots,\phi_N\in\Eset$,

\indexspace

\item Conditions (C-$p$) of $\Sset_{K}$ and $\Sset_{K}'$: $\sum_{n\in\NN}\abs{\phi_n(\vx)}^q,\sum_{n\in\NN}\abs{\adjphi_n(\vy)}^p<\infty$ for all $\vx\in\Domain$ and all $\vy\in\Domain'$ (Here $1<p,q<\infty$ and $p^{-1}+q^{-1}=1$),
\item Conditions (C-$1*$) of $\Sset_{K}$ and $\Sset_{K}'$: $\sum_{n\in\NN}\abs{\phi_n(\vx)},\sum_{n\in\NN}\abs{\adjphi_n(\vy)}<\infty$ for all $\vx\in\Domain$ and all $\vy\in\Domain'$,
\item Conditions (C-GEN) of $\Eset$ and $\Eset'$: $\sum_{n\in\NN}\abs{\phi_n(\vx)},\sum_{n\in\NN}\abs{\adjphi_n(\vy)}<\infty$ for all $\vx\in\Domain$ and all $\vy\in\Domain'$,
\item Conditions (C-PDK) of $\Lambda_K$ and $\Eset_{K}$: $\sum_{n\in\NN}\lambda_n^{1/2}\abs{e_n(\vx)}<\infty$ for all $\vx\in\Domain$,

\indexspace

\item Assumptions (A-$p$) of $K$: $\Sset_{K}$ and $\Sset_{K}'$ are linearly independent and satisfy conditions (C-$p$).
\item Assumptions (A-$1*$) of $K$: $\Sset_{K}$ and $\Sset_{K}'$ are linear independent and satisfy conditions (C-$1*$).
\item Assumptions (A-PDK) of $K$: $\Lambda_K$ and $\Eset_{K}$ satisfy conditions (C-PDK).

\indexspace

\item Assumptions (A-ELR) of $L$ and $R$: $R$ is convex and strictly increasing, and $L$ satisfies that $L(\vx,y,\cdot)$ is convex for any fixed $\vx\in\Domain$ and any fixed $y\in\RR$.
\item Assumptions (A-GLR) of $L$ and $R$: $R$ is convex, strictly increasing, and continuously differentiable, and $L$ satisfies that $L(\vx,y,\cdot)$ is convex for any fixed $\vx\in\Domain$ and any fixed $y\in\RR$, and the function $H$ driven by $L$ (see equation~\eqref{eq:H-Fun}) satisfies that $t\mapsto H(\vx,t)$ is differentiable for any fixed $\vx\in\Domain$ and $\vx\mapsto H(\vx,f(\vx))\in\Leb_1(\Domain),~\vx\mapsto H_t(\vx,f(\vx))\in\Leb_{1}(\Domain)$ whenever $f\in\Leb_{\infty}(\Domain)$.

%\subitem
%\subitem

\end{theindex}

%-----------------------------------------------------------------------
% End of index.tex
%-----------------------------------------------------------------------

\backmatter
%    Bibliography styles amsplain or harvard are also acceptable.
\bibliographystyle{amsalpha}
\bibliography{RKBSRef08}
%    See note above about multiple indexes.
\printindex

\end{document}